\documentclass[a4paper,10pt]{amsart}

\usepackage[utf8]{inputenc}
\usepackage[english,ngerman]{babel}
\usepackage{array,tikz,listings,color,graphicx,csquotes,libertine,datetime,enumitem,amssymb}
\usepackage[hypertexnames=false]{hyperref}
\usetikzlibrary{shapes.misc} 
\usepackage{pgfplots} \pgfplotsset{compat=1.14}
\usepackage{wrapfig,tikzscale}

\usepackage{colortbl}

\numberwithin{equation}{section}
\numberwithin{table}{section}
\numberwithin{figure}{section}
\makeatletter
\@addtoreset{section}{part} 
\@addtoreset{equation}{part}
\@addtoreset{figure}{part}
\@addtoreset{table}{part}
\@addtoreset{footnote}{section}
\makeatother
\newcommand\invisiblepart[1]{%
  \refstepcounter{part}%
  \addcontentsline{toc}{part}{#1}%
}
\makeatletter
\let\@cleartopmattertags\relax
\makeatother

\newif\ifen
\newcommand{\en}[1]{\ifen#1\fi} 
\newcommand{\de}[1]{\ifen\else#1\fi} 

\newtheorem{defi}{\en{Definition}\de{Definition}}[section]
\newtheorem{lem}[defi]{\en{Lemma}\de{Lemma}}
\newtheorem{thm}[defi]{\en{Proposition}\de{Satz}}

\newtheorem{bem}[defi]{\en{Remark}\de{Bemerkung}}
\usepackage{mdframed}
 	\definecolor{lightergray}{rgb}{0.9, 0.9, 0.9}
\newmdtheoremenv[innertopmargin=0pt, backgroundcolor=lightergray, linecolor=gray, linewidth=1pt, topline=false, bottomline=false]{theo}[defi]{\en{Theorem}\de{Theorem}}

\usepackage{fancyhdr}
\makeatletter
  \def\section{\@startsection{section}{1}%
    \z@{.7\linespacing\@plus\linespacing}{.6\linespacing}%
    {\Large\normalfont\scshape\bfseries\centering}}
\makeatother

\newcommand{\ko}{\en{.}\de{{,}}} 
\newcommand{\im}{\operatorname{Im}} 
\newcommand{\dsf}[2]{#1/#2} 
\DeclareMathOperator{\DIV}{DIV} 
\newcommand{\cas}[2]{\begin{cases}#1 &\text{ \en{if}\de{wenn} } m \text{ \en{is even}\de{gerade ist}}\\%
                     #2 &\text{ \en{if}\de{wenn} } m \text{ \en{is odd}\de{ungerade ist}}\end{cases}}
\newcommand{\kommentInh}[2]{{\noindent\textbf{\ref{\en{EN}#1}.~~~\nameref{\en{EN}#1}~~\dotfill~~\pageref{\en{EN}#1}}\par\vspace*{1.5pt}
                            \noindent\narrower\narrower{\itshape #2 }\par\vspace*{7pt}}}
\newcommand{\dl}[1]{\delta\mathopen{}\left(#1\right)\mathclose{}}
\def\lk{\mathopen{}\left(}
\def\rk{\right)\mathclose{}}

\begin{document}

    \title{\Large{A detailed proof of the Chudnovsky formula\\with means of basic complex analysis}\\[2ex]
        \hrule~\\[2ex]
        \Large{Ein ausführlicher Beweis der Chudnovsky-Formel\\mit elementarer Funktionentheorie}}
    \author{Lorenz Milla, 03/2021}

  \entrue
  \invisiblepart{English version}
  \selectlanguage{ngerman}

\maketitle

\thispagestyle{empty}
\vspace*{2.25cm}
$$\frac{1}{\pi} = 12\cdot\sum_{n=0}^\infty \frac{(-1)^n(6n)!}{(3n)!(n!)^3}\cdot \frac{13591409 + 545140134 n}{ 640320^{3n+3/2}}$$
\vspace*{2.25cm}
\label{\en{EN}titlepage}

{\narrower\narrower\narrower{
\noindent{\scshape Abstract.} In this paper we give another proof of the Chudnovsky formula for calculating $\pi$ -- a proof in detail with means of basic complex analysis.

With the exception of the tenth chapter, the proof is self-contained, with proofs provided for all the advanced theorems we use (e.g.~for the Clausen formula and for the Picard-Fuchs differential equation).

\vspace*{2mm}
{\begin{center}\emph{English version: pp.~\pageref{ENtitlepage}--\pageref{ENliterat}}\end{center}}

\vspace*{1.5cm}

\noindent{\scshape Zusammenfassung.} In diesem Aufsatz wird die Chudnovsky-Formel zur Berechnung von $\pi$ erneut bewiesen -- wesentlich ausführlicher, mit elementaren Methoden der Funktionentheorie und der Analysis.

Die benötigten fortgeschrittenen Sätze (z.B.~die Clausen-Formel und die Picard-Fuchs-Differentialgleichung) werden ihrerseits ausführlich bewiesen. Nur im zehnten Kapitel verweisen wir auf externe Quellen.

\vspace*{2mm}
{\begin{center}\emph{Deutsche Version: S.~\pageref{titlepage}--\pageref{literat}}\end{center}}

~
}}

\vfill\pagebreak
\en{\selectlanguage{english}}\de{\selectlanguage{ngerman}}

\fancyhead[OL]{\emph{\thesection.~\leftmark}}

\vfill\pagebreak\section*{\en{Introduction}\de{Einleitung}}\label{\en{EN}kapEinleitung}
\renewcommand{\leftmark}{\en{Introduction}\de{Einleitung}}
\en{The Chudnovsky formula for calculating $\pi$ reads}%
\de{Die Chudnovsky-Formel zur Berechnung von $\pi$ lautet}
$$\frac{1}{\pi} = 12\cdot\sum_{n=0}^\infty \frac{(-1)^n(6n)!}{(3n)!(n!)^3}\cdot \frac{13591409 + 545140134 n}{ 640320^{3n+3/2}}.$$

\en{It is particularly \emph{efficient}, because it yields on average $14\ko1816$ decimal digits of $\pi$ per iteration (see Thm.~\ref{\en{EN}14dezi}). That's why it is being used in most world record computations since 1989\footnote{The Chudnovsky algorithm develops its full speed only when the summation is done with "binary splitting" and a fast multiplication like Schönhage Strassen is implemented. Under these conditions, the time for computing $n$ digits of $\pi$ with the Chudnovsky algorithm is $O(M(n)\log^2(n))$, where $M(n)=O(n\log(n)\log(\log(n)))$ is the time for $n$-digit multiplication.}. For an overview about these world records see \cite{\en{EN}agar}, \cite{\en{EN}ycru} or Fig.~\ref{\en{EN}fighist}.}%
\de{Sie ist besonders \emph{effizient}, weil sie pro Summand durchschnittlich $14\ko1816$ weitere Dezimalen liefert (siehe dazu auch Thm.~\ref{\en{EN}14dezi}). Deshalb wird sie seit 1989 für die meisten $\pi$-Berechnungs-Weltrekorde verwendet\footnote{Die volle Geschwindigkeit entwickelt der Chudnovsky-Algorithmus erst, wenn man für die Summation \glqq binary splitting\grqq~verwendet und eine schnelle Multiplikation wie Schönhage-Strassen. Dann braucht man zur Berechnung von $n$ Dezimalen nur $O(M(n)\log^2(n))$ Zeit, wobei $M(n)=O(n\log(n)\log(\log(n)))$ die Zeit für eine $n$-stellige Multiplikation ist.}. Für einen Überblick über diese Weltrekorde siehe \cite{\en{EN}agar}, \cite{\en{EN}ycru} oder Abb.~\ref{\en{EN}fighist}.}

\en{Moreover, it is particularly \emph{beautiful}, because its coefficients are such huge natural numbers.
If you only want to know where these coefficients come from, you can read \cite{\en{EN}Milla2021}, which is much shorter (7 pages) but also more advanced.}%
\de{Außerdem ist sie besonders \emph{schön}, weil die Koeffizienten riesige natürliche Zahlen sind. Wer direkt wissen will, woher die Koeffizienten kommen, kann auch \cite{\en{EN}Milla2021} lesen, was viel kürzer ist (7 Seiten) aber natürlich mehr Kenntnisse voraussetzt.}%

\begin{figure}[h!]
\centering
\pgfkeys{/pgf/number format/.cd,1000 sep={}}
\pgfplotsset{every axis legend/.style={
cells={anchor=center},
inner xsep=3pt,inner ysep=2pt,nodes={inner sep=2pt,text depth=0.15em},
anchor=north west,
shape=rectangle,
fill=none,
at={(rel axis cs:0.4,0.3)}
}
}
\begin{tikzpicture}
\begin{semilogyaxis}[scale only axis=true,
  width=10cm, height=6.1cm,
  legend cell align={left},
  minor x tick num=4,
  minor xtick={1989,1991,1992,1993,1994,1995,1996,1997,1998,1999,
                    2001,2002,2003,2004,2005,2006,2007,2008,2009,
                    2011,2012,2013,2014,2015,2016,2017,2018,2019,2021},
  xticklabels={,,1990,,2000,,2010,,2020,},
  grid=major, 
    major grid style=gray!50,
  ytick={1e8,1e9,1e10,1e11,1e12,1e13,1e14},
  xmin=1988.5,
  xmax=2022,
  yticklabel pos=left,
legend pos = south east]
\addlegendimage{empty legend}\addlegendentry{\hspace{-.2cm}\textbf{\en{Digits of}\de{Stellen von} \boldmath{$\pi$}}, \en{computed with:}\de{berechnet mit:}}
\addplot[only marks, mark=*, mark size=1.7pt, red!60!black] coordinates {
(1989.416667, 480000000) 
(1989.500000, 535339270) 
(1989.666667, 1011196691) 
(1991.666667, 2260000000) 
(1994.416667, 4044000000) 
(1996.250000, 8000000000) 
(2010.000000, 2699999990000) 
(2010.666667, 5000000000000) 
(2011.833333, 10000000000050) 
(2014.000000, 12100000000050) 
(2014.833333, 13300000000000) 
(2016.916667, 22459157718361) 
(2019.250000, 31415926535897) 
(2020.083333, 50000000000000) 
};
\draw [fill=white,white] (2005.1,1.1e8) rectangle (2021,0.95e10);
\addplot[only marks, mark=diamond*, blue] coordinates { 
(1989.583333, 536870898) 
(1989.916667, 1073740799) 
(1995.666667, 4294960000) 
(1995.833333, 6442450000) 
(1997.333333, 17179869142) 
(1997.583333, 51539600000) 
(1999.333333, 68719470000) 
(1999.750000, 206158430000) 
(2009.333333, 2576980377524) 
};

\addplot[only marks, mark=triangle*, green!40!black] coordinates {
(2002.916667, 1241100000000) 
};
\addplot[solid]  coordinates { 
(1989.0, 201326551)
(1989.416667, 201326551) 
(1989.416667, 480000000) 
(1989.500000, 480000000) 
(1989.500000, 535339270) 
(1989.583333, 535339270) 
(1989.583333, 536870898) 
(1989.666667, 536870898) 
(1989.666667, 1011196691)
(1989.916667, 1011196691) 
(1989.916667, 1073740799) 
(1991.666667, 1073740799)
(1991.666667, 2260000000) 
(1994.416667, 2260000000)
(1994.416667, 4044000000) 
(1995.666667, 4044000000)
(1995.666667, 4294960000) 
(1995.833333, 4294960000)
(1995.833333, 6442450000) 
(1996.250000, 6442450000) 
(1996.250000, 8000000000) 
(1997.333333, 8000000000)
(1997.333333, 17179869142)
(1997.583333, 17179869142)
(1997.583333, 51539600000) 
(1999.333333, 51539600000) 
(1999.333333, 68719470000) 
(1999.750000, 68719470000) 
(1999.750000, 206158430000) 
(2002.916667, 206158430000) 
(2002.916667, 1241100000000) 
(2009.333333, 1241100000000)
(2009.333333, 2576980377524) 
(2010.000000, 2576980377524) 
(2010.000000, 2699999990000) 
(2010.666667, 2699999990000) 
(2010.666667, 5000000000000) 
(2011.833333, 5000000000000) 
(2011.833333, 10000000000050)
(2014.000000, 10000000000050) 
(2014.00000, 12100000000050) 
(2014.833333, 12100000000050)
(2014.833333, 13300000000000)
(2016.916667, 13300000000000) 
(2016.916667, 22459157718361) 
(2019.250000, 22459157718361)
(2019.250000, 31415926535897)
(2020.083333, 31415926535897)
(2020.083333, 50000000000000) 
(2021.2, 50000000000000) }; 


\addlegendentry{~\en{Chudnovsky formula}\de{Chudnovsky-Formel} (14$\small{\times}$)};
\addlegendentry{~Gauß-Legendre\en{ alg.}\de{-Alg.} (9$\small{\times}$)};
\addlegendentry{~Machin-\en{type series}\de{artigen Reihen} (1$\small{\times}$)};
\end{semilogyaxis}
\end{tikzpicture}\vspace*{-8pt}
\caption{\en{Number of known digits of $\pi$ since 1989}\de{Anzahl bekannter Stellen von $\pi$ seit 1989}}
\label{\en{EN}fighist}
\end{figure}

\en{In this paper, we prove the Chudnovsky formula in detail. We only need a basic knowledge of complex analysis and of analysis -- for example the ratio test, Leibniz's rule, Laurent series, the residue theorem and the Picard-Lindelöf theorem.}%
\de{In diesem Aufsatz beweisen wir die Chudnovsky-Formel ausführlich. Dabei werden nur grundlegende Kenntnisse in Funktionentheorie und Analysis vorausgesetzt, z.B.~Quotientenkriterium, Leibnizregel, Laurentreihen, Residuensatz und der Satz von Picard-Lin\-delöf.}

\medskip

\en{Using the normalized Eisenstein series $E_2$, $E_4$ and $E_6$}%
\de{Mit Hilfe der normierten Eisensteinreihen $E_2$, $E_4$ und $E_6$}
\begin{align*}
    E_2(\tau) &:= 1- 24 \sum_{n=1}^\infty n \frac{q^n}{1-q^n}
\qquad\text{\en{where}\de{mit} }q:=e^{2\pi i \tau}\text{ \en{and}\de{und} }\im(\tau)>0,\\
    E_4(\tau) &:= 1+ 240 \sum_{n=1}^\infty n^3 \frac{q^n}{1-q^n}\\
    \text{\en{and}\de{und}}\qquad E_6(\tau) &:= 1- 504 \sum_{n=1}^\infty n^5 \frac{q^n}{1-q^n}
\end{align*}
\en{we define these two functions:}%
\de{definieren wir diese zwei Funktionen:}
\begin{align*}
    J(\tau)   &:= \frac{E_4(\tau)^3}{E_4(\tau)^3-E_6(\tau)^2}\\
    \text{\en{and}\de{und}}\qquad s_2(\tau) &:= \frac{E_4(\tau)}{E_6(\tau)}\cdot\left(E_2(\tau)-\frac{3}{\pi \im(\tau)}\right).
\end{align*}

\en{In the first nine chapters, we develop all terms and propositions needed for our complete and self-contained proof of the following theorem:}%
\de{In den ersten neun Kapiteln entwickeln wir alle Begriffe und beweisen alle Sätze, die wir für den Beweis des folgenden Theorems  benötigen:}

\vspace*{2mm}
\begin{theo}[\en{Main Theorem}\de{Haupttheorem}~\ref{\en{EN}hauptformel}]\label{\en{EN}theo01}
\en{For all $\tau$ with $\im(\tau)>1\ko25$ we have the following identity due to David and Gregory Chudnovsky, first published in 1988 \cite[eq.~(1.4)]{\en{EN}chud1988}:}%
\de{Für alle $\tau$ mit $\im(\tau)>1\ko25$ gilt die folgende Gleichung von David und Gregory Chudnovsky aus dem Jahr 1988 \cite[Glg.~(1.4)]{\en{EN}chud1988}:}
\begin{align*}
    \frac{1}{2\pi \im(\tau)}\sqrt{\frac{J(\tau)}{J(\tau)-1}} &= \sum_{n=0}^\infty \left( \frac{1-s_2(\tau)}{6} + n \right)\cdot \frac{(6n)!}{(3n)!(n!)^3}\cdot \frac{1}{\left(1728J(\tau)\right)^n}
\end{align*}
\en{Here $\sqrt{\phantom{J}}$ denotes the principal branch of the square root.}%
\de{Hierbei bezeichnet $\sqrt{\phantom{J}}$ den Hauptzweig der Quadratwurzel.}
\end{theo}
\vspace*{3mm}

\en{The Chudnovsky formula is a special case of this identity, which we obtain by using $\tau=\tau_{163}=\frac{1+i\sqrt{163}}{2}$. There, it holds}%
\de{Die Chudnovsky-Formel erhalten wir als Spezialfall dieser Gleichung, wenn wir für $\tau$ den Wert $\tau_{163}=\frac{1+i\sqrt{163}}{2}$ einsetzen. Dort ist nämlich}
\begin{align*}
    1728 J(\tau_{163})   &= -640320^3\\[1ex]
    \text{\en{and}\de{und}}\qquad\frac{1-s_2(\tau_{163})}{6}&=\frac{13591409}{545140134}.
\end{align*}

\en{In Ch.~\ref{\en{EN}kapFormel}, we explicitly calculate these values and prove\footnote{This proof is the only part of this paper that requires more than basic complex analysis. That's why we have to refer the reader (in the proofs of Prop.~\ref{\en{EN}extjganzalgg} to~\ref{\en{EN}satztaun}) to literature proving that $1728J(\tau)\in\mathbb Z$ and $s_2(\tau)\in\mathbb Q$ holds for some $\tau$.} the exactness of our results.
For our calculation, we don’t require special software packages, but only the Fourier expansions of the Eisenstein series with a precision of $\approx 20$ decimals.

}%
\de{In Kap.~\ref{\en{EN}kapFormel} berechnen wir diese Funktionswerte explizit und beweisen, dass die Koeffizienten \emph{exakt} die berechneten Werte haben\footnote{Dieser Beweis benötigt als einziger Teil dieses Aufsatzes deutlich mehr als elementare Funktionentheorie. Deshalb müssen wir den Leser (in den Beweisen von Satz~\ref{\en{EN}extjganzalgg} bis~\ref{\en{EN}satztaun}) auf Literatur verweisen, in der bewiesen wird, dass $1728J(\tau)\in\mathbb Z$ und $s_2(\tau)\in\mathbb Q$ für gewisse $\tau$ gilt.}.
Für die Berechnung der Koeffizienten benötigen wir keine spezielle Mathematiksoftware, sondern nur die ersten Terme der Fourierentwicklungen der Eisensteinreihen mit einer Genauigkeit von etwa $20$ Dezimalen.
}%

\en{We will also use ten different values of $\tau$ to obtain ten further formulae to calculate $\pi$ (see page~\pageref{\en{EN}PiFormeln}) -- two of them were already found by Ramanujan.}%
\de{Wir werden auch noch zehn andere passende Werte für $\tau$ einsetzen und dadurch zehn weitere Formeln zur Berechnung von $\pi$ erhalten (siehe Seite~\pageref{\en{EN}PiFormeln}), wovon zwei auf Ramanujan zurück"-gehen.}

\en{An overview over this paper can be found in the commented table of contents on the next page.}%
\de{Einen Überblick über diesen Aufsatz gibt das kommentierte Inhaltsverzeichnis auf der nächsten Seite.}

\vfill\pagebreak\section*{\en{Commented Table of Contents}\de{Kommentiertes Inhaltsverzeichnis}}\label{\en{EN}kapInhalt}
\renewcommand{\leftmark}{\en{Commented Table of Contents}\de{Kommentiertes Inhaltsverzeichnis}}

\vspace*{0.2cm}

{\noindent\textbf{\nameref{\en{EN}kapEinleitung}~~\dotfill~~\pageref{\en{EN}kapEinleitung}}\par\vspace*{8pt}}

\kommentInh{ellipfkt}{\en{We develop the terms and propositions about the Weierstraß elliptic functions that we need for our proof of the Chudnovsky formula.}%
\de{Alle für die Herleitung der Chudnovsky-Formel benötigten Begriffe und Sätze über die Weierstraß'schen elliptischen Funktionen werden entwickelt.}}

\kommentInh{kapquasiint}{\en{We define the quasiperiods of a lattice with the Weierstraß $\zeta$-function. Then we give an alternative representation of the periods and quasiperiods with means of elliptic integrals.}%
\de{Die \glqq Quasiperioden\grqq~eines Gitters werden mit Hilfe der Weierstraß'schen $\zeta$-Funk"-tion definiert. Es folgt eine alternative Darstellung der Perioden und Quasiperioden mit Hilfe elliptischer Integrale.}}

\kommentInh{kapGitter}{\en{In this chapter, we will see that two lattices that are rotated and/or scaled versions of each other can be called "equivalent" and that equivalent lattices have the same value of Klein's absolute invariant~$J$.}%
\de{In diesem Kapitel werden wir sehen, dass zwei Gitter, die durch eine Drehstreckung auseinander hervorgehen, \glqq äquivalent\grqq~genannt werden können und dass äquivalente Gitter die gleiche absolute Invariante $J$ haben.}}

\kommentInh{kapFourier}{\en{We calculate the Fourier representations of the normalized Eisenstein series.}%
\de{Die Fourierentwicklungen der normierten Eisensteinreihen werden bewiesen.}}

\kommentInh{kapanhang}{\en{These estimates prove that Kummer's solution in chapter~\ref{\en{EN}kapkummer} converges, and they are needed to calculate the coefficients in chapter~\ref{\en{EN}kapFormel}.}%
\de{Die hier bewiesenen Abschätzungen garantieren, dass Kummers Lösung in Kap.~\ref{\en{EN}kapkummer} konvergiert und ermöglichen die Berechnungen der Koeffizienten in Kap.~\ref{\en{EN}kapFormel}.}}

\kommentInh{kapClausen}{\en{We prove Clausen's formula and the hypergeometric differential equations needed for the proof. This chapter is self-contained.}%
\de{Wir beweisen die Clausen-Formel und die für den Beweis benötigten hypergeometrischen Differentialgleichungen. Das Kapitel ist unabhängig von den vorigen.}}

\kommentInh{kappicardfuchs}{\en{This proof of the Picard Fuchs differential equation can be read straight after chapter~\ref{\en{EN}kapGitter}.}%
\de{Dieser Beweis der Picard-Fuchs-Differentialgleichung kann direkt im Anschluss an Kap.~\ref{\en{EN}kapGitter} gelesen werden.}}

\kommentInh{kapkummer}{\en{We use one of Kummer's solutions of the Picard Fuchs differential equation to prove a connection between the periods of a lattice and a hypergeometric function.}%
\de{Mit Hilfe einer Kummer'schen Lösung der Picard-Fuchs-Differentialgleichung wird ein Zusammenhang zwischen den Perioden eines Gitters und einer hypergeometrischen Funktion hergestellt.}}

\kommentInh{kaphaupt}{\en{We prove the Main Theorem~\ref{\en{EN}hauptformel} using Kummer's solution, Clausen's formula and the Fourier representations.}%
\de{Das Haupttheorem~\ref{\en{EN}hauptformel} wird bewiesen, ausgehend von Kum"-mers Lö"-sung und mit Hilfe der Clausen-Formel und der Fourierdarstellungen.}}

\kommentInh{kapFormel}{\en{We explicitly calculate the exact values of $s_2(\tau_{N})$ and $J(\tau_{N})$ using the estimates from chapter~\ref{\en{EN}kapanhang}. Thus we obtain the Chudnovsky formula and ten further formulae to calculate $\pi$.}%
\de{Die exakten Werte von $s_2(\tau_{N})$ und $J(\tau_{N})$ werden mit Hilfe der Abschätzungen aus Kap.~\ref{\en{EN}kapanhang} explizit berechnet. Somit erhalten wir die Chudnovsky-Formel und zehn weitere Formeln zur Berechnung von $\pi$.}}

\kommentInh{kapdivi}{\en{We prove that $m\cdot\wp(u;L)$ is an algebraic integer of~~$\mathbb Z\mathopen{}\left[\frac 1 4 g_2(L); \frac 1 4 g_3(L)\right]\mathclose{}$ for all positive integers $m$ and for all $u\in\mathbb C-L$ with $m\cdot u\in L$.}%
\de{Wir beweisen, dass $m\cdot\wp(u;L)$ für alle natürlichen Zahlen $m\neq 0$ und für alle $u\in\mathbb C-L$ mit $m\cdot u\in L$ ganzalgebraisch in $\mathbb Z\mathopen{}\left[\frac 1 4 g_2(L); \frac 1 4 g_3(L)\right]\mathclose{}$ ist.}}

\kommentInh{kapmult}{\en{We use Appendix~\ref{\en{EN}kapdivi} to prove that $\sqrt{D}\cdot\frac{E_2^*(\tau)}{\eta^4(\tau)}\cdot(AC)^2$ is an algebraic integer if $\tau$ satisfies $C\tau^2+B\tau+A=0$ with discriminant $D$.}%
\de{Wir beweisen mit Hilfe von Anhang~\ref{\en{EN}kapdivi}, dass $\sqrt{D}\cdot\frac{E_2^*(\tau)}{\eta^4(\tau)}\cdot(AC)^2$ ganzalgebraisch ist, falls $\tau$ eine Lösung von $C\tau^2+B\tau+A=0$ mit Diskriminante $D$ ist.}}

{\noindent\textbf{\hyperref[\en{EN}literat]{\en{References}\de{Literatur}}~~\dotfill~~\pageref{\en{EN}literat}}}

\vfill\pagebreak\section{\en{Elliptic Functions}\de{Elliptische Funktionen}}\label{\en{EN}ellipfkt}
\renewcommand{\leftmark}{\en{Elliptic Functions}\de{Elliptische Funktionen}}
\en{In this chapter, we develop the terms and propositions about the Weierstraß elliptic functions that we need for our proof of the Chudnovsky formula. The notation and some proofs are based on \cite{\en{EN}Freitag2000}, where one can find more detailed explanations.}%
\de{In diesem Kapitel werden alle für den Beweis der Chudnovsky-Formel benötigten Begriffe und Sätze entwickelt. Die Notation und einige Beweise orientieren sich an \cite{\en{EN}Freitag2000}. Dort findet man auch ausführlichere Erläuterungen und Zusammenhänge.}

\begin{defi}\label{\en{EN}defgitter} \en{For each pair $(\omega_1,\omega_2)$ of complex numbers which is $\mathbb R$-linearly independent (which means $\omega_2/\omega_1\notin\mathbb R$) we call}%
\de{Zu jedem Paar $(\omega_1,\omega_2)$ komplexer Zahlen, die $\mathbb R$-linear unabhängig sind (also gilt $\omega_2/\omega_1\notin\mathbb R$) nennt man}
$$L = \mathbb Z \omega_1 + \mathbb Z \omega_2 = \left\{m\omega_1+n\omega_2~|~m,n\in\mathbb Z\right\}\subset\mathbb C$$
\en{a "lattice". $\omega_1$ and $\omega_2$ are called "basic periods" of the lattice.}%
\de{ein \glqq Gitter\grqq. $\omega_1$ und $\omega_2$ heißen dann auch \glqq Basisperioden\grqq~des Gitters.}
\end{defi}

\begin{defi}
\en{An "elliptic function" is a meromorphic function $f:\mathbb C \rightarrow \mathbb C \cup \{\infty\}$ with the property}%
\de{Eine \glqq elliptische Funktion zum Gitter $L$\grqq~ist eine meromorphe Funktion $f:\mathbb C \rightarrow \mathbb C \cup \{\infty\}$ mit der Eigenschaft} $$f(z+\omega)=f(z)\qquad\text{\en{for all}\de{für alle} }~\omega\in L~\text{ \en{and}\de{und} }~z\in\mathbb C$$
\en{"Meromorphic" means that $f$ has no essential singularities, that the set of poles of $f$ has no accumulation point, and that $f$ is holomorphic apart from the poles.}%
\de{\glqq Meromorph\grqq~bedeutet, dass $f$ keine außerwesentlichen Singularitäten hat, dass die Polstellenmenge von $f$ keinen Häufungspunkt hat, und dass $f$ außerhalb der Polstellen analytisch ist.}
\en{To show that a meromorphic function is elliptic it suffices to check if $f(z+\omega_1)=f(z)=f(z+\omega_2)$ holds for both basic periods of the lattice -- that's why elliptic functions are also called "doubly periodic".}%
\de{Um nachzuweisen, dass eine meromorphe Funktion elliptisch ist, reicht es zu prüfen, ob $f(z+\omega_1)=f(z)=f(z+\omega_2)$ für die beiden Basisperioden des Gitters gilt -- deshalb nennt man elliptische Funktionen auch \glqq doppeltperiodisch\grqq.}
\end{defi}

\begin{defi}\label{\en{EN}defimodz}
\en{Every lattice $L$ produces an equivalence relation on the complex numbers: We call $z_1\in \mathbb C$ and $z_2\in \mathbb C$ "equivalent modulo $L$", iff it holds $z_1-z_2\in L$ (since it doesn't matter if one uses $z_1$ or $z_2$ in an elliptic function of the lattice $L$).}%
\de{Jedes Gitter $L$ erzeugt eine Äquivalenzrelation auf den komplexen Zahlen: Wir nennen $z_1\in \mathbb C$ und $z_2\in \mathbb C$ \glqq äquivalent modulo $L$\grqq, falls $z_1-z_2\in L$ ist (dann ist es nämlich egal, ob man $z_1$ oder $z_2$ in eine elliptische Funktion zum Gitter $L$ einsetzt).}
\end{defi}

\begin{defi}\label{\en{EN}fund}
\en{The "fundamental parallelogram" $\mathcal P$ and its closure $\overline{\mathcal P}$ are:}%
\de{Das \glqq Periodenparallelogramm\grqq~$\mathcal P$ und sein Abschluss $\overline{\mathcal P}$ lauten:}
$$\mathcal P = \left\{~s\omega_1+t\omega_2~|~0\leq s,t< 1~\right\}
\qquad\text{\en{and}\de{und}}\qquad
\overline{\mathcal P}= \left\{~s\omega_1+t\omega_2~|~0\leq s,t\leq 1~\right\}$$
\en{Since $\omega_1$ and $\omega_2$ are $\mathbb R$-linearly independent, it holds: for all $z\in\mathbb C$ there is exactly one $z'\in\mathcal P$ which is equivalent to $z$ (modulo $L$).
Figure~\ref{\en{EN}abbwege} on p.~\pageref{\en{EN}abbwege} depicts $\overline{\mathcal P}$.}%
\de{Da $\omega_1$ und $\omega_2$ $\mathbb R$-linear unabhängig sind, gibt es zu jedem $z\in\mathbb C$ genau ein $z'\in\mathcal P$, das modulo $L$ äquivalent zu $z$ ist.
Eine Abbildung von $\overline{\mathcal P}$ befindet sich auf Seite~\pageref{\en{EN}abbwege} (Abb.~\ref{\en{EN}abbwege}).}
\end{defi}

\begin{thm}[\en{Liouville's Theorem}\de{Satz von Liouville}]\label{\en{EN}liouville0}
\en{Any bounded analytic function $\mathbb C \rightarrow \mathbb C$ is constant.}%
\de{Jede beschränkte analytische Funktion $\mathbb C \rightarrow \mathbb C$ ist konstant.}
\end{thm}
\begin{proof}
\en{Given $z\in \mathbb C$, we will prove $f'(z)=0$: By deriving Cauchy's integral formula (with Leibniz's rule) we obtain for all $r>0$:}%
\de{Für $z\in \mathbb C$ beweisen wir zunächst $f'(z)=0$: Durch Ableiten der Cauchy'schen Integralformel (mit der Leibniz'schen Regel) folgt für alle $r>0$:}
\begin{align*}
    |f'(z)| &= \left|\frac{1}{2\pi i} \oint_{|\zeta-z|=r}\frac{f(\zeta)}{(\zeta-z)^2}d\zeta\right|
\leq \frac{1}{2\pi} \cdot \frac{C}{r^2}\cdot 2\pi r = \frac{C}{r}
\end{align*}
\en{Here we used the boundedness $|f(\zeta)|\leq C$ and the perimeter of the circle.}%
\de{Hier haben wir die Beschränktheit $|f(\zeta)|\leq C$ und den Umfang des Kreises genutzt.}
\en{For $r\rightarrow\infty$ we obtain $f'(z)=0$ -- thus $f$ is constant.}%
\de{Für $r\rightarrow\infty$ folgt also $f'(z)=0$ und somit dass $f$ konstant ist.}
\end{proof}

\begin{thm}[\en{First Liouville Theorem}\de{Erster Liouville'scher Satz}]\label{\en{EN}liouville1}
\en{Any elliptic function without poles is constant.}%
\de{Jede elliptische Funktion ohne Polstellen ist konstant.}
\end{thm}
\begin{proof}
\en{Any elliptic function $f$ with basic periods $\omega_1$ and $\omega_2$ takes any of its values in the fundamental parallelogram $\mathcal P$ (cf.~Def.~\ref{\en{EN}fund}). But its closure $\overline{\mathcal P}$ is closed and bounded (see Fig.~\ref{\en{EN}abbwege} on p.~\pageref{\en{EN}abbwege}). Since $f$ has no poles, $|f|$ is continuous and must have a maximum in $\overline{\mathcal P}$. But then, because of its periodicity, $f$ is bounded on the whole complex plane. From Liouville's theorem (Prop.~\ref{\en{EN}liouville0}) we deduce that $f$ must be constant.}%
\de{Eine elliptische Funktion $f$ nimmt jeden ihrer Werte schon im Periodenparallelogramm $\mathcal P$ an. Aber $\overline{\mathcal P}$ ist beschränkt und abgeschlossen (siehe Abb.~\ref{\en{EN}abbwege} auf S.~\pageref{\en{EN}abbwege}), deshalb besitzt $f$ wie jede stetige Funktion in $\overline{\mathcal P}$ ein Maximum. Aus der Periodizität folgt, dass $f$ auf ganz $\mathbb C$ beschränkt und somit nach dem Satz von Liouville (Satz~\ref{\en{EN}liouville0}) konstant ist.}
\end{proof}

\begin{thm}[\en{Second Liouville Theorem}\de{Zweiter Liouville'scher Satz}]\label{\en{EN}liouville2}
\en{Any elliptic function has only finitely many poles (modulo $L$) and the sum of their residues vanishes.}%
\de{Jede elliptische Funktion hat nur endlich viele Pole (modulo $L$) und die Summe ihrer Rediduen verschwindet.}
\end{thm}
\begin{proof}
\en{For any pole of an elliptic function, there is an equivalent pole in $\mathcal P$ (cf.~Def.~\ref{\en{EN}fund}). The set of poles of an elliptic function is discrete, thus only finitely many poles are in the closure $\overline{\mathcal P}$ of the fundamental parallelogram ($\overline{\mathcal P}$ is compact).}%
\de{Zu jeder Polstelle einer elliptischen Funktion gibt es eine äquivalente Polstelle in $\mathcal P$ (vgl.~Def.~\ref{\en{EN}fund}). Die Menge der Pole einer elliptischen Funktion ist diskret, also liegen nur endlich viele Pole in  der kompakten Menge $\overline{\mathcal P}$ aus Def.~\ref{\en{EN}fund}.
}
\en{Then we move $\overline{\mathcal P}$ so that no more poles are on its border, and integrate along the border. Since $\wp$ is doubly periodic, this integral vanishes (since the integrals along parallel parts of the border cancel each other out). From the residue theorem we deduce that the sum of the residues vanishes.}%
\de{Wenn wir nun $\overline{\mathcal P}$ so verschieben, dass keine Pole mehr auf dem Rand liegen, und einmal entlang des Randes integrieren, folgt aus der Periodizität der $\wp$-Funktion, dass das Integral den Wert Null hat (die Beiträge gegenüberliegender Kanten heben sich gegenseitig auf). Aus dem Residuensatz folgt nun, dass die Summe der Residuen verschwindet.}
\end{proof}

\begin{thm}[\en{Third Liouville Theorem}\de{Dritter Liouville'scher Satz}]\label{\en{EN}liouville3}
\en{Any non-constant elliptic function $f$ has the same number of zeros and poles modulo $L$, if they are counted with their multiplicities.}%
\de{Jede nichtkonstante elliptische Funktion $f$ hat modulo $L$ gleich viele Null- und Polstellen, wobei diese mit ihrer Vielfachheit zu rechnen sind.}
\end{thm}
\begin{proof}
\en{If $f$ is a non-constant elliptic function, $g(z):=\frac{f'(z)}{f(z)}$ also is a non-constant elliptic function.}%
\de{Wenn $f$ eine nichtkonstante elliptische Funktion ist, ist auch $g(z):=\frac{f'(z)}{f(z)}$ eine nichtkonstante elliptische Funktion.}

\en{If the Laurent series of $f$ in $z_0$ starts with $f(z)\approx a\cdot (z-z_0)^{k}$ (where $k\in \mathbb Z$, $k\neq 0$), it holds $f'(z)\approx k\cdot a\cdot(z-z_0)^{k-1}$ and $g(z)\approx \frac{k}{z-z_0}$. Thus every pole and every zero of $f(z)$ produces a pole of $g(z)$ of order one with residue $k$.}%
\de{Wenn die Laurentreihe von $f$ in $z_0$ mit $f(z)\approx a\cdot (z-z_0)^{k}$ beginnt ($k\in \mathbb Z$, $k\neq 0$), gilt $f'(z)\approx k\cdot a\cdot(z-z_0)^{k-1}$ und $g(z)\approx \frac{k}{z-z_0}$. Also gilt für alle Null- und Polstellen von $f(z)$, dass $g(z)$ dort einen Pol erster Ordnung mit Residuum $k$ hat.}

\en{From its definition $g(z):=\frac{f'(z)}{f(z)}$ we see that $g$ has no further poles.}%
\de{Aus der Definition $g(z):=\frac{f'(z)}{f(z)}$ folgt, dass $g$ keine weiteren Pole hat.}

\en{The sum of the residues of $g$ vanishes (second Liouville theorem, Prop.~\ref{\en{EN}liouville2}), thus we have:
The sum of the positive residues of $g$ (the sum of the multiplicities of the zeros of $f$) has the same absolute value as the sum of the negative residues of $g$ (the sum of the multiplicities of the poles of $f$).}%
\de{Die Summe der Residuen von $g$ verschwindet (zweiter Liouville'scher Satz~\ref{\en{EN}liouville2}), also gilt:
Die Summe der positiven Residuen von $g$ (die Summe der Nullstellenordnungen von $f$) ist betragsmäßig gleich groß wie die Summe der negativen Residuen von $g$ (die Summe der Polstellenordnungen von $f$).}
\end{proof}

\begin{defi}\label{\en{EN}defisigma}
\en{The Weierstraß $\sigma$-function of the lattice $L$ is defined as follows:}%
\de{Die Weierstraß'sche $\sigma$-Funktion zum Gitter $L$ ist wie folgt definiert:}
$$\sigma(z;L):=z\cdot\prod_{\substack{\omega\in L\\\omega\neq 0}}\left\{\left(1-\frac z \omega \right)\cdot\exp\lk\frac z \omega + \frac 1 2 \left(\frac z \omega\right)^2\rk\right\}$$
\en{The $\sigma$-function will be analyzed further in chapter~\ref{\en{EN}kapFourier}, see for example Prop.~\ref{\en{EN}sigmaodd},~\ref{\en{EN}trafosigma} and~\ref{\en{EN}fouriersigma}.}%
\de{Die $\sigma$-Funktion wird in Kapitel~\ref{\en{EN}kapFourier} weiter untersucht, siehe z.B.~Satz~\ref{\en{EN}sigmaodd},~\ref{\en{EN}trafosigma} und~\ref{\en{EN}fouriersigma}.}
\end{defi}

\begin{bem}\label{\en{EN}bemsigma}
\en{This product converges absolutely because of the exponential factor, and the zeros of $\sigma(z;L)$ are exactly the points of the lattice $L$ and are zeros of order $1$. Nevertheless, the $\sigma$-function is \emph{not} doubly periodic (cf.~Prop.~\ref{\en{EN}trafosigma}).}%
\de{Dieses Produkt konvergiert aufgrund des Exponentialfaktors absolut, und die Nullstellen von $\sigma(z;L)$ sind genau die Punkte des Gitters $L$. Es sind einfache Nullstellen. Trotzdem ist die $\sigma$-Funktion \emph{nicht} doppeltperiodisch (vgl.~Satz~\ref{\en{EN}trafosigma}).}
\end{bem}

\begin{defi}\label{\en{EN}defizeta}
\en{The Weierstraß $\zeta$-function of a lattice $L$ is defined as the logarithmic derivative of the $\sigma$-function, whose product yields a sum because of~~$\ln(a\cdot b)=\ln a +\ln b$:}%
\de{Die Weierstraß'sche $\zeta$-Funktion zu einem Gitter $L$ ist als logarithmische Ableitung der $\sigma$-Funktion definiert:}
\begin{align*}
\zeta(z;L) &:= \frac {d}{dz}\ln\sigma(z;L) = \frac {d}{dz}\lk\ln z\rk + \sum_{\substack{\omega\in L\\\omega\neq 0}}\frac {d}{dz}\left\{\ln\left(1-\frac z \omega \right)+\frac z \omega + \frac 1 2 \left(\frac z \omega\right)^2\right\}\\
&= \frac 1 z + \sum_{\substack{\omega\in L\\ \omega\neq 0}} \left(\frac 1 {z-\omega} +\frac 1 \omega +\frac z {\omega^2}\right)
\end{align*}
\en{The $\zeta$-function will be analyzed further in chapter~\ref{\en{EN}kapquasiint}, see for example Def.~\ref{\en{EN}defetak} and Rem.~\ref{\en{EN}bempitch}.}%
\de{Die $\zeta$-Funktion wird in Kapitel~\ref{\en{EN}kapquasiint} weiter untersucht, siehe z.B.~Def.~\ref{\en{EN}defetak} und Bem.~\ref{\en{EN}bempitch}.}
\end{defi}

\begin{defi}\label{\en{EN}defiwp}
\en{The Weierstraß $\wp$-function denotes the negative derivative of the Weierstraß $\zeta$-function:}%
\de{Die Weierstraß'sche $\wp$-Funktion ist definiert als die negative Ableitung der Weierstraß'schen $\zeta$-Funktion:}
$$\wp(z;L) :=  -\zeta'(z;L) = \frac 1 {z^2} + \sum_{\substack{\omega\in L\\ \omega\neq 0}} \left(\frac 1 {(z-\omega)^2} -\frac 1 {\omega^2}\right)$$
\end{defi}

\begin{bem}\label{\en{EN}pstrich}
\en{The derivative of the Weierstraß $\wp$-function reads:}%
\de{Die Ableitung der Weierstraß'schen $\wp$-Funktion lautet:}
$$\wp'(z;L)=\sum_{\omega\in L}\frac{-2}{(z-\omega)^3}$$
\end{bem}

\begin{thm}\label{\en{EN}pgerade}
\en{$\wp(z;L)$ is an even function and $\wp'(z;L)$ is an odd function, i.e.}%
\de{$\wp(z;L)$ ist eine gerade und $\wp'(z;L)$ ist eine ungerade Funktion, d.h.}
$$\wp(-z;L)=\wp(z;L)\qquad\text{\en{and}\de{und}}\qquad\wp'(-z;L)=-\wp'(z;L)$$
\end{thm}
\begin{proof}
\en{If $\omega$ runs through all points of the lattice, then $-\omega$ does it too:}%
\de{Mit $\omega$ durchläuft auch $-\omega$ alle Gitterpunkte. Hieraus folgt:}
\begin{align*}
    \wp(-z;L) &= \frac 1 {(-z)^2} + \sum_{\substack{\omega\in L\\ \omega\neq 0}} \left(\frac 1 {(-z-\omega)^2} -\frac 1 {\omega^2}\right)\\
            &= \frac 1 {z^2} + \sum_{\substack{-\omega\in L\\ -\omega\neq 0}} \left(\frac 1 {(z-(-\omega))^2} -\frac 1 {(-\omega)^2}\right) = \wp(z;L)
\end{align*}
\en{And for $\wp'(z)$ it holds:}%
\de{Und für $\wp'(z)$ gilt:}\belowdisplayskip=-12pt
\begin{align*}
    \wp'(-z;L) &= \sum_{\omega\in L}\frac{-2}{(-z-\omega)^3} = - \sum_{-\omega\in L}\frac{-2}{(z-(-\omega))^3} = -\wp'(z;L)
\end{align*}
\end{proof}

\begin{thm}\label{\en{EN}wpdoppeltper}
\en{The Weierstraß $\wp$-function is doubly periodic, i.e.~for all $\omega\in L$ we have $\wp(z+\omega;L)=\wp(z;L)$.}%
\de{Die Weierstraß'sche $\wp$-Funktion ist doppeltperiodisch, d.h.~für alle $\omega\in L$ gilt $\wp(z+\omega;L)=\wp(z;L)$.}
\end{thm}
\begin{proof}
\en{$\wp'$ is doubly periodic, since the summation runs through all lattice points and since there are no further terms (see Remark~\ref{\en{EN}pstrich}).
So we get $\wp'(z+\omega)-\wp'(z)=0$ and thus $\wp(z+\omega)-\wp(z)=\text{const}$.
If $\omega$ is a basic period of the lattice, then $-\frac \omega 2 \notin L$.
We get the value of the constant with Prop.~\ref{\en{EN}pgerade}: $\wp\lk-\frac \omega 2 + \omega\rk - \wp\lk-\frac \omega 2\rk =\wp\lk\frac \omega 2\rk - \wp\lk-\frac \omega 2\rk = 0$.
This yields $\wp(z+\omega)=\wp(z)$ for all \emph{basic} periods of the lattice $L$ and thus for \emph{all} points of the lattice.}%
\de{Man sieht sofort, dass $\wp'$ doppeltperiodisch ist, weil über alle Gitterpunkte summiert wird und keine weiteren Terme in der Summe stehen.
Also gilt $\wp'(z+\omega)-\wp'(z)=0$ und somit $\wp(z+\omega)-\wp(z)=\text{const}$.
Wenn wir für $\omega$ eine der Basisperioden des Gitters wählen, dann ist $-\frac \omega 2 \notin L$.
Wir erhalten den Wert der Konstanten mit Hilfe von Satz~\ref{\en{EN}pgerade}: $\wp\lk-\frac \omega 2 + \omega\rk - \wp\lk-\frac \omega 2\rk =\wp\lk\frac \omega 2\rk - \wp\lk-\frac \omega 2\rk = 0$.
Es folgt $\wp(z+\omega)=\wp(z)$ für alle Basisperioden des Gitters $L$ und somit auch für alle anderen Gitterpunkte.}
\end{proof}

\begin{thm}\label{\en{EN}zerowp}
\en{The zeros of $\wp'$ are exactly those points $\frac\omega 2$, for which $\omega\in L$ but $\frac\omega 2\notin L$.
If $L=\mathbb Z \omega_1 + \mathbb Z \omega_2$, this yields the following three zeros (see Fig.~\ref{\en{EN}abbwege} on p.~\pageref{\en{EN}abbwege}):}%
\de{Die Nullstellen von $\wp'$ sind genau diejenigen Stellen $\frac\omega 2$, die selbst nicht im Gitter liegen, für die aber $\omega$ im Gitter liegt.
Wenn $L=\mathbb Z \omega_1 + \mathbb Z \omega_2$ ist, gilt also}
$$\wp'\lk\frac{\omega_1}{2}\rk = \wp'\lk\frac{\omega_2}{2}\rk = \wp'\lk\frac{\omega_1+\omega_2}{2}\rk=0$$
\end{thm}
\begin{proof}
\en{Choose $\omega_k\in L$ so that $\frac{\omega_k}2\notin L$. Then we get: if $\omega$ runs through all points of the lattice, then also $\omega'=\omega+\omega_k$ does it. This yields:}%
\de{Sei $\omega_k\in L$ so gewählt, dass $\frac{\omega_k}2\notin L$. Dann durchläuft mit $\omega$ auch $\omega'=\omega+\omega_k$ alle Gitterpunkte und es gilt:}
\begin{align*}
    \wp'\lk-\frac{\omega_k}2;L\rk &= \sum_{\omega\in L}\frac{-2}{\left(-\frac{\omega_k}2-\omega\right)^3} = \sum_{\omega'\in L}\frac{-2}{\left(-\frac{\omega_k}2-(\omega'-\omega_k)\right)^3}\\
    &= \sum_{\omega'\in L}\frac{-2}{\left(\frac{\omega_k}2-\omega'\right)^3} = \wp'\lk\frac{\omega_k}2;L\rk
\end{align*}
\en{From Prop.~\ref{\en{EN}pgerade} we know that $\wp'$ is \emph{odd} and according to our premises $\pm\frac{\omega_k}2$ is not in $L$. From this we get $\wp'\lk-\frac{\omega_k}2;L\rk = -\wp'\lk\frac{\omega_k}2;L\rk$.
This yields $\wp'\lk\frac{\omega_k}2;L\rk=0$.}%
\de{Weil $\wp'$ nach Satz~\ref{\en{EN}pgerade} eine \emph{ungerade} Funktion ist und weil nach Voraussetzung $\pm\frac{\omega_k}2\notin L$ ist, gilt $\wp'\lk-\frac{\omega_k}2;L\rk = -\wp'\lk\frac{\omega_k}2;L\rk$.
Folglich muss $\wp'\lk\frac{\omega_k}2;L\rk=0$ sein.}
\en{Using the third Liouville theorem (Prop.~\ref{\en{EN}liouville3}) we see that $\wp'$ has no further zeros (modulo $L$).}%
\de{Aus dem dritten Liouville'schen Satz~\ref{\en{EN}liouville3} folgt, dass $\wp'$ (modulo $L$) keine weiteren Nullstellen hat.}
\end{proof}

\begin{defi}\label{\en{EN}defeisen}
\en{The series}\de{Die Reihen} $\displaystyle G_n = G_n(L) := \sum_{\substack{\omega\in L\\ \omega\neq 0}} \omega^{-n}$
\en{are called "Eisenstein series of weight $n$" and converge absolutely for natural $n\geq 3$.}%
\de{heißen \glqq Eisensteinreihen zum Gitter $L$\grqq\ und konvergieren für natürliche $n\geq 3$ absolut.}
\end{defi}

\begin{thm}\label{\en{EN}eisenungerade}
\en{The Eisenstein series of odd weight vanish (i.e.~they take on the value $0$).}%
\de{Die Eisensteinreihen mit ungeradem Gewicht verschwinden.}
\end{thm}
\begin{proof}
\en{Since $n$ is odd, we can deduce that for all $\omega\in L-\{0\}$ the summands $\omega^{-n}$ and $(-\omega)^{-n}=-\left(\omega^{-n}\right)$ cancel each other out. Thus the full sum takes on the value $0$.}%
\de{Wenn $n$ ungerade ist, dann gilt für alle $\omega\in L-\{0\}$, dass sich die Summanden $\omega^{-n}$ und $(-\omega)^{-n}=-\left(\omega^{-n}\right)$ gegenseitig aufheben.}
\end{proof}

\begin{thm}\label{\en{EN}laurentwp}
\en{The Weierstraß $\wp$-function admits the following Laurent series expansion around $z=0$ without a constant term:}%
\de{Die Weierstraß'sche $\wp$-Funktion lässt sich in der Nähe von $z=0$ durch folgende Laurentreihe ohne konstanten Term darstellen:}
\begin{align*}
\wp(z;L)&=\frac 1 {z^2} + \sum_{n=1}^\infty (2n+1)\cdot G_{2n+2}(L)\cdot z^{2n}
\end{align*}
\end{thm}
\begin{proof}
\en{First we analyze $f(z):=\wp(z;L)-\frac 1{z^2}$. From Def.~\ref{\en{EN}defiwp} we get $f(0)=0$.
Then we get the derivatives of $f(z)$ at $z=0$ with the representation of $\wp'$ from Remark~\ref{\en{EN}pstrich}:}%
\de{Wir untersuchen zunächst $f(z):=\wp(z;L)-\frac 1{z^2}$. Bei $z=0$ folgt direkt aus der Definition~\ref{\en{EN}defiwp} der $\wp$-Funktion, dass $f(0)=0$ ist.
Für die Ableitungen von $f(z)$ bei $z=0$ folgt dann mit Hilfe der Darstellung von $\wp'$ aus Bemerkung~\ref{\en{EN}pstrich}:}
$$f^{(n)}(z) = (-1)^n(n+1)!\sum_{\substack{\omega\in L\\ \omega\neq 0}} \frac{1}{(z-\omega)^{n+2}}\qquad\text{\en{if}\de{falls} }n\geq 1$$
\pagebreak\\
\en{From Prop.~\ref{\en{EN}eisenungerade} we deduce that the odd derivatives vanish at $z=0$, and that the even derivatives are:}%
\de{Hieraus folgt (wegen Satz~\ref{\en{EN}eisenungerade}), dass die ungeraden Ableitungen bei $z=0$ verschwinden, und dass für die geraden gilt:}
$$f^{(2n)}(0) = (-1)^{2n}(2n+1)!\sum_{\substack{\omega\in L\\ \omega\neq 0}} \frac{1}{(-\omega)^{2n+2}}=(2n+1)!\sum_{\substack{\omega\in L\\ \omega\neq 0}} \frac{1}{\omega^{2n+2}}=(2n+1)!\cdot G_{2n+2}$$
\en{with the Eisenstein series from Definition~\ref{\en{EN}defeisen}.
Thus we have shown that it holds $f(z)=\sum_{n=1}^\infty \frac{f^{2n}(0)}{(2n)!}\cdot z^{2n} = \sum_{n=1}^\infty (2n+1)G_{2n+2}\cdot z^{2n}$ and the proposition is proven.}%
\de{mit den Eisensteinreihen aus Definition~\ref{\en{EN}defeisen}.
Insgesamt haben wir also bewiesen, dass $f(z)=\sum_{n=1}^\infty \frac{f^{2n}(0)}{(2n)!}\cdot z^{2n} = \sum_{n=1}^\infty (2n+1)G_{2n+2}\cdot z^{2n}$ gilt, und wir sind fertig.}
\end{proof}

\begin{thm}\label{\en{EN}dglP}
\en{The Weierstraß $\wp$-function satisfies the algebraic differential equation:}%
\de{Es gilt folgende algebraische Differentialgleichung der $\wp$-Funktion zum Gitter $L$:}
\begin{align*}
    \wp'(z)^2 &= 4 \wp(z)^3 - g_2 \wp(z) - g_3\\
    \text{\en{with}\de{mit}}\quad g_2 &= g_2(L) := 60 G_4(L) = 60 \sum_{\substack{\omega\in L\\ \omega\neq 0}} \omega^{-4}\\
    \text{\en{and}\de{und}}\quad g_3 &= g_3(L) := 140 G_6(L) = 140 \sum_{\substack{\omega\in L\\ \omega\neq 0}} \omega^{-6}
\end{align*}
\end{thm}
\begin{proof}
\en{We use the beginning of the Laurent series expansion from Prop.~\ref{\en{EN}laurentwp} and show that $h(z):=\wp'(z)^2-4\wp(z)^3+60G_4\wp(z)$ has no poles:}%
\de{Wir verwenden den Anfang der Laurentreihe aus Satz~\ref{\en{EN}laurentwp} und zeigen, dass die Funktion $h(z):=\wp'(z)^2-4\wp(z)^3+60G_4\wp(z)$ keine Pole hat:}
\begin{align*}
    \wp(z;L) &= z^{-2} + 3 G_4 z^2 + 5 G_6 z^4 + O(z^6)\\
    \Longrightarrow\quad\wp(z;L)^2 &= z^{-4} + 6 G_4 + 10 G_6 z^2 + O(z^4)\\
    \Longrightarrow\quad\wp(z;L)^3 &= \wp(z;L)^2\cdot\wp(z;L) = z^{-6} + 9 G_4 z^{-2} + 15 G_6 + O(z^2)\\
    \text{\en{and}\de{und}}\quad\wp'(z;L) &= -2z^{-3} + 6 G_4 z + 20 G_6 z^3 + O(z^5)\\
    \Longrightarrow\quad\wp'(z;L)^2 &= 4z^{-6} -24 G_4 z^{-2} - 80 G_6 + O(z^2)\\
    \Longrightarrow\wp'(z;L)^2-4\wp(z;L)^3 &= -60 G_4 z^{-2} - 140 G_6 + O(z^2)\\
    \Longrightarrow\wp'(z;L)^2 -4\wp(z;L)^3 &+ 60 G_4 \wp(z;L) = - 140 G_6 + O(z^2)
\end{align*}
\en{Here we recognize that $h(z)$ has no pole at $z=0$. From the definition of $h(z)$ we see that it is doubly periodic, and that $h(z)$ has no poles in any lattice points. Since neither $\wp$ nor $\wp'$ have poles besides the lattice points we deduce that $h(z)$ is an elliptic function without poles.
Thus, by the first Liouville theorem (Prop.~\ref{\en{EN}liouville1}), $h(z)$ is constant. The value of this constant is $-140G_6$ (see above), so we get
$\wp'(z;L)^2 = 4\wp(z;L)^3 - 60 G_4(L)\wp(z;L) - 140 G_6(L)$.}%
\de{$h(z)$ hat also bei $z=0$ keinen Pol. Weil $h(z)$ aufgrund seiner Definition doppeltperiodisch ist, hat $h(z)$ also auch in den anderen Gitterpunkten keine Pole, und außerhalb der Gitterpunkte haben sowieso weder $\wp$ noch $\wp'$ Pole, insofern ist $h(z)$ eine elliptische Funktion ohne Pole und nach dem ersten Liouville'schen Satz~\ref{\en{EN}liouville1} konstant. Der konstante Wert von $h(z)$ ergibt sich zu $-140G_6$ und wir erhalten $\wp'(z;L)^2 = 4\wp(z;L)^3 - 60 G_4(L)\wp(z;L) - 140 G_6(L)$, was zu zeigen war.}
\end{proof}

\begin{thm}\label{\en{EN}pstrichprod}
\en{If $L=\mathbb Z \omega_1 + \mathbb Z \omega_2$, then it holds:}%
\de{Wenn $L=\mathbb Z \omega_1 + \mathbb Z \omega_2$ ist, gilt:} 
\begin{align*}
    \wp'(z)^2&=4\cdot \left(\wp(z)-e_1\right) \cdot\left(\wp(z)-e_2\right) \cdot\left(\wp(z)-e_3\right)\\
    \intertext{\en{with the pairwise distinct half lattice values of the $\wp$-function}\de{mit den paarweise verschiedenen Halbwerten der $\wp$-Funktion}}
    e_1 &:= \wp\lk\frac{\omega_1}{2}\rk;\qquad e_2 := \wp\lk\frac{\omega_2}{2}\rk;\qquad e_3 := \wp\lk\frac{\omega_1+\omega_2}{2}\rk
\end{align*}
\end{thm}
\begin{proof}
\en{Prop.~\ref{\en{EN}zerowp} tells (for example) $\wp'\lk\frac{\omega_1}{2}\rk=0$.
If we set $f(z):=\wp(z)-e_1$, we get both $f\lk\frac{\omega_1}{2}\rk=0$ and $f'\lk\frac{\omega_1}{2}\rk=0$ -- thus $f$ has a double zero at $\frac{\omega_1}{2}$.
Now the third Liouville theorem (Prop.~\ref{\en{EN}liouville3}) tells that $\wp(z)-e_1$ has no further zeros.
Thus the $e_{1;2;3}$ are pairwise distinct.}%
\de{Nach Satz~\ref{\en{EN}zerowp} gilt z.B.~$\wp'\lk\frac{\omega_1}{2}\rk=0$.
Also gilt für $f(z):=\wp(z)-e_1$ sowohl $f\lk\frac{\omega_1}{2}\rk=0$ als auch $f'\lk\frac{\omega_1}{2}\rk=0$ -- somit hat $f$ bei $\frac{\omega_1}{2}$ eine doppelte Nullstelle.
Aus dem dritten Liouville'schen Satz~\ref{\en{EN}liouville3} folgt, dass $\wp(z)-e_1$ keine weiteren Nullstellen hat.
Somit sind die $e_{1;2;3}$ paarweise verschieden.}

\en{From Prop.~\ref{\en{EN}dglP} and~\ref{\en{EN}zerowp} we deduce that $P(X):=4X^3-g_2X-g_3$ has the three distinct zeros $e_1$, $e_2$ and $e_3$. This proves $P(X)=4(X-e_1)(X-e_2)(X-e_3)$.

Using $X=\wp(z)$ and Prop.~\ref{\en{EN}dglP} proves the Proposition.}%
\de{Aus Satz~\ref{\en{EN}dglP} und~\ref{\en{EN}zerowp} folgt dann, dass $P(X):=4X^3-g_2X-g_3$ die drei verschiedenen Nullstellen $e_1$, $e_2$ und $e_3$ hat. Hieraus folgt $P(X)=4(X-e_1)(X-e_2)(X-e_3)$.
Mit $X=\wp(z)$ folgt dann aus Satz~\ref{\en{EN}dglP} die zu beweisende Aussage.}
\end{proof}

\vfill\pagebreak\section{\en{Quasiperiods and their Representation by Integrals}\de{Quasiperioden und ihre Integraldarstellung}}\label{\en{EN}kapquasiint}
\renewcommand{\leftmark}{\en{Quasiperiods and their Representation by Integrals}\de{Quasiperioden und ihre Integraldarstellung}}
\en{In this chapter, we define the "quasiperiods" of a lattice with help of the Weierstraß $\zeta$-function.
We also give an alternative representation of the periods and quasiperiods with help of elliptic integrals.
For this, we use the algebraic differential equation of $\wp$ from Prop.~\ref{\en{EN}dglP}.}%
\de{In diesem Abschnitt definieren wir den Begriff \glqq Quasiperiode\grqq~eines Gitters mit Hilfe der Weierstraß'schen $\zeta$-Funktion und geben eine alternative Darstellung der Perioden und Quasiperioden mit Hilfe elliptischer Integrale an. Hierfür benötigen wir die Differentialgleichung der $\wp$-Funktion aus Satz~\ref{\en{EN}dglP}.}

\begin{thm}\label{\en{EN}eta}
\en{The Weierstraß $\zeta$-function from Def.~\ref{\en{EN}defizeta} is not doubly periodic, but the following value of the "quasiperiod"}%
\de{Die Weierstraß'sche $\zeta$-Funktion aus Def.~\ref{\en{EN}defizeta} ist zwar nicht doppeltperiodisch, aber immerhin ist der Wert der \glqq Quasiperiode\grqq}
$$\eta(\omega;L):=\zeta(z+\omega;L)-\zeta(z;L)$$
\en{is independent of the choice of $z$ (as long as $z\notin L$).}%
\de{unabhängig von der Wahl von $z$ (solange $z\notin L$).}
\end{thm}
\begin{proof}
\en{If we call the right hand side $R(z):=\zeta(z+\omega;L)-\zeta(z;L)$ and derive by $z$, we get from Def.~\ref{\en{EN}defiwp} that $R'(z)=-\wp(z+\omega;L)-(-\wp(z;L))=\wp(z;L)-\wp(z+\omega;L)$.
Prop.~\ref{\en{EN}wpdoppeltper} tells us that this is zero, so $R(z)$ is constant with respect to $z$.
The value of this constant thus depends only on the lattice $L$ and on the choice of $\omega$ -- and we can call it $\eta(\omega;L)$.}%
\de{Wenn wir die rechte Seite mit $R(z):=\zeta(z+\omega;L)-\zeta(z;L)$ bezeichnen und nach $z$ ableiten, erhalten wir nach Def.~\ref{\en{EN}defiwp}: $R'(z)=-\wp(z+\omega;L)-(-\wp(z;L))=\wp(z;L)-\wp(z+\omega;L)$.
Nach Satz~\ref{\en{EN}wpdoppeltper} ist das Null, also ist $R(z)$ konstant bezüglich $z$. Diese Konstante hängt dann noch vom Gitter $L$ und von der Wahl von $\omega$ ab -- und wir können sie mit $\eta(\omega;L)$ bezeichnen.}
\end{proof}

\begin{defi}\label{\en{EN}defetak}
\en{The following values $\eta_1(L)$ and $\eta_2(L)$ are called "basic quasiperiods" of the lattice $L=\mathbb Z \omega_1 + \mathbb Z \omega_2$:}%
\de{Man nennt die folgenden Werte $\eta_1(L)$ und $\eta_2(L)$ auch \glqq Basis-Quasi"-perioden\grqq~ des Gitters $L=\mathbb Z \omega_1 + \mathbb Z \omega_2$:}
$$\eta_k(L):=\zeta(z+\omega_k;L)-\zeta(z;L)$$
\en{Remark: by applying Prop.~\ref{\en{EN}eta} repeatedly, we see that these two values generate all other quasiperiods $\eta(\omega;L)$ like a lattice -- this is the reason they are called "basic" quasiperiods.}%
\de{Bemerkung: durch wiederholte Anwendung der Beziehung aus Satz~\ref{\en{EN}eta} folgt, dass diese beiden Werte sämtliche anderen Quasiperioden $\eta(\omega;L)$ wie ein Gitter erzeugen, deshalb heißen sie \glqq Basis\grqq-Quasiperioden.}
\end{defi}

\begin{bem}\label{\en{EN}bempitch}
\en{In Def.~\ref{\en{EN}defetak} we see that $\eta_k$ gives the difference in the value of the $\zeta$-function if the argument is changed by $\omega_k$.
The commonly used terms "period of the lattice" and "quasiperiod of the lattice" are thus a bit inappropriate or misleading:
\begin{itemize}
    \item Instead of "period of the lattice $L$", $\omega_k$ should be called "period of the associated $\wp$-function".
    \item Instead of "quasiperiod of the lattice $L$", $\eta_k$ should be called "pitch of the associated $\zeta$-function" (cf.~"pitch of a helix").
\end{itemize}
Anyway, we will continue with the commonly used terms.}%
\de{In Definition~\ref{\en{EN}defetak} erkennen wir, dass $\eta_k$ angibt, um wie viel der Wert der $\zeta$-Funktion zunimmt, wenn man das Argument um $\omega_k$ verändert.
Die in der Literatur üblichen Begriffe \glqq Periode des Gitters\grqq~und \glqq Quasiperiode des Gitters\grqq~sind also eigentlich unpassend:
\begin{itemize}
    \item Für $\omega_k$ sollte man statt \glqq Periode des Gitters $L$\grqq~besser \glqq Periode der $\wp$-Funktion zum Gitter $L$\grqq~sagen.
    \item Für $\eta_k$ sollte man statt \glqq Quasiperiode des Gitters $L$\grqq~besser \glqq Gewindesteigung oder Ganghöhe der $\zeta$-Funktion zum Gitter $L$\grqq~sagen.
\end{itemize}
Wir werden trotzdem die üblichen Begriffe weiter verwenden.}
\end{bem}

\begin{thm}[\en{Legendre's relation}\de{Legendre'sche Relation}]\label{\en{EN}legendre}
\en{For the basic periods and the associated basic quasi\-periods of a lattice $L=\mathbb Z \omega_1 + \mathbb Z \omega_2$ it holds:}%
\de{Für die Basisperioden und zugehörigen Basisquasiperioden des Gitters $L=\mathbb Z \omega_1 + \mathbb Z \omega_2$ gilt die sogenannte \glqq Legendre'sche Relation\grqq :} $$\eta_1\omega_2-\eta_2\omega_1 = 2\pi i$$
\end{thm}
\begin{proof}
\en{We shift the fundamental parallelogram $\overline{\mathcal P}$ (Def.~\ref{\en{EN}fund}), so that there are no lattice points on the border of $\overline{\mathcal P}_v=\overline{\mathcal P}+v$ with $v\in\mathbb C$. Then the residue theorem yields}%
\de{Wir verschieben das Periodenparallelogramm $\overline{\mathcal P}$ (Def.~\ref{\en{EN}fund}) so, dass keine Gitterpunkte auf dem Rand liegen: $\overline{\mathcal P}_v=\overline{\mathcal P}+v$ mit $v\in\mathbb C$. Dann liefert der Residuensatz:}
$$\oint_{\delta\overline{\mathcal P}_v} \zeta(z)dz = 2 \pi i,$$
\en{because the $\zeta$-function has (modulo $L$) only one pole with residue $1$ (see Def.~\ref{\en{EN}defizeta}).
On the other hand one can combine the values of the integrals along opposite sides (using Def.~\ref{\en{EN}defetak}):
The sides parallel to $\omega_1$ contribute $-\eta_2\omega_1$, the sides parallel to $\omega_2$ contribute $\eta_1\omega_2$. In total, the value of the integral is $\eta_1\omega_2-\eta_2\omega_1 = 2\pi i$.}%
\de{weil die $\zeta$-Funktion (modulo $L$) nur einen Pol mit Residuum $1$ hat (siehe Def.~\ref{\en{EN}defizeta}).
Andererseits kann man unter Verwendung der Definition~\ref{\en{EN}defetak} der Quasiperioden die Integrale längs gegenüberliegender Kanten zusammenfassen: die beiden zu $\omega_1$ parallelen Kanten liefern einen Beitrag von $-\eta_2\omega_1$, die beiden zu $\omega_2$ parallelen Kanten liefern $\eta_1\omega_2$.
Insgesamt lautet der Wert des Integrals also $\eta_1\omega_2-\eta_2\omega_1 = 2\pi i$.}
\end{proof}

\begin{defi}\label{\en{EN}defX}
\en{Let $g_2$ and $g_3$ be two complex numbers. Then}%
\de{Seien $g_2$ und $g_3$ zwei komplexe Zahlen. Dann ist}
$$X(g_2,g_3) := \left\{\left.(x,y)\in\mathbb C^2~\right|~y^2 = 4x^3 - g_2 x - g_3\right\}$$
\en{is an example of a "plane affine algebraic curve". Given a lattice $L$, we use $g_2=g_2(L)$ and $g_3=g_3(L)$ as in Prop.~\ref{\en{EN}dglP} and obtain:}%
\de{ein Beispiel einer \glqq ebenen affinen Kurve\grqq. Wenn ein Gitter $L$ gegeben ist, dann kann man mit $g_2=g_2(L)$ und $g_3=g_3(L)$ aus Satz~\ref{\en{EN}dglP} auch schreiben:}
$$X(L)=X(g_2(L),g_3(L))$$
\end{defi}

\begin{thm}\label{\en{EN}EbeneAffineKurve}
\en{The mapping $\Phi$ with}%
\de{Die Zuordnung $\Phi$ mit}
$$\begin{aligned}
\Phi: \mathbb C - L~& \to~  X(g_2(L),g_3(L))\subset\mathbb C^2\\
 z ~~~~&\mapsto~ (\wp(z;L),\wp'(z;L))
\end{aligned}$$
\en{is well-defined, differentiable and doubly periodic.}%
\de{ist wohldefiniert, differenzierbar und doppeltperiodisch.}
\end{thm}
\begin{proof}
\en{From the differential equation of the $\wp$-function (Prop.~\ref{\en{EN}dglP}) and the compatible definition of $X(g_2(L),g_3(L))$ we get that $\Phi$ is well-defined.
Since both $\wp$ and $\wp'$ are doubly periodic and differentiable, the same holds for $\Phi$.}%
\de{Die Wohldefiniertheit folgt aus der algebraischen Differentialgleichung der $\wp$-Funk"-tion in Satz~\ref{\en{EN}dglP} und der dazu passend gewählten Definition von $X(g_2(L),g_3(L))$. $\Phi$ ist doppeltperiodisch und differenzierbar, weil sowohl $\wp$ als auch $\wp'$ doppeltperiodisch und differenzierbar sind.}
\end{proof}

\begin{defi}\label{\en{EN}defwege}
\en{Let $L=\mathbb Z \omega_1 + \mathbb Z \omega_2$ be a lattice with basic periods $\omega_1$ and $\omega_2$. Then we define the paths $\beta_1$ and $\beta_2$ as follows:}%
\de{Es sei $L=\mathbb Z \omega_1 + \mathbb Z \omega_2$ ein Gitter mit Basisperioden $\omega_1$ und $\omega_2$. Dann definieren wir die beiden Wege $\beta_1$ und $\beta_2$ wie folgt:}
\begin{align*}
    \beta_1(t):=\frac 1 4 \cdot \omega_2 + t \cdot \omega_1 \qquad\text{\en{for}\de{für}~ } 0\leq t\leq 1\\
    \beta_2(t):=\frac 1 4 \cdot \omega_1 + t \cdot \omega_2 \qquad\text{\en{for}\de{für}~ } 0\leq t\leq 1
\end{align*}
\end{defi}

\begin{bem}\label{\en{EN}bemwege}
\en{The paths $\beta_k$ from Def.~\ref{\en{EN}defwege} are shown in Fig.~\ref{\en{EN}abbwege}.
On these paths, there are no poles of $\wp$ and $\wp'$ (black dots in the figure) and no zeros of $\wp'$ (circles in the figure, cf.~Prop.~\ref{\en{EN}zerowp}).}%
\de{Die in~\ref{\en{EN}defwege} definierten Wege $\beta_k$ sind in Abbildung~\ref{\en{EN}abbwege} zu sehen.
Auf den Wegen liegen keine Polstellen von $\wp$ und $\wp'$ (schwarze Punkte im Bild); und keine Nullstellen von $\wp'$ (Kreise im Bild, vgl.~Satz~\ref{\en{EN}zerowp}).}

\begin{figure}[ht]\begin{tikzpicture}[x=25mm, y=25mm] 
   \path[fill=lightgray] (0,0) -- (2,0) -- (2.3,1) -- (0.3,1) -- cycle;
   \draw[fill=black] (0,0) circle (1.5pt) node[right] {$0$};
   \draw[fill=black] (2,0) circle (1.5pt) node[right] {$\omega_1$};
   \draw[fill=black] (0.3,1) circle (1.5pt) node[right] {$\omega_2$};
   \draw[fill=black] (2.3,1) circle (1.5pt) node[right] {$\omega_1+\omega_2$};
   \draw             (0.15,0.5) circle (1.5pt) node[right] {$\frac{\omega_2}{2}$};
   \draw             (1.15,0.5) circle (1.5pt) node[right] {$\frac{\omega_1+\omega_2}{2}$};
   \draw             (1,0) circle (1.5pt) node[right] {$\frac{\omega_1}{2}$};
   \draw             (2.15,0.5) circle (1.5pt);
   \draw             (1.3,1) circle (1.5pt);
   \draw[->,dashed,thick] (0.075,0.25) -- (2.075,0.25) ;
   \draw[->,dashed,thick] (0.5,0) -- (0.8,1) ;
   \draw (1.8,0.22) node[above] {$\beta_1$};
   \draw (0.75,0.75) node[left] {$\beta_2$};
\end{tikzpicture}
\caption{\en{Fundamental parallelogram $\overline{\mathcal P}$ (cf.~Def.~\ref{\en{EN}fund}) with lattice points (poles of $\wp$ and $\wp'$) and half lattice point (zeros of $\wp'$) and paths $\beta_k$ from Def.~\ref{\en{EN}defwege}.}%
\de{Periodenparallelogramm $\overline{\mathcal P}$ (vgl.~Def.~\ref{\en{EN}fund}) mit Gitterpunkten (Polstellen von $\wp$ und $\wp'$) und halben Gitterpunkten (Nullstellen von $\wp'$) und den Wegen $\beta_k$ aus Def.~\ref{\en{EN}defwege}.}}\label{\en{EN}abbwege}
\end{figure}
\end{bem}

\begin{thm}\label{\en{EN}satzint}
\en{Let}\de{Es sei} $L=\mathbb Z \omega_1 + \mathbb Z \omega_2$.
\en{Then we use the paths $\beta_k$ from Def.~\ref{\en{EN}defwege} to define two new paths $\alpha_k:=(\wp(\beta_k),\wp'(\beta_k))$.
These $\alpha_k$ are closed paths in the plane affine algebraic curve $X(g_2(L),g_3(L))$.
Here, the basic periods and basic quasiperiods of the lattice admit the following representation by elliptic integrals:}%
\de{Dann definieren wir mit Hilfe der Wege $\beta_k$ aus Def.~\ref{\en{EN}defwege} zwei Wege $\alpha_k:=(\wp(\beta_k),\wp'(\beta_k))$.
Diese $\alpha_k$ sind geschlossene Wege durch die ebene affine Kurve $X(g_2(L),g_3(L))$.
Die Basisperioden bzw.~Basisquasiperioden des Gitters kann man dann durch die folgenden elliptischen Integrale darstellen:}
\begin{align*}
    \omega_k = \oint_{\alpha_k} \frac{dx}{y} \qquad \text{ \en{and}\de{sowie} }\qquad \eta_k(L)  = -\oint_{\alpha_k} \frac{x~dx}{y}
\end{align*}
\end{thm}
\begin{proof}\belowdisplayskip=-12pt
\en{The paths $\alpha_k$ are indeed paths in $X(g_2(L),g_3(L))$, because the differential equation from Prop.~\ref{\en{EN}dglP} guarantees, that the equation from Def.~\ref{\en{EN}defX} is fulfilled everywhere on $\alpha_k$.
From $\beta_k(1)=\beta_k(0)+\omega_k$ we deduce $\wp(\beta_k(0))=\wp(\beta_k(1))$ and the same for $\wp'$. Thus it holds $\alpha_k(0)=\alpha_k(1)$ and the paths $\alpha_k$ are closed.
With $(x,y)=(\wp(z),\wp'(z))$ along the paths $\alpha_k$ we get $\frac{dx}{dz}=\wp'(z)$ and thus}%
\de{Die Wege $\alpha_k$ sind tatsächlich Wege durch $X(g_2(L),g_3(L))$, weil mit der Differentialgleichung der $\wp$-Funktion (Satz~\ref{\en{EN}dglP}) folgt, dass die definierende Gleichung aus Def.~\ref{\en{EN}defX} für alle Punkte auf $\alpha_k$ erfüllt ist.
Aus $\beta_k(1)=\beta_k(0)+\omega_k$ folgt $\wp(\beta_k(0))=\wp(\beta_k(1))$ und Gleiches für $\wp'$. Also ist $\alpha_k(0)=\alpha_k(1)$ und somit sind die Wege $\alpha_k$ geschlossen.
Mit $(x,y)=(\wp(z),\wp'(z))$ entlang der Wege $\alpha_k$ folgt dann $\frac{dx}{dz}=\wp'(z)$ und somit}
\begin{align*}
    \oint_{\alpha_k} \frac{dx}{y}
    &=\int_{\beta_k}\frac{\wp'(z)dz}{\wp'(z)}
    =\int_{\beta_k} dz = \beta_k(1)-\beta_k(0) = \omega_k\\
    \text{\en{and}\de{und}}\qquad-\oint_{\alpha_k} \frac{x~dx}{y}
    &=-\int_{\beta_k}\frac{\wp(z)\wp'(z)dz}{\wp'(z)}
    =\int_{\beta_k} -\wp(z) dz\\
    &=\int_{\beta_k} \zeta'(z) dz
    = \zeta(z+\omega_k;L)-\zeta(z;L) = \eta_k(L)
\end{align*}
\end{proof}

\vfill\pagebreak\section{\en{Equivalent Lattices and Klein's Absolute Invariant \texorpdfstring{$J$}{\emph{J}}}\de{Äquivalente Gitter und die absolute Invariante \texorpdfstring{$J$}{\emph{J}}}}\label{\en{EN}kapGitter}
\renewcommand{\leftmark}{\en{Equivalent Lattices and Klein's Absolute Invariant $J$}\de{Äquivalente Gitter und die absolute Invariante $J$}}
\en{In this chapter we will see that two lattices that are rotated and/or scaled versions of each other can be called "equivalent" and that equivalent lattices have the same value of Klein's absolute invariant~$J$.}%
\de{In diesem Kapitel werden wir sehen, dass zwei Gitter, die durch eine Drehstreckung auseinander hervorgehen, \glqq äqui"-valent\grqq~genannt werden können und dass äquivalente Gitter die gleiche absolute Invariante $J$ haben.}

\begin{defi}
\en{Two lattices $L,L'\subset \mathbb C$ are called "equivalent", iff they can be obtained from each other by rotation and scaling, i.e.~iff there is $a\in\mathbb C$ with $L'=a\cdot L$ and $a\neq 0$.}%
\de{Zwei Gitter $L$ und $L'$, die durch eine Drehstreckung $L'=a\cdot L$ mit $a\in\mathbb C$ auseinander hervorgehen $(a\neq 0)$, heißen äquivalent.}
\end{defi}

\begin{bem}
\en{Any elliptic function $f(z)$ of the lattice $L$ yields an elliptic function $g(z)=f\lk\frac z a\rk$ of the lattice $L'=a\cdot L$ and vice versa. That's why we call $L$ and $L'$ equivalent.}%
\de{Für jede elliptische Funktion $f(z)$ zum Gitter $L$ ist $g(z)=f\lk\frac z a\rk$ eine elliptische Funktion zum Gitter $L'=a\cdot L$. Deshalb nennt man $L$ und $L'$ auch äquivalent.}
\end{bem}

\begin{thm}\label{\en{EN}ltau}
\en{For each lattice $L=\mathbb Z \omega_1 + \mathbb Z \omega_2$ there is an equivalent lattice $L_\tau=\mathbb Z + \mathbb Z \tau$ with $\tau$ from the upper half plane $\mathbb H$ (i.e.~$\im(\tau)>0$).}%
\de{Zu jedem Gitter $L=\mathbb Z \omega_1 + \mathbb Z \omega_2$ gibt es ein äquivalentes Gitter $L_\tau=\mathbb Z + \mathbb Z \tau$, wobei $\tau$ in der oberen Halbebene $\mathbb H$ liegt.}
\end{thm}
\begin{proof}
\en{Choose}\de{Wähle den Faktor} $a=\frac{1}{\omega_1}$, \en{then we get}\de{dann gilt} $L'=a\cdot L = \mathbb Z +\mathbb Z \cdot \frac{\omega_2}{\omega_1}$.
\en{If}\de{Falls} $\im\lk\frac{\omega_2}{\omega_1}\rk>0$\en{, then we set}\de{ ist, setzen wir} $\tau=\frac{\omega_2}{\omega_1}$.
\en{If}\de{Falls} $\im\lk\frac{\omega_2}{\omega_1}\rk<0$\en{, then we set}\de{ ist, setzen wir} $\tau=-\frac{\omega_2}{\omega_1}$ \en{(this is still the same lattice, only another basic period).}\de{(das ist immer noch das gleiche Gitter, nur ein anderer Basisvektor).}
\en{The case}\de{Der Fall} $\im\lk\frac{\omega_2}{\omega_1}\rk=0$ \en{is impossible, since it would yield}\de{ist ausgeschlossen, weil sonst} $\frac{\omega_2}{\omega_1}\in\mathbb R$ \en{and $L$ wouldn't be a lattice (cf.~Def.~\ref{\en{EN}defgitter}).}\de{wäre und somit $L$ kein Gitter wäre (vgl.~Def.~\ref{\en{EN}defgitter}).}
\end{proof}

\begin{defi}\label{\en{EN}defimodtau}
\en{We call $\tau_1\in \mathbb H$ and $\tau_2\in \mathbb H$ "equivalent", iff the lattices $L_{\tau_1}$ and $L_{\tau_2}$ are equivalent.
For example, $\tau$ and $\tau+1$ are equivalent (since they generate the same lattice), but also $\tau$ and $-1/\tau$ are equivalent (since it holds $L_{-1/\tau}=1/\tau\cdot L_\tau$).
This explains why such equivalent $\tau_1$ and $\tau_2$ are also called "equivalent under modular transformations".}%
\de{Wir nennen $\tau_1\in \mathbb H$ und $\tau_2\in \mathbb H$ \glqq äquivalent\grqq, wenn die Gitter $L_{\tau_1}$ und $L_{\tau_2}$ äquivalent sind.
Beispielsweise sind $\tau$ und $\tau+1$ äquivalent (weil sie das gleiche Gitter erzeugen), aber auch $\tau$ und $-1/\tau$ sind äquivalent (weil $L_{-1/\tau}=1/\tau\cdot L_\tau$ ist).\\
Statt \glqqäquivalent\grqq~sagt man daher auch \glqqäquivalent unter Modultransformationen\grqq.}
\end{defi}

\begin{defi}\label{\en{EN}defijdelta}
\en{Given a lattice $L\subset\mathbb C$. Using the definitions of $g_2(L)$ and $g_3(L)$ from Prop.~\ref{\en{EN}dglP} we define the "discriminant" $\Delta$ and Klein's absolute invariant $J$ of the lattice:}%
\de{Gegeben sei ein Gitter $L\subset\mathbb C$. Mit den Definitionen von $g_2(L)$ und $g_3(L)$ aus Satz~\ref{\en{EN}dglP} definieren wir die \glqq Diskriminante\grqq\ $\Delta$ des Gitters und die \glqq absolute Invariante\grqq\ $J$ des Gitters:}
\begin{align*}
    \Delta(L) &:= g_2^3(L)-27g_3^2(L)\\
    J(L)&:=\frac{g_2^3(L)}{g_2^3(L)-27g_3^2(L)}
\end{align*}
\end{defi}

\begin{bem}
\en{If the lattice is of the form $L_\tau=\mathbb Z + \mathbb Z \tau$, we denote $g_2(\tau)$ instead of $g_2(L_\tau)$. In the same way, we write $g_3(\tau)$, $G_k(\tau)$, $\Delta(\tau)$ and $J(\tau)$.}%
\de{Wenn wir uns auf ein Gitter der Form $L_\tau=\mathbb Z + \mathbb Z \tau$ beziehen, schreiben wir kurz $g_2(\tau)$ statt $g_2(L_\tau)$. Ebenso schreiben wir abkürzend $g_3(\tau)$, $G_k(\tau)$, $\Delta(\tau)$ und $J(\tau)$.}
\end{bem}

\begin{thm}\label{\en{EN}trafog23}
\en{If the lattice $L'=a\cdot L$ is equivalent to $L$, then the following transformation formula for the Eisenstein series holds for $a\neq 0$:}%
\de{Wenn $L'=a\cdot L$ ein zu $L$ äquivalentes Gitter mit $a\neq 0$ ist, dann gilt folgende Transformationsformel für die Eisensteinreihen:}
$$G_k(aL) = a^{-k}\cdot G_k(L)$$
\en{and thus:}\de{und folglich:}
$$g_2(aL)=a^{-4}g_2(L)\qquad\text{\en{and}\de{und}}\qquad g_3(aL)=a^{-6}g_3(L)$$
\en{From this we get}\de{Hieraus wiederum folgt}
$$\Delta(aL)=a^{-12}\Delta(L)\qquad\text{\en{and}\de{und}}\qquad J(aL)=J(L)$$
\en{In particular Klein's absolute invariant $J$ has the same value if the lattices are equivalent -- this is why $J$ is called "invariant".}%
\de{Insbesondere ändert sich der Wert der absoluten Invariante $J$ nicht, wenn man zu einem äquivalenten Gitter übergeht. Das rechtfertigt den Namen Invariante.}
\end{thm}
\begin{proof}
\en{This is a consequence of the Def.~\ref{\en{EN}defeisen} of the Eisenstein series:}%
\de{Das ist eine Folge aus der Definition~\ref{\en{EN}defeisen} der Eisensteinreihen:}
$$G_k(aL)=\sum_{\substack{\omega'\in aL\\ \omega'\neq 0}}\omega'^{-k}=\sum_{\substack{\omega'\in aL\\ \omega'\neq 0}}\left(\frac{\omega'} a\right)^{-k}\cdot a^{-k} = a^{-k}\cdot \sum_{\substack{\omega\in L\\ \omega\neq 0}}\omega^{-k}= a^{-k}\cdot G_k(L)$$
\en{Here we used}\de{wobei} $\omega'=a\cdot\omega$\en{. This yields, with the definitions of $g_{2;3}$ from Prop.~\ref{\en{EN}dglP}, that}\de{ verwendet wurde. Es folgt} $g_2(aL)=60G_4(aL)=a^{-4}\cdot g_2(L)$\en{ and}\de{\linebreak
sowie} $g_3(aL)=140G_6(aL)= a^{-6}\cdot g_3(L)$.
\en{Finally, we get the discriminant}\de{Schließlich erhalten wir die Diskriminante}
$\Delta(aL)=(a^{-4})^3g_2^3(L) - 27 (a^{-6})^2g_3^2(L)=a^{-12}\cdot\Delta(L)$ \en{and Klein's absolute invariant $J(aL)=J(L)$, which shows that it doesn't change when the lattice is rotated and/or stretched.}\de{und die absolute Invariante $J(aL)=J(L)$, die sich also bei einer Drehstreckung des Gitters nicht ändert.}
\end{proof}

\begin{thm}\label{\en{EN}etatransf}
\en{For the basic periods and basic quasi periods of $L'=a\cdot L$ it holds:}%
\de{Für die Perioden und Quasiperioden von $L'=a\cdot L$ gilt:}
$$\omega_k'=a\cdot\omega_k\qquad\text{\en{and}\de{und}}\qquad \eta_k(L') = \frac 1 a \cdot \eta_k(L).$$
\end{thm}
\begin{proof}
\en{The first identity is proven by multiplicating the lattice with $a$.
Then, by Def.~\ref{\en{EN}defetak} and Prop.~\ref{\en{EN}eta}, it holds for any $z\in\mathbb C\setminus(L\cup L')$:}%
\de{Die erste Gleichung folgt aus der Multiplikation des Gitters mit der Zahl $a$.
Weiter folgt mit Def.~\ref{\en{EN}defetak} und Satz~\ref{\en{EN}eta} für beliebiges $z\in\mathbb C\setminus(L\cup L')$:}
\begin{align*}
    \eta_k(L') &= \eta_k(aL) = \zeta(z+a\omega_k;aL)-\zeta(z;aL) = \zeta(az+a\omega_k;aL)-\zeta(az;aL)
\end{align*}
\en{Next we use the Definition~\ref{\en{EN}defizeta} of the Weierstraß $\zeta$-function and get:}%
\de{Dann verwenden wir die Definition~\ref{\en{EN}defizeta} der Weierstraß'schen $\zeta$-Funktion und erhalten:}
\begin{align*}
\zeta(az;aL) &= \frac 1 {az} + \sum_{\substack{\omega\in aL\\ \omega\neq 0}} \left(\frac 1 {az-\omega} +\frac 1 \omega +\frac {az} {\omega^2}\right)
\end{align*}
\en{Now we change summation variables by setting $v := \omega/a$. Then, from $\omega\in aL$, we get $v\in L$ and thus:}%
\de{Jetzt folgt ein Variablenwechsel $v := \omega/a$. Aus $\omega\in aL$ folgt dann $v\in L$ und somit:}
\begin{align*}
    \zeta(az;aL)&=\frac 1 {az} + \sum_{\substack{v\in L\\ v\neq 0}} \left(\frac 1 {az-av} +\frac 1 {av} +\frac {az} {(av)^2}\right)=\frac 1 a \zeta(z;L)
\end{align*}
\en{In the same way (i.e.~setting $v := \omega/a$) we get $\zeta(az+a\omega_k;aL)=\frac 1 a \zeta(z+\omega_k;L)$ and}%
\de{Derselbe Variablenwechsel liefert analog $\zeta(az+a\omega_k;aL)=\frac 1 a \zeta(z+\omega_k;L)$ und}\belowdisplayskip=-12pt
\begin{align*}
    \eta_k(L') &= \zeta(az+a\omega_k;aL)-\zeta(az;aL) = \frac 1 a \zeta(z+\omega_k;L)-\frac 1 a \zeta(z;L) = \frac 1 a \cdot \eta_k(L)
\end{align*}
\end{proof}

\begin{defi}\label{\en{EN}definLJ}
\en{Given the lattice $L_\tau=\mathbb Z  + \mathbb Z \tau$, we define the equivalent lattice $L_J$ by}%
\de{Gegeben ist das Gitter $L_\tau=\mathbb Z  + \mathbb Z \tau$. Dann definieren wir ein zu $L_\tau$ äquivalentes Gitter $L_J$ durch}
\begin{align*}
    L_J &:= \mu(\tau) \cdot L_\tau\qquad\text{ \en{with}\de{mit} }\qquad\mu(\tau):=\sqrt{\frac{g_3(L_\tau)}{g_2(L_\tau)}}
\end{align*}
\en{From chapter~\ref{\en{EN}kappicardfuchs} onward, we will denote the basic periods of $L_J$ with $(\Omega_1,\Omega_2)$, and the corresponding basic quasi periods $\eta_k(L_J)$ will be called $(H_1,H_2)$.}%
\de{Die Basisperioden von $L_J$ bezeichnen wir im Folgenden mit $(\Omega_1,\Omega_2)$ und die zugehörigen Basisquasiperioden $\eta_k(L_J)$ bezeichnen wir mit $(H_1,H_2)$.}
\end{defi}

\begin{bem}\label{\en{EN}bemlj}
\en{It doesn't matter which branch of the square root is being chosen when calculating $\mu(\tau)$, because the negated basic periods generate the same lattice:}%
\de{Es ist egal, für welchen Zweig der Quadratwurzel man sich bei $\mu(\tau)$ entscheidet, denn negierte Basisperioden erzeugen das gleiche Gitter:}
$$\mathbb Z \omega_1 + \mathbb Z \omega_2 = \mathbb Z\cdot(-\omega_1) + \mathbb Z\cdot(-\omega_2).$$
\end{bem}

\begin{thm}\label{\en{EN}satz44}
\en{The plane affine algebraic curve $X(L_J)$ has a representation that depends only on the value of Klein's absolute invariant $J$ (that's why the lattice is called $L_J$). This representation reads:}%
\de{Die ebene affine Kurve zum Gitter $L_J$ hat eine Darstellung, die nur von der absoluten Invarianten $J$ des Gitters $L_J$ abhängt (deshalb nennt man das Gitter $L_J$). Diese lautet:}
$$X(L_J)=\left\{~(x,y)\in\mathbb C^2~\left|~ y^2 = 4 x^3 - \frac{27J}{J-1}( x +1)\right.~\right\}$$
\end{thm}
\begin{proof}
\en{From the transformation formula of $g_2$ and $g_3$ in Prop.~\ref{\en{EN}trafog23} and using $L_J=\mu(\tau)\cdot L_\tau$ we get:}%
\de{Aus dem Transformationsverhalten der $g_2$ und $g_3$ in Satz~\ref{\en{EN}trafog23} folgt mit $L_J=\mu(\tau)\cdot L_\tau$:}
\begin{align*}
        g_2(L_J) = \mu(\tau)^{-4}\cdot g_2(L_\tau) = \frac{g_2(L_\tau)^2}{g_3(L_\tau)^2}\cdot g_2(L_\tau) = \frac{g_2(L_\tau)^3}{g_3(L_\tau)^2}\\
        g_3(L_J) = \mu(\tau)^{-6}\cdot g_3(L_\tau) = \frac{g_2(L_\tau)^3}{g_3(L_\tau)^3}\cdot g_3(L_\tau) = \frac{g_2(L_\tau)^3}{g_3(L_\tau)^2}
\end{align*}
\en{In the lattice $L_J$ we thus have $g_2(L_J)=g_3(L_J)=:g$. Then we get the value of Klein's absolute invariant of $L_J$ from Def.~\ref{\en{EN}defijdelta}:}%
\de{Im Gitter $L_J$ gilt also $g_2(L_J)=g_3(L_J)=:g$. Die absolute Invariante des Gitters $L_J$ ist dann nach Definition~\ref{\en{EN}defijdelta}:}
$$J = \frac{g^3}{g^3-27g^2} = \frac{g}{g-27} \quad\Longrightarrow\quad g = \frac{27J}{J-1}$$
\en{which yields said equation of the plane affine algebraic curve $X(L_J)$:}%
\de{und wir erhalten die angekündigte ebene affine Kurve mit der Gleichung}\belowdisplayskip=-12pt
$$y^2 = 4 x^3 - g_2(L_J) x - g_3(L_J) = 4x^3 - g (x+1) = 4x^3 - \frac{27J}{J-1}(x+1).$$
\end{proof}

\vfill\pagebreak\section{\en{Fourier Representations of the Eisenstein Series}\de{Fourierentwicklungen der Eisensteinreihen}}\label{\en{EN}kapFourier}
\renewcommand{\leftmark}{\en{Fourier Representations of the Eisenstein Series}\de{Fourierentwicklungen der Eisensteinreihen}}
\en{In this chapter, we prove some properties of the Weierstraß $\sigma$-function and the Fourier representations of Thm.~\ref{\en{EN}fouriertheorem}.
The proof follows \cite[ch.~18, §1-3]{\en{EN}Lang1987}.}%
\de{In diesem Kapitel beweisen wir einige Eigenschaften der Weierstraß'schen $\sigma$-Funktion und die Fourierdarstellungen in Thm.~\ref{\en{EN}fouriertheorem}. Der Beweis folgt \cite[Kap.~18, §1-3]{\en{EN}Lang1987}.}

\begin{thm}\label{\en{EN}sigmaodd}
\en{The Weierstraß $\sigma$-function is an odd function: $\sigma(-z;L)=-\sigma(z;L)$.}%
\de{Die Weierstraß'sche $\sigma$-Funktion ist ungerade: $\sigma(-z;L)=-\sigma(z;L)$.}
\end{thm}
\begin{proof}
\en{We recall Definition~\ref{\en{EN}defisigma} from page~\pageref{\en{EN}defisigma}:}%
\de{Wir rufen uns Definition~\ref{\en{EN}defisigma} von Seite~\pageref{\en{EN}defisigma} in Erinnerung:}
$$\sigma(z;L):=z\cdot\prod_{\substack{\omega\in L\\\omega\neq 0}}\left\{\left(1-\frac z \omega \right)\cdot\exp\lk\frac z \omega + \frac 1 2 \left(\frac z \omega\right)^2\rk\right\}$$
\en{Here we realize that if $\omega$ runs through the whole lattice $L$, then $-\omega$ does the same. This yields $\sigma(-z;L)=-\sigma(z;L)$, where the additional minus sign comes from the factor $z$ in front of the product sign.}%
\de{Hier erkennen wir, dass mit $\omega$ auch $-\omega$ alle Punkte des Gitters $L$ durchläuft, und dass folglich $\sigma(-z;L)=-\sigma(z;L)$ gilt, wobei das zusätzliche Minuszeichen vom ersten Faktor $z$ stammt.}
\end{proof}

\begin{thm}\label{\en{EN}trafosigma}
\en{When translating $z$ by one of the basic periods of the lattice $L=\mathbb Z \omega_1 + \mathbb Z \omega_2$, the Weierstraß $\sigma$-function transforms as follows: }%
\de{Es gilt das folgende Transformationsverhalten der $\sigma$-Funktion bei Translation um eine der beiden Basisperioden des Gitters $L=\mathbb Z \omega_1 + \mathbb Z \omega_2$:}
$$\sigma(z+\omega_k)=-\exp\lk\eta_k\cdot\left(z+\frac{\omega_k}{2}\right)\rk\cdot \sigma(z)$$
\end{thm}
\begin{proof}
\en{From the Definitions~\ref{\en{EN}defizeta} and~\ref{\en{EN}defetak} of $\zeta$ and the $\eta_k$ we get:}%
\de{Zunächst folgt aus den Definitionen~\ref{\en{EN}defizeta} und~\ref{\en{EN}defetak} für $\zeta$ und die $\eta_k$:}
\begin{align*}
    \frac{d}{dz}\log\lk\frac{\sigma(z+\omega_k)}{\sigma(z)}\rk
    &= \frac{d}{dz}\log\lk\sigma(z+\omega_k)\rk - \frac{d}{dz}\log\lk\sigma(z)\rk\\
    &= \frac{\sigma'(z+\omega_k;L)}{\sigma(z+\omega_k;L)} - \frac{\sigma'(z;L)}{\sigma(z;L)} =\zeta(z+\omega_k;L)-\zeta(z;L) =\eta_k\\
    \Longrightarrow\quad \log\lk\frac{\sigma(z+\omega_k)}{\sigma(z)}\rk &= \eta_k\cdot z + c(\omega_k)\\
    \Longrightarrow\quad \frac{\sigma(z+\omega_k)}{\sigma(z)} &= \exp\lk\eta_k\cdot z + c(\omega_k)\rk = \exp\lk\eta_k\cdot z\rk\cdot\exp\lk c(\omega_k)\rk
\end{align*}
\en{We get the value of $\exp\lk c(\omega_k)\rk$ by setting $z=-\frac{\omega_k}{2}$ and using the fact that $-\frac{\omega_k}{2}\notin L_\tau$ and that $\sigma$ is odd (see Prop.~\ref{\en{EN}sigmaodd}):}%
\de{Den Wert $\exp\lk c(\omega_k)\rk$ bestimmen wir, indem wir $z=-\frac{\omega_k}{2}$ einsetzen. Dabei verwenden wir, dass dieser Wert nicht im Gitter $L_\tau$ liegt und dass $\sigma$ ungerade ist:}\belowdisplayskip=-6pt
\begin{align*}
     \exp\lk-\eta_k\cdot\frac{\omega_k}{2} + c(\omega_k)\rk
     &= \frac{\sigma\lk-\frac{\omega_k}{2}+\omega_k\rk}{\sigma\lk-\frac{\omega_k}{2}\rk}
     = \frac{\sigma\lk\frac{\omega_k}{2}\rk}{\sigma\lk-\frac{\omega_k}{2}\rk}
     = -1\quad\left|~\cdot~\exp\lk\eta_k\cdot\frac{\omega_k}{2}\rk\right.\\
     \Longrightarrow\quad \exp\lk c(\omega_k)\rk &= -\exp\lk\eta_k\cdot\frac{\omega_k}{2}\rk\\
     \Longrightarrow\quad \frac{\sigma(z+\omega_k)}{\sigma(z)}
     &= -\exp\lk\eta_k\cdot z\rk\cdot\exp\lk\eta_k\cdot\frac{\omega_k}{2}\rk
     = -\exp\lk\eta_k\cdot\left(z+\frac{\omega_k}{2}\right)\rk
\end{align*}
\end{proof}

\begin{thm}\label{\en{EN}satzphi}
 \en{We define the function}\de{Wir definieren die Funktion}
 $$\varphi(z;L_\tau):=\exp\lk-\frac {\eta_1}{2}\cdot z^2+i\pi z\rk\cdot\sigma(z;L_\tau)$$
 \en{using the first basic quasiperiod $\eta_1=\eta_1(L_\tau)$. For this function, it holds:}%
 \de{mit Hilfe der ersten Basisquasiperiode $\eta_1=\eta_1(L_\tau)$. Für diese Funktion gilt dann:}
 $$\varphi(z+1;L_\tau)=\varphi(z;L_\tau)\qquad\text{\en{and}\de{und}}\qquad\varphi(z+\tau;L_\tau) = -\exp(2\pi i z)\cdot \varphi(z;L_\tau)$$
\end{thm}
\begin{proof}
\en{We use the transformation formula of the $\sigma$-function from Prop.~\ref{\en{EN}trafosigma}:}%
\de{Wir verwenden das Transformationsverhalten der $\sigma$-Funktion aus Satz~\ref{\en{EN}trafosigma}:}
\begin{align*}
    \varphi(z+1;L_\tau) &= \exp\lk-\frac 1 2 \eta_1\cdot (z+1)^2+i\pi (z+1)\rk\cdot\sigma(z+1;L_\tau)\\
    &= -\exp\lk-\frac {\eta_1}{2}\cdot (z^2+2z+1) +i\pi (z+1)+\eta_1\cdot\left(z+\frac{\omega_1}{2}\right)\rk\cdot\sigma(z;L_\tau)\\
    &= -\exp\lk-\frac {\eta_1}{2}\cdot (2z+1) +i\pi +\eta_1\cdot\left(z+\frac{1}{2}\right)\rk\cdot\varphi(z;L_\tau)\\
    &= -\exp\lk i\pi\rk\cdot\varphi(z;L_\tau) = \varphi(z;L_\tau)
\end{align*}
\en{For the second basic period of $L_\tau$, Prop.~\ref{\en{EN}trafosigma} yields:}%
\de{Und für die andere Basisperiode von $L_\tau$ folgt ebenfalls mit Satz~\ref{\en{EN}trafosigma}:}
\begin{align*}
    \varphi(z+\tau;L_\tau) &= \exp\lk-\frac {\eta_1}{2}\cdot (z+\tau)^2+i\pi (z+\tau)\rk\cdot\sigma(z+\tau;L_\tau)\\
    &= -\exp\lk-\frac {\eta_1}{2}\cdot (z^2+2z\tau+\tau^2) +i\pi (z+\tau)+\eta_2\cdot\left(z+\frac{\omega_2}{2}\right)\rk\cdot\sigma(z;L_\tau)\\
    &= -\exp\lk-\frac {\eta_1}{2}\cdot (2z\tau+\tau^2) +i\pi \tau+\eta_2\cdot\left(z+\frac{\tau}{2}\right)\rk\cdot\varphi(z;L_\tau)\\
    &= -\exp\lk-\eta_1\cdot\tau \left(z+\frac \tau 2\right) +i\pi \tau+\eta_2\cdot\left(z+\frac{\tau}{2}\right)\rk\cdot\varphi(z;L_\tau)
\end{align*}
\en{Then we use Legendre's relation of the lattice $L_\tau$ from Prop.~\ref{\en{EN}legendre}, which reads $\eta_1\cdot\tau = 2\pi i + \eta_2$ (because $\omega_1=1$ and $\omega_2=\tau$). This yields:}%
\de{Jetzt verwenden wir noch die Legendre-Relation des Gitters $L_\tau$ aus Satz~\ref{\en{EN}legendre} in der Form $\eta_1\cdot\tau = 2\pi i + \eta_2$ (setze $\omega_1=1$ und $\omega_2=\tau$ ein) und erhalten:}\belowdisplayskip=0pt
\begin{align*}
    \varphi(z+\tau;L_\tau)
        &= -\exp\lk-(2\pi i + \eta_2)\cdot \left(z+\frac \tau 2\right) +i\pi \tau+\eta_2\cdot\left(z+\frac{\tau}{2}\right)\rk\cdot\varphi(z;L_\tau)\\
        &= -\exp\lk-2\pi i \left(z+\frac \tau 2\right) +i\pi \tau\rk\cdot\varphi(z;L_\tau)= -\exp\lk-2\pi i z\rk\cdot\varphi(z;L_\tau)
\end{align*}
\end{proof}

\begin{thm}\label{\en{EN}fouriersigma}
\en{The Weierstraß $\sigma$-function of the lattice $L_\tau$ admits the following Fourier series expansion with $q_z=e^{2\pi i z}$ and $q_\tau=e^{2\pi i \tau}$ and $\eta_1=\eta_1(L_\tau)$:}%
\de{Die Weierstraß'sche $\sigma$-Funktion zum Gitter $L_\tau$ hat die folgende Fourierentwicklung mit $q_z=e^{2\pi i z}$ und $q_\tau=e^{2\pi i \tau}$ und $\eta_1=\eta_1(L_\tau)$:}
$$\sigma(z;\tau)=\frac 1 {2\pi i}e^{\eta_1 \cdot z^2/2}\cdot(q_z^{1/2}-q_z^{-1/2})\cdot\prod_{n=1}^\infty\frac{(1-q_\tau^n q_z)(1-q_\tau^n / q_z)}{(1-q_\tau^n)^2}$$
\end{thm}
\begin{proof}
\en{First we will prove that for the $\varphi$-function from Prop.~\ref{\en{EN}satzphi} it holds:}%
\de{Wir beweisen zunächst, dass für die in Satz~\ref{\en{EN}satzphi} definierte $\varphi$-Funktion gilt:}
\begin{align}
    \varphi(z;L_\tau)=\frac {q_z-1} {2\pi i}\cdot\prod_{n=1}^\infty\frac{(1-q_\tau^n q_z)(1-q_\tau^n / q_z)}{(1-q_\tau^n)^2} \label{\en{EN}glgphi}
\end{align}
\en{We denote the right hand side of~(\ref{\en{EN}glgphi}) by $g(z;L_\tau)$ and prove that $g(z;L_\tau)=\varphi(z;L_\tau)$:
From $q_{z+1}=e^{2\pi i (z+1)} = e^{2\pi i z}= q_z$ we get $g(z+1;L_\tau)=g(z;L_\tau)$, like with $\varphi(z;L_\tau)$ (cf.~Prop.~\ref{\en{EN}satzphi}).
Then it holds $q_{z+\tau}=q_z\cdot q_\tau$ and thus}%
\de{Die rechte Seite von~(\ref{\en{EN}glgphi}) nennen wir $g(z;L_\tau)$ und beweisen dann $g(z;L_\tau)=\varphi(z;L_\tau)$:
Es gilt $q_{z+1}=e^{2\pi i (z+1)} = e^{2\pi i z}= q_z$ und somit $g(z+1;L_\tau)=g(z;L_\tau)$, genau wie bei $\varphi(z;L_\tau)$ (vgl.~Satz~\ref{\en{EN}satzphi}). Außerdem ist $q_{z+\tau}=q_z\cdot q_\tau$ und somit}
\begin{align*}
    g(z+\tau;L_\tau) &=\frac {q_zq_\tau-1} {2\pi i}\cdot\prod_{n=1}^\infty\frac{(1-q_\tau^{n+1} q_z)(1-q_\tau^{n-1} / q_z)}{(1-q_\tau^n)^2}\\
    &= \frac {q_zq_\tau-1} {2\pi i}\cdot\frac{\left\{\prod_{n=2}^\infty(1-q_\tau^{n} q_z)\right\}\cdot\left\{\prod_{n=0}^\infty(1-q_\tau^n / q_z)\right\}}{\prod_{n=1}^\infty(1-q_\tau^n)^2}\\
    &= \frac {q_zq_\tau-1} {2\pi i}\cdot \frac{1-q_\tau^0/q_z}{1-q_\tau^1q_z}
    \cdot \prod_{n=1}^\infty\frac{(1-q_\tau^{n+1} q_z)(1-q_\tau^{n-1} / q_z)}{(1-q_\tau^n)^2}\\
    &= \frac {q_zq_\tau-1} {2\pi i}\cdot \frac{1-1/q_z}{1-q_\tau q_z}
    \cdot \frac {q_z-1}{q_z-1}\cdot\prod_{n=1}^\infty\frac{(1-q_\tau^{n+1} q_z)(1-q_\tau^{n-1} / q_z)}{(1-q_\tau^n)^2}\\
    &=\frac {q_zq_\tau-1} {q_z-1}\cdot \frac{1-1/q_z}{1-q_\tau q_z} \cdot g(z;L_\tau) =\frac {q_zq_\tau-1}{1-q_\tau q_z}\cdot \frac{q_z^{-1}(q_z-1)} {q_z-1} \cdot g(z;L_\tau)\\
    &= - q_z^{-1}\cdot  g(z;L_\tau) = -\exp(-2\pi i z) \cdot g(z;L_\tau)
\end{align*}
\en{Thus we have shown that $g$ and $\varphi$ act in the same way when transforming $z\rightarrow z+1$ or $z\rightarrow z+\tau$. This shows that $\frac{\varphi}g$ is doubly periodic with period lattice $L_\tau$.}%
\de{Somit haben wir bewiesen, dass sich $g$ und $\varphi$ bei $z\rightarrow z+1$ und bei $z\rightarrow z+\tau$ genau gleich verhalten, dass also $\frac{\varphi}g$ das Gitter $L_\tau$ als Periodengitter hat.}

\en{Now we analyze the zeros of $g(z;L_\tau)$. From the rule of zero product we deduce that $g(z;L_\tau)=0$ iff $q_z=q_\tau^m$ for any $m\in \mathbb Z$.
This yields the condition $e^{2\pi i z} = e^{2\pi i m \tau}$ for the zeros of $g$.
But since the natural exponential function has the complex period $2\pi i$, every $(l,m)\in\mathbb Z^2$ produces a zero of $g$:
$2\pi i \cdot z =  2\pi i \cdot l + 2\pi i \cdot m \tau $ or $z=l + m\tau$. Thus $g(z;L_\tau)$ has zeros of order one for $z\in L_\tau$, like $\sigma(z;L_\tau)$ (cf.~Remark~\ref{\en{EN}bemsigma}) and like $\varphi(z;L_\tau)$ (cf.~Prop.~\ref{\en{EN}satzphi}).
}%
\de{Nun kommen wir zu den Nullstellen von $g(z;L_\tau)$. Aus dem Satz vom Nullprodukt folgt, dass $g(z;L_\tau)=0$ genau dann gilt, wenn $q_z=q_\tau^m$ für ein $m\in \mathbb Z$.
Dies liefert die Gleichung $e^{2\pi i z} = e^{2\pi i m \tau}$. Aufgrund der komplexen Periode $2\pi i$ der $e$-Funktion erhalten wir für alle $(l,m)\in\mathbb Z^2$ eine Lösung:
$2\pi i \cdot z =  2\pi i \cdot l + 2\pi i \cdot m \tau $ bzw.~$z=l + m\tau$. Folglich hat $g(z;L_\tau)$ in allen Punkten des Gitters $L_\tau$ eine einfache Nullstelle, genau wie $\sigma(z;L_\tau)$ (siehe Bemerkung~\ref{\en{EN}bemsigma}) und somit genau wie $\varphi(z;L_\tau)$ (siehe Satz~\ref{\en{EN}satzphi}).}

\en{From its definition in Prop.~\ref{\en{EN}satzphi} we deduce that $\varphi$ has no poles, and thus $\frac{\varphi}{g}$ is an elliptic function without poles and must be constant (first Liouvielle theorem, Prop.~\ref{\en{EN}liouville1}).}%
\de{Weil zudem $\varphi$ aufgrund der Definition in Satz~\ref{\en{EN}satzphi} keine Polstellen hat, ist die elliptische Funktion $\frac{\varphi}{g}$ wegen des ersten Liouville'schen Satzes~\ref{\en{EN}liouville1} konstant.}

\en{Next we calculate this constant value for $z\rightarrow 0$.
There we have $q_z= 1+2\pi i z + O(z^2)$ and thus $g(z;L_\tau)\approx \frac{1+2\pi i z-1}{2\pi i}\cdot \prod_{n=1}^\infty\frac{(1-q_\tau^n)(1-q_\tau^n )}{(1-q_\tau^n)^2}=\frac {2\pi i z}{2\pi i}\cdot 1 = z$.
We deduce the same approximation $\sigma(z;L_\tau)\approx z$ around $z=0$ from the Definition~\ref{\en{EN}defisigma} of $\sigma(z;L_\tau)$.
Thus it holds $\varphi(z;L_\tau)\approx z$ around $z=0$ and we get:}%
\de{Den Wert der Konstante bestimmen wir für $z\rightarrow 0$.
Dort ist $q_z= 1+2\pi i z + O(z^2)$ und somit $g(z;L_\tau)\approx \frac{1+2\pi i z-1}{2\pi i}\cdot \prod_{n=1}^\infty\frac{(1-q_\tau^n)(1-q_\tau^n )}{(1-q_\tau^n)^2}=\frac {2\pi i z}{2\pi i}\cdot 1 = z$.
Die gleiche Näherung $\sigma(z;L_\tau)\approx z$ in der Nähe von $z=0$ erkennt man an der Def.~\ref{\en{EN}defisigma} von $\sigma(z;L_\tau)$.
Somit gilt auch $\varphi(z;L_\tau)\approx z$ in der Nähe von $z=0$ und somit:}
$$\frac{\varphi(z;L_\tau)}{g(z;L_\tau)}=\lim_{z\rightarrow 0}\frac{\varphi(z;L_\tau)}{g(z;L_\tau)}=\lim_{z\rightarrow 0} \frac z z = 1$$
\en{This yields $\varphi(z;L_\tau)=g(z;L_\tau)$, which proves the Fourier series expansion~(\ref{\en{EN}glgphi}) of $\varphi$.}%
\de{Somit haben wir $\varphi(z;L_\tau)=g(z;L_\tau)$ bewiesen, also die Gleichung~(\ref{\en{EN}glgphi}).}

\en{Finally, we use the definition of $\varphi$ from Prop.~\ref{\en{EN}satzphi} to write $\sigma$ in terms of $\varphi$:}%
\de{Wir müssen nur noch die Definition von $\varphi$ nach $\sigma$ auflösen und erhalten:}
\begin{align*}\sigma(z;L_\tau)&=\exp\lk\frac {\eta_1}{2}\cdot z^2-i\pi z\rk\cdot \varphi(z;L_\tau)\\
&=\exp\lk\frac {\eta_1}{2}\cdot z^2\rk\cdot q_z^{-\frac 1 2}\cdot \frac {q_z-1} {2\pi i}\cdot\prod_{n=1}^\infty\frac{(1-q_\tau^n q_z)(1-q_\tau^n / q_z)}{(1-q_\tau^n)^2} \\
&= \frac 1 {2\pi i}e^{\eta_1\cdot z^2/2}\cdot(q_z^{1/2}-q_z^{-1/2})\cdot\prod_{n=1}^\infty\frac{(1-q_\tau^n q_z)(1-q_\tau^n / q_z)}{(1-q_\tau^n)^2}
\end{align*}
\en{This is the Fourier series representation of the Weierstraß $\sigma$-function from Prop.~\ref{\en{EN}fouriersigma}.}%
\de{Das ist die Fourierdarstellung der Weierstraß'schen $\sigma$-Funktion aus Satz~\ref{\en{EN}fouriersigma}.}
\end{proof}

\begin{theo}\label{\en{EN}fouriertheorem}
\en{Let $q=e^{2\pi i \tau}$. Then $\im(\tau)>0$ yields $|q|<1$ and the following "normalized Eisenstein series" $E_2$, $E_4$ and $E_6$ converge absolutely:}%
\de{Wenn $\im(\tau)>0$ ist, dann gilt für $q=e^{2\pi i \tau}$, dass $|q|<1$ ist und somit die folgenden \glqq normierten Eisensteinreihen\grqq~ $E_2$, $E_4$ und $E_6$ absolut konvergieren:}
\begin{align*}
    E_2(\tau) &:= 1- 24 \sum_{n=1}^\infty n \frac{q^n}{1-q^n}\\
    E_4(\tau) &:= 1+ 240 \sum_{n=1}^\infty n^3 \frac{q^n}{1-q^n}\\
    E_6(\tau) &:= 1- 504 \sum_{n=1}^\infty n^5 \frac{q^n}{1-q^n}
\end{align*}
\en{Using these, we can equivalently redefine the functions $\eta_1(L_\tau)$, $g_2(L_\tau)$ and $g_3(L_\tau)$ which had previously been defined in Def.~\ref{\en{EN}defetak} and Prop.~\ref{\en{EN}dglP}:}%
\de{Mit ihrer Hilfe können wir die in Def.~\ref{\en{EN}defetak} und Satz~\ref{\en{EN}dglP} bereits anders definierten Ausdrücke $\eta_1(L_\tau)$, $g_2(L_\tau)$ und $g_3(L_\tau)$ äquivalent darstellen:}
\begin{align*}
    \eta_1(L_\tau)&=\zeta(z+1;L_\tau)-\zeta(z;L_\tau) = \frac{\pi^2}{3}\cdot E_2(\tau)\\
    g_2(\tau) &= g_2(L_\tau) =60\cdot G_4(L_\tau) =  \frac 4 3 \pi^4 \cdot E_4(\tau)\\
    g_3(\tau) &= g_3(L_\tau) = 140\cdot G_6(L_\tau) = \frac 8 {27} \pi^6 \cdot E_6(\tau)
\end{align*}
\en{And with these new representations of $g_2(L_\tau)$ and $g_3(L_\tau)$ we can also equivalently redefine the discriminant of the lattice $L_\tau$ and its absolute invariant $J$ from Def.~\ref{\en{EN}defijdelta}:}%
\de{Und mit Hilfe der neuen Darstellungen für $g_2(L_\tau)$ und $g_3(L_\tau)$ können wir dann auch die Diskriminante des Gitters $L_\tau$ und die absolute Invariante $J$ aus Def.~\ref{\en{EN}defijdelta} äquivalent darstellen:}
\begin{align*}
    \Delta(\tau) &= \Delta(L_\tau) = \frac{(2\pi)^{12}}{1728} \cdot (E_4(\tau)^3-E_6(\tau)^2)\\
    J(\tau) &= J(L_\tau) = \frac{E_4(\tau)^3}{E_4(\tau)^3-E_6(\tau)^2}
\end{align*}
\en{This expression will not only be called "Klein's absolute invariant of the lattice $L_\tau$", but also "$J$-function", since it associates each $\tau$ from the upper half plane to a complex number.}%
\de{Diesen Ausdruck nennen wir ab jetzt nicht nur \glqq absolute Invariante des Gitters $L_\tau$\grqq, sondern auch \glqq$J$-Funktion\grqq, da sie jedem $\tau$ aus der oberen Halbebene eine komplexe Zahl zuordnet.}
\end{theo}

\begin{proof}
\en{First we calculate the logarithmic derivative of the Fourier series expansion from Prop.~\ref{\en{EN}fouriersigma}.
Then the product yields a sum and we get:}%
\de{Zunächst berechnen wir die logarithmische Ableitung der Fourierentwicklung aus Satz~\ref{\en{EN}fouriersigma}.
Dann geht das Produkt in eine Summe über und wir erhalten:}
\begin{align}
\frac{\sigma'(z;L_\tau)}{\sigma(z;L_\tau)} &= \eta_1\cdot z + \pi i\cdot \frac{e^{\pi i z}+e^{-\pi i z}}{e^{\pi i z}-e^{-\pi i z}}
        + 2\pi i\cdot\sum_{n=1}^\infty \left(\frac{q_\tau^n/q_z}{1-q_\tau^n/q_z}-\frac{q_\tau^n\cdot q_z}{1-q_\tau^n\cdot q_z}\right)\nonumber\\
        &=\eta_1\cdot z + \pi \cdot\frac{\cos(\pi z)}{\sin(\pi z)} + 2\pi i\cdot\sum_{n=1}^\infty \left(\frac{q^n/w}{1-q^n/w}-\frac{q^n\cdot w}{1-q^n\cdot w}\right)\label{\en{EN}koeffi}
\end{align}
\en{where we used $q=q_\tau=e^{2\pi i\tau}$ and $w=q_z=e^{2\pi i z}$ in the last line.
Next we simplify the remaining sum, using the summation formula for geometric series several times:}%
\de{wobei wir $q=q_\tau=e^{2\pi i\tau}$ und $w=q_z=e^{2\pi i z}$ abgekürzt haben.
Als Nächstes vereinfachen wir die verbleibende Summe mit Hilfe der Formel für die geometrische Reihe:}
\begin{align*}
&\phantom{=}~~2\pi i\sum_{n=1}^\infty \left(\frac{q^n/w}{1-q^n/w}-\frac{q^n\cdot w}{1-q^n\cdot w}\right)\\
&= 2\pi i\sum_{n=1}^\infty \sum_{m=1}^\infty\left(\left(q^n/w\right)^m-\left(q^n\cdot w\right)^m\right)= 2\pi i\sum_{n=1}^\infty \sum_{m=1}^\infty(q^m)^n\cdot\left(w^{-m}-w^m\right)\\
&= 2\pi i\sum_{m=1}^\infty \sum_{n=1}^\infty(q^m)^n\cdot\left(w^{-m}-w^m\right) = 2\pi i\sum_{m=1}^\infty \frac{q^m}{1-q^m}\cdot\left(w^{-m}-w^m\right)\\
&= 2\pi i\sum_{m=1}^\infty \frac{q_\tau^m}{1-q_\tau^m}\cdot\left(e^{-2\pi i m z}-e^{2\pi i m z}\right) = 4\pi \sum_{m=1}^\infty \frac{q_\tau^m}{1-q_\tau^m}\cdot\sin(2\pi m z)
\end{align*}
\en{If we put this into equation~(\ref{\en{EN}koeffi}), we get a representation of the $\wp$-function:}%
\de{Wenn wir das jetzt einsetzen, erhalten wir eine Darstellung der $\wp$-Funktion:}
\begin{align*}
\frac{\sigma'(z;L_\tau)}{\sigma(z;L_\tau)} &= \eta_1\cdot z + \pi \cdot\frac{\cos(\pi z)}{\sin(\pi z)} + 4\pi \cdot\sum_{m=1}^\infty \frac{q_\tau^m}{1-q_\tau^m}\cdot\sin(2\pi m z)\quad\left|-\frac{d}{dz}\right.\\
\Longrightarrow\quad \wp(z;L_\tau) &= - \eta_1 + \left(\frac{\pi}{\sin(\pi z)}\right)^2 - 8\pi^2\cdot\sum_{m=1}^\infty \frac{m\cdot q_\tau^m}{1-q_\tau^m}\cdot\cos(2\pi m z)
\end{align*}
\en{Here we use the known Taylor series of $\sin(x)$ and $\cos(x)$ around $x=0$:}%
\de{Nun verwenden wir die Taylorreihen von $\sin(x)$ und $\cos(x)$ bei $x=0$:}
\begin{align*}
    \cos(x)&=1-\frac {x^2}{2!} + \frac {x^4}{4!} + O(x^6)\\
    \Longrightarrow~~ \cos(2\pi m z) &= 1 - 2\pi^2m^2 z^2 + \frac 2 3 \pi^4 m^4 z^4 + O(z^6)\\
    \text{\en{and}\de{und}}\qquad\sin(x)&=x-\frac {x^3}{3!} + \frac {x^5}{5!} - \frac{x^7}{7!}+O(x^9)\\
    \Longrightarrow~~ \frac{\sin(\pi z)}{\pi z} &= 1 - \frac{\pi^2}{6} z^2 + \frac{\pi^4}{120}z^4 - \frac{\pi^6}{5040}z^6 + O(z^8)\\
    \Longrightarrow~~ \left(\frac{\sin(\pi z)}{\pi z}\right)^2 &= 1 - \frac{2\pi^2}{6} z^2 +\left(\frac{2\pi^4}{120}+\frac{\pi^4}{36}\right)z^4 - \left(\frac{2\pi^6}{5040}+\frac{2\pi^6}{6\cdot 120}\right)z^6 + O(z^{8})\\
                    &=1 - \frac{\pi^2}{3} z^2 +\frac{2\pi^4}{45} z^4 - \frac{\pi^6}{315}z^6 + O(z^8)\\
    \Longrightarrow~~ \left(\frac{\pi z}{\sin(\pi z)}\right)^2 &= \left.\left(1-\left(\frac{\pi^2}{3} z^2 - \frac{2\pi^4}{45} z^4 + \frac{\pi^6}{315}z^6 + O(z^8)\right)\right)^{-1}\quad\right|\text{\en{geom.~series}\de{geom.~Reihe}}\\
    &= 1+\left(\frac{\pi^2}{3} z^2 - \frac{2\pi^4}{45} z^4+ \frac{\pi^6}{315}z^6\right)\\
   &~~~~+\left(\frac{\pi^2}{3} z^2 - \frac{2\pi^4}{45} z^4\right)^2+\left(\frac{\pi^2}{3} z^2\right)^3 + O(z^8)\\
    &=\left.1 + \frac{\pi^2}{3} z^2 +\frac{\pi^4}{15} z^4 + \frac{2\pi^6}{189} z^6 + O(z^8)\qquad\right|: z^2\end{align*}\begin{align*}
    \Longrightarrow~~ \left(\frac{\pi}{\sin(\pi z)}\right)^2 &= \frac 1{z^2} + \frac{\pi^2}{3} +\frac{\pi^4}{15} z^2 + \frac{2\pi^6}{189} z^4 + O(z^6)
\end{align*}
\en{This yields the beginning of the Laurent series of $\wp(z;L_\tau)$ around $z=0$:}%
\de{Hiermit erhalten wir den Anfang der Laurentreihe von $\wp(z;L_\tau)$ um $z=0$:}
\begin{align*}
\wp(z;L_\tau) &= - \eta_1 + \left(\frac{\pi}{\sin(\pi z)}\right)^2 - 8\pi^2\cdot\sum_{m=1}^\infty \frac{m\cdot q_\tau^m}{1-q_\tau^m}\cdot\cos(2\pi m z)\\
&= -\eta_1 + \frac{1}{z^2} + \frac{\pi^2}{3} +\frac{\pi^4}{15} z^2 + \frac{2\pi^6}{189} z^4\\
&~~~- 8\pi^2\cdot\sum_{m=1}^\infty \frac{m\cdot q_\tau^m}{1-q_\tau^m}\cdot\left(1 - 2\pi^2m^2 z^2 + \frac 2 3 \pi^4 m^4 z^4\right) + O(z^6)
\end{align*}
\en{In Prop.~\ref{\en{EN}laurentwp} we already calculated the Laurent series of $\wp$:}%
\de{Zur Erinnerung und zum Vergleich hier nochmal die Laurentreihe aus Satz~\ref{\en{EN}laurentwp}:}
$$\wp(z;L)=\frac 1 {z^2} + 3 G_4(L) z^2 + 5 G_6(L) z^4 + \sum_{n=3}^\infty (2n+1)G_{2n+2}(L)\cdot z^{2n}$$
\en{Equating the coefficients of $z^0$, $z^2$ and $z^4$ in these two Laurent series of $\wp$ yields:}%
\de{Ein Vergleich der Koeffizienten vor $z^0$, $z^2$ und $z^4$ in diesen beiden Darstellungen liefert}
\begin{align*}
    0 &= - \eta_1(L_\tau) + \frac {\pi^2}{3} - 8\pi^2\cdot\sum_{m=1}^\infty \frac{m\cdot q_\tau^m}{1-q_\tau^m}\\
    3 G_4(L_\tau) &= \frac{\pi^4}{15} +16\pi^4\cdot\sum_{m=1}^\infty \frac{m^3\cdot q_\tau^m}{1-q_\tau^m}\\
    5 G_6(L_\tau) &= \frac{2\pi^6}{189} - \frac{16}3\pi^6\cdot\sum_{m=1}^\infty \frac{m^5\cdot q_\tau^m}{1-q_\tau^m}
\end{align*}
\en{and we get:}\de{und wir erhalten:}
\begin{align*}
    \eta_1(L_\tau) &= \frac{\pi^2}{3}\left(1-24\sum_{m=1}^\infty \frac{m\cdot q_\tau^m}{1-q_\tau^m}\right)\\
    g_2(L_\tau)&=60 G_4(L_\tau)=\frac{4}{3}\pi^4\left(1+240\sum_{m=1}^\infty \frac{m^3\cdot q_\tau^m}{1-q_\tau^m}\right)\\
    g_3(L_\tau)&=140 G_6(L_\tau)=\frac{8}{27}\pi^6\left(1-504\sum_{m=1}^\infty \frac{m^5\cdot q_\tau^m}{1-q_\tau^m}\right)
\end{align*}

\en{Now we denote the brackets by $E_2(\tau)$, $E_4(\tau)$ and $E_6(\tau)$ and obtain the representations of $\eta_1(\tau)$, $g_2(\tau)$ and $g_3(\tau)$ from Thm.~\ref{\en{EN}fouriertheorem}.}%
\de{Jetzt benennen wir den Inhalt der Klammern mit $E_2(\tau)$, $E_4(\tau)$ und $E_6(\tau)$ und erhalten die in Thm.~\ref{\en{EN}fouriertheorem} genannten Darstellungen für $\eta_1(\tau)$, $g_2(\tau)$ und $g_3(\tau)$.}

\en{Finally, we use these new representations of $g_2$ and $g_3$ in the old Def.~\ref{\en{EN}defijdelta} of $\Delta$ and $J$:}%
\de{Schließlich setzen wir wie angekündigt die neuen Darstellungen von $g_2$ und $g_3$ in die alte Def.~\ref{\en{EN}defijdelta} von $\Delta$ und $J$ ein:}
\begin{align*}
    \Delta(\tau)&=g_2^3(L_\tau)-27g_3^2(L_\tau) = \left(\frac 4 3 \pi^4 \cdot E_4(\tau)\right)^3 - 27 \cdot \left(\frac 8 {27} \pi^6 \cdot E_6(\tau)\right)^2\\
    &=\frac{(2\pi)^{12}}{1728}\cdot\left(E_4(\tau)^3-E_6(\tau)^2\right) \\
    J(\tau)&=\frac{g_2^3(\tau)}{\Delta(\tau)} = \frac{\left(\frac 4 3 \pi^4 \cdot E_4(\tau)\right)^3}{\frac{(2\pi)^{12}}{1728}\cdot\left(E_4(\tau)^3-E_6(\tau)^2\right)} = \frac{E_4(\tau)^3}{E_4(\tau)^3-E_6(\tau)^2}
\end{align*}
\en{Thus we have proven all statements from Theorem~\ref{\en{EN}fouriertheorem}.}%
\de{Somit haben wir alle Aussagen aus Theorem~\ref{\en{EN}fouriertheorem} bewiesen.}
\end{proof}

\vfill\pagebreak\section{\en{Some Estimates for the \texorpdfstring{$J$}{\emph{J}}- and the \texorpdfstring{$s_2$}{\emph{s\tiny{2}}}-Function}\de{Einige Abschätzungen der \texorpdfstring{$J$}{\emph{J}}- und der \texorpdfstring{$s_2$}{\emph{s\tiny{2}}}-Funktion}}\label{\en{EN}kapanhang}
\renewcommand{\leftmark}{\en{Some Estimates for the $J$- and the $s_2$-Function}\de{Einige Abschätzungen der $J$- und der $s_2$-Funktion}}
\en{In this chapter, we prove the estimates and approximations for $1728J(\tau)$ and $s_2(\tau)$ phrased in the following two theorems.}%
\de{In diesem Kapitel beweisen wir die Abschätzungen und Näherungen für $1728J(\tau)$ und $s_2(\tau)$, die in den folgenden beiden Theoremen formuliert sind.}
\en{These will be used to prove that Kummer's solution in Ch.~\ref{\en{EN}kapkummer} converges, and they are needed to calculate the coefficients in Ch.~\ref{\en{EN}kapFormel}.}%
\de{Diese garantieren, dass Kummers Lösung in Kap.~\ref{\en{EN}kapkummer} konvergiert und ermöglichen die Berechnungen der Koeffizienten in Kap.~\ref{\en{EN}kapFormel}.}
\medskip

\begin{theo}\label{\en{EN}theonaeherJ}
    \en{For the $J$-function from Theorem~\ref{\en{EN}fouriertheorem}:}%
    \de{Für die $J$-Funktion aus Theorem~\ref{\en{EN}fouriertheorem}:}
    $$ J(\tau) := \frac{ E_4(\tau)^3}{ E_4(\tau)^3- E_6(\tau)^2}$$
    \en{the following estimates hold in the region $\im(\tau)>1\ko25$:}%
    \de{gelten im Bereich $\im(\tau)>1\ko25$ folgende Abschätzungen:}
    $$|J(\tau)|>1\ko096>1\qquad\text{\en{and}\de{sowie}}\qquad\frac{0\ko737}{|q|}<|1728J(\tau)|<\frac{1\ko321}{|q|}$$
    \en{For calculating values of $J(\tau)$, one can use the following approximation:}%
    \de{Zur Berechnung kann man die folgende Näherung verwenden:}
    $$\tilde J(\tau) := \frac{\left(1 + 240\left( q + 9 q^2\right)\right)^3}{1728 q \cdot (1-q-q^2)^{24}}\qquad\text{\en{with}\de{mit}}\quad q=e^{2\pi i \tau}$$
    \en{This differs by less than $0\ko2$ from the exact value, if~~$\im(\tau)>1\ko25$:}%
    \de{Diese weicht für $\im(\tau)>1\ko25$ um weniger als $0\ko2$ vom exakten Wert ab:}
    $$|1728J(\tau)-1728\tilde J(\tau)|< 500|q| < 0\ko2 $$
\end{theo}

\begin{bem}\label{\en{EN}penta}
\en{The representation of $\tilde J$ is a consequence of the Dedekind $\eta$ function and Euler's pentagonal number theorem.
Both are not proven here, because the estimations we need can be proven without them.
For even better approximations of the $J$-function, we could add more terms of Prop.~\ref{\en{EN}sigmaeisen} in the nominator
and more terms of the pentagonal number theorem in the denominator, i.e.~more terms of $\left(1-q-q^2+q^5+q^7-q^{12}-q^{15}+q^{22}+q^{26}\pm\ldots\right)^{24}$.}%
\de{Die Darstellung des Nenners von $\tilde J$ folgt aus der Dedekind'schen $\eta$-Funk"-tion und dem Euler'schen Pentagonalzahlensatz.
Beide beweisen wir hier nicht, weil die vorliegende Abschätzung auch ohne den vollen Beweis nachgerechnet werden kann.
Wollte man eine genauere Näherung für die $J$-Funktion erhalten, könnte man im Zähler weitere Terme aus Satz~\ref{\en{EN}sigmaeisen} verwenden
und im Nenner weitere Terme des Pentagonalzahlensatzes, also mehr von $\left(1-q-q^2+q^5+q^7-q^{12}-q^{15}+q^{22}+q^{26}\pm\ldots\right)^{24}$ ergänzen.}
\end{bem}

\begin{theo}\label{\en{EN}theonaehers2}
    \en{The function $s_2$, defined in the upper half plane by:}%
    \de{Die Funktion $s_2$, die auf der oberen Halbebene wie folgt definiert ist:}
    $$s_2(\tau):=\frac{E_4(\tau)}{E_6(\tau)}\cdot\left(E_2(\tau)-\frac{3}{\pi \im(\tau)}\right)$$
    \en{can be replaced by the following approximation using $q=e^{2\pi i \tau}$:}%
    \de{kann mit $q=e^{2\pi i \tau}$ durch die folgende Näherung ersetzt werden:}
    $$\tilde s_2(\tau):=\frac{1+240 (q+9q^2)}{1 - 504 (q+33q^2)}\cdot\left(1 - 24  (q+3q^2) -\frac{3}{\pi \im(\tau)}\right)$$
    \en{If~~$\im(\tau)>1\ko25$, the following estimate holds for this approximation:}%
    \de{Für diese Näherung gilt im Bereich $\im(\tau)>1\ko25$ folgende Abschätzung:}
    $$|s_2(\tau)-\tilde s_2 (\tau)| < 222000|q|^3 $$
\end{theo}
\medskip
\en{In order to prove these two theorems, we prove first:}%
\de{Für den Beweis der beiden Theoreme beweisen wir zunächst:}

\begin{thm}\label{\en{EN}sigmaeisen}
    \en{With help of the sum of the divisors of the number $n$, more accurate with $$\sigma_k(n):=\sum_{d|n}d^k$$ we get the following equivalent representation of the normalized Eisenstein series:}%
    \de{Mit Hilfe der Summe über die Teiler der Zahl $n$, genauer mit $$\sigma_k(n):=\sum_{d|n}d^k$$ erhalten wir äquivalente Darstellungen der normierten Eisensteinreihen:}
    \begin{align*}
    E_2(\tau) &= 1- 24 \sum_{n=1}^\infty \sigma_1(n)\cdot q^n = 1 - 24\left( q + 3 q^2 +\sum_{n=3}^\infty \sigma_1(n)\cdot q^n\right)\\
    E_4(\tau) &= 1+ 240 \sum_{n=1}^\infty \sigma_3(n)\cdot q^n = 1 + 240\left( q + 9 q^2 +\sum_{n=3}^\infty \sigma_3(n)\cdot q^n\right)\\
    E_6(\tau) &= 1- 504 \sum_{n=1}^\infty \sigma_5(n)\cdot q^n = 1 -504\left( q + 33 q^2 +\sum_{n=3}^\infty \sigma_5(n)\cdot q^n\right)
    \end{align*}
\end{thm}
\begin{proof}
\en{In Theorem~\ref{\en{EN}fouriertheorem} we defined the normalized Eisenstein series $E_k$. Now we use the summation formula of the geometric series:}%
\de{In Theorem~\ref{\en{EN}fouriertheorem} wurden die normierten Eisensteinreihen $E_k$ definiert. Jetzt formen wir mit der Formel für die geometrische Reihe um:}
$$\sum_{n=1}^\infty n^k\cdot \frac{q^n}{1-q^n} = \sum_{n=1}^\infty n^k \cdot \sum_{m=1}^\infty (q^n)^m = \sum_{m,n=1}^\infty n^k\cdot q^{m\cdot n}$$
\en{Then we combine all summands with the same exponent $m\cdot n = p$ and get}%
\de{Dann fassen wir alle Summanden, die zum selben Exponenten $m\cdot n = p$ führen zusammen und erhalten}
$$\sum_{m,n=1}^\infty n^k\cdot q^{m\cdot n} = \sum_{p=1}^\infty\left(\sum_{n|p} n^k\right)\cdot q^p=\sum_{p=1}^\infty\sigma_k(p)\cdot q^p$$
\en{Thus we get the new representations of the normalized Eisenstein series. The beginnings of the $q$-series from Prop.~\ref{\en{EN}sigmaeisen} are a consequence of $\sigma_k(1)=1^k$ and $\sigma_k(2)=1^k+2^k$.}%
\de{Somit erhalten wir die neuen Darstellungen der Eisensteinreihen. Die angegebenen An"-fänge der $q$-Entwicklungen folgen aus $\sigma_k(1)=1^k$ und $\sigma_k(2)=1^k+2^k$.}
\end{proof}

\begin{bem}
\en{We calculated the first eight values of $\sigma_{1;3;5}$ for Table~\ref{\en{EN}tab:WerteSigmaK}, but we will use only the first three of each.}%
\de{Für Tabelle~\ref{\en{EN}tab:WerteSigmaK} wurden jeweils die ersten acht Werte von $\sigma_{1;3;5}$ berechnet. Wir werden allerdings nur jeweils die ersten drei verwenden.}
\end{bem}

\begin{table}[ht]
    \centering\renewcommand{\arraystretch}{1.5}
    \begin{tabular}{c||c|c|c|c|c|c|c|c}
        $n$ & $1$ & $2$ & $3$ & $4$ & $5$ & $6$ & $7$ & $8$\\
        \hline
        $\{d|n\}$ & $\{1\}$ & $\{1;2\}$ & $\{1;3\}$ & $\{1;2;4\}$ & $\{1;5\}$ & $\{1;2;3;6\}$ & $\{1;7\}$ & $\{1;2;4;8\}$\\
        \hline
        $\sigma_1(n)$ & $1$ & $3$ & $4$ & $7$ & $6$ & $12$ & $8$ & $15$\\
        $\sigma_3(n)$ & $1$ & $9$ & $28$ & $73$ & $126$ & $252$ & $344$ & $585$\\
        $\sigma_5(n)$ & $1$ & $33$ & $244$ & $1057$ & $3126$ & $8052$ & $16808$ & $33825$\\
    \end{tabular}\bigskip 
    \caption{\en{Some values of the divisor functions $\sigma_k$}\de{Einige Werte der $\sigma_k$-Funktionen}}
    \label{\en{EN}tab:WerteSigmaK}
\end{table}

\begin{lem}\label{\en{EN}sigmaschaetz}
\en{For all natural $n$ and $k$ it holds}%
\de{Für alle natürlichen $k$ und $n$ gilt die Abschätzung}
$$\sigma_k(n)\leq n^{k+1}.$$
\end{lem}
\begin{proof}
\en{We extend the summation over the divisors of $n$ to all natural numbers until $n$:}%
\de{Wir erweitern die Summe über die Teiler der Zahl $n$ auf die Summe über alle Zahlen bis $n$:}\belowdisplayskip=-12pt
$$\sigma_k(n) = \sum_{d|n}d^k \leq \sum_{d=1}^n d^k \leq \sum_{d=1}^n n^k = n\cdot n^k = n^{k+1}$$
\end{proof}

\begin{lem}\label{\en{EN}lemrestchaezt}
\en{For the remainder $R_k^{(l)}$ in the normalized Eisenstein series $E_k$ of Prop.~\ref{\en{EN}sigmaeisen}}%
\de{Für das Restglied $R_k^{(l)}$ in den normierten Eisensteinreihen $E_k$ von Satz~\ref{\en{EN}sigmaeisen}}
$$R_k^{(l)} := \sum_{n=l}^\infty \sigma_{k-1}(n)\cdot q^n$$
\en{the following estimate holds, if $\left(1+\frac 1 l\right)^k\cdot |q|<1$:}%
\de{gilt die folgende Abschätzung, falls $\left(1+\frac 1 l\right)^k\cdot |q|<1$ ist:}
$$\left|R_k^{(l)}\right| \leq \frac{l^k\cdot |q|^l}{1-\left(1+\frac 1 l\right)^k\cdot |q|}$$
\end{lem}
\begin{proof}
\en{Lemma~\ref{\en{EN}sigmaschaetz} yields:}%
\de{Zunächst folgt aus Lemma~\ref{\en{EN}sigmaschaetz}:}
$$\left|R_k^{(l)}\right| \leq \sum_{n=l}^\infty \sigma_{k-1}(n)\cdot |q|^n\leq \sum_{n=l}^\infty \underbrace{n^k\cdot |q|^n}_{=:r_n}$$
\en{In this sum, we apply the ratio test:}%
\de{In der Summe wenden wir das Quotientenkriterium an:}
\begin{align*}
    \frac{r_{n+1}}{r_n} = \frac{(n+1)^k\cdot |q|^{n+1}}{n^k\cdot |q|^n}
    = \left(1 +\frac{1}{n}\right)^k\cdot |q|
    \leq \left(1 +\frac{1}{l}\right)^k\cdot |q| =: s
\end{align*}
\en{Here, direct comparison $r_n \leq r_l \cdot s^{n-l}$ yields a geometric series, converging if $|s|<1$: }%
\de{also kann man mit $r_n \leq r_l \cdot s^{n-l}$ die geometrische Reihe als Majorante verwenden:}\belowdisplayskip=-12pt
$$\left|R_k^{(l)}\right| \leq r_l\cdot\sum_{n=l}^\infty s^{n-l} = r_l\cdot\frac{1}{1-s} = \frac{l^k\cdot |q|^l}{1-\left(1+\frac 1 l\right)^k\cdot |q|}$$
\end{proof}

\begin{lem}[Archimedes]\label{\en{EN}archim-lem}
\en{It holds $3+\frac{10}{71}<\pi<3+\frac{1}{7}$. This implies:}%
\de{Es gilt $3+\frac{10}{71}<\pi<3+\frac{1}{7}$. Daraus folgt:}
$$\text{\en{If}\de{Wenn} }\im(\tau)>1\ko25\text{, \en{then}\de{dann ist} }|q|<e^{-7\ko852}.$$
\end{lem}
\begin{proof}
\en{Archimedes proved in \cite[p.~91--98]{\en{EN}archimedes}, that $3+\frac{1}{7}>\pi>3+\frac{10}{71}> 3\ko1408$\footnote{For an alternative proof of this, evaluate Dalzell's integral $I:=\int_{0}^{1}{\frac {x^{4}\left(1-x\right)^{4}}{1+x^{2}}}\,dx = {\frac {22}{7}}-\pi$.
Then observe $0< I < \int_{0}^{1}{x^{4}\left(1-x\right)^{4}}\,dx=\frac 1 {630}$ and thus $3+\frac{10}{71}<\frac{22}{7}-\frac 1 {630}<\pi<\frac{22}{7}$.}.
If we now denote $\tau=x+iy$, we get $q=e^{2\pi i \tau} = e^{2\pi i x}\cdot e^{-2\pi y}$ and $$|q|=e^{-2\pi \im(\tau)}$$
Then $\im(\tau)>1\ko25$ and $\pi>3\ko1408$ yields $2\pi \im(\tau)>7\ko852$ and $|q|<e^{-7\ko852}$.}%
\de{Archimedes bewies in \cite[S.~91--98]{\en{EN}archimedes}, dass $3+\frac{1}{7}>\pi>3+\frac{10}{71}> 3\ko1408$ gilt\footnote{Einen Alternativbeweis für diese Abschätzung liefert das Dalzell-Integral $I:=\int_{0}^{1}{\frac {x^{4}\left(1-x\right)^{4}}{1+x^{2}}}\,dx = {\frac {22}{7}}-\pi$.
Für dieses gilt nämlich $0< I < \int_{0}^{1}{x^{4}\left(1-x\right)^{4}}\,dx=\frac 1 {630}$ und somit $3+\frac{10}{71}<\frac{22}{7}-\frac 1 {630}<\pi<\frac{22}{7}$.}.
Mit $\tau=x+iy$ erhalten wir $q=e^{2\pi i \tau} = e^{2\pi i x}\cdot e^{-2\pi y}$ und $$|q|=e^{-2\pi \im(\tau)}$$
Aus $\im(\tau)>1\ko25$ und $\pi>3\ko1408$ folgt dann $2\pi \im(\tau)>7\ko852$ bzw.~$|q|<e^{-7\ko852}$.}
\end{proof}

\begin{bem}
\en{From here onwards, the calculations are merely technical. One misses next to nothing by skipping the remainder of this chapter.}%
\de{Ab hier werden die Rechnungen rein technisch, man verpasst nichts wenn man den Rest dieses Kapitels überblättert.}
\end{bem}

\begin{lem}\label{\en{EN}lemrestkonkr}
\en{In the region $\im(\tau)>1\ko25$, it holds for the remainders from Lemma~\ref{\en{EN}lemrestchaezt}:}%
\de{Im Bereich $\im(\tau)>1\ko25$ gilt für die Restglieder aus Lemma~\ref{\en{EN}lemrestchaezt}:}
\begin{align*}
    \left|R_2^{(3)}\right| &\leq ~4\ko007|q|^3
    \qquad\text{\en{and}\de{und}}\qquad
    \left|R_4^{(3)}\right| \leq ~28\ko1|q|^3
    \qquad\text{\en{and}\de{und}}\qquad
    \left|R_6^{(3)}\right| \leq 245\ko6|q|^3
\end{align*}
\end{lem}
\begin{proof}
\en{First we use Lemma~\ref{\en{EN}lemrestchaezt} and then $|q|<e^{-7\ko852}$ from Lemma~\ref{\en{EN}archim-lem}:}%
\de{Zunächst nutzen wir Lemma~\ref{\en{EN}lemrestchaezt} und dann $|q|<e^{-7\ko852}$ aus Lemma~\ref{\en{EN}archim-lem}:}
\begin{align*}
    \left|R_2^{(4)}\right| &\leq \frac{4^2\cdot |q|^4}{1-\left(1+\frac 1 4\right)^2\cdot |q|}\leq ~~16\ko01|q|^4\\
    \left|R_4^{(4)}\right| &\leq \frac{4^4\cdot |q|^4}{1-\left(1+\frac 1 4\right)^4\cdot |q|}\leq 256\ko25|q|^4\\
    \left|R_6^{(4)}\right| &\leq \frac{4^6\cdot |q|^4}{1-\left(1+\frac 1 4\right)^6\cdot |q|}\leq 4102\ko1|q|^4
\end{align*}
\en{Then we use the values of $\sigma_k$ from Table~\ref{\en{EN}tab:WerteSigmaK} and obtain the stated estimates:}%
\de{Dann verwenden wir die Werte der $\sigma_k$ aus Tabelle~\ref{\en{EN}tab:WerteSigmaK} und erhalten die zu beweisenden besseren Abschätzungen:}\belowdisplayskip=-12pt
\begin{align*}
    \left|R_2^{(3)}\right| &= \left|\sigma_1(3)\cdot q^3 + R_2^{(4)}\right| \leq ~~4|q|^3 + ~~~16\ko01|q|^4 \leq 4\ko007|q|^3\\
    \left|R_4^{(3)}\right| &= \left|\sigma_3(3)\cdot q^3 + R_4^{(4)}\right| \leq ~28|q|^3 + 256\ko25|q|^4 \leq ~28\ko1|q|^3\\
    \left|R_6^{(3)}\right| &= \left|\sigma_5(3)\cdot q^3 + R_6^{(4)}\right| \leq 244|q|^3 + 4102\ko1|q|^4 \leq 245\ko6|q|^3
\end{align*}
\end{proof}

\begin{lem}\label{\en{EN}lemE6}
\en{If $\im(\tau)>1\ko25$ it holds $|E_6(\tau)|>0\ko8$ and in particular $E_6(\tau)\neq 0$.}%
\de{Im Bereich $\im(\tau)>1\ko25$ gilt $|E_6(\tau)|>0\ko8$ und insbesondere $E_6(\tau)\neq 0$.}
\end{lem}
\begin{proof}
\en{It holds}\de{Es gilt} \belowdisplayskip=-12pt
\begin{align*}
|E_6(\tau)| &= \left|1-504\left(q+33q^2+R_6^{(3)}\right)\right|\\
&\geq 1 - 504 \left(|q| +33|q|^2+245\ko6|q|^3\right) >0\ko8
\end{align*}
\end{proof}

\begin{defi}\label{\en{EN}delta}
\en{We denote the difference between a function $f$ and its approximation $\tilde f$ with $\dl{\tilde f}$:}%
\de{Wir bezeichnen die Differenz zwischen einer Funktion $f$ und ihrer Näherung $\tilde f$ mit $\dl{\tilde f}$:}
$$f = \tilde f + \dl{\tilde f}$$
\end{defi}

\pagebreak

\begin{lem}\label{\en{EN}lemxy}
\en{For the quadratic approximations}%
\de{Für die drei quadratischen Näherungen}
\begin{align*}
    X&:=E_4^{(2)}=1+240(q+9q^2)\\
    Y&:=E_6^{(2)}=1-504(q+33q^2)\\
    Z&:=E_2^{(2)}-\frac{3}{\pi \im(\tau)}=1-24(q+3q^2)-\frac{3}{\pi \im(\tau)}
\end{align*}
\en{with $E_4(\tau)=X+\dl{X}$, $E_6(\tau)=Y+\dl{Y}$ and $E_2(\tau)-\frac{3}{\pi \im(\tau)}=Z+\dl{Z}$, the following estimates hold in the region $\im(\tau)>1\ko25$:}
\de{mit $E_4(\tau)=X+\dl{X}$, $E_6(\tau)=Y+\dl{Y}$ und $E_2(\tau)-\frac{3}{\pi \im(\tau)}=Z+\dl{Z}$ gelten im Bereich $\im(\tau)>1\ko25$ die Abschätzungen:}
\begin{align*}
    \begin{aligned}
    |\dl{X}|& \leq 6744 |q|^3\\
    |\dl{Y}|& \leq 123783 |q|^3\\
    |\dl{X^3}|& \leq 24202 |q|^3\\
    |\dl{Y^2}|& \leq 296734 |q|^3\\
    |\dl{Z}|& \leq 96\ko2 |q|^3
    \end{aligned}
    \qquad \text{\en{and}\de{und}} \qquad\quad 
    \begin{aligned}
    0\ko9063 \leq |X| &\leq 1\ko0937\\
    0\ko8014 \leq |Y| &\leq 1\ko1986\\
    0\ko7444 \leq |X^3| &\leq 1\ko3083\\
    0\ko6422 \leq |Y^2| &\leq 1\ko4367\\
                   |Z| &\leq 1\ko0094
    \end{aligned}
\end{align*}
\end{lem}
\begin{proof}
\en{From the definition of $X$ and $Y$ we obtain:}%
\de{Zunächst folgt aus der Definition von $X$ und $Y$:}
\begin{align*}
    |X - 1|&\leq 240(|q|+\phantom{3}9|q|^2)\leq 0\ko0937\qquad\Longrightarrow\qquad 0\ko9063 \leq |X| \leq 1\ko0937\\
    |Y - 1|&\leq 504(|q|+33|q|^2)\leq 0\ko1986\qquad\Longrightarrow\qquad 0\ko8014 \leq |Y| \leq 1\ko1986
\end{align*}
\en{This yields:}\de{Hieraus folgt}
\begin{align*}
    0\ko7444\leq 0\ko9063^3 \leq |X^3| \leq 1\ko0937^3\leq 1\ko3083\\
    0\ko6422 \leq 0\ko8014^2 \leq |Y^2| \leq 1\ko1986^2 \leq 1\ko4367
\end{align*}
\en{From the definition of $Z$ we obtain:}%
\de{Dann folgt aus der Definition von $Z$:}
$$|Z|\leq\left|1-\frac{3}{\pi \im(\tau)}\right|+24(|q|+3|q|^2) \leq 1+24(|q|+3|q|^2)\leq 1\ko0094$$
\en{Then we get from}\de{Außerdem gilt mit} Lemma~\ref{\en{EN}lemrestkonkr}:
\begin{align*}
    |\dl{X}|&= 240\left|R_4^{(3)}\right|\leq 6744 |q|^3\\
    |\dl{Y}|&= 504\left|R_6^{(3)}\right|\leq 123783 |q|^3\\
    |\dl{Z}|&= 24\left|R_2^{(3)}\right|\leq 96\ko2 |q|^3
\end{align*}
\en{This yields the errors of $X^3$ and $Y^2$ compared to $E_4^3$ and $E_6^2$:}%
\de{Hieraus folgen die Fehler von $X^3$ und $Y^2$ im Vergleich zu $E_4^3$ und $E_6^2$:}\belowdisplayskip=-12pt
\begin{align*}
    E_4^3=(X+\dl{X})^3 &= X^3 + \dl{X}\cdot\left(3 X^2 + 3 X \dl{X} + (\dl{X})^2\right)=X^3 + \dl{X^3}\\
    \Longrightarrow\quad |\dl{X^3}| &\leq |\dl{X}|\cdot\left(3 |X|^2 + 3 |X|\cdot |\dl{X}| + |\dl{X}|^2\right)\leq 24202|q|^3\\
    E_6^2 = (Y+\dl{Y})^2 &= Y^2 + \dl{Y}\cdot\left(2 Y + \dl{Y}\right) =Y^2 + \dl{Y^2}\\
    \Longrightarrow\quad |\dl{Y^2}| &\leq |\dl{Y}|\cdot\left(2 |Y| + |\dl{Y}|\right)\leq 296734|q|^3
\end{align*}
\end{proof}

\begin{proof}[{\textbf{\en{Proof of Theorem}\de{Beweis des Theorems}~\ref{\en{EN}theonaehers2}}}]
\en{In the notation of Lemma~\ref{\en{EN}lemxy}, the definitions of Thm.~\ref{\en{EN}theonaehers2} read $\tilde s_2 = \frac{X}{Y}\cdot Z$ and:}%
\de{In der Notation von Lemma~\ref{\en{EN}lemxy} übersetzen sich die Definitionen von Thm.~\ref{\en{EN}theonaehers2} zu $\tilde s_2 = \frac{X}{Y}\cdot Z$ und:}
\begin{align*}
    s_2(\tau) = \tilde s_2 + \dl{\tilde s_2} &= \left.\frac{X+\dl{X}}{Y+\dl{Y}}\cdot\left(Z+\dl{Z}\right)\quad\right|\cdot\left(Y + \dl{Y}\right)\\
    \left(\tilde s_2 + \dl{\tilde s_2}\right)\cdot\left(Y + \dl{Y}\right) &= \left(X + \dl{X}\right)\cdot\left(Z + \dl{Z}\right)\\
    \tilde s_2\cdot Y + \tilde s_2\cdot\dl{Y} + \dl{\tilde s_2}\cdot\left(Y + \dl{Y}\right) &= X\cdot Z + \dl{X}\cdot Z + X\cdot\dl{Z} + \dl{X}\cdot\dl{Z}
\end{align*}

\pagebreak

\noindent\en{From this equation we substract $\tilde s_2\cdot Y = X\cdot Z$ and obtain:}%
\de{Von dieser Gleichung subtrahieren wir $\tilde s_2\cdot Y = X\cdot Z$ und erhalten:}
\begin{align}
\dl{\tilde s_2}&=\frac{\dl{X}\cdot Z + X\cdot\dl{Z} + \dl{X}\cdot\dl{Z} - \tilde s_2\cdot\dl{Y}}{Y + \dl{Y}}\label{\en{EN}glgxyz}
\end{align}
\en{From Lemma~\ref{\en{EN}lemxy} we deduce:}%
\de{Aus Lemma~\ref{\en{EN}lemxy} folgt nun noch:}
\begin{align*}
    |\tilde s_2| &= \left|\frac{X}{Y}\cdot Z\right| \leq \frac{1\ko0937}{0\ko8014}\cdot 1\ko0094\leq 1\ko3776
\end{align*}
\en{If we take this and the other estimates of Lemma~\ref{\en{EN}lemxy} into~(\ref{\en{EN}glgxyz}) we obtain:}%
\de{Das und die anderen Abschätzungen aus Lemma~\ref{\en{EN}lemxy} setzen wir in~(\ref{\en{EN}glgxyz}) ein und erhalten:}
\begin{align*}
    |\dl{s_2}| &\leq \frac{|\dl{X}|\cdot |Z| + |X|\cdot|\dl{Z}| + |\dl{X}|\cdot|\dl{Z}| + |\tilde s_2|\cdot|\dl{Y}|}{|Y| - |\dl{Y}|}\\
    &\leq \frac{6808|q|^3+106|q|^3+4\cdot10^{-5}|q|^3+170550|q|^3}{0\ko8014-123800|q|^3} < 222000|q|^3
\end{align*}
\en{Here we used $|q|<e^{-7\ko852}$ (Lemma~\ref{\en{EN}archim-lem}). Thus the estimate from Thm.~\ref{\en{EN}theonaehers2} is proven.}
\de{Hier haben wir $|q|<e^{-7\ko852}$ (Lemma~\ref{\en{EN}archim-lem}) benutzt. Somit ist die Abschätzung aus Theorem~\ref{\en{EN}theonaehers2} bewiesen.}
\end{proof}

\begin{lem}\label{\en{EN}lemk}
\en{We define the following function $k$, which is analytic in the upper half plane, and its approximation $\tilde k$:}%
\de{Wir definieren die folgende in der oberen Halbebene analytische Funktion $k$ und ihre Näherung $\tilde k$:}
$$k(\tau):=\frac{E_4^3-E_6^2}{1728q}\qquad \text{\en{and}\de{und}} \qquad \tilde k(\tau):=(1-q-q^2)^{24}$$
\en{Then we get the following estimates in the region $\im(\tau)>1\ko25$:}%
\de{Dann gelten im Bereich $\im(\tau)>1\ko25$ folgenden Abschätzungen:}
\begin{align*}
    |k-\tilde k| \leq 365\ko6 |q|^2
    \qquad \text{\en{and}\de{und}} \qquad\quad 
        0\ko9907 \leq |\tilde k|  \leq 1\ko0094
\end{align*}
\end{lem}
\begin{proof}
\en{From $|q|<e^{-7\ko852}$ (Lemma~\ref{\en{EN}archim-lem}) we get the estimate of $|\tilde k|$:}%
\de{Die Abschätzung für $|\tilde k|$ folgt direkt aus $|q|<e^{-7\ko852}$ (Lemma~\ref{\en{EN}archim-lem}), denn}
$$0\ko9907\leq (1-|q|-|q|^2)^{24}\leq |\tilde k| \leq (1+|q|+|q|^2)^{24}\leq 1\ko0094$$
\en{Now we add another term into the error estimation, so that we can use Lemma~\ref{\en{EN}lemxy}:}%
\de{Dann fügen wir für die Fehlerabschätzung einen zusätzlichen Term ein, um Lemma~\ref{\en{EN}lemxy} verwenden zu können:}
\begin{align}
    |k-\tilde k|&=\left|\frac{E_4^3-E_6^2}{1728q}-(1-q-q^2)^{24}\right|\nonumber\\
    &\leq\left|\frac{E_4^3-E_6^2}{1728q}-\frac{X^3-Y^2}{1728q}\right|+\left|\frac{X^3-Y^2}{1728q}-(1-q-q^2)^{24}\right|\nonumber\\
    &\leq\left|\frac{E_4^3-X^3}{1728q}\right| + \left|\frac{E_6^2-Y^2}{1728q}\right| + \left|\frac{X^3-Y^2}{1728q}-(1-q-q^2)^{24}\right|\nonumber\\
    &\leq \frac{24202|q|^3}{1728|q|} + \frac{296734|q|^3}{1728|q|} + \left|\frac{X^3-Y^2}{1728q}-(1-q-q^2)^{24}\right|\nonumber\\
    &\leq 185\ko8|q|^2 + \left|\frac{X^3-Y^2}{1728q}-(1-q-q^2)^{24}\right|\label{\en{EN}kkstr}
\end{align}
\en{It remains to find an estimate for the last term. We derive several times by $q$:}%
\de{Es fehlt also noch eine Abschätzung für den letzten Ausdruck. Wir leiten einige Male nach $q$ ab:}
\begin{align*}
    f(q) &:=(1-q-q^2)^{24}\qquad\Longrightarrow\qquad f(0)=1\\
    f'(q) &= -24(1+2q)(1-q-q^2)^{23}\qquad\Longrightarrow\qquad f'(0)=-24\\
    f''(q) &= 24(21+94q+94q^2)(1-q-q^2)^{22}\qquad\Longrightarrow\qquad f''(0)=504\\
    f'''(q) &= -1104(1+2q)(8+47q+47q^2)(1-q-q^2)^{21}\qquad\Longrightarrow\qquad |f'''(q)|\leq 8932
\end{align*}
\en{This proves the existence of a $\xi$ with}%
\de{Somit folgt, dass es ein $\xi$ gibt mit}
$$(1-q-q^2)^{24} = 1 -24 q + \frac{504}{2}q^2 + \frac{f'''(\xi)}{6}\cdot q^3$$
\en{Furthermore, we expand the other expression:}%
\de{Außerdem lautet die ausmultiplizierte Form des anderen Ausdrucks}
\begin{align*}
    \frac{X^3-Y^2}{1728q} &= \frac{(1+240(q+9q^2))^3-(1-504(q+33q^2))^2}{1728q}\\
    &= 1-24q+98q^2+64017q^3+1944000q^4+5832000q^5
\end{align*}
\en{Subtracting these two representations yields}%
\de{Wenn wir diese beiden Darstellungen subtrahieren, erhalten wir}
\begin{align*}
    ~&\left|\frac{X^3-Y^2}{1728q}-(1-q-q^2)^{24}\right|\\
    =& \left| (98-252)q^2+64017q^3+1944000q^4+5832000q^5 - \frac{f'''(\xi)}{6}\cdot q^3\right|\\
    \leq & ~154 |q|^2 + 64017|q|^3+1944000|q|^4+5832000|q|^5 + \frac{8932}{6}\cdot |q|^3 \leq 179\ko8|q|^2
\end{align*}
\en{If we use this, we obtain from eq.~(\ref{\en{EN}kkstr}):}%
\de{Wenn wir das in Glg.~(\ref{\en{EN}kkstr}) einsetzen, erhalten wir die angekündigte Abschätzung:}
$$|k-\tilde k| \leq 185\ko8|q|^2 + 179\ko8|q|^2\leq 365\ko6|q|^2$$
\en{Using the pentagonal number theorem (Remark~\ref{\en{EN}penta}) one could prove the better error estimation $|k-\tilde k| < 25|q|^5$, but for our means $|k-\tilde k| \leq 365\ko6|q|^2$ is enough.}%
\de{Mit dem Pentagonalzahlensatz aus Bemerkung~\ref{\en{EN}penta} könnte man sogar beweisen, dass die Abweichung $|k-\tilde k| < 25|q|^5$ ist, aber das wird wie gesagt nicht benötigt.}
\end{proof}

\begin{lem}\label{\en{EN}lem3}
\en{We define $J_2$ and its approximation $\tilde J_2$ as follows:}%
\de{Wir definieren die folgende Funktion $J_2$ und ihre Näherung $\tilde J_2$:}
$$J_2(\tau):=1728q\cdot J(\tau) = \frac{E_4^3}{k}\qquad \text{\en{and}\de{und}} \qquad \tilde J_2(\tau):=1728q\cdot \tilde J(\tau) = \frac{X^3}{\tilde k}$$
\en{Then the following estimates hold in the region $\im(\tau)>1\ko25$:}%
\de{Dann gelten im Bereich $\im(\tau)>1\ko25$ folgenden Abschätzungen:}
\begin{align*}
    \begin{aligned}
        |\delta(\tilde J_2)|&:=|J_2-\tilde J_2| < 500 |q|^2
    \end{aligned}
    \qquad \text{\en{and}\de{und}} \qquad\quad 
    \begin{aligned}
        0\ko7374 \leq |\tilde J_2| \leq 1\ko3206
    \end{aligned}
\end{align*}
\end{lem}
\begin{proof}
\en{The estimate of $|\tilde J_2|$ is a consequence of those of $|X^3|$ and $|\tilde k|$ from Lemma~\ref{\en{EN}lemxy} and~\ref{\en{EN}lemk}:}%
\de{Die Abschätzung für $|\tilde J_2|$ folgt aus denen für $|X^3|$ und $|\tilde k|$ (Lemma~\ref{\en{EN}lemxy} und~\ref{\en{EN}lemk}):}
$$0\ko7374 \leq \frac{0\ko7444}{1\ko0094} \leq |\tilde J_2| \leq \frac{1\ko3083}{0\ko9907} \leq 1\ko3206$$
\en{From the definitions of $J_2$ and $\tilde J_2$ we obtain (as in eq.~(\ref{\en{EN}glgxyz})):}%
\de{Aus der Definition von $J_2$ und $\tilde J_2$ folgt dann, ähnlich wie in Glg.~(\ref{\en{EN}glgxyz}):}
\begin{align*}
    \tilde J_2 +\delta(\tilde J_2) = \frac{X^3+ \delta(X^3)}{\tilde k + \delta(\tilde k)}\qquad\Longrightarrow\qquad
    \delta(\tilde J_2) = \frac{\delta(X^3) -\tilde J_2 \cdot\delta(\tilde k)}{\tilde k + \delta(\tilde k)}
\end{align*}
\en{and again, using Lemma~\ref{\en{EN}lemxy} and~\ref{\en{EN}lemk}:}%
\de{und somit, wieder mit Lemma~\ref{\en{EN}lemxy} und~\ref{\en{EN}lemk}:}\belowdisplayskip=-12pt
\begin{align*}
|\delta(\tilde J_2)| \leq \frac{24202 |q|^3 + 1\ko3206 \cdot 365\ko6 |q|^2}{0\ko9907 - 365\ko6 |q|^2} < 496\ko9 |q|^2 < 500 |q|^2
\end{align*}
\end{proof}

\begin{proof}[{\textbf{\en{Proof of Theorem}\de{Beweis des Theorems}~\ref{\en{EN}theonaeherJ}}}]
\en{We use the estimates of $J_2$ from Lemma~\ref{\en{EN}lem3} and apply them to $J$ with $1728J = \frac{J_2}{q}$:}%
\de{Wir verwenden die Abschätzungen für $J_2$ aus Lemma~\ref{\en{EN}lem3} und übertragen sie auf $J$ mit Hilfe von $1728J = \frac{J_2}{q}$:}
\begin{align*}
    |1728J-1728\tilde J| &\leq \frac{|\delta(\tilde J_2)|}{|q|}< \frac{500|q|^2}{|q|} = 500|q| < 0\ko2\\
    \text{\en{and}\de{und}}\qquad |J(\tau)|&\geq \frac{|\tilde J_2|-|\delta(\tilde J_2)|}{1728|q|} \geq \frac{0\ko7374-500|q|^2}{1728|q|} > 1\ko096 > 1\\
    \text{\en{and}\de{und}}\qquad |1728J(\tau)|&\leq \frac{|\tilde J_2|+|\delta(\tilde J_2)|}{|q|} \leq \frac{1\ko3206+500|q|^2}{|q|} < \frac{1\ko321}{|q|}\\
    \text{\en{and}\de{und}}\qquad |1728J(\tau)|&\geq \frac{|\tilde J_2|-|\delta(\tilde J_2)|}{|q|} \geq \frac{0\ko7374-500|q|^2}{|q|} > \frac{0\ko737}{|q|}
\end{align*}
\en{Thus we have proven all estimates of Theorem~\ref{\en{EN}theonaeherJ}.}%
\de{Somit sind alle Abschätzungen des Theorems~\ref{\en{EN}theonaeherJ} bewiesen.}
\end{proof}

\vfill\pagebreak\section{\en{Hypergeometric Functions and Clausen's Formula}\de{Hypergeometrische Funktionen und Clausen-Formel}}\label{\en{EN}kapClausen}
\renewcommand{\leftmark}{\en{Hypergeometric Functions and Clausen's Formula}\de{Hypergeometrische Funktionen und Clausen-Formel}}
\en{In this chapter, we prove Clausen's formula (Theorem~\ref{\en{EN}satzclausen}). The proof follows Thomas Clausen's paper \cite[p.~89-91]{\en{EN}clausen} from 1828, but we use modern notations with Pochhammer symbols. Additionally, we prove the hypergeometric differential equations needed for the proof.}%
\de{Ziel dieses Kapitels ist der Beweis der Formel von Clausen (siehe Theorem~\ref{\en{EN}satzclausen}). Der Beweis folgt dem Originalartikel \cite[S.~89-91]{\en{EN}clausen} von Thomas Clausen aus dem Jahr 1828, allerdings in heutiger Notation mit Pochhammer-Symbolen. Auch die für den Beweis benötigten hypergeometrischen Differentialgleichungen werden bewiesen.}

\en{This chapter does not rely on the previous chapters, it is self-contained.}%
\de{Man kann dieses Kapitel unabhängig von den vorherigen Kapiteln lesen.}

\begin{defi}\label{\en{EN}defpoch} \en{For $n\in\mathbb N$ we define the Pochhammer symbol $(a)_n$ as follows:}%
\de{Das Pochhammer-Symbol $(a)_n$ ist für natürliches $n$ wie folgt definiert:}
\begin{align*}
    (a)_0 := 1 \qquad\text{\en{and}\de{und}}\qquad (a)_{n+1} := (a)_n \cdot (a+n)
\end{align*}
\en{This implies $(1)_n=n!$ for all $n$; and $(a)_n = a (a+1) (a+2) \cdots (a+n-1)$, if $n>0$.}%
\de{Hieraus folgt $(1)_n=n!$ für alle $n$; und $(a)_n = a (a+1) (a+2) \cdots (a+n-1)$, falls $n>0$ ist.}
\end{defi}

\begin{defi} \label{\en{EN}defihyp}
\en{The hypergeometric functions ${_2F_1}$ and ${_3F_2}$ are defined as follows:}%
\de{Die hypergeometrische Funktion ${_2F_1}$ und die verallgemeinerte hypergeometrische Funktion ${_3F_2}$ lauten:}
\begin{align*}
    {_2F_1}(a,b;c;z) &= \sum_{n=0}^{\infty} \frac{(a)_n\cdot (b)_n}{(c)_n}\cdot\frac{z^n}{n!}\\
    {_3F_2}(\alpha,\beta,\gamma;\delta,\varepsilon;z) &= \sum_{n=0}^{\infty} \frac{(\alpha)_n\cdot(\beta)_n\cdot(\gamma)_n}{(\delta)_n \cdot (\varepsilon)_n}\cdot\frac{z^n}{n!}
\end{align*}
\end{defi}

\begin{thm}\label{\en{EN}satzkonv}
    ${_2F_1}(a,b;c;z)$ \en{and}\de{und} ${_3F_2}(\alpha,\beta,\gamma;\delta,\varepsilon;z)$ \en{converge absolutely for $|z|<1$.}\de{konvergieren für $|z|<1$ absolut.}
\end{thm}
\begin{proof}
\en{We use the ratio test:}\de{Das folgt aus dem Quotientenkriterium:}
\begin{align*}
    \frac{(a)_{n+1}\cdot (b)_{n+1}}{(c)_{n+1}}\cdot \frac{z^{n+1}}{(n+1)!}\cdot\frac{(c)_n}{(a)_n\cdot (b)_n}\cdot \frac{n!}{z^n} &= \frac{(a+n)(b+n)}{(c+n)(n+1)}\cdot z
\end{align*}
\en{The fraction before $z$ approaches $1$, so the absolute value of the whole expression will be smaller than $1$ for large $n$ if $|z|<1$. The convergence of ${_3F_2}$ is proven in the same way.}%
\de{wobei der Bruch vor $z$ gegen $1$ geht und somit, falls $|z|<1$ ist, der Ausdruck für große $n$ betragsmäßig auch kleiner als $1$ wird. Die Konvergenz von ${_3F_2}$ folgt analog.}
\end{proof}

\begin{thm}\label{\en{EN}koeffvergl}
\en{If}\de{Wenn} $f(z)=\sum_{n=0}^\infty A_n\frac{z^n}{n!}$ \en{is given as a power series, then for all $z$ with absolute convergence of $f$ we get:}\de{als Potenzreihe gegeben ist, dann gilt im absoluten Konvergenzbereich der Reihe:}
\begin{align*}\renewcommand*{\arraystretch}{1.6}
    \begin{array}{rrl|rcl}  f(z)    &:=&\displaystyle\sum_{n=0}^\infty A_n\frac{z^n}{n!}\\
                    zf'(z)  &=& \displaystyle\sum_{n=0}^\infty n A_n\frac{z^n}{n!}
                    &    f'(z) &=&\displaystyle\sum_{n=0}^\infty A_{n+1}\frac{z^n}{n!}\\
                    z^2f''(z) &=& \displaystyle\sum_{n=0}^\infty n(n-1) A_n\frac{z^n}{n!}
                    & zf''(z) &=&\displaystyle\sum_{n=0}^\infty n A_{n+1}\frac{z^n}{n!}\\
                    z^3f'''(z) &=& \displaystyle\sum_{n=0}^\infty n(n-1)(n-2) A_n\frac{z^n}{n!}~~~
                    &~~ z^2f'''(z) &=&\displaystyle\sum_{n=0}^\infty n(n-1) A_{n+1}\frac{z^n}{n!}
    \end{array}
\end{align*}
\end{thm}
\begin{proof}
  \en{Since the power series $f(z)$ converges absolutely according to our premises, we may interchange summation and derivation.
  The identities are a direct consequence of this and the definition of $f(z)$:
  
  The reason we didn't reduce the fractions in the left identities with $n$ etc.~is that it will make equating coefficients easier.
  For the identites on the right side, we skip the first summand in $f'(z)=\sum_{n=0}^\infty A_n\frac{n\cdot z^{n-1}}{n!}=\sum_{n=1}^\infty A_n\frac{n\cdot z^{n-1}}{n!}=\sum_{n=1}^\infty A_n\frac{z^{n-1}}{(n-1)!}$ (which is zero anyway), then we reduce the fraction by $n$.
  Then we do an index shift by $1$ and get the first identity on the right side, i.e.~$f'(z)=\sum_{m=0}^\infty A_{m+1}\frac{z^m}{m!}$.
  For the further formulae on the right side, we don't use any further index shifts.\pagebreak
  }%
  \de{Weil die Reihe $f(z)$ nach Voraussetzung absolut konvergiert, dürfen wir Summation und Ableitung vertauschen.
  Die Formeln folgen dann direkt aus der Definition von $f(z)$.
  Bei den linken Formeln wird nicht gekürzt, damit man es später beim Koeffizientenvergleich leichter hat.
  Für die Formeln der rechten Spalte wird zunächst bei $f'(z)=\sum_{n=0}^\infty A_n\frac{n\cdot z^{n-1}}{n!}=\sum_{n=1}^\infty A_n\frac{n\cdot z^{n-1}}{n!}=\sum_{n=1}^\infty A_n\frac{z^{n-1}}{(n-1)!}$ der nullte Summand weggelassen (der sowieso Null ist), dann wird mit $n$ gekürzt.
  Jetzt folgt noch eine Indexverschiebung um $1$ und wir erhalten die oberste Formel der rechten Spalte, nämlich $f'(z)=\sum_{m=0}^\infty A_{m+1}\frac{z^m}{m!}$.
  Für die weiteren Formeln der rechten Spalte wird wieder aufs Kürzen und Indexverschieben verzichtet.}
  \end{proof}

\pagebreak

\begin{theo}\label{\en{EN}dgl2f1}
\en{The hypergeometric function $f(z)={_2F_1}(a,b;c;z)$ satisfies the "hypergeometric differential equation":}%
\de{Die hypergeometrische Funktion $f(z)={_2F_1}(a,b;c;z)$ erfüllt die hypergeometrische Differentialgleichung:}
$$z(z-1) f''(z) + \left[(a+b+1) z - c\right] f'(z) + ab f(z) = 0$$
\end{theo}
\begin{proof}
    \en{We prove this by equating coefficients.
    The hypergeometric series can be written as $\displaystyle f(z)=\sum_{n=0}^\infty A_n\frac{z^n}{n!}$ with $\displaystyle A_n:=\frac{(a)_n\cdot(b)_n}{(c)_n}$.
    The Definition~\ref{\en{EN}defpoch} of the Pochhammer symbols yields $(a)_{n+1}=(a)_n\cdot(a+n)$ and thus }%
    \de{Wir führen den Beweis mit Hilfe eines Koeffizientenvergleichs.
    Die hypergeometrische Funktion lautet $\displaystyle f(z)=\sum_{n=0}^\infty A_n\frac{z^n}{n!}$ mit $\displaystyle A_n:=\frac{(a)_n\cdot(b)_n}{(c)_n}$.
    Die Definition~\ref{\en{EN}defpoch} der Pochhammersymbole sagt $(a)_{n+1}=(a)_n\cdot(a+n)$ und somit }
    \begin{align*}
        A_{n+1}&=\frac{(a+n)(b+n)}{(c+n)}\cdot A_n\\
        \Longrightarrow (c+n)\cdot A_{n+1} &= (n^2+(a+b)n+ab)\cdot A_n\\
        \Longrightarrow (c+n)\cdot A_{n+1} &= (n(n-1)+(a+b+1)n+ab)\cdot A_n\\
        \Longrightarrow c\cdot A_{n+1} + n\cdot A_{n+1} &= n(n-1)\cdot A_n+(a+b+1) n \cdot A_n+ab\cdot A_n
    \end{align*}
    \en{Now we recognize the coefficients of Prop.~\ref{\en{EN}koeffvergl} and get:}%
    \de{Wenn wir nun Satz~\ref{\en{EN}koeffvergl} verwenden, erkennen wir an dieser Koeffizientengleichung:}
    \begin{align*}
        cf'(z) + zf''(z) = z^2f''(z) + (a+b+1)zf'(z)+abf(z)\\
        \Longrightarrow\qquad z(z-1) f''(z) + \left[(a+b+1) z - c\right] f'(z) + ab f(z) = 0
    \end{align*}
    \en{Thus we have proven that ${_2F_1}$ satisfies the hypergeometric differential equation.}%
    \de{Somit ist bewiesen, dass die hypergeometrische Funktion die genannte Differentialgleichung erfüllt.}
\end{proof}

\begin{thm}\label{\en{EN}dgl3f2}
\en{The hypergeometric function $g(z)={_3F_2}(\alpha,\beta,\gamma;\delta,\varepsilon;z)$ satisfies this differential equation:}%
\de{Die verallgemeinerte hypergeometrische Funktion $g(z)={_3F_2}(\alpha,\beta,\gamma;\linebreak[3]\delta,\varepsilon;z)$ erfüllt die folgende Differentialgleichung:}
\begin{align*}
    (z^3-z^2)\cdot g'''(z) + [(\alpha+\beta+\gamma+3)z^2-(\delta+\varepsilon+1)z]\cdot g''(z)&\\
    +~[(1+\alpha+\beta+\gamma+\alpha\beta+\alpha\gamma+\beta\gamma)z-\delta\varepsilon]\cdot g'(z)+\alpha\beta\gamma\cdot g(z)&=0
\end{align*}
\end{thm}
\begin{proof}
    \en{We prove this like in Thm.~\ref{\en{EN}dgl2f1} by equating coefficients.
    The hypergeometric function now is $\displaystyle g(z)=\sum_{n=0}^\infty A_n\frac{z^n}{n!}$ with the coefficients $\displaystyle A_n:=\frac{(\alpha)_n\cdot(\beta)_n\cdot(\gamma)_n}{(\delta)_n \cdot (\varepsilon)_n}$.
    The Definition~\ref{\en{EN}defpoch} of the Pochhammer symbols yields $(a)_{n+1}=(a)_n\cdot(a+n)$ and thus}%
    \de{Wir verwenden wieder einen Koeffizientenvergleich als Beweis, völlig analog zum Beweis von Thm.~\ref{\en{EN}dgl2f1}.
    Die verallgemeinerte hypergeometrische Funktion lautet $\displaystyle g(z)=\sum_{n=0}^\infty A_n\frac{z^n}{n!}$ mit den Koeffizienten $\displaystyle A_n:=\frac{(\alpha)_n\cdot(\beta)_n\cdot(\gamma)_n}{(\delta)_n \cdot (\varepsilon)_n}$.
    Die Definition~\ref{\en{EN}defpoch} der Pochhammersymbole sagt $(a)_{n+1}=(a)_n\cdot(a+n)$ und somit }
    \begin{align*}
        A_{n+1}&=\frac{(\alpha+n)(\beta+n)(\gamma+n)}{(\delta+n)(\varepsilon+n)}\cdot A_n\\
        \Longrightarrow (\delta+n)(\varepsilon+n)\cdot A_{n+1} &= (\alpha+n)(\beta+n)(\gamma+n)\cdot A_n\\
        \Longrightarrow [n^2+(\delta+\varepsilon)n+\delta\varepsilon] A_{n+1} &= [n^3+(\alpha+\beta+\gamma)n^2+(\alpha\beta+\alpha\gamma+\beta\gamma)n+\alpha\beta\gamma] A_n
    \end{align*}
    \en{But it is $n^2=n(n-1)+1n$ and $n^3=n(n-1)(n-2)+3n^2-2n$, so we get}%
    \de{Jetzt ist aber $n^2=n(n-1)+1n$ und $n^3=n(n-1)(n-2)+3n^2-2n$, also gilt}
    \begin{align*}
         &~[n(n-1)+(\delta+\varepsilon+1)n+\delta\varepsilon]\cdot  A_{n+1} \\
        =&~[n(n-1)(n-2)+(\alpha+\beta+\gamma+3)n^2+(\alpha\beta+\alpha\gamma+\beta\gamma-2)n+\alpha\beta\gamma]\cdot  A_n\\
        \Longrightarrow~~~~&~n(n-1)  A_{n+1}+(\delta+\varepsilon+1)n  A_{n+1}+\delta\varepsilon  A_{n+1} = n(n-1)(n-2) A_n\\
        &\qquad\qquad+~(\alpha+\beta+\gamma+3)n^2 A_n+(\alpha\beta+\alpha\gamma+\beta\gamma-2)n A_n+\alpha\beta\gamma A_n\\
        \Longrightarrow~~~~&~n(n-1)  A_{n+1}+(\delta+\varepsilon+1)n  A_{n+1}+\delta\varepsilon  A_{n+1} = n(n-1)(n-2) A_n\\
        &\qquad\qquad+~(\alpha+\beta+\gamma+3)n(n-1) A_n\\
        &\qquad\qquad+~(\alpha\beta+\alpha\gamma+\beta\gamma-2+\alpha+\beta+\gamma+3)n A_n+\alpha\beta\gamma A_n
    \end{align*}
    \en{Again we recognize the coefficients of Prop.~\ref{\en{EN}koeffvergl} and get:}%
    \de{Wenn wir nun Satz~\ref{\en{EN}koeffvergl} verwenden, erkennen wir an dieser Koeffizientengleichung:}
    \begin{align*}
        &z^2\cdot g'''(z)+(\delta+\varepsilon+1)z\cdot g''(z)+\delta\varepsilon\cdot g'(z)\\
        &\qquad= z^3\cdot g'''(z)+(\alpha+\beta+\gamma+3)z^2g''(z)\\
        &\qquad+(\alpha\beta+\alpha\gamma+\beta\gamma+\alpha+\beta+\gamma+1)z\cdot g'(z)+\alpha\beta\gamma\cdot g(z)\\
        \Longrightarrow~~~~&~\left[z^3-z^2\right]\cdot g'''(z) + [(\alpha+\beta+\gamma+3)z^2-(\delta+\varepsilon+1)z]\cdot g''(z)\\
        &\qquad+[(\alpha\beta+\alpha\gamma+\beta\gamma+\alpha+\beta+\gamma+1)z-\delta\varepsilon]\cdot g'(z)+\alpha\beta\gamma\cdot g(z)=0
    \end{align*}
    \en{Thus we have proven that ${_3F_2}$ satisfies said differential equation.}%
    \de{Somit ist bewiesen, dass die verallgemeinerte hypergeometrische Funktion die genannte Differentialgleichung erfüllt.}
\end{proof}

\begin{theo}\label{\en{EN}satzclausen}
\en{The following formula (published and proven in 1828 by Thomas Clausen) applies:}\de{Es gilt die Formel von Thomas Clausen aus dem Jahr 1828, nämlich}
$$\left({_2F_1}\lk a,b;a+b+\frac 1 2;z\rk\right)^2 = {_3F_2}\lk 2a,2b,a+b;2a+2b,a+b+\frac 1 2;z\rk.$$
\en{Setting $a=\frac{1}{12}$ and $b=\frac{5}{12}$ yields:}\de{Insbesondere gilt für $a=\frac{1}{12}$ und $b=\frac{5}{12}$:}
$$\left({_2F_1}\lk\frac{1}{12},\frac{5}{12};1;z\rk\right)^2 = {_3F_2}\lk\frac{1}{6},\frac{5}{6},\frac{1}{2};1,1;z\rk.$$
\end{theo}
\begin{proof}
    \en{We prove that both sides of this equation satisfy the same third order differential equation:}%
    \de{Wir zeigen zunächst, dass beide Seiten der Gleichung derselben Differentialgleichung dritter Ordnung genügen.}
    
    \en{First we look at the right side, which we call $g(z)$. Here we recognize the hypergeometric function ${_3F_2}$ with its differential equation from Prop.~\ref{\en{EN}dgl3f2}. Setting $\alpha=2a$, $\beta=2b$, $\gamma=a+b$, $\delta=2a+2b$ and $\varepsilon=a+b+\frac 1 2$ we get:}%
    \de{Zunächst betrachten wir die rechte Seite, die wir $g(z)$ nennen. Das ist die verallgemeinerte hypergeometrische Funktion ${_3F_2}$, für die wir in Satz~\ref{\en{EN}dgl3f2} bereits eine Differentialgleichung bewiesen haben. Wir müssen nur noch $\alpha=2a$, $\beta=2b$, $\gamma=a+b$, $\delta=2a+2b$ und $\varepsilon=a+b+\frac 1 2$ einsetzen und erhalten:}
    \begin{align}
        \left[z^3-z^2\right]\cdot g'''(z) + \left[3\left(a+b+1\right)z^2-3\left(a+b+\frac 1 2\right)z\right]\cdot g''(z)&\nonumber\\
        +~\left[\left(1+3a+3b+8ab+2a^2+2b^2\right)z-\left(a+b\right)\left(2a+2b+1\right)\right]\cdot g'(z)&\label{\en{EN}ue13}\\
        +~4ab\left(a+b\right)\cdot g(z)&=0\nonumber
    \end{align}
    
    \en{Next we look at the left side, which we call $h(z):=\left({_2F_1}\lk a,b;a+b+\frac 1 2;z\rk\right)^2$. Here it gets more complicated, since we have the square of a power series. But we will show that $h(z)$ also satisfies the differential equation~(\ref{\en{EN}ue13}).}%
    \de{Jetzt kommen wir zur linken Seite, die wir $h(z)$ nennen. Hier steht das Quadrat einer Potenzreihe, deshalb wird es wesentlich komplizierter. Wir werden zeigen, dass auch $h(z):=\left({_2F_1}\lk a,b;a+b+\frac 1 2;z\rk\right)^2$ eine Lösung der Differentialgleichung~(\ref{\en{EN}ue13}) ist.}
    
    \en{We start with $f(z)={_2F_1}\left(a,b;a+b+\frac 1 2;z\right)$, for which we have the differential equation from Thm.~\ref{\en{EN}dgl2f1} with $c=a+b+\frac 1 2$:}%
    \de{Für $f(z)={_2F_1}\left(a,b;a+b+\frac 1 2;z\right)$ gilt die Differentialgleichung aus Thm.~\ref{\en{EN}dgl2f1}, wobei $c=a+b+\frac 1 2$ ist:}
    \begin{align}
        \left(z^2-z\right) f''(z) + \left[(a+b+1) z - c\right] f'(z) + ab f(z) &= 0\qquad\left|~\cdot~z\right.\label{\en{EN}C0}\\
        \left(z^3-z^2\right) f''(z) + \left[(a+b+1) z^2 - cz\right] f'(z) + abz f(z) &= 0\qquad\left|~\frac{d}{dz}\right.\label{\en{EN}C1}\\
        \left(z^3-z^2\right) f'''(z) + \left[(a+b+4) z^2 - \left(c+2\right)z\right] f''(z)&\nonumber\\
        +~\left[\left(ab+2a+2b+2\right)z-c\right]f'(z) + ab f(z) &= 0\label{\en{EN}C2}
    \end{align}
    \en{In the last step we derived the equation by $z$ minding the product rule and grouped similar terms.
    Now we form a linear combination of the equations~(\ref{\en{EN}C0}),~(\ref{\en{EN}C1}) and~(\ref{\en{EN}C2}), as Clausen suggested in \cite{\en{EN}clausen}:}%
    \de{Im letzten Schritt haben wir die Gleichung unter Beachtung der Produktregel nach $z$ abgeleitet und dann ähnliche Terme gruppiert.
    Nun bilden wir (wie von Clausen in \cite{\en{EN}clausen} vorgeschlagen) eine Linearkombination der Gleichungen~(\ref{\en{EN}C0}),~(\ref{\en{EN}C1}) und~(\ref{\en{EN}C2}), nämlich:}
    $$(2a+2b-1)\cdot2f(z)\cdot(\ref{\en{EN}C0}) + 6f'(z)\cdot(\ref{\en{EN}C1})+ 2f(z)\cdot(\ref{\en{EN}C2})$$

    \pagebreak
    \en{This linear combination reads:\\}\de{Diese Linearkombination lautet ausgeschrieben:\\}
    \noindent\scalebox{0.93}{\begin{minipage}{\the\textwidth}
\begin{align}
    0 = (2a+2b-1)\cdot 2f \cdot \left[ \underline{\underline{(z^2-z) f''}} + \underline{\left((a+b+1)z-c\right)f'} + abf\right]&\nonumber\\
    +~6f'\cdot\left[\underline{\underline{\underline{(z^3-z^2) f''}}} + \underline{\underline{\left((a+b+1)z^2-cz\right)f'}} + \underline{abz f}\right]&\label{\en{EN}unterstr}\\
    +~2f\cdot\left[\underline{\underline{\underline{(z^3-z^2) f'''}}} + \underline{\underline{\left((a+b+4)z^2-(c+2)z\right)f''}} + \underline{\left((ab+2a+2b+2)z-c\right)f'} + ab f\right] &\nonumber
\end{align}
\end{minipage}}\\[1ex]

\en{Now we will combine the different terms in equation~(\ref{\en{EN}unterstr}) in order to get equation~(\ref{\en{EN}linkombclausen}) later on.
The terms that are three times underlined in~(\ref{\en{EN}unterstr}) contain the third derivatives:}%
\de{Jetzt werden wir die verschiedenen Terme in Gleichung~(\ref{\en{EN}unterstr}) zusammenfassen, um weiter unten Gleichung~(\ref{\en{EN}linkombclausen}) zu erhalten.
Die dreifach unterstrichenen Summanden in~(\ref{\en{EN}unterstr}) enthalten die dritten Ableitungen:}
$$\underline{\underline{\underline{\left\{2f f''' + 6 f' f''\right\}}}}\cdot\left[z^3-z^2\right]$$
\en{The terms in~(\ref{\en{EN}unterstr}) that are twice underlined contain the second derivatives:}%
\de{Die zweifach unterstrichenen Summanden in~(\ref{\en{EN}unterstr}) enthalten die zweiten Ableitungen:}
\begin{align*}
    &\underline{\underline{\left\{2f f''\right\}}}\cdot \underbrace{\left[(2a+2b-1)\cdot(z^2-z)+(a+b+4)z^2-(c+2)z\right]}_{=:A_1}\\
    +~&\underline{\underline{\left\{2f'^2\right\}}}\cdot \underbrace{\left[3(a+b+1)z^2-3(a+b+1/2)z\right]}_{=:A_2}\\
    =~&\underline{\underline{\left\{2f'f'' + 2f'^2\right\}}}\cdot \left[3(a+b+1)z^2-3(a+b+1/2)z\right],\\
    \text{\en{since}\de{denn} }A_1 =~ & (2a+2b-1+a+b+4)z^2-(2a+2b-1+a+b+1/2+2)z = A_2
\end{align*}
\en{The terms that are underlined once contain the first derivatives:}%
\de{Die einfach unterstrichenen Summanden in~(\ref{\en{EN}unterstr}) enthalten die ersten Ableitungen:}
\begin{align*}
&\underline{\left\{2f f'\right\}}\cdot \left[(2a+2b-1)\cdot\left((a+b+1)z-c\right)+3abz+\left((ab+2a+2b+2)z-c\right)\right]\\
=~&\underline{\left\{2f f'\right\}}\cdot \left[ \left( (2a+2b-1)(a+b+1)+4ab+2a+2b+2 \right)z-\left((2a+2b-1)c+c\right) \right]\\
=~&\underline{\left\{2f f'\right\}}\cdot \left[ \left( 1+3a+3b+8ab+2a^2+2b^2 \right)z-(a+b)(2a+2b+1) \right]
\end{align*}
\en{Finally, the terms that aren't underlined contain no derivatives:}%
\de{Die nicht unterstrichenen Summanden in~(\ref{\en{EN}unterstr}) enthalten keine Ableitungen:}
\begin{align*}
\left\{f^2\right\}\cdot \left[2ab\cdot(2a+2b-1)+2ab\right] = \left\{f^2\right\}\cdot \left[4ab(a+b)\right]
\end{align*}
    \en{If we combine all this, equation~(\ref{\en{EN}unterstr}) yields:}%
    \de{Insgesamt geht Gleichung~(\ref{\en{EN}unterstr}) also über in:}
    \begin{align}
        &\left[z^3-z^2\right]\cdot\left\{2ff'''+6f'f''\right\}\nonumber\\
        +&\left[3\left(a+b+1\right)z^2-3\left(a+b+\frac 1 2\right)z\right]\cdot \left\{2ff''+2f'^2\right\}\label{\en{EN}linkombclausen}\\
        +&\left[\left(1+3a+3b+8ab+2a^2+2b^2\right)z-\left(a+b\right)\left(2a+2b+1\right)\right]\cdot \left\{2ff'\right\}\nonumber\\
        +&\left[4ab\left(a+b\right)\right]\cdot \left\{f^2\right\} =0\nonumber
    \end{align}
    \en{and here we recognize the content of the square brackets from equation~(\ref{\en{EN}ue13}).\\
    Next we use $h(z)=\left(f(z)\right)^2$:}%
    \de{und wir erkennen, dass der Inhalt der eckigen Klammern genau so auch in Gleichung~(\ref{\en{EN}ue13}) vorkam.
    Nun gilt aber auch noch, dass $h(z)=\left(f(z)\right)^2$, und somit:}
    \begin{align}
        h(z) &= (f(z))^2\nonumber\\
        h'(z) &= 2f(z)f'(z)\nonumber\\
        h''(z) &= 2f(z)f''(z)+2(f'(z))^2\label{\en{EN}ablH}\\
        h'''(z) &= 2f'(z)f''(z)+2f(z)f'''(z)+4f'(z)f''(z)\nonumber\\
                &= 2f(z)f'''(z)+6f'(z)f''(z)\nonumber
    \end{align}
    \en{If we now replace the curly brackets in~(\ref{\en{EN}linkombclausen}) by these derivatives of $h(z)$, we realize that $h(z)=\left({_2F_1}\lk a,b;a+b+\frac 1 2;z\rk\right)^2$ is another solution to the differential equation~(\ref{\en{EN}ue13}).\pagebreak}%
    \de{Wir können also für die geschweiften Klammern in~(\ref{\en{EN}linkombclausen}) die Ableitungen von $h(z)$ einsetzen und haben bewiesen, dass auch $h(z)=\left({_2F_1}\lk a,b;a+b+\frac 1 2;z\rk\right)^2$ eine Lösung der Differentialgleichung~(\ref{\en{EN}ue13}) ist.}

    \en{We still have to show that these two solutions are the same, that we have $h(z)=g(z)$.
    Since they both are solutions to the same ordinary differential equation of third order, it suffices to show that $h(0)=g(0)$ and $h'(0)=g'(0)$ and $h''(0)=g''(0)$.
    In order to prove this, we use the notation $A_n:=\frac{(a)_n\cdot(b)_n}{(a+b+1/2)_n}$ for the coefficients of $f(z)$.
    The first three of these coefficients are (cf.~Def.~\ref{\en{EN}defpoch}):}%
    \de{Wir müssen jetzt noch zeigen, dass die beiden Lösungen sogar gleich sind, also dass $h(z) = g(z)$ gilt.
    Da es sich um eine Differentialgleichung dritter Ordnung handelt, müssen wir z.B.~für $z=0$ zeigen, dass die Funktionswerte und die ersten beiden Ableitungen paarweise übereinstimmen.
    Hierfür verwenden wir die Notation $A_n:=\frac{(a)_n\cdot(b)_n}{(a+b+1/2)_n}$ für die Koeffizienten von $f(z)$.
    Die ersten drei dieser Koeffizienten lauten (vgl.~Def.~\ref{\en{EN}defpoch}):}
    \begin{align*}
        A_0 &= \frac{(a)_0\cdot(b)_0}{(a+b+1/2)_0}=\frac{1\cdot 1} 1 = 1\\
        A_1 &= \frac{(a)_1\cdot(b)_1}{(a+b+1/2)_1}=\frac{a\cdot b} {a+b+\frac 1 2}\\
        A_2 &= \frac{(a)_2\cdot(b)_2}{(a+b+1/2)_2}=\frac{a(a+1)\cdot b(b+1)} {\left(a+b+\frac 1 2\right)\left(a+b+\frac 3 2\right)}
    \end{align*}
    \en{Then we use the derivatives in~(\ref{\en{EN}ablH}) and obtain the following values of $h(z)=(f(z))^2$:}%
    \de{Somit gilt für $h(z)=(f(z))^2$ wegen der Ableitungsdarstellungen~(\ref{\en{EN}ablH}):}
    \begin{align*}
        h(0) &= (f(0))^2 = A_0^2 = 1\\
        h'(0) &= 2 f(0) f'(0) = 2 A_0 A_1 = \frac{2 ab} {a+b+\frac 1 2}\\
        h''(0) &= 2f(0)f''(0) + 2(f'(0))^2 = 2A_0A_2+2A_1^2\\
            &=\frac{2ab(a+1)(b+1)} {(a+b+1/2)(a+b+3/2)} + \frac{2a^2b^2} {(a+b+1/2)^2}\\
            &=\frac{ab(4a^2b+4ab^2+8ab+2a^2+2b^2+3a+3b+1)}{(a+b+1/2)^2(a+b+3/2)}
    \end{align*}
    \en{For the other solution $g(z)={_3F_2}(\alpha,\beta,\gamma;\delta,\varepsilon;z)$ it holds:}%
    \de{Für die andere Lösung $g(z)={_3F_2}(\alpha,\beta,\gamma;\delta,\varepsilon;z)$ gilt:}
    \begin{align*}
        g(0)&=\frac{(\alpha)_0\cdot(\beta)_0\cdot(\gamma)_0}{(\delta)_0 \cdot (\varepsilon)_0} = 1 = h(0)\\
        g'(0)&=\frac{(\alpha)_1\cdot(\beta)_1\cdot(\gamma)_1}{(\delta)_1 \cdot (\varepsilon)_1} = \frac{\alpha\beta\gamma}{\delta\varepsilon}=\frac{2a\cdot2b\cdot(a+b)}{(2a+2b)\cdot\left(a+b+\frac 1 2\right)}= \frac{2 ab} {a+b+\frac 1 2}=h'(0)\\
        g''(0)&=\frac{(\alpha)_2\cdot(\beta)_2\cdot(\gamma)_2}{(\delta)_2 \cdot (\varepsilon)_2}
        = \frac{\alpha(\alpha+1)\beta(\beta+1)\gamma(\gamma+1)}{\delta(\delta+1)\varepsilon(\varepsilon+1)}\\
        &= \frac{2a(2a+1)2b(2b+1)(a+b)(a+b+1)}{(2a+2b)(2a+2b+1)\left(a+b+\frac 1 2\right)\left(a+b+\frac 3 2\right)}\\
        &= \frac{ab(2a+1)(2b+1)(a+b+1)}{(a+b+1/2)^2(a+b+3/2)}\\
        &=\frac{ab(4a^2b+4ab^2+8ab+2a^2+2b^2+3a+3b+1)}{(a+b+1/2)^2(a+b+3/2)}=h''(0)
    \end{align*}
    \en{From the Picard-Lindelöf theorem we deduce that $h(z)=g(z)$ and thus we have proven Clausen's formula.}%
    \de{Insgesamt gilt also wegen des Satzes von Picard-Lindelöf $h(z)=g(z)$ und wir haben die Formel von Clausen bewiesen.}
\end{proof}

\vfill\pagebreak\section{\en{Picard Fuchs Differential Equation}\de{Picard-Fuchs-Differentialgleichung}}\label{\en{EN}kappicardfuchs}
\renewcommand{\leftmark}{\en{Picard Fuchs Differential Equation}\de{Picard-Fuchs-Differentialgleichung}}
\en{In this chapter, which can be read straight after chapter~\ref{\en{EN}kapGitter}, we prove the Picard Fuchs differential equation.
The proof follows \cite[p.~33-34, ch.~I.2, §3]{\en{EN}klein} -- but there every $g$ has a different sign.}%
\de{In diesem Kapitel, das direkt im Anschluss an Kapitel~\ref{\en{EN}kapGitter} gelesen werden kann, beweisen wir die Picard-Fuchs-Differentialgleichung.
Der Beweis orientiert sich stark an \cite[S.~33-34, Kap.~I.2, §3]{\en{EN}klein}. Dort hat allerdings jedes $g$ genau das andere Vorzeichen.}\\

\begin{theo}[Picard\en{ }\de{-}Fuchs]\label{\en{EN}picardfuchs}
\en{The periods $\Omega_{1,2}(J)$ of $L_J$ from Def.~\ref{\en{EN}definLJ} are solutions to the following differential equation:}%
\de{Für die Perioden $\Omega_{1,2}(J)$ von $L_J$ aus Def.~\ref{\en{EN}definLJ} gilt folgende Differentialgleichung:}
$$\frac{d^2\Omega}{dJ^2}+\frac 1 J \cdot \frac{d\Omega}{dJ} + \frac{31J-4}{144J^2(J-1)^2}\cdot\Omega = 0$$
\end{theo}
\begin{proof} 
\en{We start with the representations of the basic periods $\Omega_{1,2}$ and basic quasiperiods $H_{1,2}$ of the lattice $L_J$ from Prop.~\ref{\en{EN}satzint}:}%
\de{Zunächst gilt nach Satz~\ref{\en{EN}satzint} für die Perioden $\Omega_{1,2}$ und Quasiperioden $H_{1,2}$ des Gitters $L_J$:}
\begin{align*}
    \Omega_k &= \oint_{\alpha_k} \frac{dx}{y}\qquad\text{\en{and}\de{und}}\qquad H_k = -\oint_{\alpha_k} \frac{x~dx}{y}
\end{align*}
\en{The defining equation of $X(L_J)$ reads $y^2=4x^3-g(x+1)$ with $g=\frac{27J}{J-1}$ (cf.~Prop.~\ref{\en{EN}satz44}). This yields:}%
\de{Die definierende Gleichung von $X(L_J)$ lautet $y^2=4x^3-g(x+1)$ mit $g=\frac{27J}{J-1}$ (vgl.~Satz~\ref{\en{EN}satz44}). Hieraus folgt:}
\begin{align}
    \frac{d}{dg}(y^2) &= \frac{d}{dg}\lk\left(\frac 1 y\right)^{-2}\rk
    = -2 \cdot \left(\frac 1 y\right)^{-3} \cdot \frac{d}{dg}\lk\frac 1 y\rk\label{\en{EN}ablxyg}\\
    \Longrightarrow\quad\frac{d}{dg}\lk\frac 1 y\rk &= \frac{-1}{2y^3}\cdot \frac{d}{dg}(y^2) 
    = \frac{-1}{2y^3}\cdot \frac{d}{dg}(4x^3-g(x+1)) = \frac{x+1}{2y^3}\nonumber
\end{align}
\en{From this we obtain the derivatives of $\Omega$ and $H$ with respect to $g$ using Leibniz's rule:}%
\de{Somit folgen die Ableitungen von $\Omega$ und $H$ nach $g$ mit der Leibnizregel:}
\begin{align}
    \frac{d\Omega}{dg} &= \frac{d}{dg}\lk\oint_{\alpha} \frac{dx}{y}\rk
                        = \oint_{\alpha}\frac{d}{dg}\lk\frac{1}{y}\rk dx
                        = \oint_{\alpha} \frac{x~dx}{2y^3} + \oint_{\alpha} \frac{dx}{2y^3}\label{\en{EN}stern}\\
    \text{\en{and}\de{und} }\quad\frac{dH}{dg} &= \frac{d}{dg}\lk-\oint_{\alpha} \frac{x~dx}{y}\rk
                        = -\oint_{\alpha}\frac{d}{dg}\lk\frac{1}{y}\rk\cdot x~dx
                        = -\oint_{\alpha} \frac{x^2~dx}{2y^3} - \oint_{\alpha} \frac{x~dx}{2y^3}\nonumber
\end{align}
\en{We have to calculate the values $I_n:=\oint_{\alpha} \frac{x^n~dx}{2y^3}$ for $n=0,1,2$.
In order to do this, we use the functions $f_n(x):=\frac{x^n}{y}$ with $n=0,1,2$ and start by deriving $f_0(x)=\frac 1 y$ as in~(\ref{\en{EN}ablxyg}):}%
\de{Wir müssen also noch die Werte der Integrale $I_n:=\oint_{\alpha} \frac{x^n~dx}{2y^3}$ für $n=0,1,2$ berechnen.
Hierfür verwenden wir die Funktionen $f_n(x):=\frac{x^n}{y}$ mit $n=0,1,2$, wobei wir zunächst $f_0(x)=\frac 1 y$ so wie in~(\ref{\en{EN}ablxyg}) ableiten:}
\begin{align*}
    f_0 (x) &=\frac 1 y\\
    f_0'(x) &= \frac{d}{dx}\lk\frac 1 y\rk = \frac{-1}{2y^3}\cdot\frac{d}{dx}(y^2) = \frac{-1}{2y^3}\cdot \frac{d}{dx}(4x^3-g(x+1))\\
    &= \frac{-1}{2y^3}\cdot(12x^2-g)=\frac{g-12x^2}{2y^3}\\
    f_1 (x) &=\frac x y = x \cdot f_0(x)\\
    f_1'(x) &=f_0(x)+x\cdot f_0'(x) = \frac 1 y + x \cdot \frac{g-12x^2}{2y^3} = \frac 1 y + \frac{gx-3\cdot 4x^3}{2y^3}\\
            &= \frac 1 y + \frac{gx-3\cdot \left(y^2 + g x +g\right)}{2y^3}
            = \frac 1 y + \frac{-3y^2 - 2 g x -3 g}{2y^3}\\ &= \frac 2 {2y} - \frac{3y^2}{2y^3} - \frac{2gx+3g}{2y^3}
            = -\frac {1} {2y} - \frac{2gx+3g}{2y^3}\end{align*}\begin{align*}
    f_2 (x) &= \frac{x^2}{y} = x\cdot f_1(x)\\
    f_2'(x) &= f_1(x) + x \cdot f_1'(x) = \frac x y + x \cdot \left(-\frac {1} {2y} - \frac{2gx+3g}{2y^3}\right)\\
            &=\frac {2x} {2y} -\frac {x} {2y} - \frac{2gx^2+3gx}{2y^3} = \frac {x} {2y} - \frac{2gx^2+3gx}{2y^3}
\end{align*}

\en{Prop.~\ref{\en{EN}satzint} tells us that $\alpha$ is a closed path, which avoids the zeros and poles of $\wp$ (cf.~Remark~\ref{\en{EN}bemwege}). From this we deduce that}%
\de{Aus Satz~\ref{\en{EN}satzint} folgt, dass $\alpha$ eine geschlossene Kurve ist, die die Nullstellen und Polstellen der $\wp$-Funktion vermeidet (vgl.~Bem.~\ref{\en{EN}bemwege}). Hieraus folgt:}
$$\oint_{\alpha} f_n'(x)dx=\int_0^1f_n'(\alpha(t))\alpha'(t)dt=\left[f_n(\alpha(t))\right]_0^1=0$$
\en{Using the derivatives of $f_0$, $f_1$ and $f_2$ we calculated above we get the following relations between the values $I_n=\oint_{\alpha} \frac{x^n~dx}{2y^3}$:}%
\de{Wenn wir die oben berechneten Ableitungen von $f_0$, $f_1$ und $f_2$ verwenden erhalten wir die folgenden Beziehungen zwischen den Integralen $I_n=\oint_{\alpha} \frac{x^n~dx}{2y^3}$:}
\begin{align*}
    0= \oint_\alpha f_0'(x)dx &= \qquad\qquad\qquad~ g \cdot \oint_\alpha \frac {dx} {2y^3}\quad ~~- 12 \cdot \oint_\alpha \frac {x^2~dx} {2y^3}\\
    0= \oint_\alpha f_1'(x)dx &= - \underbrace{\oint_\alpha \frac {dx} {2y}}_{=~\frac 1 2 \Omega}\quad  - ~ 2g \cdot \oint_\alpha \frac {x~dx} {2y^3}\quad - 3g\cdot \oint_\alpha \frac {dx} {2y^3}\\
    0=\oint_\alpha f_2'(x)dx &= \underbrace{\oint_\alpha \frac {x~dx} {2y}}_{=~-\frac 1 2 H}\quad - ~2g \cdot \oint_\alpha \frac {x^2~dx} {2y^3}\quad  -3g \cdot \oint_\alpha \frac {x~dx} {2y^3}
\end{align*}
\en{This produces the following system of linear equations in $I_0,I_1,I_2$:}%
\de{Also gilt für die Integrale $I_0,I_1,I_2$ folgendes lineares Gleichungssystem:}
\begin{align*}
    \left|
    \begin{aligned}
        \phantom{3}g\cdot I_0 \phantom{~+2g\cdot I_1}-12\cdot I_2 &= 0 & \text{(I)}~\\
        ~3 g \cdot I_0 + 2g\cdot I_1 \phantom{-12g\cdot I_2} &= -\frac 1 2 \cdot \Omega & \text{(II)}~\\
        \phantom{3g\cdot I_0 +} 3g\cdot I_1 + 2g\cdot I_2 &= -\frac 1 2\cdot H & \text{(III)}~
    \end{aligned}
    \right|
    \quad \Longrightarrow \quad 
    \left|
    \begin{aligned}
        I_0 &= \displaystyle\frac{9\Omega-6H}{2g(g-27)}\\[1ex]
        ~I_1 &= \displaystyle \frac{18H-g\Omega}{4g(g-27)}~\\[1ex]
        I_2 &= \displaystyle\frac{3\Omega-2H}{8(g-27)}
    \end{aligned}
    \right|
\end{align*}
\en{The value of $I_2$ comes from (III) -- 1\ko5 $\cdot$ (II) + 4\ko5 $\cdot$ (I). Then, (I) yields $I_0$ and (II) yields $I_1$.
Now we use these results in~(\ref{\en{EN}stern}) to get the desired values of $\frac{d\Omega}{dg}$ and $\frac{dH}{dg}$:}%
\de{Den Wert von $I_2$ erhält man z.B.~aus (III) -- 1\ko5 $\cdot$ (II) + 4\ko5 $\cdot$ (I). Dann erhält man aus $I_0$ aus (I) und dann $I_1$ aus (II).
Das setzen wir in~(\ref{\en{EN}stern}) ein und erhalten $\frac{d\Omega}{dg}$ und $\frac{dH}{dg}$:}
\begin{align*}
    \frac{d\Omega}{dg} &= I_0 + I_1 = \frac{9\Omega-6H}{2g(g-27)}+\frac{18H-g\Omega}{4g(g-27)} = \frac{(18-g)\Omega +6H}{4g(g-27)}\\
    \text{\en{and}\de{und} }\quad\frac{dH}{dg} &= -I_1 - I_2 = -\frac{18H-g\Omega}{4g(g-27)}-\frac{3\Omega-2H}{8(g-27)} = \frac{(2g-36)H-g\Omega}{8g(g-27)}
\end{align*}
\en{Using $g=\frac{27J}{J-1}$ we can transform these equations in $g$ into equations in $J$.
For this transformation, we use $\frac{dg}{dJ} = \frac{-27}{(J-1)^2}$ and $\frac{d}{dg} = \left(\frac{dg}{dJ}\right)^{-1}\cdot\frac{d}{dJ}$.
This transforms the two equations into}%
\de{Mit $g=\frac{27J}{J-1}$ kann man diese beiden Gleichungen, die von $g$ abhängen, in Gleichungen umwandeln, die stattdessen von $J$ abhängen.
Dazu nutzen wir $\frac{dg}{dJ} = \frac{-27}{(J-1)^2}$ und dass wir die Ableitung nach $g$ somit durch eine nach $J$ ersetzen können: $\frac{d}{dg} = \left(\frac{dg}{dJ}\right)^{-1}\cdot\frac{d}{dJ}$

Deshalb übersetzen sich die beiden Gleichungen in}
\begin{align*}
    \frac{(J-1)^2}{-27}\cdot\frac{d\Omega}{dJ}
    &= \frac{\left(18-\frac{27J}{J-1}\right)\Omega +6H}{4\cdot\frac{27J}{J-1}\cdot\left(\frac{27J}{J-1}-27\right)}\\
    \text{\en{and}\de{und}}\qquad\frac{(J-1)^2}{-27}\cdot\frac{dH}{dJ}
    &= \frac{\left(2\cdot\frac{27J}{J-1}-36\right)H-\frac{27J}{J-1}\Omega}{8\cdot\frac{27J}{J-1}\cdot\left(\frac{27J}{J-1}-27\right)}
\end{align*}
\en{which simplifies to:}\de{was sich wie folgt vereinfachen lässt:}
\begin{align}
    36 J(J-1)\frac{d\Omega}{dJ} &= 3(J+2)\Omega -2(J-1)H\label{\en{EN}dOdJ}\\
    \intertext{\en{and}\de{und}} 24 J(J-1)\frac{dH}{dJ} &= 3J \Omega -2(J+2)H\label{\en{EN}dHdJ}
\end{align}
\en{Now we derive equation~(\ref{\en{EN}dOdJ}) by $J$ once more (using the product rule), and obtain after grouping similar terms:}%
\de{Wenn wir nun Gleichung~(\ref{\en{EN}dOdJ}) nochmals nach $J$ ableiten (unter Beachtung der Produktregel), erhalten wir nach Zusammenfassen gleichartiger Terme:}
\begin{align*}
    36J(J-1)\frac{d^2\Omega}{dJ^2}+(69J-42)\frac{d\Omega}{dJ}+2(J-1)\frac{dH}{dJ}-3\Omega+2H&=0
\end{align*}
\en{Next we multiply this equation with $12J$ and eliminate $\frac{dH}{dJ}$ using~(\ref{\en{EN}dHdJ}). This yields:}%
\de{Nun multiplizieren wir diese Gleichung mit $12J$ und eliminieren $\frac{dH}{dJ}$ mit Hilfe von~(\ref{\en{EN}dHdJ}). Das führt auf:}
\begin{align*}
    432 J^2(J-1)\cdot\Omega'' + 12 J\cdot(69J-42)\cdot\Omega' - 33 J \cdot \Omega + (11J-2) \cdot 2H = 0
\end{align*}
\en{Here, we multiply with $(J-1)$ and eliminate $H$ using~(\ref{\en{EN}dOdJ}). This yields:}%
\de{Das multiplizieren wir mit $(J-1)$ und eliminieren $H$ mit Hilfe von~(\ref{\en{EN}dOdJ}). Wir erhalten:}
\begin{align*}
    432 J^2(J-1)^2\cdot\Omega'' + 12 J(J-1)(36J-36)\cdot\Omega' + (93J-12) \cdot \Omega = 0
\end{align*}
\en{One last division by $432 J^2(J-1)$ yields the Picard Fuchs differential equation:}%
\de{Eine letzte Division durch $432 J^2(J-1)^2$ liefert die Picard-Fuchs-Differentialgleichung:}\belowdisplayskip=-12pt
\begin{align*}
    \frac{d^2\Omega}{dJ^2} + \frac 1 J \cdot \frac{d\Omega}{dJ} + \frac{31J-4}{144J^2(J-1)^2}\cdot\Omega &= 0
\end{align*}
\end{proof}

\vfill\pagebreak\section{\en{Kummer's Solution}\de{Kummers Lösung}}\label{\en{EN}kapkummer}
\renewcommand{\leftmark}{\en{Kummer's Solution}\de{Kummers Lösung}}
\en{In this chapter, we use the Picard Fuchs differential equation to establish a connection between the periods of the lattice $\tilde L = \Delta(\tau)^{\frac{1}{12}}\cdot L_\tau$ and the hypergeometric function ${_2F_1}\lk\frac{1}{12},\frac{5}{12};1;\frac 1 J\rk$ with $J=J(\tau)$.}%
\de{Ziel dieses Kapitels ist es, mit Hilfe der Picard-Fuchs-Differentialgleichung einen Zusammenhang zwischen den Perioden des Gitters $\tilde L = \Delta(\tau)^{\frac{1}{12}}\cdot L_\tau$ und der hypergeometrischen Funktion ${_2F_1}\lk\frac{1}{12},\frac{5}{12};1;\frac 1 J\rk$ herzustellen, wobei $J=J(\tau)$ ist.}

\begin{thm}\label{\en{EN}satzbj}
\en{The following function $b(J)$, which is defined for $|J|>1$ by:}%
\de{Die Funktion $b(J)$, die im Bereich $|J|>1$ wie folgt definiert ist:}
$$b(J):= J^{-\frac 1 4}\cdot(1-J)^{\frac 1 4}\cdot{_2F_1}\lk\frac{1}{12},\frac{5}{12};1;\frac 1 J\rk$$
\en{is a solution to the Picard Fuchs differential equation from Thm.~\ref{\en{EN}picardfuchs}, no matter which fourth root we choose.
This is one of the 16 solutions found by Ernst Eduard Kummer (1810-1893).}%
\de{erfüllt die Picard-Fuchs-Differentialgleichung aus Theorem~\ref{\en{EN}picardfuchs}, ganz egal für welche vierte Wurzel wir uns entscheiden. Diese Lösung ist eine der 16 Lösungen, die auf Ernst Eduard Kummer (1810-1893) zurückgehen.}
\end{thm}
\begin{proof}
\en{Convergence of the hypergeometric sum for $|J|>1$ follows from Prop.~\ref{\en{EN}satzkonv}.
Next we realize that the Picard Fuchs differential equation is homogenous, which means that for any solution $b(J)$, $c\cdot b(J)$ is another solution -- this means that we don't have to be concerned with the choice of roots.
From the definition of $b(J)$ we obtain ${_2F_1}\lk\frac{1}{12},\frac{5}{12};1;\frac 1 J\rk = J^{\frac 1 4}\cdot(1-J)^{-\frac 1 4}\cdot b(J)$. With the new variable $z=\frac 1 J$, we obtain}%
\de{Aus Satz~\ref{\en{EN}satzkonv} folgt die Konvergenz der hypergeometrischen Summe im Bereich $|J|>1$. Dann bemerken wir, dass die Picard-Fuchs-Differentialgleichung homogen ist, also dass mit jeder Lösung $b(J)$ auch $c\cdot b(J)$ eine Lösung ist -- somit müssen wir uns wieder keine Gedanken über die Wahl der Wurzeln machen.
Aus der Definition von $b(J)$ folgt dann, dass ${_2F_1}\lk\frac{1}{12},\frac{5}{12};1;\frac 1 J\rk = J^{\frac 1 4}\cdot(1-J)^{-\frac 1 4}\cdot b(J)$ ist. Nun führen wir eine neue Variable $z=\frac 1 J$ ein. Dann erhalten wir}
$$\underbrace{{_2F_1}\lk\frac{1}{12},\frac{5}{12};1;z\rk}_{=:f(z)} = \underbrace{J^{\frac 1 4}\cdot(1-J)^{-\frac 1 4}\cdot b(J)}_{=:g(J)}\qquad\text{\en{with}\de{mit}}\quad z=\frac 1 J.$$
\en{We proved in Thm.~\ref{\en{EN}dgl2f1}, that $f(z)$ is a solution to the hypergeometric differential equation with $a=\frac{1}{12}$, $b=\frac{5}{12}$ and $c=1$:}%
\de{Für $f(z)$ gilt aber nach Theorem~\ref{\en{EN}dgl2f1} die hypergeometrische Differentialgleichung, wobei $a=\frac{1}{12}$, $b=\frac{5}{12}$ und $c=1$ gilt:}
\begin{align}\label{\en{EN}tmp1}
z(z-1) f''(z) + \left(\frac 3 2 z - 1\right) f'(z) + \frac{5}{144}~f(z) = 0
\end{align}
\en{Now we transform this into a differential equation of $g(J)$ by setting $z=\frac 1 J$.
This yields $\frac{dz}{dJ}=\frac{-1}{J^2}$ and $\frac{dJ}{dz}=-J^2$:}%
\de{Diese überführen wir nun in eine Differentialgleichung für $g(J)$, indem wir $z=\frac 1 J$ setzen.
Daraus folgt $\frac{dz}{dJ}=\frac{-1}{J^2}$ und somit $\frac{dJ}{dz}=-J^2$:}
\begin{align*}
    f(z) &= g(J)\qquad\text{\en{with}\de{mit}}\quad z=\frac 1 J\\
    \frac{df}{dz} &= \frac{dg}{dJ}\cdot\frac{dJ}{dz} = -J^2\cdot\frac{dg}{dJ}\\
    \frac{d^2f}{dz^2} &= -J^2\frac{d}{dJ}\lk-J^2\cdot\frac{dg}{dJ}\rk = J^4\cdot\frac{d^2g}{dJ^2}+2J^3\cdot\frac{dg}{dJ}
\end{align*}
\en{Using all this in~(\ref{\en{EN}tmp1}) we obtain a differential equation of $g(J)$:}%
\de{Das setzen wir jetzt in~(\ref{\en{EN}tmp1}) ein und erhalten eine Differentialgleichung für $g(J)$:}
\begin{align}
    \frac{1}{J}\left(\frac 1 J - 1\right) \cdot \left(J^4\cdot g''(J)+2J^3\cdot g'(J)\right)&\nonumber\\
    + \left(\frac 3 2 \cdot \frac 1 J - 1\right) \cdot \left(-J^2\cdot g'(J)\right) + \frac{5}{144}~g(J) &= 0\nonumber\\
     \Longrightarrow\quad J^2\left(1-J\right)g''(J) + \left(2J(1-J)-\frac 3 2 J + J^2\right) g'(J)+ \frac{5}{144}~g(J) &= 0\nonumber\\
     \Longrightarrow\quad J^2\left(1-J\right)g''(J) + \left(-J^2 + \frac 1 2 J\right) g'(J)+ \frac{5}{144}~g(J) &= 0\label{\en{EN}tmp2}
\end{align}
\pagebreak

\en{We have defined $g(J):=J^{\frac 1 4}\cdot(1-J)^{-\frac 1 4}\cdot b(J)$ and are looking for a differential equation of $b(J)$.
In order to find such, we denote $a(J):=J^{\frac 1 4}\cdot(1-J)^{-\frac 1 4}$, so that $g(J)=a(J)\cdot b(J)$.
With help of the derivatives of $a(J)$ we will transform the differential equation of $g(J)$ into one of $b(J)$:}%
\de{Wir haben $g(J):=J^{\frac 1 4}\cdot(1-J)^{-\frac 1 4}\cdot b(J)$ definiert und suchen eigentlich eine Differentialgleichung für $b(J)$.

Hierzu benennen wir den Faktor mit $a(J):=J^{\frac 1 4}\cdot(1-J)^{-\frac 1 4}$, sodass $g(J)=a(J)\cdot b(J)$ ist. Mit Hilfe der Ableitungen von $a(J)$ werden wir dann die Differentialgleichung für $g(J)$ in eine für $b(J)$ überführen. Doch zunächst die Ableitungen von $a(J)$:}
\begin{align*}
    a(J)&=J^c\cdot(1-J)^{-c}\qquad\text{\en{with}\de{mit} }c=\frac 1 4\\
    a'(J) &= cJ^{c-1}(1-J)^{-c}+cJ^c(1-J)^{-c-1}\\
    a''(J) &= c(c-1)J^{c-2}(1-J)^{-c}+c^2J^{c-1}(1-J)^{-c-1}\cdot 2 + c(c+1)J^c(1-J)^{-c-2}
\end{align*}
\en{We write these as multiples of $a(J)$, so that we can divide by $a(J)$ later:}%
\de{Diese lassen sich als Vielfache von $a(J)$ darstellen, um später mit $a(J)$ kürzen zu können:}
\begin{align*}
a'(J) &= a(J)\cdot\left(c J^{-1} + c (1-J)^{-1}\right) = \left(\frac{1}{4J} + \frac{1}{4(1-J)}\right) a(J) = \frac{1}{4J(1-J)}\cdot a(J)\\
a''(J) &= a(J)\cdot\left(c(c-1)J^{-2}+2c^2J^{-1}(1-J)^{-1}+c(c+1)(1-J)^{-2}\right)\\
&=\left(\frac{-3}{16J^2} + \frac{2}{16J(1-J)} + \frac{5}{16(1-J)^2}\right)a(J) = \frac{8J-3}{16J^2(1-J)^2}\cdot a(J)
\end{align*}
\en{This yields the derivatives of $g(J)$:}%
\de{Somit folgt für $g(J)$:}
\begin{align*}
    g(J)&=a(J)\cdot b(J)\\
    g'(J) &= a'(J)\cdot b(J) + a(J) \cdot b'(J)\\
          &= \frac{1}{4J(1-J)}\cdot a(J) \cdot b(J) + a(J) \cdot b'(J)\\
    g''(J) &= a''(J)\cdot b(J) + 2 a'(J) \cdot b'(J) + a(J) \cdot b''(J)\\
            &= \frac{8J-3}{16J^2(1-J)^2}\cdot a(J)\cdot b(J) + \frac{2}{4J(1-J)}\cdot a(J)\cdot b'(J) + a(J) \cdot b''(J)
\end{align*}
\en{We use these in~(\ref{\en{EN}tmp2}) and get:}%
\de{Das setzen wir jetzt in~(\ref{\en{EN}tmp2}) ein und erhalten:}
\begin{align*}
    J^2\left(1-J\right)\left(\frac{8J-3}{16J^2(1-J)^2}\cdot a(J)\cdot b(J) + \frac{2}{4J(1-J)}\cdot a(J)\cdot b'(J) + a(J) \cdot b''(J)\right)\\
    + \left(-J^2 + \frac 1 2 J\right) \left(\frac{1}{4J(1-J)}\cdot a(J) \cdot b(J) + a(J) \cdot b'(J)\right)+ \frac{5}{144}~a(J)\cdot b(J) = 0
\end{align*}
\en{After a division by $a(J)$ and some sorting we obtain:}%
\de{Wir dividieren durch $a(J)$ und sortieren ein bisschen:}
\begin{align*}
    J^2\left(1-J\right) b''(J) +\left(\frac{2J^2(1-J)}{4J(1-J)}-J^2+\frac 1 2 J\right) b'(J)&\\
    + \left(\frac{8J-3}{16(1-J)}+ \frac{-J^2 + \frac 1 2 J}{4J(1-J)}+\frac{5}{144}\right) b(J) &= 0\\
    \Longrightarrow\quad J^2\left(1-J\right) b''(J) +J(1-J) b'(J)+ \frac{31J-4}{144(1-J)} b(J) &= 0\qquad\quad|:(J^2(1-J))\\
    \Longrightarrow\quad b''(J) +\frac 1 J\cdot b'(J)+ \frac{31J-4}{144J^2(1-J)^2}\cdot b(J) &= 0
\end{align*}
\en{Here we recognize the Picard Fuchs differential equation from Thm.~\ref{\en{EN}picardfuchs}.}%
\de{Hier erkennen wir genau die Picard-Fuchs-Differentialgleichung aus Theorem~\ref{\en{EN}picardfuchs}.}
\end{proof}

\begin{bem}\label{\en{EN}bemwurzel}
\en{From here on, many $n$-th roots appear, for example the twelfth root in Def.~\ref{\en{EN}defltilde}. In the intermediate calculations, we won't fix which branch of the root we use, so that the equations are only correct up to a $n$-th root of unity. The main result from Thm.~\ref{\en{EN}hauptformel} will be exact if the main branch is used.}%
\de{Ab hier tauchen verschiedene $n$-te Wurzeln auf. Wir werden uns bei Zwischenrechnungen nicht festlegen, welchen Zweig dieser Wurzeln wir verwenden, sodass die Gleichungen nur bis auf eine $n$-te Einheitswurzel korrekt sind. Das Hauptresultat in Thm.~\ref{\en{EN}hauptformel} gilt jedoch exakt, wenn man den Hauptzweig der Wurzel verwendet.}
\end{bem}

\pagebreak

\begin{defi}\label{\en{EN}defltilde}
\en{We call the following lattice $\tilde L$. It is equivalent to $L_\tau$.}%
\de{Wir nennen das folgende Gitter $\tilde L$. Es ist äquivalent zu $L_\tau$.}
$$\tilde L = \mathbb Z \tilde \omega_1 + \mathbb Z \tilde \omega_1\quad\text{ \en{with}\de{mit} }\quad(\tilde\omega_1,\tilde\omega_2)=\Delta(\tau)^{\frac 1 {12}}\cdot(1,\tau)$$
\end{defi}

\begin{theo}\label{\en{EN}omegastrich}
\en{For all $\tau$ with $\im(\tau)>1\ko25$ it holds:}%
\de{Für alle $\tau$ mit $\im(\tau)>1\ko25$ gilt:}
$$\tilde\omega_1=\Delta(\tau)^{\frac 1 {12}} = \frac{2\pi}{\sqrt[4]{12}}\cdot J(\tau)^{-\frac{1}{12}}\cdot{_2F_1}\lk\frac{1}{12},\frac{5}{12};1;\frac 1 {J(\tau)}\rk$$
\end{theo}
\begin{bem}
\en{Actually, the formula from Thm.~\ref{\en{EN}omegastrich} holds for all $\tau$ with $\im(\tau)>1$ and $|J(\tau)|>1$ (see the colored area in Fig.~\ref{\en{EN}abb1}), but for proving the Chudnovsky type formulae the region $\im(\tau)>1\ko25$ is sufficient -- and in Thm.~\ref{\en{EN}theonaeherJ} we proved that in this region it holds $|J(\tau)|>1$.}%
\de{Eigentlich gilt die Formel aus Thm.~\ref{\en{EN}omegastrich} sogar für alle $\tau$ mit $\im(\tau)>1$\linebreak und $|J(\tau)|>1$ (siehe den in Abbildung~\ref{\en{EN}abb1} gefärbten Bereich), aber für die Herleitung der Chudnovsky-Formel und der zehn anderen Formeln wird der in Theorem~\ref{\en{EN}theonaeherJ} bewiesene Bereich $\im(\tau)>1\ko25$ ausreichen.}
\end{bem}

\begin{figure}[ht]\begin{tikzpicture}[x=2cm, y=2cm] \def\anzPunkte{127}
   \def\mypath{{0.087793*pow(abs(\x-round(\x)),4)+0.099371*pow(abs(\x-round(\x)),3)-1.118425*pow(abs(\x-round(\x)),2)+(abs(\x-round(\x)))+1}}
   \path[fill=lightgray] (-1.5,3.4) -- plot[samples=\anzPunkte,domain=-1.5:1.5] (\x,\mypath) -- (1.5,3.4) -- cycle;
   \draw[->,thick] (-1.6,0) -- (1.7,0) node[right] {$x=\operatorname{Re}(\tau)$};
   \draw[->,thick] (0,0) -- (0,3.5) node[above] {$y=\im(\tau)$};
   \foreach \c in {-1,0,1}{ \draw (\c,-.08) -- (\c,.08) node[below=12pt] {$\c$}; }
   \draw (-.08,1) -- (.08,1) node[left=12pt] {$i$};
   \foreach \c in {2,3} { \draw (-.08,\c) -- (.08,\c) node[left=12pt] {$\c i$}; }
   \foreach \c in {-1.5,-0.5,0.5,1.5}{ \draw (\c,-.05) -- (\c,.05); }
   \foreach \c in {0.5,1.5,2.5} { \draw (-.05,\c) -- (.05,\c); }
   \draw[dashed](-1.75,1.25)node[left]{$\im(\tau)=1\ko25$} -- (1.4,1.25) ;
   \draw plot[samples=\anzPunkte,domain=-1.5:1.5] (\x,\mypath)  node[right] {$|J(\tau)|=1$};
   \draw (1.5,2.25) node[right] {$|J(\tau)|>1$};
   \draw[fill=black] (0.0,1.414) circle (1.5pt) node[right] {$\tau_8$};
    \draw[fill=black] (0,1.732) circle (1.5pt) node[right] {$\tau_{12}$};
	\draw[fill=black] (0,2) circle (1.5pt) node[right] {$\tau_{16}$};
	\draw[fill=black] (0,2.646) circle (1.5pt) node[right] {$\tau_{28}$};
	\draw[fill=black] (0.5,1.323) circle (1.5pt) node[right] {$\tau_{7}$};
	\draw[fill=black] (0.5,1.658) circle (1.5pt) node[right] {$\tau_{11}$};
	\draw[fill=black] (0.5,2.179) circle (1.5pt) node[right] {$\tau_{19}$};
	\draw[fill=black] (0.5,2.598) circle (1.5pt) node[right] {$\tau_{27}$};
	\draw[fill=black] (0.5,3.279) circle (1.5pt) node[right] {$\tau_{43}$};
\end{tikzpicture} 
\caption{\en{The region with $\im(\tau)>1$ and $|J(\tau)|>1$ has been calculated with Mathematica and is colored gray here.\newline
Also depicted: The values of $\tau_N$ (cf.~Prop.~\ref{\en{EN}satztaun}) which will lead to a Chudnovsky type formula.
$\tau_{67}$ and $\tau_{163}$ are outside the depicted area, above $\tau_{43}.$}%
\de{Hier ist das Gebiet mit $\im(\tau)>1$ und $|J(\tau)|>1$ grau gefärbt, das mit Mathematica berechnet wurde.\newline
Außerdem sind die Werte $\tau_N$ (vgl.~Satz~\ref{\en{EN}satztaun}) eingezeichnet, die auf eine Formel zur Berechnung von $\pi$ führen werden.
$\tau_{67}$ und $\tau_{163}$ liegen oberhalb von $\tau_{43}$ außerhalb des dargestellten Bereichs.}}\label{\en{EN}abb1}
\end{figure}

\begin{proof}[\en{Proof of Theorem}\de{Beweis des Theorems}~\ref{\en{EN}omegastrich}]
\en{In Thm.~\ref{\en{EN}theonaeherJ} we proved that for all $\tau$ with $\im(\tau)>1\ko25$ it holds: $|J(\tau)|>1$.
From Prop.~\ref{\en{EN}satzkonv} we deduce absolute convergence of the hypergeometric function ${_2F_1}$ in the region $\im(\tau)>1\ko25$ using $|z|=\left|\frac 1 J\right|<1$.
The present proof of Thm.~\ref{\en{EN}omegastrich} follows \cite[ch.~2.3 and 2.5]{\en{EN}lit02}:}%
\de{In Thm.~\ref{\en{EN}theonaeherJ} haben wir bewiesen, dass $|J(\tau)|>1$ für alle $\tau$ mit $\im(\tau)>1\ko25$ gilt. Aus Satz~\ref{\en{EN}satzkonv} folgt dann mit $|z|=\left|\frac 1 J\right|<1$ die absolute Konvergenz der angegebenen hypergeometrischen Funktion ${_2F_1}$ im Bereich $\im(\tau)>1\ko25$. Der jetzt folgende Beweis des Theorems~\ref{\en{EN}omegastrich} orientiert sich an \cite[Kap.~2.3 und 2.5]{\en{EN}lit02}:}

\en{We proved in Prop.~\ref{\en{EN}satzbj} that $b(J)= J^{-\frac 1 4}\cdot(1-J)^{\frac 1 4}\cdot{_2F_1}\lk\frac{1}{12},\frac{5}{12};1;\frac 1 J\rk$ is a solution of the Picard Fuchs differential equation for $|J|>1$.
But we already know two independent solutions of this second order differential equation -- the two basic periods $\Omega_1$ and $\Omega_2$ of $L_J$.
The Picard-Lindelöf theorem tells us that the solution $b(J)$ can be written as a linear combination of these basic periods, i.e.~that there are two complex numbers $A$ and $B$ satisfying
$b(J)=A\cdot\Omega_1(J) + B\cdot \Omega_2(J)$.}%
\de{Wir haben in Satz~\ref{\en{EN}satzbj} bewiesen, dass $b(J)= J^{-\frac 1 4}\cdot(1-J)^{\frac 1 4}\cdot{_2F_1}\lk\frac{1}{12},\frac{5}{12};1;\frac 1 J\rk$ für $|J|>1$ eine Lösung der Picard-Fuchs-Differential"-gleichung ist.
Wir kennen aber bereits zwei unabhängige Lösungen dieser Differentialgleichung zweiter Ordnung -- nämlich die beiden Basisperioden $\Omega_1$ und $\Omega_2$ von $L_J$.
Hieraus folgt mit dem Satz von Picard-Lindelöf, dass man die Lösung $b(J)$ als Linearkombination dieser beiden Perioden schreiben kann, also dass es komplexe Zahlen $A$ und $B$ gibt, sodass
$b(J)=A\cdot\Omega_1(J) + B\cdot \Omega_2(J)$ ist.}

\pagebreak

\en{In Def.~\ref{\en{EN}definLJ} we see that the basic periods $(\Omega_1,\Omega_2)$ of $L_J$ can be written in the form $\mu(\tau)\cdot(1,\tau)$ with $\mu(\tau)=\sqrt{\frac{g_3(\tau)}{g_2(\tau)}}$. We will write $\mu$ in terms of $J$:
}%
\de{In Definition~\ref{\en{EN}definLJ} erkennen wir, dass man die Basisperioden $(\Omega_1,\Omega_2)$ von $L_J$ in der Form $\mu(\tau)\cdot(1,\tau)$ mit $\mu(\tau)=\sqrt{\frac{g_3(\tau)}{g_2(\tau)}}$ schreiben kann. Wir werden zunächst $\mu$ mit Hilfe der $J$-Funktion schreiben:}
\begin{align}
    \frac{27J}{J-1} &= \frac{27g_2^3}{\Delta\cdot\left(\frac{g_2^3}{\Delta}-1\right)} = \frac{27g_2^3}{g_2^3-\Delta} = \frac{27g_2^3}{g_2^3-\left(g_2^3-27g_3^2\right)} = \frac{g_2^3}{g_3^2}\nonumber\\
    \Longrightarrow~~\mu&=\sqrt{\frac{g_3}{g_2}} = \left(\frac{g_3^2}{g_2^2}\right)^{\frac 1 4} = \left(g_2\cdot \frac{g_3^2}{g_2^3}\right)^{\frac 1 4} = \left(g_2\cdot \frac{J-1}{27J}\right)^{\frac 1 4}\nonumber\\
    &=\left(\left(J\cdot\Delta\right)^{\frac 1 3}\cdot \frac{J-1}{27J}\right)^{\frac 1 4} = 27^{-\frac 1 4} \cdot J^{-\frac 1 6} \cdot \left(J-1\right)^{\frac 1 4} \cdot \Delta^\frac 1 {12}\label{\en{EN}mue}
\end{align}

\en{This proves the existence of complex numbers $A$ and $B$ with $b(J)=A\cdot\Omega_1+B\cdot\Omega_2=(A+B\tau)\cdot\mu$ and thus}%
\de{Also gibt es komplexe $A$ und $B$, sodass $b(J)=A\cdot\Omega_1+B\cdot\Omega_2=(A+B\tau)\cdot\mu$ und somit}
\begin{align*}
    J^{-\frac 1 4}\cdot(1-J)^{\frac 1 4}\cdot{_2F_1}\lk\frac{1}{12},\frac{5}{12};1;\frac 1 J\rk &=\left(A + B\tau\right)\cdot 27^{-\frac 1 4} \cdot J^{-\frac 1 6} \cdot \left(J-1\right)^{\frac 1 4} \cdot \Delta^\frac 1 {12}\\
    \Longrightarrow\quad J^{-\frac{1}{12}}\cdot\Delta^{-\frac 1 {12}}\cdot{_2F_1}\lk\frac{1}{12},\frac{5}{12};1;\frac 1 J\rk &= C + D\tau
\end{align*}
\en{for two complex $C$ and $D$. Here we included $27^{-\frac 1 4}$ into the new numbers $C$ and $D$.}%
\de{für passende komplexe Zahlen $C$ und $D$. Hierbei wurde der Faktor $27^{-\frac 1 4}$ in die neuen komplexen Zahlen $C$ und $D$ integriert.}

\en{First we calculate $D$: It holds $e^{2\pi i (\tau+1)}=e^{2\pi i \tau}\cdot e^{2\pi i}=e^{2\pi i \tau}$.
Then we deduce from the Fourier representations in Thm.~\ref{\en{EN}fouriertheorem}, that $E_4(\tau+1)=E_4(\tau)$ and $E_6(\tau+1)=E_6(\tau)$.
This yields (again with Thm.~\ref{\en{EN}fouriertheorem}), that $J(\tau+1)=J(\tau)$ and $\Delta(\tau+1)=\Delta(\tau)$.
This shows that the left hand side of the above equation is invariant under the transformation $\tau\mapsto\tau+1$.
Thus the right hand side must also be invariant under this transformation -- we obtain $D=0$.}%
\de{Um $D$ zu berechnen, bemerken wir zunächst, dass $e^{2\pi i (\tau+1)}=e^{2\pi i \tau}\cdot e^{2\pi i}=e^{2\pi i \tau}$ gilt.
Dann folgt mit den Darstellungen in Theorem~\ref{\en{EN}fouriertheorem}, dass $E_4(\tau+1)=E_4(\tau)$ und $E_6(\tau+1)=E_6(\tau)$ gilt.
Hieraus folgt (ebenfalls mit den Darstellungen aus Theorem~\ref{\en{EN}fouriertheorem}), dass $J(\tau+1)=J(\tau)$ und $\Delta(\tau+1)=\Delta(\tau)$ ist. Folglich ist die linke Seite dieser Gleichung invariant unter der Transformation $\tau\mapsto\tau+1$, also muss es auch die rechte Seite sein -- wir erhalten $D=0$.}

\en{We get the value of $C$ by calculating the limit $\tau\rightarrow i\infty$ (which is allowed since our solution holds in the whole region $\im(\tau)>1\ko25$). This limit implies $q=e^{2\pi i \tau}\rightarrow 0$.
But from Thm.~\ref{\en{EN}theonaeherJ} we get $|1728J|>\frac{0\ko737}{|q|}$, which implies $\frac{1}{J(\tau)}\rightarrow 0$.
Then Def.~\ref{\en{EN}defihyp} shows that ${_2F_1}\lk\frac{1}{12},\frac{5}{12};1;\frac 1 J\rk\rightarrow 1$.
From the representation of $J$ and $\Delta$ in Thm.~\ref{\en{EN}fouriertheorem} und from the Fourier series of $E_4$ in the same Thm.~\ref{\en{EN}fouriertheorem} we get:}%
\de{Den Wert von $C$ erhalten wir z.B., indem wir auf der linken Seite $\tau\rightarrow i\infty$ gehen lassen (das dürfen wir, weil diese Lösung im ganzen Bereich $\im(\tau)>1\ko25$ gilt) und somit $q=e^{2\pi i \tau}\rightarrow 0$.
Dort gilt allerdings nach Theorem~\ref{\en{EN}theonaeherJ}, dass $|1728J|>\frac{0\ko737}{|q|}$, also dass  $\frac{1}{J(\tau)}\rightarrow 0$ gilt
und somit aus Def.~\ref{\en{EN}defihyp} folgt, dass ${_2F_1}\lk\frac{1}{12},\frac{5}{12};1;\frac 1 J\rk\rightarrow 1$.
Somit gilt wegen der Darstellung von $J$ und $\Delta$ aus Theorem~\ref{\en{EN}fouriertheorem} und wegen der dort befindlichen Darstellung von $E_4$:}
\begin{align*}
C&=\lim_{\tau\rightarrow i\infty}\left(J(\tau)\cdot\Delta(\tau)\right)^{-\frac 1 {12}}
= \lim_{\tau\rightarrow i\infty}\left(\frac{(2\pi)^{12}}{1728}\cdot E_4(\tau)^3\right)^{-\frac 1 {12}}\\
&=\left(\frac{(2\pi)^{12}}{12^3}\cdot 1\right)^{-\frac 1 {12}} = \frac{12^{\frac 1 4}}{2\pi}
\end{align*}
\en{This yields}\de{Schließlich erhalten wir}
$${_2F_1}\lk\frac{1}{12},\frac{5}{12};1;\frac 1 {J(\tau)}\rk = \frac{\sqrt[4]{12}}{2\pi}\cdot J(\tau)^{\frac{1}{12}}\cdot\Delta(\tau)^{\frac 1 {12}}$$
\en{where we recognize the equation from Thm.~\ref{\en{EN}omegastrich} because of $\tilde\omega_1 = \Delta(\tau)^{\frac 1 {12}}$.}%
\de{und erkennen darin wegen $\tilde\omega_1 = \Delta(\tau)^{\frac 1 {12}}$ die zu beweisende Gleichung.}
\end{proof}

\vfill\pagebreak\section{\en{Proof of the Main Theorem}\de{Beweis des Haupttheorems}}\label{\en{EN}kaphaupt}
\renewcommand{\leftmark}{\en{Proof of the Main Theorem}\de{Beweis des Haupttheorems}}
\en{We will start with Kummer's solution (Thm.~\ref{\en{EN}omegastrich}) and use Clausen's Formula (Thm.~\ref{\en{EN}satzclausen}) and the Fourier series (Thm.~\ref{\en{EN}fouriertheorem}) to prove the Main Theorem~\ref{\en{EN}hauptformel}.
This chapter follows the paper \cite{\en{EN}glebov} of Chen and Glebov.}%
\de{Wir werden, ausgehend von Kummers Lösung (Thm.~\ref{\en{EN}omegastrich}) und mit Hilfe der Clausen-Formel (Thm.~\ref{\en{EN}satzclausen}) und den Fourierdarstellungen (Thm.~\ref{\en{EN}fouriertheorem}), das Haupttheorem~\ref{\en{EN}hauptformel} beweisen.
Dieses Kapitel orientiert sich stark am Paper von Chen und Glebov \cite{\en{EN}glebov}.}

\en{First we give an overview over the names of the basic periods and basic quasiperiods of the equivalent lattices $L_\tau$, $L_J$ and $\tilde L$ (cf.~Def.~\ref{\en{EN}definLJ} and Def.~\ref{\en{EN}defltilde}). And throughout this chapter, we assume that $L=\mathbb Z \omega_1 + \mathbb Z \omega_2$ equals $L_\tau$ with $\omega_1=1$:}%
\de{Zunächst stellen wir die Bezeichnungen für die Perioden und Quasiperioden der drei äquivalenten Gitter $L_\tau$, $L_J$ und $\tilde L$ zusammen (vgl.~Def.~\ref{\en{EN}definLJ} und Def.~\ref{\en{EN}defltilde}), wobei wir in diesem Kapitel voraussetzen, dass $L=\mathbb Z \omega_1 + \mathbb Z \omega_2$ von der Form $L_\tau$ mit $\omega_1=1$ ist:}

{\vspace*{1mm}\hspace*{2mm}\renewcommand{\arraystretch}{2}%
    \begin{tabular}{l|l}\label{\en{EN}tabgitter}
    \textbf{ \en{Lattices and Basic Periods}\de{Gitter und Perioden}} & \textbf{ \en{Basic Quasiperiods}\de{Quasiperioden} }\\
    \hline
    ~$\displaystyle L_\tau: (\omega_1,\omega_2)=(1,\tau)$
    & ~$\displaystyle (\eta_1,\eta_2)=(\eta_1(L_\tau),\eta_2(L_\tau))$\\
    ~$\displaystyle L_J: (\Omega_1,\Omega_2)=\sqrt{\frac{g_3(\tau)}{g_2(\tau)}}\cdot(1,\tau)$~
    &~$\displaystyle (H_1,H_2)=\sqrt{\frac{g_2(\tau)}{g_3(\tau)}}\cdot(\eta_1(L_\tau),\eta_2(L_\tau))$\\
    ~$\displaystyle \tilde L: (\tilde\omega_1,\tilde\omega_2)=\Delta(\tau)^{\frac 1 {12}}\cdot(1,\tau)$~
    & ~$\displaystyle (\tilde\eta_1,\tilde\eta_2)=\Delta(\tau)^{-\frac 1 {12}}\cdot(\eta_1(L_\tau),\eta_2(L_\tau))$\\
    \end{tabular}\vspace*{2mm}}

\en{If we use one of the expressions $\eta_1$, $\eta_2$, $g_2$, $g_3$ or $\Delta$ in this chapter, it shall denote $\eta_1(L_\tau)$, $\eta_2(L_\tau)$, $g_2(\tau)$, $g_3(\tau)$ or $\Delta(\tau)$.
And (like in Remark~\ref{\en{EN}bemwurzel}) we won't care about the choice of the branch of the roots in the intermediate steps, only when we arrive at the Main Theorem~\ref{\en{EN}hauptformel}.}%
\de{Wenn wir im Folgenden $\eta_1$, $\eta_2$, $g_2$, $g_3$ oder $\Delta$ schreiben, meinen wir damit $\eta_1(L_\tau)$, $\eta_2(L_\tau)$, $g_2(\tau)$, $g_3(\tau)$ bzw.~$\Delta(\tau)$.
Außerdem werden wir uns wie in Bem.~\ref{\en{EN}bemwurzel} bei den Zwischenschritten keine Gedanken zur Wahl der Wurzel machen, erst im Haupttheorem~\ref{\en{EN}hauptformel}.}

\begin{thm}\label{\en{EN}thm35}
\en{For $k=1$ and for $k=2$ it holds:}%
\de{Für $k=1$ und für $k=2$ gilt:}
$$\tilde\eta_k = -\sqrt{12}J^\frac 2 3\sqrt{J-1}\cdot\frac{d\tilde\omega_k}{dJ}$$
\end{thm}
\begin{proof}
\en{For $\sqrt{\frac{g_3}{g_2}}$, we calculated the representation~(\ref{\en{EN}mue}) on page~\pageref{\en{EN}mue}:}%
\de{Für $\sqrt{\frac{g_3}{g_2}}$ haben wir auf S.~\pageref{\en{EN}mue} die Darstellung~(\ref{\en{EN}mue}) berechnet:}
\begin{align}
    \sqrt{\frac{g_3}{g_2}}&=\underbrace{J^{-\frac 1 6} \cdot \left(\frac{J-1}{27}\right)^{\frac 1 4}}_{=:A(J)} \cdot~\Delta^\frac 1 {12}\label{\en{EN}g2g3}
\end{align}
\en{This yields a direct connection between $L_J$ and $\tilde L$:}%
\de{Hieraus folgt ein direkter Zusammenhang zwischen $L_J$ und $\tilde L$:}
\begin{align*}
    (\Omega_1,\Omega_2)&=\sqrt{\frac{g_3}{g_2}}\cdot(1,\tau)=\sqrt{\frac{g_3}{g_2}}\cdot\Delta^{-\frac 1{12}}\cdot(\tilde\omega_1,\tilde\omega_2) = A(J) \cdot(\tilde\omega_1,\tilde\omega_2)\\
    (H_1,H_2)&=\sqrt{\frac{g_2}{g_3}}\cdot(\eta_1,\eta_2)=\sqrt{\frac{g_2}{g_3}}\cdot\Delta^{\frac 1{12}}\cdot(\tilde\eta_1,\tilde\eta_2) = \frac{1}{A(J)} \cdot(\tilde\eta_1,\tilde\eta_2)
\end{align*}
\en{We derive the first line by $J$ and obtain:}%
\de{Die erste Zeile leiten wir noch nach $J$ ab und erhalten:}
\begin{align*}
    \frac{d\Omega}{dJ} &= A(J) \cdot\frac{d\tilde\omega}{dJ} + \left(-\frac{1}{6J}+\frac{1}{4(J-1)}\right)\cdot A(J) \cdot\tilde\omega\\
    &=A(J) \cdot \left(\frac{d\tilde\omega}{dJ} + \frac{J+2}{12J(J-1)}\cdot\tilde\omega\right)
\end{align*}
\en{Now we use this new equation in~(\ref{\en{EN}dOdJ}) on p.~\pageref{\en{EN}dOdJ} and get the desired equation:}%
\de{Wir setzen die soeben gefundenen Zusammenhänge nun auf Seite~\pageref{\en{EN}dOdJ} in Gleichung~(\ref{\en{EN}dOdJ}) ein und erhalten die zu beweisende Gleichung:}\belowdisplayskip=-12pt
\begin{align*}
    36 J(J-1)\overbrace{A(J)\cdot \left(\frac{d\tilde\omega}{dJ} + \frac{J+2}{12J(J-1)}\cdot\tilde\omega\right)}^{\frac{d\Omega}{dJ}} &= 3(J+2)\overbrace{A(J)\tilde\omega}^\Omega -2(J-1)\overbrace{\frac{1}{A(J)}\tilde\eta}^H\\
    \Longrightarrow\qquad \frac{2(J-1)}{A(J)}\cdot\tilde\eta &= -36J(J-1)A(J)\cdot\frac{d\tilde\omega}{dJ}\\
    \Longrightarrow\qquad \tilde\eta
    &=-\sqrt{12}J^\frac 2 3\sqrt{J-1}\cdot\frac{d\tilde\omega}{dJ}
\end{align*}\end{proof}

\begin{defi}\label{\en{EN}defis2}
\en{We denote the following non-holomorphic function by $s_2$:}%
\de{Die folgende nicht-holomorphe Funktion nennen wir $s_2$:}
$$s_2(\tau):=\frac{E_4(\tau)}{E_6(\tau)}\cdot E_2^*(\tau)\qquad\text{\en{with}\de{mit}}\qquad E_2^*(\tau) := E_2(\tau)-\frac{3}{\pi \im(\tau)}$$
\en{where $E_k(\tau)$ are the normalized Eisenstein series from Thm.~\ref{\en{EN}fouriertheorem}.}%
\de{wobei $E_k(\tau)$ die normierten Eisensteinreihen aus Thm.~\ref{\en{EN}fouriertheorem} sind.}
\end{defi}

\begin{bem}
\en{$s_2$ is an example of an "almost holomorphic modular form". These are functions in $\mathbb H$ which transform like a modular form -- i.e.~the values of $s_2$ are equal for equivalent lattices -- but instead of being holomorphic, they are polynomials in $\frac{1}{\im(\tau)}$ with holomorphic coefficients.
And in Prop.~\ref{\en{EN}exts2rat} we will see that certain values of $s_2(\tau)$ are rational.}%
\de{$s_2$ ist eine \glqq fast holomorphe Modulform\grqq. Diese sind Funktionen in $\mathbb H$, die wie eine Modulform transformieren -- auch für $s_2$ gilt, dass es bei äquivalenten Gittern den gleichen Wert annimmt -- aber sie sind nicht holomorph, sondern Polynome in $\frac{1}{\im(\tau)}$ mit holomorphen Koeffizienten.
Und in Satz~\ref{\en{EN}exts2rat} werden wir sehen, dass gewisse Werte von $s_2(\tau)$ rational sind.}
\end{bem}

\begin{thm}\label{\en{EN}thmglg10}
\en{It holds}\de{Es gilt} $$ \eta_1-\frac{3g_3}{2g_2}s_2(\tau) = \frac \pi {\im(\tau)} $$
\end{thm}
\begin{proof}
\en{We change the representations of $\eta_1$, $g_2$ and $g_3$ from Thm.~\ref{\en{EN}fouriertheorem} into representations of $E_2$, $E_4$ and $E_6$.
Then we put these into the definition of $s_2$ and obtain:}%
\de{Wir lösen die Darstellungen von $\eta_1$, $g_2$ und $g_3$ aus Theorem~\ref{\en{EN}fouriertheorem} nach $E_2$, $E_4$ und $E_6$ auf, setzen die Ergebnisse in die Definition von $s_2(\tau)$ ein und erhalten:}\belowdisplayskip=-12pt
\begin{align*}
    s_2(\tau) = \frac{\frac{3g_2}{4\pi^4}}{\frac{27g_3}{8\pi^6}}\left(\frac{3\eta_1}{\pi^2}-\frac{3}{\pi \im(\tau)}\right)
    &=\frac{2\pi^2}{9}\cdot\frac{g_2}{g_3}\left(\frac{3\eta_1}{\pi^2}-\frac{3}{\pi \im(\tau)}\right)\\
    &= \frac{2g_2}{3g_3}\cdot\eta_1-\frac{2\pi g_2}{3g_3 \im(\tau)}\\
    \Longrightarrow\qquad \frac{2g_2}{
    3g_3}\cdot\eta_1 - s_2(\tau) &= \frac{2\pi g_2}{3g_3 \im(\tau)}\quad\left|\cdot\frac{3g_3}{2 g_2}\right.\\
    \Longrightarrow\qquad \eta_1 - \frac{3g_3}{2 g_2}\cdot s_2(\tau) &= \frac{\pi}{\im(\tau)}
\end{align*}
\end{proof}

\begin{thm}\label{\en{EN}thm42}
\en{For all $\tau$ with $\im(\tau)>1\ko25$ it holds}%
\de{Für alle $\tau$ mit $\im(\tau)>1\ko25$ gilt}
$$\frac{1}{2\pi \im(\tau)}\sqrt{\frac{J}{J-1}} = \frac{1-s_2(\tau)} 6 F^2 - J \frac d {dJ} F^2$$
\en{where we denote}\de{Dabei ist} $F = {_2F_1}\left(\frac{1}{12},\frac{5}{12};1;\frac 1 {J}\right)$ \en{and}\de{und} $J=J(\tau)$.
\end{thm}

\begin{proof}
\en{First we will prove Eq.~(\ref{\en{EN}etaeins}) and~(\ref{\en{EN}gdreigzwei}). Then we will combine them with Prop.~\ref{\en{EN}thmglg10}.}%
\de{Wir werden Glg.~(\ref{\en{EN}etaeins}) und~(\ref{\en{EN}gdreigzwei}) beweisen und sie dann in Satz~\ref{\en{EN}thmglg10} einsetzen.}

\en{Thm.~\ref{\en{EN}omegastrich} tells us that $\tilde\omega_1 = \Delta^{\frac 1 {12}} = \frac{2\pi}{\sqrt[4]{12}}\cdot J^{-\frac{1}{12}}\cdot F$ holds in the given region. We derive this by $J$ and obtain (using the product rule):}%
\de{Zunächst sagt Theorem~\ref{\en{EN}omegastrich}, dass $\tilde\omega_1 = \Delta^{\frac 1 {12}} = \frac{2\pi}{\sqrt[4]{12}}\cdot J^{-\frac{1}{12}}\cdot F$ im angegebenen Bereich gilt. Die Ableitung dieser Gleichung nach $J$ liefert mit der Produktregel:}
$$\frac{d\tilde\omega_1}{dJ}=\frac{2\pi}{\sqrt[4]{12}}\cdot J^{-\frac 1 {12}} \cdot \left(\frac{-1}{12J}\cdot F + \frac{dF}{dJ}  \right)$$
\en{Using this in the equation of Prop.~\ref{\en{EN}thm35} we get:}%
\de{Das setzen wir nun in die Gleichung aus Satz~\ref{\en{EN}thm35} ein:}
\begin{align*}
\tilde\eta_1 &= -\sqrt{12}J^\frac 2 3\sqrt{J-1}\cdot\overbrace{\frac{2\pi}{\sqrt[4]{12}}\cdot J^{-\frac 1 {12}} \cdot \left(\frac{-1}{12J}\cdot F + \frac{dF}{dJ}  \right)}^{\frac{d\tilde\omega_1}{dJ}}\nonumber\\
&= -2\pi\sqrt[4]{12}J^\frac 7{12} \sqrt{J-1} \cdot \left(\frac{-1}{12J}\cdot F + \frac{dF}{dJ}  \right)
\end{align*}
\en{From the table on p.~\pageref{\en{EN}tabgitter} we conclude $\eta_1 = \tilde\eta_1\cdot\Delta^{\frac{1}{12}}$.
Here we can use the newly found representation of $\tilde \eta_1$ and the representation of $\Delta^{\frac{1}{12}}$ from Thm.~\ref{\en{EN}omegastrich}. This yields:}%
\de{Aus der Tabelle auf S.~\pageref{\en{EN}tabgitter} entnehmen wir $\eta_1 = \tilde\eta_1\cdot\Delta^{\frac{1}{12}}$. Hier setzen wir die eben gefundene Darstellung von $\tilde \eta_1$ sowie die Darstellung von $\Delta^{\frac{1}{12}}$ aus Thm.~\ref{\en{EN}omegastrich} ein und erhalten:}
\begin{align}
    \eta_1 = \tilde\eta_1\cdot\Delta^{\frac{1}{12}} &= - 2\pi\sqrt[4]{12}J^\frac 7{12} \sqrt{J-1} \cdot \left(\frac{-1}{12J}\cdot F + \frac{dF}{dJ}  \right)\cdot \frac{2\pi}{\sqrt[4]{12}}\cdot J^{-\frac{1}{12}}\cdot F \nonumber\\
    &=\frac{\pi^2}{3}\cdot \sqrt{\frac{J-1}{J}}\cdot F^2  -  2\pi^2\sqrt{J(J-1)}\cdot \underbrace{2F\frac{dF}{dJ}}_{\frac{d}{dJ}\lk F^2\rk}\label{\en{EN}etaeins}
\end{align}
\en{Next we take eq.~(\ref{\en{EN}g2g3}) from p.~\pageref{\en{EN}g2g3} and use the representation of $\Delta^{\frac{1}{12}}$ from Thm.~\ref{\en{EN}omegastrich}:}%
\de{Nun nehmen wir Glg.~(\ref{\en{EN}g2g3}) von S.~\pageref{\en{EN}g2g3} und setzen die Darstellung von $\Delta^{\frac{1}{12}}$ aus Thm.~\ref{\en{EN}omegastrich} ein:}
\begin{align}
    \frac{3g_3}{2g_2} &= \frac 3 2 \cdot \left(J^{-\frac 1 6} \cdot \left(\frac{J-1}{27}\right)^{\frac 1 4} \cdot~\Delta^\frac 1 {12}\right)^2
    = \frac 3 2 \cdot J^{-\frac 1 3} \cdot \left(\frac{J-1}{27}\right)^{\frac 1 2} \cdot \Delta^\frac 1 {6}\nonumber\\
    &= \frac 3 2 \cdot J^{-\frac 1 3} \cdot \frac{\sqrt{J-1}}{\sqrt{27}}\cdot \frac{4\pi^2}{\sqrt{12}}\cdot J^{-\frac{1}{6}}\cdot F^2 = \frac{\pi^2}{3}\cdot\sqrt{\frac{J-1}{J}}\cdot F^2\label{\en{EN}gdreigzwei}
\end{align}
\en{Now we use~(\ref{\en{EN}etaeins}) and~(\ref{\en{EN}gdreigzwei}) in the equation of Prop.~\ref{\en{EN}thmglg10}:}%
\de{Nun setzen wir~(\ref{\en{EN}etaeins}) und~(\ref{\en{EN}gdreigzwei}) in die Gleichung aus Satz~\ref{\en{EN}thmglg10} ein:}
\begin{align*}
    \overbrace{\frac{\pi^2}{3}\cdot \sqrt{\frac{J-1}{J}}\cdot F^2  -  2\pi^2\sqrt{J(J-1)}\cdot \frac{d}{dJ}\lk F^2\rk}^{\eta_1} - \overbrace{\frac{\pi^2}{3}\cdot\sqrt{\frac{J-1}{J}}\cdot F^2 }^{\frac{3g_3}{2g_2}}\cdot~s_2(\tau) &= \frac \pi {\im(\tau)}\\
    \Longrightarrow\quad \frac{\pi^2}{3}\cdot \sqrt{\frac{J-1}{J}}\cdot F^2 \cdot \left(1-s_2(\tau)\right) - 2\pi^2\sqrt{J(J-1)}\cdot \frac{d}{dJ}\lk F^2\rk &= \frac \pi {\im(\tau)}
\end{align*}
\en{Finally we multiply with $\sqrt{\frac{J}{J-1}}\cdot\frac{1}{2\pi^2}$ and obtain:}%
\de{Schließlich multiplizieren wir das mit $\sqrt{\frac{J}{J-1}}\cdot\frac{1}{2\pi^2}$ und erhalten:}
\begin{align*}
    \frac{1-s_2(\tau)}{6}\cdot F^2 - J\cdot \frac{d}{dJ}\lk F^2\rk &= \sqrt{\frac{J}{J-1}}\cdot\frac{1}{2\pi\im(\tau)}
\end{align*}
\en{This finishes the proof of Prop.~\ref{\en{EN}thm42}.}%
\de{Somit ist Satz~\ref{\en{EN}thm42} bewiesen.}
\end{proof}

\begin{thm}\label{\en{EN}darst}
\en{For the square of the following hypergeometric function it holds:}%
\de{Für das Quadrat der folgenden hypergeometrischen Funktion gilt:}
\begin{align*}
    \left({_2F_1}\lk\frac{1}{12},\frac{5}{12};1;z\rk\right)^2 &= \sum_{n=0}^\infty \frac{(6n)!}{(3n)!(n!)^3} \frac{z^n}{12^{3n}}
\end{align*}
\end{thm}
\begin{proof}
  \en{From Clausen's formula (Thm.~\ref{\en{EN}satzclausen} on p.~\pageref{\en{EN}satzclausen}) we get using Def.~\ref{\en{EN}defihyp} on p.~\pageref{\en{EN}defihyp}:}%
  \de{Aus der Formel von Clausen (Theorem~\ref{\en{EN}satzclausen} auf Seite~\pageref{\en{EN}satzclausen}) folgt mit Hilfe von Definition~\ref{\en{EN}defihyp} auf Seite~\pageref{\en{EN}defihyp}:}
  \begin{align}
      \left({_2F_1}\lk\frac{1}{12},\frac{5}{12};1;z\rk\right)^2 = {_3F_2}\lk\frac{1}{6},\frac{5}{6},\frac{1}{2};1,1;z\rk = \sum_{n=0}^{\infty} \frac{\left(\frac{1}{6}\right)_n\cdot\left(\frac{5}{6}\right)_n\cdot\left(\frac{1}{2}\right)_n}{(1)_n \cdot (1)_n}\cdot\frac{z^n}{n!}\label{\en{EN}asdf}
  \end{align}
  \en{Since $(1)_n=n!$ we only have to express $\left(\frac{1}{6}\right)_n \cdot \left(\frac{5}{6}\right)_n \cdot \left(\frac{1}{2}\right)_n$ in terms of factorials:}%
  \de{Es gilt aber $(1)_n=n!$, also müssen wir nur noch $\left(\frac{1}{6}\right)_n \cdot \left(\frac{5}{6}\right)_n \cdot \left(\frac{1}{2}\right)_n$ vereinfachen:}
  
    \en{If $a=\frac p q$ is the quotient of two natural numbers, the Def.~\ref{\en{EN}defpoch} of the Pochhammer symbols yields:}%
    \de{Wenn $a=\frac p q$ ein Verhältnis zweier natürlicher Zahlen ist, dann gilt aufgrund der Definition~\ref{\en{EN}defpoch} der Pochhammer-Symbole:}
    $$\left(\frac p q\right)_n = \prod_{k=1}^n\left(\frac p q + k - 1\right) = q^{-n} \prod_{k=1}^n\left(p + kq - q\right).$$
    \en{This yields (cf.~\cite[Lemma 4.1]{\en{EN}glebov}):}%
    \de{Hieraus folgt (vgl.~\cite[Lemma 4.1]{\en{EN}glebov}):}
    \begin{align*}
         \left(\frac{1}{6}\right)_n \cdot \left(\frac{5}{6}\right)_n \cdot \left(\frac{3}{6}\right)_n
         &= 6^{-3n}\prod_{k=1}^n(6k-5)(6k-3)(6k-1)\\
         &= 6^{-3n}\cdot 1\cdot 3\cdot 5 \cdot 7 \cdots (6n-1)\\
         &= 6^{-3n}\cdot\frac{(6n)!}{2\cdot 4\cdot 6\cdots 6n}\\
         &= 6^{-3n}\cdot\frac{(6n)!}{2^{3n}\cdot (3n)!}= \frac{(6n)!}{(3n)!\cdot 12^{3n}}
    \end{align*}
    \en{If we put this along with $(1)_n=n!$ into eq.~(\ref{\en{EN}asdf}) we get:}%
    \de{Wenn wir das in Gleichung~(\ref{\en{EN}asdf}) einsetzen, erhalten wir mit $(1)_n=n!$:}
    \begin{align*}
      \left({_2F_1}\lk\frac{1}{12},\frac{5}{12};1;z\rk\right)^2  = \sum_{n=0}^{\infty} \frac{(6n)!}{(3n)!\cdot 12^{3n}\cdot n! \cdot n!}\cdot\frac{z^n}{n!}
  \end{align*}
  \en{which concludes the proof.}%
  \de{und das ist was zu zeigen war.}
\end{proof}

\pagebreak
\begin{theo}[\en{Main Theorem}\de{Haupttheorem}]\label{\en{EN}hauptformel}
\en{For all $\tau$ with $\im(\tau)>1\ko25$ we have the following identity due to David and Gregory Chudnovsky, first published in 1988 \cite[Eq.~(1.4)]{\en{EN}chud1988}:}%
\de{Für alle $\tau$ mit $\im(\tau)>1\ko25$ gilt die folgende Gleichung von David und Gregory Chudnovsky aus dem Jahr 1988 \cite[Glg.~(1.4)]{\en{EN}chud1988}:}
\begin{align*}
    \frac{1}{2\pi \im(\tau)}\sqrt{\frac{J(\tau)}{J(\tau)-1}} &= \sum_{n=0}^\infty \left( \frac{1-s_2(\tau)}{6} + n \right)\cdot \frac{(6n)!}{(3n)!(n!)^3}\cdot \frac{1}{\left(1728J(\tau)\right)^n}
\end{align*}
\en{Here $\sqrt{\phantom{J}}$ denotes the principal branch of the square root.}%
\de{Hierbei bezeichnet $\sqrt{\phantom{J}}$ den Hauptzweig der Quadratwurzel.}
\end{theo}

\begin{proof}
\en{For the proof we combine the differential equation from Prop.~\ref{\en{EN}thm42} with the representation from Prop.~\ref{\en{EN}darst}:}%
\de{Für den Beweis kombinieren wir die Differentialgleichung aus Satz~\ref{\en{EN}thm42} mit der Darstellung aus Satz~\ref{\en{EN}darst}:}

\en{First we use the same notation like in Prop.~\ref{\en{EN}thm42} and denote the function $F(J)={_2F_1}\lk\frac{1}{12},\frac{5}{12};1;\frac 1 {J}\rk$ with $J=J(\tau)$.
Then we call $G(z)=\left({_2F_1}\lk\frac{1}{12},\frac{5}{12};1;z\rk\right)^2$.
Then it holds $\left(F(J)\right)^2=G(z)$ with $z=\frac 1 J$.
In Prop.~\ref{\en{EN}thm42} we proved: $$\frac{1}{2\pi \im(\tau)}\sqrt{\frac{J}{J-1}} = \frac{1-s_2(\tau)} 6 F^2 - J \frac d {dJ} F^2$$}%
\de{Wir bezeichnen zunächst wie in Satz~\ref{\en{EN}thm42} die Funktion $F(J)={_2F_1}\lk\frac{1}{12},\frac{5}{12};1;\frac 1 {J}\rk$ mit $J=J(\tau)$.
Darüberhinaus nennen wir $G(z)=\left({_2F_1}\lk\frac{1}{12},\frac{5}{12};1;z\rk\right)^2$.
Dann haben wir mit $z=\frac 1 J$ nämlich $\left(F(J)\right)^2=G(z)$.
In Satz~\ref{\en{EN}thm42} haben wir bereits $$\frac{1}{2\pi \im(\tau)}\sqrt{\frac{J}{J-1}} = \frac{1-s_2(\tau)} 6 F^2 - J \frac d {dJ} F^2$$
bewiesen.}
\en{Now we transform this differential equation of $F(J)$ into one of $G(z)$, like in the proof of Prop.~\ref{\en{EN}satzbj}.
From $J=\frac{1}{z}$ we get $\frac{dJ}{dz}=\frac{-1}{z^2}$ and $\frac{dz}{dJ}=-z^2$. This yields:}%
\de{Wir wandeln diese Differentialgleichung für $F(J)$ in eine für $G(z)$ um, ähnlich wie im Beweis von Satz~\ref{\en{EN}satzbj}.
Aus $J=\frac{1}{z}$ folgt $\frac{dJ}{dz}=\frac{-1}{z^2}$ und $\frac{dz}{dJ}=-z^2$. Also gilt}
$$J \frac d {dJ} (F(J))^2 = \frac 1 z \cdot \frac {dz}{dJ} \cdot \frac{d}{dz} G(z) = -z \cdot \frac{d}{dz} G(z)$$
\en{This produces a differential equation of $G(z)$:}%
\de{Dies liefert uns eine Differentialgleichung für $G(z)$:}
$$\frac{1}{2\pi \im(\tau)}\sqrt{\frac{J}{J-1}} = \frac{1-s_2(\tau)} 6 G(z) + z \frac d {dz} G(z)\qquad\text{\en{with}\de{mit} }J=\frac 1 z$$
\en{Here we use the representation of $G(z)$ from Prop.~\ref{\en{EN}darst}:}%
\de{Hier können wir noch die Darstellung von $G(z)$ aus Satz~\ref{\en{EN}darst} einsetzen:}
\begin{align*}
    G(z)&=\sum_{n=0}^\infty \frac{(6n)!}{(3n)!(n!)^3} \frac{z^n}{12^{3n}}\\
    \Longrightarrow\quad z \frac d {dz} G(z) &= \sum_{n=0}^\infty \frac{(6n)!}{(3n)!(n!)^3} \frac{n\cdot z^n}{12^{3n}}\\
    \Longrightarrow\quad \frac{1}{2\pi \im(\tau)}\sqrt{\frac{J}{J-1}} &= \sum_{n=0}^\infty \left(\frac{1-s_2(\tau)} 6 + n\right)\cdot \frac{(6n)!}{(3n)!(n!)^3}\cdot \frac{z^n}{12^{3n}}
\end{align*}
\en{And finally we get the statement of Thm.~\ref{\en{EN}hauptformel}, if we use $z=\frac 1 J$ and $12^3=1728$.}%
\de{Und wir erhalten schließlich mit $z=\frac 1 J$ und $12^3=1728$ die zu beweisende Aussage.}

\en{In the intermediate steps until here we never cared about the choice of the root (cf.~Remark~\ref{\en{EN}bemwurzel}).
Our result in Thm.~\ref{\en{EN}hauptformel} is thus (until now) only proven correct up to a root of unity.
Thus we still have to prove for any $\tau$ with $\im(\tau)>1\ko25$ that our result doesn't contain any hidden complex roots of unity:}%
\de{In allen Zwischenschritten bis hier war uns die Wahl des Zweigs der Wurzeln egal (vgl.~Bemerkung~\ref{\en{EN}bemwurzel}),
deshalb haben wir bisher auch nur bewiesen, dass die Gleichung aus Thm.~\ref{\en{EN}hauptformel} korrekt bis auf eine komplexe Einheitswurzel ist.
Wir müssen also noch beweisen, dass die Gleichung für irgendein $\tau$ mit $\im(\tau)>1\ko25$ exakt ist und keine versteckten Einheitswurzeln enthält:}

\en{We choose $\tau_8=\frac{i\sqrt{8}}{2}=i\sqrt{2}$.
Here we get that $q=e^{2\pi i \tau_8} = e^{-2\pi\sqrt{2}}$ is a real number.
Thus (from Thm.~\ref{\en{EN}fouriertheorem}) both $J(\tau_8)$ and $s_2(\tau_8)$ are also real-valued.
The approximations together with the error estimates from Thm.~\ref{\en{EN}theonaeherJ} and~\ref{\en{EN}theonaehers2} tell us that both $J(\tau_8)$ and $\frac{1-s_2(\tau_8)}{6}$ are \emph{positive} real numbers.
This shows that all quantities in the equation of Thm.~\ref{\en{EN}hauptformel} are real-valued and positive at $\tau=\tau_8$.

This proves that if we choose the principal branch of the square root (which is positive on the positive real axis), the equation is exact without any hidden roots of unity at $\tau=\tau_8$.
The region $\im(\tau)>1\ko25$ is connected, both sides of the equation depend continuously on $\tau$ and are not zero.
This proves that the equation is exact without any hidden roots of unity for all $\tau$ with $\im(\tau)>1\ko25$, if the principal branch of the square root is chosen.}%
\de{Wir wählen $\tau_8=\frac{i\sqrt{8}}{2}=i\sqrt{2}$.
Hier gilt, dass $q=e^{2\pi i \tau_8} = e^{-2\pi\sqrt{2}}$ eine reelle Zahl ist.
Die Darstellungen in Thm.~\ref{\en{EN}fouriertheorem} zeigen dann, dass $J(\tau_8)$ und $s_2(\tau_8)$ ebenfalls reellwertig sind.
Weiter folgt aus den Näherungen und den Fehlerabschätzungen aus Thm.~\ref{\en{EN}theonaeherJ} und~\ref{\en{EN}theonaehers2}, dass sowohl $J(\tau_8)$ und $\frac{1-s_2(\tau_8)}{6}$ \emph{positive} reelle Zahlen sind.
Somit ist bewiesen, dass für $\tau=\tau_8$ alle Größen in der Gleichung des Haupttheorems~\ref{\en{EN}hauptformel} reellwertig und positiv sind.

Das zeigt, dass in der Gleichung bei $\tau=\tau_8$ keine komplexen Einheitswurzeln vergessen wurden, sofern man den Hauptzweig der Quadratwurzel wählt, welcher den postiven Radikanden positive Werte zuordnet.
Das Gebiet $\im(\tau)>1\ko25$ ist zusammenhängend, beide Seiten der Gleichung hängen stetig von $\tau$ ab und sind nicht Null.
Hieraus folgt, dass die Gleichung aus Thm.~\ref{\en{EN}hauptformel} auch für alle anderen $\tau$ mit $\im(\tau)>1\ko25$ exakt erfüllt ist, sofern man wie gefordert den Hauptzweig der Quadratwurzel wählt.}
\end{proof}

\vfill\pagebreak\section{\en{Determination of the Coefficients}\de{Bestimmung der Koeffizienten}}\label{\en{EN}kapFormel}
\renewcommand{\leftmark}{\en{Determination of the Coefficients}\de{Bestimmung der Koeffizienten}}
\en{In this chapter, we determine some values of $s_2(\tau)$ and of $j(\tau):=1728J(\tau)$.
For this we calculate \emph{approximations} as in Thm.~\ref{\en{EN}theonaeherJ} and \ref{\en{EN}theonaehers2} which we use to obtain the \emph{exact} values. We also prove the exactness of these values.
Finally, we use them in the Main Theorem~\ref{\en{EN}hauptformel} and obtain eleven formulae for calculating $\pi$ -- one of them is the Chudnovsky formula.
}%
\de{In diesem Kapitel bestimmen wir einige Funktionswerte von $s_2(\tau)$ und von $j(\tau):=1728J(\tau)$.
Dafür berechnen wir zunächst \emph{Näherungswerte} wie in Thm.~\ref{\en{EN}theonaeherJ} und \ref{\en{EN}theonaehers2}, mit Hilfe derer wir dann auf die \emph{exakten} Funktionswerte schließen und und ihre Exaktheit auch beweisen.
Schließlich setzen wir diese Funktionswerte ins Haupttheorem~\ref{\en{EN}hauptformel} ein und erhalten elf Formeln zur Berechnung von $\pi$ -- eine davon ist die Chudnovsky-Formel.}

\en{Our paper \cite{\en{EN}Milla2021} in the Ramanujan Journal is based on this chapter.}%
\de{Unser Paper \cite{\en{EN}Milla2021} im Ramanujan Journal basiert auf diesem Kapitel.}
\begin{defi}\label{\en{EN}defiCM}
\begin{enumerate}[leftmargin=*]
    \item \en{Those $\tau\in\mathbb H$ for which the elliptic curve associated with $L_\tau$ has complex multiplication are called "CM-points".
They are those $\tau$ for which it exists $a\in\mathbb C - \mathbb Z$ such that $a\cdot L_\tau \subseteq L_\tau$.
This yields $a\cdot 1\in L_\tau$ and $a\cdot\tau\in L_\tau$ -- or after some calculations:}%
    \de{Diejenigen $\tau\in\mathbb H$, für die die zu $L_\tau$ gehörige elliptische Kurve eine komplexe Multiplikation hat, heißen \glqq CM-Punkte\grqq.
Das sind diejenigen $\tau$, für die es eine Zahl $a\in\mathbb C - \mathbb Z$ gibt mit $a\cdot L_\tau \subseteq L_\tau$.
Das bedeutet $a\cdot 1\in L_\tau$ und $a\cdot\tau\in L_\tau$ bzw.:}
$$\text{CM} := \left\{~\tau\in\mathbb H~\left|~A+B\tau+C\tau^2=0;~(A,B,C)\in\mathbb Z^3;~\operatorname{\en{gcd}\de{ggT}}(A,B,C)=1~\right.\right\}$$
\item \en{For each $\tau\in CM$, we call $D = B^2-4AC$ the "discriminant of $\tau$".}%
\de{Für alle $\tau\in \text{CM}$ bezeichnen wir $D = B^2-4AC$ als \glqq Diskriminante von $\tau$\grqq.}
\item \en{All CM-points $\tau$ for which the imaginary quadratic field $K=\mathbb Q(\tau)$ has the class number $h_K=1$ are called \glqq $\text{CM}_1$-points\grqq:}%
\de{Alle CM-Punkte $\tau$, für die der imaginärquadratische Zahlkörper $K=\mathbb Q(\tau)$ die Klassenzahl $h_K=1$ hat, nennen wir auch \glqq $\text{CM}_1$-Punkte\grqq:}
$$\text{CM}_1 := \left\{~\tau\in \text{CM}~\left|~K=\mathbb Q(\tau)\text{ \en{has class number}\de{hat die Klassenzahl} }h_K=1~\right.\right\}$$
\end{enumerate}
\end{defi}

\begin{thm}\label{\en{EN}extjganzalgg}
    \en{For all CM-points $\tau$, the value $j(\tau):=1728J(\tau)$ is an algebraic integer.
    The degree of this algebraic integer is the class number of the imaginary quadratic field $\mathbb Q(\tau)$.}%
    \de{Für alle CM-Punkte $\tau$ ist der Wert $j(\tau):=1728J(\tau)$ ganzalgebraisch.
    Und der Grad dieser ganzalgebraischen Zahl ist gleich der Klassenzahl des imaginärquadratischen Zahlkörpers~$\mathbb Q(\tau)$.}
\end{thm}
\begin{proof}
\en{In \cite[Thm.~II.4.3 (b)]{\en{EN}silverman1994advanced} it is proven that for all $\tau\in \text{CM}$ it holds: $j(\tau)=1728J(\tau)$ is an algebraic number whose degree is the class number of $\mathbb Q(\tau)$.
And in \cite[Thm.~II.6.1]{\en{EN}silverman1994advanced} it is proven that these $j(\tau)$ are even algebraic integers.}%
\de{In \cite[Thm.~II.4.3 (b)]{\en{EN}silverman1994advanced} wird bewiesen, dass der Wert von $j(\tau):=1728J(\tau)$ für alle $\tau\in \text{CM}$ eine algebraische Zahl ist, deren Grad der Klassenzahl von $\mathbb Q(\tau)$ entspricht.
Und in \cite[Thm.~II.6.1]{\en{EN}silverman1994advanced} wird bewiesen, dass diese Werte $j(\tau)$ sogar ganzalgebraisch sind.}
\end{proof}

\begin{thm}\label{\en{EN}exts2rat}
\en{For all CM-points $\tau$ which are not equivalent to $i$ under modular transformations (see Def.~\ref{\en{EN}defimodtau}), $s_2(\tau)$ lies in the field generated over $\mathbb Q$ by $j(\tau)= 1728 J(\tau)$.}%
\de{Für alle CM-Punkte $\tau$, die zu $i$ nicht äquivalent unter Modultransformationen sind (siehe Def.~\ref{\en{EN}defimodtau}), liegt $s_2(\tau)$ in der Körper\-erwei\-terung $\mathbb Q\lk j(\tau)\rk=\mathbb Q\lk1728J(\tau)\rk$.}
\end{thm}
\begin{proof}
\en{See \cite[Appendix~A1, Thm.~A1]{\en{EN}Masser1975}. Notation there: $\Psi(\tau)=\frac{3}{2}s_2(\tau)$.}%
\de{Siehe \cite[Anhang A1, Thm.~A1]{\en{EN}Masser1975}. Notation dort: $\Psi(\tau)=\frac{3}{2}s_2(\tau)$.}\qedhere\\
\en{\emph{Remark:} At $\tau=i$ it holds $E_6(i)=0=E_2^*(i)$ and $s_2(i)$ is undefined.}%
\de{\emph{Bemerkung:} Bei $\tau=i$ gilt $E_6(i)=0=E_2^*(i)$ und $s_2(i)$ ist nicht definiert.}
\end{proof}

\begin{thm}\label{\en{EN}satztaun}
\en{Let}\de{Es sei} $\mathcal H := \left\{ 3;4;7;8;11;12;16;19;27;28;43;67;163\right\}$
\en{and}\de{und}
$$\tau_N:=\begin{cases}\frac{0+i\sqrt{N}}{2}&\text{\en{if}\de{falls} }N\equiv 0 \mod 4\\ \frac{1+i\sqrt{N}}{2}&\text{\en{if}\de{falls} }N\equiv 3 \mod 4\end{cases}\qquad\text{\en{for}\de{für} }N\in\mathcal H$$
\en{Then each of these $\tau_N$ with $N\in\mathcal H$ is a $\text{CM}_1$-point.}%
\de{Dann ist jedes dieser $\tau_N$ mit $N\in\mathcal H$ ein $\text{CM}_1$-Punkt.}
\end{thm}
\begin{proof}
\en{The class numbers $h_K$ of these $K=\mathbb Q(\tau_N)$ are determined in \cite[Ch.~1-2]{\en{EN}Buell1989}, and for all of them it holds $h_K=1$.
 \emph{Remark:} One could even prove that each $\text{CM}_1$-point is equivalent to one of these $\tau_N$, but we don't need this for our proof.}%
\de{Die Klassenzahlen $h_K$ dieser $K=\mathbb Q(\tau_N)$ werden z.B. in \cite[Kap.~1-2]{\en{EN}Buell1989} berechnet, bei allen ist $h_K=1$.
\emph{Bemerkung:} Es gilt sogar, dass jeder $\text{CM}_1$-Punkt zu einem dieser $\tau_N$ äquivalent ist, aber das wird für unseren Beweis nicht benötigt.}
\end{proof}

\begin{thm}\label{\en{EN}satzklassenzahl}
    \en{For all $\text{CM}_1$-points $\tau$, in particular for all $\tau_N$ with $N\in\mathcal H$, it holds $j(\tau)\in\mathbb Z$.
    And for all $\text{CM}_1$-points $\tau$ which are not equivalent to $\tau_4=i$, it holds $s_2(\tau)\in\mathbb Q$.}%
    \de{Für alle $\text{CM}_1$-Punkte $\tau$, also insbesondere für alle $\tau_N$ mit $N\in\mathcal H$, gilt $j(\tau)\in\mathbb Z$.    Und für alle $\text{CM}_1$-Punkte $\tau$, die nicht äquivalent zu $\tau_4=i$ sind, gilt $s_2(\tau)\in\mathbb Q$.}
\end{thm}
\begin{proof}
\en{Prop.~\ref{\en{EN}extjganzalgg} tells that $j(\tau)$ is an algebraic integer of degree $1$ for all $\text{CM}_1$-points $\tau$, thus it holds $j(\tau)\in\mathbb Z$.
This yields $\mathbb Q\lk j(\tau)\rk=\mathbb Q$ and Prop.~\ref{\en{EN}exts2rat} proves $s_2(\tau)\in\mathbb Q$ if this $\text{CM}_1$-point $\tau$ is not equivalent to $i$.}%
\de{Satz \ref{\en{EN}extjganzalgg} besagt für $\text{CM}_1$-Punkte $\tau$, dass $j(\tau)$ ganzalgebraisch vom Grad $1$ ist, also muss $j(\tau)\in\mathbb Z$ gelten.
Hieraus folgt $\mathbb Q\lk j(\tau)\rk=\mathbb Q$ und Satz \ref{\en{EN}exts2rat} liefert $s_2(\tau)\in\mathbb Q$, falls dieser $\text{CM}_1$-Punkt $\tau$ nicht äquivalent zu $i$ ist.}
\end{proof}

\begin{bem}\label{\en{EN}bemceulen}
\en{Now we need some digits of $\pi$ to calculate approximate values of $j(\tau)$ and of $s_2(\tau)$ with sufficient precision.
These digits mustn't be computed with the Chudnovsky algorithm (since we still have to prove it).

Ludolph van Ceulen (1540--1610) used a regular $2^{62}$-gon to calculate the following $35$ digits of $\pi$.
They have been published by his student Willebrord Snell \cite[p.~55]{\en{EN}snellius} in 1621:}%
\de{Um jetzt hinreichend gute Näherungswerte für $j(\tau)$ und $s_2(\tau)$ zu berechnen, benötigen wir einige Stellen von $\pi$. Diese dürfen natürlich nicht mit der noch zu beweisenden Chudnovsky-Formel berechnet werden.

Ludolph van Ceulen (1540--1610) berechnete mit Hilfe eines regelmäßigen $2^{62}$-Ecks 35~Stellen von $\pi$, veröffentlicht im Jahr 1621 von seinem Schüler Willebrord Snell \cite[S.~55]{\en{EN}snellius}:}
$$\pi=3\ko14159~26535~89793~23846~26433~83279~50288\ldots$$
\end{bem}

\pagebreak

\begin{thm}\label{\en{EN}jwert2}
   \en{All values of $j(\tau)=1728J(\tau)$ given in Tab.~\ref{\en{EN}tabJwerte} are exact and correct.}%
   \de{Alle Werte von $j(\tau)=1728J(\tau)$ in Tab.~\ref{\en{EN}tabJwerte} sind exakt und korrekt.}
\end{thm}
\begin{proof}\begin{itemize}[leftmargin=*]
    \item \en{Prop.~\ref{\en{EN}satztaun} tells that all $\tau_N$ in Tab.~\ref{\en{EN}tabJwerte} are $\text{CM}_1$-points.}%
    \de{Satz~\ref{\en{EN}satztaun} besagt, dass alle $\tau_N$ in Tab.~\ref{\en{EN}tabJwerte} $\text{CM}_1$-Punkte sind.}
    \item \en{For these $\tau_N$, we calculate the approximations as in Thm.~\ref{\en{EN}theonaeherJ}:}%
    \de{Für diese $\tau_N$ berechnen wir nun die Näherungswerte wie in Thm.~\ref{\en{EN}theonaeherJ} beschrieben:}
    $$1728\tilde J(\tau) := \frac{\left(1 + 240\left( q + 9 q^2\right)\right)^3}{ q \cdot (1-q-q^2)^{24}}$$
    \en{Here we use $q=e^{2\pi i \tau_N}=(-1)^N\cdot e^{-\pi\sqrt{N}}$ and 25 digits of $\pi$ (which we can take from Rem.~\ref{\en{EN}bemceulen}).
    The results are given in Tab.~\ref{\en{EN}tabJwerte} (third column).}%
    \de{Dabei setzen wir $q=e^{2\pi i \tau_N}=(-1)^N\cdot e^{-\pi\sqrt{N}}$ und 25 Stellen von $\pi$ ein (die wir z.B. aus Bem.~\ref{\en{EN}bemceulen} entnehmen).
    Die Ergebnisse finden sich in Tab.~\ref{\en{EN}tabJwerte} (dritte Spalte).}
    \item \en{For all $\tau_N$ in Tab.~\ref{\en{EN}tabJwerte} it holds $\im(\tau_N)>1\ko25$, thus Thm.~\ref{\en{EN}theonaeherJ} gives an error bound for the approximations: $|1728J(\tau)-1728\tilde J(\tau)|<  0\ko2 $.}%
    \de{Für alle gelisteten $\tau_N$ gilt $\im(\tau_N)>1\ko25$, also besagt Thm.~\ref{\en{EN}theonaeherJ}, dass der Fehler der Näherungen $|1728J(\tau)-1728\tilde J(\tau)|<  0\ko2 $ ist.}
    \item \en{Prop.~\ref{\en{EN}satzklassenzahl} tells that the unknown values of $j(\tau_N)=1728J(\tau_N)$ are integers.}%
    \de{Satz~\ref{\en{EN}satzklassenzahl} besagt, dass die gesuchten Werte von $j(\tau_N)=1728J(\tau_N)$ ganzzahlig sind.}
    \item \en{Thus $j(\tau)$ must have \emph{exactly} the values from Tab.~\ref{\en{EN}tabJwerte} (last column), because these are the only integers close enough to the approximations.}%
    \de{Somit muss $j(\tau)$ \emph{genau} die Werte aus Tab.~\ref{\en{EN}tabJwerte} (letzte Spalte) annehmen, weil das die einzigen ganzen Zahlen sind, die nahe genug am Wert der Näherungen liegen.}\qedhere
\end{itemize}\end{proof}

\begin{table}[ht]\begin{center}\renewcommand{\arraystretch}{1.3}%
    \begin{tabular}{c|>{\columncolor[gray]{0.9}}c|r|>{\columncolor[gray]{0.9}}r}
    && \en{Approximation}\de{Näherungswert} (Thm.~\ref{\en{EN}theonaeherJ}) & \en{Exact value}\de{Exakter Wert}\\
 $D=-N$ & $\tau_N$ & $1728\tilde J(\tau_N)$ & $j(\tau_N)=1728J(\tau_N)$\\
\hline
 $-8$ & $i\sqrt{2}$ & $7999\ko99959$&  $20^3$ \\
 $-12$ & $i\sqrt{3}$ & $53999\ko99999$&  $2\cdot 30^3$  \\
 $-16$ & $i\sqrt{4}$ & $287496\ko00000$&  $66^3$\\
 $-28$ & $i\sqrt{7}$ & $16581375\ko00000$&  $255^3$\\
\hline
 $-7$ & $\frac{1+i\sqrt{7}}{2}$ & $-3375\ko00107$ & $-15^3$\\
 $-11$ & $\frac{1+i\sqrt{11}}{2}$ & $-32768\ko00002$& $-32^3$ \\
 $-19$ & $\frac{1+i\sqrt{19}}{2}$ & $-884736\ko00000$&  $-96^3$ \\
 $-27$ & $\frac{1+i\sqrt{27}}{2}$ & $-12288000\ko00000$&  $-3\cdot 160^3$\\
 $-43$ & $\frac{1+i\sqrt{43}}{2}$ & $-884736000\ko00000$&  $-960^3$  \\
 $-67$ & $\frac{1+i\sqrt{67}}{2}$ & $-147197952000\ko00000$&  $-5280^3$\\
 $-163$ & $\frac{1+i\sqrt{163}}{2}$ & $-262537412640768000\ko00000$&  $-640320^3$
    \end{tabular}\end{center}\vspace*{6pt}\caption{\en{Calculation of $j(\tau)$ at some $\text{CM}_1$-points}\de{Berechnung von $j(\tau)$ bei einigen $\text{CM}_1$-Punkten}}\label{\en{EN}tabJwerte}\end{table}

\begin{bem}
\en{In fact, $\tau_3=\varrho=\frac{1+i\sqrt{3}}{2}$ and $\tau_4=i$ are also $\text{CM}_1$-points.
But $E_4(\varrho)=0$ yields $J(\varrho)=0$ and $E_6(i)=0$ yields $J(i)=1$.
By Prop.~\ref{\en{EN}satzkonv}, the hypergeometric sum in Thm.~\ref{\en{EN}hauptformel} converges only if $|J(\tau)|>1$.
Thus $\tau_3$ and $\tau_4$ won't produce formulae for calculating $\pi$ and aren't listed in Tab.~\ref{\en{EN}tabJwerte}.}%
\de{Eigentlich sind $\tau_3=\varrho=\frac{1+i\sqrt{3}}{2}$ und $\tau_4=i$ ebenfalls $\text{CM}_1$-Punkte.
Aber aus $E_4(\varrho)=0$ folgt $J(\varrho)=0$ und aus $E_6(i)=0$ folgt $J(i)=1$, und nach Satz~\ref{\en{EN}satzkonv} konvergiert die hypergeometrische Summe in Thm.~\ref{\en{EN}hauptformel} nur, falls $|J(\tau)|>1$ ist.
Deshalb liefern $\tau_3$ und $\tau_4$ keine Formeln zur Berechnung von $\pi$ und werden in Tab.~\ref{\en{EN}tabJwerte} nicht gelistet.}
\end{bem}

\begin{thm}\label{\en{EN}satzezweistern}
\en{For all CM-points $\tau$, the following expressions are algebraic integers:}%
\de{Für alle CM-Punkte $\tau$ sind die folgenden Ausdrücke ganzalgebraisch:}
\begin{align*}
    \sqrt{D}\cdot\frac{E_2^*(\tau)}{\eta^4(\tau)}\cdot(AC)^2
    \qquad\text{\en{and}\de{und}}\qquad
    \frac{E_4(\tau)}{\eta(\tau)^8}
    \qquad\text{\en{and}\de{und}}\qquad
    \frac{E_6(\tau)}{\eta(\tau)^{12}}~.
\end{align*}
\en{Here, $E_{4}(\tau)$ and $E_{6}(\tau)$ denote the normalized Eisenstein series from Thm.~\ref{\en{EN}fouriertheorem}; $\eta(\tau)$ denotes the Dedekind $\eta$-Function with $1728\eta^{24}:=E_4^3-E_6^2$ (no further properties of the $\eta$-function are needed); $E_2^*(\tau)$ is defined as in Def.~\ref{\en{EN}defis2}; $D=B^2-4AC$ is the discriminant of the quadratic equation $A+B\tau+C\tau^2=0$.}%
\de{Dabei bezeichnen $E_{4}(\tau)$ und $E_{6}(\tau)$ die normierten Eisensteinreihen aus Thm.~\ref{\en{EN}fouriertheorem}; $\eta(\tau)$ bezeichne die Dedekind'sche $\eta$-Funktion mit $1728\eta^{24}:=E_4^3-E_6^2$ (weitere Eigenschaften der $\eta$-Funktion werden nicht benötigt); $E_2^*(\tau)$ ist wie in Def.~\ref{\en{EN}defis2} definiert; $D=B^2-4AC$ ist die Diskriminante der quadratischen Gleichung $A+B\tau+C\tau^2=0$.}
\end{thm}
\pagebreak
\begin{proof} \en{From the definitions of $J(\tau)$ and $\eta(\tau)$ we get}%
\de{Aus den Definitionen von $J(\tau)$ und $\eta(\tau)$ erhalten wir}
\begin{align}
    j(\tau)&=1728J(\tau)=\frac{1728 E_4^3}{E_4^3-E_6^2}=\frac{E_4(\tau)^3}{\eta(\tau)^{24}}=\left(\frac{E_4(\tau)}{\eta(\tau)^{8}}\right)^3\nonumber\\
\intertext{\en{and}\de{und}}
\label{\en{EN}eq101}j(\tau)-1728&=\frac{1728 E_4^3}{E_4^3-E_6^2}-1728\frac{E_4^3-E_6^2}{E_4^3-E_6^2}=\frac{E_6(\tau)^2}{\eta(\tau)^{24}}=\left(\frac{E_6(\tau)}{\eta(\tau)^{12}}\right)^2~.
\end{align}
\en{So $\frac{E_4(\tau)}{\eta(\tau)^8}$ is a zero of $P(X)=X^3-j(\tau)$ and $\frac{E_6(\tau)}{\eta(\tau)^{12}}$ is a zero of $Q(X)=X^2-j(\tau)+1728$.
Since $j(\tau)$ is an algebraic integer (Prop.~\ref{\en{EN}extjganzalgg}), both terms $\left(\frac{E_4(\tau)}{\eta(\tau)^8}\right.$ and $\left.\frac{E_6(\tau)}{\eta(\tau)^{12}}\right)$ are algebraic integers.}%
\de{Also ist $\frac{E_4(\tau)}{\eta(\tau)^8}$ eine Nullstelle von $P(X)=X^3-j(\tau)$ und $\frac{E_6(\tau)}{\eta(\tau)^{12}}$ ist eine Nullstelle von $Q(X)=X^2-j(\tau)+1728$.
Weil nach Satz~\ref{\en{EN}extjganzalgg} $j(\tau)$ ganzalgebraisch ist, sind $\frac{E_4(\tau)}{\eta(\tau)^8}$ und $\frac{E_6(\tau)}{\eta(\tau)^{12}}$ ebenfalls ganzalgebraisch.}

\en{It remains to prove that the first term is also an algebraic integer.
A complete and self-contained proof for this can be found in Appendix~\ref{\en{EN}kapmult} (Thm.~\ref{\en{EN}theoezweisternrest}), which uses App.~\ref{\en{EN}kapdivi}.}%
\de{Der vollständige Beweis, dass auch der erste Ausdruck ganzalgebraisch ist, folgt im Anhang~\ref{\en{EN}kapmult} (Thm.~\ref{\en{EN}theoezweisternrest}), welcher auf Anhang~\ref{\en{EN}kapdivi} aufbaut.}
\end{proof}

\begin{theo}\label{\en{EN}s2nenner}
    \en{For $\text{CM}_1$-points $\tau$ which are not equivalent to $i$, if the class number of $\mathbb Q(\tau)$ is $1$ and if its minimal equation is $A+B\tau+C\tau^2=0$ with discriminant $D=B^2-4AC$,
    then there is $c\in\mathbb Z$, so that
    $$b := \sqrt{c\cdot D\cdot(j(\tau)-1728)}\cdot(AC)^2\in\mathbb Z.$$
    And then 
    $a := s_2(\tau)\cdot b\in\mathbb Z$
    is also integer and we obtain a representation of $s_2(\tau)=a/b$ as the ratio of two integers.}%
    \de{Für $\text{CM}_1$-Punkte $\tau$, die nicht äquivalent zu $i$ sind und bei denen $\mathbb Q(\tau)$ die Klassenzahl $1$ hat, gilt:
    Wenn die Minimalgleichung $A+B\tau+C\tau^2=0$ mit Diskriminante $D=B^2-4AC$ ist,
    dann gibt es $c\in\mathbb Z$, sodass
    $$b := \sqrt{c\cdot D\cdot(j(\tau)-1728)}\cdot(AC)^2\in\mathbb Z$$
    ganzzahlig ist. Außerdem ist dann auch
    $a := s_2(\tau)\cdot b\in\mathbb Z$
    ganzzahlig und wir erhalten eine explizite Darstellung von $s_2(\tau)=a/b$ als Verhältnis zweier ganzer Zahlen.}
\end{theo}
\begin{proof}

\en{The integrality of $b$ follows from the integrality of $j(\tau)$ (Prop.~\ref{\en{EN}satzklassenzahl}) and the fact that $c$ can be chosen so that the radicand is a square -- for example $c=D\cdot(j(\tau)-1728)$.
But it is possible to choose $c$ with much smaller absolute values, they are listed in Tab.~\ref{\en{EN}tabel7}.}%
\de{Die Ganzzahligkeit von $b$ folgt aus der Ganzzahligkeit von $j(\tau)$ (Satz~\ref{\en{EN}satzklassenzahl}) und der Tatsache, dass man $c$ so wählen kann, dass der Radikand eine Quadratzahl wird -- das gilt z.B.~für $c=D\cdot(j(\tau)-1728)$, aber auch für betragsmäßig deutlich kleinere~$c$ (sie werden in Tab.~\ref{\en{EN}tabel7} aufgelistet).}

\en{From its definition $a:=s_2(\tau)\cdot b$ we deduce that $a$ is the product of a rational number (Prop.~\ref{\en{EN}satzklassenzahl}) and an integer -- 
thus $a$ must be \emph{rational} for all $\text{CM}_1$-points $\tau$ not equivalent to $i$.
It remains to prove that these $a$ are even \emph{integral}.
From~(\ref{\en{EN}eq101}) we obtain}%
\de{Aus der Definition $a:=s_2(\tau)\cdot b$ folgt, dass $a$ als Produkt einer rationalen Zahl (Satz~\ref{\en{EN}satzklassenzahl}) mit einer ganzen Zahl selbst \emph{rational} sein muss, falls $\tau$ nicht äquivalent zu $i$ ist.
Es fehlt noch zu zeigen, dass diese $a$ sogar \emph{ganze} Zahlen sind.
Zunächst folgt aus~(\ref{\en{EN}eq101}):}
$$\sqrt{j(\tau)-1728}=\pm\frac{E_6(\tau)}{\eta^{12}(\tau)}$$

\en{We use this in the definition of $a$ and (using Def.~\ref{\en{EN}defis2}) we obtain:}%
\de{Das setzen wir in die Definition für $a$ ein und erhalten mit Def.~\ref{\en{EN}defis2}:}
\begin{align}
    a :=&~s_2(\tau)\cdot b = s_2(\tau)\cdot \sqrt{c\cdot D\cdot(j(\tau)-1728)} \cdot(AC)^2\nonumber\\
     =&~ \frac{E_4(\tau)}{E_6(\tau)}\cdot E^*_2(\tau)\cdot \sqrt{c\cdot D}\cdot \pm\frac{E_6(\tau)}{\eta^{12}(\tau)}\cdot(AC)^2\nonumber\\
    =&~ \frac{E_4(\tau)}{\eta^8(\tau)}\cdot 
    \frac{\sqrt{D}\cdot E^*_2(\tau)\cdot(AC)^2}{\eta^4(\tau)}
    \cdot\pm\sqrt{c}\label{\en{EN}produ}
\end{align}
\en{Here we could reduce by $E_6(\tau)$, because Lemma~\ref{\en{EN}lemE6} tells $E_6(\tau)\neq 0$ for $\im(\tau)\geq 1\ko25$.}%
\de{Hier durften wir wegen Lemma~\ref{\en{EN}lemE6} mit $E_6(\tau_N)\neq 0$ kürzen (beachte, dass für $N\geq 7$ gilt: $\im(\tau_N)\geq\frac{\sqrt{7}}{2}\geq 1\ko25$).}
\begin{itemize}[leftmargin=*]
\item \en{The first two factors in Eq.~(\ref{\en{EN}produ}) are algebraic integers because of Prop.~\ref{\en{EN}satzezweistern}.}%
\de{Die ersten beiden Faktoren in Gl.~(\ref{\en{EN}produ}) sind nach Satz~\ref{\en{EN}satzezweistern} ganzalgebraisch.}
\item \en{For the remaining factor $X=\pm\sqrt{c}$ it holds $X^2 = c$; it is thus an algebraic integer.}%
\de{Für $X=\pm\sqrt{c}$ gilt $X^2 = c$; es ist also ebenfalls ganzalgebraisch.}
\end{itemize}
\en{This proves that $a$ is the product of algebraic integers, thus it is an algebraic integer as well.
But we know already that $a\in\mathbb Q$ (if $\tau$ is not equivalent to $i$). The rational root theorem tells us that $a$ must be \emph{integral}.}%
\de{$a$ ist also das Produkt ganzalgebraischer Zahlen und somit selbst ganz"-algebraisch.
Weil wir bereits wissen, dass $a\in \mathbb Q$ ist (falls $\tau$ nicht äquivalent zu $i$ ist), folgt aus dem Satz über rationale Nullstellen, dass $a$ sogar eine \emph{ganze} Zahl ist.}\qedhere\\
\en{\emph{Remark:} For $\tau_4=i$ we have $E_2^*(i)=0$ and we could define $a_4=0$ using Eq.~(\ref{\en{EN}produ}).}%
\de{\emph{Bemerkung:} Für $\tau_4=i$ gilt $E_2^*(i)=0$ und wir könnten mit Gl.~(\ref{\en{EN}produ}) $a_4=0$ definieren.}
\end{proof}

\begin{table}[ht]\centering\vspace*{-4pt}
\scalebox{0.88}{{\renewcommand{\arraystretch}{2.2}%
\begin{tabular}{c|c|>{\columncolor[gray]{0.9}}c|c|c|c|>{\columncolor[gray]{0.9}}c}
\multicolumn{3}{c|}{\en{Equation, Discriminant and Solution}\de{Gleichung, Diskriminante und Lösung}} & \multicolumn{3}{c|}{\en{Intermediate values from}\de{Hilfsgrößen aus} Thm.~\ref{\en{EN}s2nenner}} & \multicolumn{1}{c}{\en{Result}\de{Ergebnis}}\\
\scalebox{.9}{$C\tau^2+B\tau+A=0$} & $-N$ & $\tau_N$  & $c_N$  & $b_N$ & $a_N$ & $s_2(\tau_N)$\\
    \hline
     $\tau^2+2=0$ & $-8$ & $\sqrt{2}i$         & $-1$ & $896$ & $320$ & $\dsf{5}{14}$ \\
     $\tau^2+3=0$ & $-12$ & $\sqrt{3}i$   & $-1$ & $7128$ & $3240$ & $\dsf{5}{11}$ \\
     $\tau^2+4=0$ & $-16$ & $\sqrt{4}i$          & $-2$  & $48384$ & $25344$ & $\dsf{11}{21}$ \\
     $\tau^2+7=0$ & $-28$ & $\sqrt{7}i$        & $-1$ & $1055754$ & $674730$ & $\dsf{85}{133}$ \\
    \hline
     $\tau^2-\tau+2=0$ & $-7$ & $\displaystyle\frac{1+i\sqrt{7}}{2}$         & $1$  & $756$ & $180$ & $\dsf{5}{21}$\\
     $\tau^2-\tau+3=0$ & $-11$ & $\displaystyle\frac{1+i\sqrt{11}}{2}$        & $1$  & $5544$ & $2304$ & $\dsf{32}{77}$ \\
     $\tau^2-\tau+5=0$ & $-19$ & $\displaystyle\frac{1+i\sqrt{19}}{2}$         & $1$ & $102600$ & $57600$ & $\dsf{32}{57}$\\
     $\tau^2-\tau+7=0$ & $-27$ & $\displaystyle\frac{1+i\sqrt{27}}{2}$ & $1$ & $892584$ & $564480$  &  $\dsf{160}{253}$\\
     $\tau^2-\tau+11=0$ & $-43$ & $\displaystyle\frac{1+i\sqrt{43}}{2}$       & $1$  & $23600808$ & $16727040$&  $\dsf{640}{903}$ \\
     $\tau^2-\tau+17=0$ & $-67$ & $\displaystyle\frac{1+i\sqrt{67}}{2}$       & $1$ & $907582536$ & $695819520$ &  $\displaystyle\frac{33440}{43617}$\\
     $\tau^2-\tau+41=0$ & $-163$ & $\displaystyle\frac{1+i\sqrt{163}}{2}$    & $1$ & $10996566783048$ & $9351571368960$ & $\displaystyle\frac{77265280}{90856689}$\\[1ex]
\end{tabular}}
}\vspace{6pt}
\caption{\en{Calculation of $s_2(\tau)$ at some $\text{CM}_1$-points}\de{Berechnung von $s_2(\tau)$ bei einigen $\text{CM}_1$-Punkten}}\label{\en{EN}tabel7}\vspace*{-12pt}
\end{table}

\begin{thm}\label{\en{EN}satzhilfszahlen}
    \en{All values of $s_2(\tau)$ given in Tab.~\ref{\en{EN}tabel7} are exact and correct.}%
    \de{Alle Werte von $s_2(\tau)$ in Tab.~\ref{\en{EN}tabel7} sind exakt und korrekt.}
\end{thm}
\begin{proof}

\begin{itemize}[leftmargin=*]
    \item \en{In the first three columns of Tab.~\ref{\en{EN}tabel7} we can find some $\text{CM}_1$-points $\tau$ from Prop.~\ref{\en{EN}satztaun} and their quadratic equation and discriminant.}%
    \de{Zunächst finden wir in Tab.~\ref{\en{EN}tabel7} einige $\text{CM}_1$-Punkte $\tau$ aus Satz~\ref{\en{EN}satztaun} und die zugehörige quadratische Gleichung mit ihrer Diskriminante.}
    \pagebreak
    \item \en{To calculate $b := \sqrt{c\cdot D\cdot(j(\tau)-1728)}\cdot(AC)^2$, we need the values of $j(\tau)$ from Tab.~\ref{\en{EN}tabJwerte} and we have to choose a suitable $c\in\mathbb Z$. Our choice of $c$ and the resulting values of $b$ are in Tab.~\ref{\en{EN}tabel7}.}%
    \de{Um $b := \sqrt{c\cdot D\cdot(j(\tau)-1728)}\cdot(AC)^2$ zu berechnen, benötigen wir die Werte von $j(\tau)$ aus Tab.~\ref{\en{EN}tabJwerte} und müssen uns dann jeweils für ein passendes $c\in\mathbb Z$ entscheiden. Unsere Wahl von $c$ und die daraus folgenden Werte von $b$ sind in Tab.~\ref{\en{EN}tabel7} zu finden.}
    \item \en{To calculate $a:=s_2(\tau)\cdot b$, we use the approximation from Thm.~\ref{\en{EN}theonaehers2}:}%
    \de{Um jetzt $a:=s_2(\tau)\cdot b$ zu berechnen, nutzen wir die Näherung aus Thm.~\ref{\en{EN}theonaehers2}:}
    $$\tilde s_2(\tau):=\frac{1+240 (q+9q^2)}{1 - 504 (q+33q^2)}\cdot\left(1 - 24  (q+3q^2) -\frac{3}{\pi \im(\tau)}\right).$$
    \en{Since $q=e^{2\pi i \tau_N}=(-1)^N\cdot e^{-\pi\sqrt{N}}$, we need again 25 digits of $\pi$.
    Using the approximations $\tilde s_2(\tau_N)$ we obtain $\tilde a_N:=\tilde s_2(\tau_N)\cdot b_N$ as an approximation for $a_N$:}%
    \de{Wegen $q=e^{2\pi i \tau_N}=(-1)^N\cdot e^{-\pi\sqrt{N}}$ benötigen wir hier wieder 25 Stellen von $\pi$.
    Mit den Näherungen $\tilde s_2(\tau_N)$ berechnen wir dann $\tilde a_N:=\tilde s_2(\tau_N)\cdot b_N$ als Näherung für $a_N$:}
\begin{align*}
\tilde s_2(\tau_{7})  &\approx 0\ko23809~56479~14958~22417 &&\Longrightarrow& 
\tilde a_{7} &\approx&           180\ko00031\\[-0.2ex]
\tilde s_2(\tau_{8})  &\approx 0\ko35714~27261~48252~57875 &&\Longrightarrow& 
\tilde a_{8} &\approx&           319\ko99988\\[-0.2ex]
\tilde s_2(\tau_{11}) &\approx 0\ko41558~44169~95050~54414 &&\Longrightarrow& 
\tilde a_{11} &\approx&         2304\ko00001\\[-0.2ex]
\tilde s_2(\tau_{12}) &\approx 0\ko45454~54541~52238~44453 &&\Longrightarrow& 
\tilde a_{12} &\approx&         3239\ko00000\\[-0.2ex]
\tilde s_2(\tau_{16}) &\approx 0\ko52380~95238~06641~89452 &&\Longrightarrow& 
\tilde a_{16} &\approx&        25343\ko000\\[-0.2ex]
\tilde s_2(\tau_{19}) &\approx 0\ko56140~35087~72034~50431 &&\Longrightarrow& 
\tilde a_{19} &\approx&        57600\ko000\\[-0.2ex]
\tilde s_2(\tau_{27}) &\approx 0\ko63241~10671~93675~93347 &&\Longrightarrow& 
\tilde a_{27} &\approx&       564480\ko000\\[-0.2ex]
\tilde s_2(\tau_{28}) &\approx 0\ko63909~77443~60902~23748 &&\Longrightarrow& 
\tilde a_{28} &\approx&       674730\ko000\\[-0.2ex]
\tilde s_2(\tau_{43}) &\approx 0\ko70874~86157~25359~91141 &&\Longrightarrow& 
\tilde a_{43} &\approx&      16727040\ko000\\[-0.2ex]
\tilde s_2(\tau_{67}) &\approx 0\ko76667~35447~18802~30185 &&\Longrightarrow& 
\tilde a_{67} &\approx&     695819520\ko000\\[-0.2ex]
\tilde s_2(\tau_{163})&\approx 0\ko85040~82731~87238~86141 &&\Longrightarrow&
\tilde a_{163} &\approx& 9351571368960\ko000
\end{align*}
\en{Here we already recognize \emph{approximately} the values of $a_N$ from Tab.~\ref{\en{EN}tabel7}.}%
\de{Hier erkennen wir bereits \emph{ungefähr} die Werte der $a_N$ aus Tab.~\ref{\en{EN}tabel7}.}

\item \en{For $N\geq 7$ it holds $\im(\tau_N)=\sqrt{N}/2>1\ko25$ so we can use the error estimate of $\tilde s_2(\tau)$ from Thm.~\ref{\en{EN}theonaehers2} which yields:}%
\de{Für $N\geq 7$ gilt $\im(\tau_N)=\sqrt{N}/2>1\ko25$ und wir können die Fehlerabschätzung für $\tilde s_2(\tau)$ aus Thm.~\ref{\en{EN}theonaehers2} nutzen:}
\begin{align*}
    |\tilde a_N - a_N| &=|\tilde s_2(\tau_N)-s_2(\tau_N)|\cdot |b_N| \leq 222000\cdot |q|^3\cdot |b_N|
\end{align*}
\en{From the values of $b_N$ given in Tab.~\ref{\en{EN}tabel7} we observe that $|b_N|\leq e^{3\cdot\sqrt{N}}$ for all these $N$. Further we have $|q|=e^{-2\pi\im(\tau_N)}=e^{-\pi\sqrt{N}}$ and $\pi>3+\frac{10}{71}$ (Lemma~\ref{\en{EN}archim-lem}):}%
\de{An den bereits berechneten Werten der $b_N$ erkennen wir $|b_N|\leq e^{3\cdot\sqrt{N}}$ für alle diese $N$. Außerdem gilt $|q|=e^{-2\pi\im(\tau_N)}=e^{-\pi\sqrt{N}}$ und $\pi>3+\frac{10}{71}$ (Lemma~\ref{\en{EN}archim-lem}):}
\begin{align*}
    |\tilde a_N - a_N|&\leq 222000\cdot e^{-3\pi\sqrt{N}}\cdot e^{3\cdot\sqrt{N}}
=222000\cdot e^{-3(\pi-1)\sqrt{N}}\\
&\leq 222000\cdot e^{-3\cdot\left(2+\frac{10}{71}\right)\cdot\sqrt{7}} \leq 0\ko01
\end{align*}

\item
\en{In Prop.~\ref{\en{EN}s2nenner} we have proved that the $a_N$ are integral. Thus the values of $a_N$ given in Tab.~\ref{\en{EN}tabel7} are \emph{exact}, because they are the only integers close enough to the approximations.}%
\de{Aus der Ganzzahligkeit der $a_N$ in Thm.~\ref{\en{EN}s2nenner} folgt, dass $a_N$ \emph{exakt} die in Tab.~\ref{\en{EN}tabel7} angegebenen Werte hat, weil das die einzigen ganzen Zahlen sind, die hinreichend nahe an $\tilde a_N$ liegen.}
\en{The values of $s_2(\tau)$ now follow by reducing the fraction $s_2(\tau_N)=a_N/b_N$.}%
\de{Die Werte von $s_2(\tau)$ folgen schließlich aus $s_2(\tau_N)=a_N/b_N$.}\qedhere
\end{itemize}
\end{proof}

\begin{theo}\label{\en{EN}theohud}
\en{The "Chudnovsky formula" for calculating $\pi$ applies:}%
\de{Es gilt die \glqq Chudnovsky-Formel\grqq~zur Berechnung von $\pi$:}
\begin{align*}
    \frac{\sqrt{640320^3}}{12 \cdot \pi} &= \sum_{n=0}^\infty\frac{\left(6n\right)!}{\left(3n\right)!\left(n!\right)^3}\cdot\frac{13591409 + 545140134\cdot n}{\left(-640320^3\right)^n}
\end{align*}
\en{It was published in 1988 by David and Gregroy Chudnovsky (see \cite[\en{eq.}\de{Gl.} (1.5)]{\en{EN}chud1988}).}%
\de{Sie wurde 1988 von den Chudnovsky-Brüdern veröffentlicht (siehe \cite[\en{eq.}\de{Gl.} (1.5)]{\en{EN}chud1988}).}
\end{theo}
\begin{proof}
\en{We use $\tau_{163}=\frac{1+i\sqrt{163}}{2}$ in the Main Theorem~\ref{\en{EN}hauptformel}. For this we use the values $j(\tau_{163})=1728J(\tau_{163})$ from Tab.~\ref{\en{EN}tabJwerte} and $s_2(\tau_{163})$ from Tab.~\ref{\en{EN}tabel7}:}
\de{Wir setzen $\tau_{163}=\frac{1+i\sqrt{163}}{2}$ ins Haupttheorem~\ref{\en{EN}hauptformel} ein und verwenden die Werte von $j(\tau_{163})=1728J(\tau_{163})$ aus Tab.~\ref{\en{EN}tabJwerte} und die von $s_2(\tau_{163})$ aus Tab.~\ref{\en{EN}tabel7}:}
\begin{align*}
    \frac{1}{2\pi \im(\tau)}\sqrt{\frac{J(\tau)}{J(\tau)-1}} &= \sum_{n=0}^\infty \left( \frac{1-s_2(\tau)}{6} + n \right)\cdot \frac{(6n)!}{(3n)!(n!)^3}\cdot \frac{1}{\left(1728J(\tau)\right)^n}\\
    \frac{1}{\pi\sqrt{163}}\cdot\sqrt{\frac{-1728J(\tau_{163})}{1728-1728J(\tau_{163})}} &=\sum_{n=0}^\infty\frac{(6n)!}{(3n)!(n!)^3}\cdot \frac{\left(1-s_2(\tau_{163})\right)/6 + n}{\left(1728J(\tau_{163})\right)^n}\\
    \frac{1}{\pi\sqrt{163}}\cdot\sqrt{\frac{640320^3}{1728+640320^3}} &=\sum_{n=0}^\infty\frac{(6n)!}{(3n)!(n!)^3}\cdot \frac{\left(1-\frac{77265280}{90856689}\right)/{6} + n}{\left(-640320^3\right)^n}\\
    \frac{\sqrt{640320^3}}{12\cdot\pi\cdot 545140134} &= 
    \sum_{n=0}^\infty\frac{(6n)!}{(3n)!(n!)^3}\cdot \frac{\frac{13591409}{545140134} + n}{\left(-640320^3\right)^n}
\end{align*}
\en{By multiplying this equation with $545140134$, we obtain the Chudnovsky formula.}%
\de{Eine Multiplikation mit $545140134$ liefert nun die Chudnovsky-Formel.}
\end{proof}

\begin{theo}\label{\en{EN}14dezi}
\en{If we use only the first $N$ terms of the Chudnovskys' series:}%
\de{Wenn man in der Chudnovsky-Formel nur die ersten $N$ Summanden}
$$\pi_N=\frac{\sqrt{640320^3}}{12}\,\left(\sum_{n=0}^{N-1} s_n\right)^{\hspace*{-3pt}-1\hspace*{3pt}} \text{\en{with}\de{mit}}\quad s_n=\frac{\left(6n\right)!}{\left(3n\right)!\left(n!\right)^3}\,\frac{13591409 + 545140134 n}{\left(-640320^3\right)^n}$$
\en{then for all $N\geq 1$ it holds}%
\de{verwendet, gilt für alle $N\geq 1$ die Fehlerabschätzung}
$$\left|\pi_N-\pi\right|<11\ko315\cdot \frac{53360^{-3N}}{\sqrt{N}}.$$
\en{For $N\geq 129$ terms, this yields the weaker estimate $\left|\pi_N-\pi\right|<53360^{-3N}$ and we obtain on average $\log_{10}\lk53360^{3}\rk\approx14\ko1816$ decimal digits of $\pi$ per iteration.}%
\de{Für $N\geq 129$ liefert das die etwas schwächere Abschätzung
$\left|\pi_N-\pi\right|<53360^{-3N}$
und man erhält durchschnittlich $\log_{10}\lk53360^{3}\rk\approx14\ko1816$ Dezimalen pro Summand.}
\end{theo}
\begin{proof}
\en{We use Stirling's approximation (see \cite{\en{EN}Robbins}):}%
\de{Wir verwenden zuerst die Stirling'sche Näherung (siehe z.B. \cite{\en{EN}Robbins}):}
$$n!=\sqrt{2\pi n}\cdot\left(\frac{n}{e}\right)^n\cdot e^{r_n}\qquad\text{\en{with}\de{mit}}\qquad \frac{1}{12n+1}<r_n<\frac{1}{12n}.$$
\en{This yields}\de{Daraus folgt nämlich}
\begin{align*}
    \frac{(6n)!}{(3n)!(n!)^3} &= \frac{\sqrt{2\pi\cdot 6n}\cdot\left(\frac{6n}{e}\right)^{6n}}{\sqrt{2\pi\cdot 3n}\cdot\left(\frac{3n}{e}\right)^{3n}\cdot\left(\sqrt{2\pi\cdot n}\cdot\left(\frac{n}{e}\right)^{n}\right)^3}\cdot\frac{e^{r_{6n}}}{e^{r_{3n}}\cdot e^{3\cdot r_{n}}}\\
    &= \frac{\sqrt{2}\cdot\left(6^6/3^3\right)^n}{\left(\sqrt{2\pi n}\right)^{3}}\cdot e^{r_{6n}-r_{3n}-3r_{n}}= \frac{1728^n}{2(\pi n)^{3/2}}\cdot e^{-(3r_{n}+r_{3n}-r_{6n})}
\end{align*}
\en{with the following bound, valid for $n\geq 1$:}%
\de{mit der für $n\geq 1$ gültigen Abschätzung:}
\begin{align*}
    3r_{n}+r_{3n}-r_{6n} &> \frac{3}{12n+1}+\frac{1}{12\cdot 3n +1}-\frac{1}{12\cdot 6n}\\
    &> \frac{3}{12n+1}+\frac{1}{3\cdot(12n+1)}-\frac{1}{12\cdot 6n} = \frac{228 n - 1}{ 864 n^2 + 72 n} > \frac{13}{54 n}
\end{align*}
\en{The last step is equivalent to $(228n-1) 54n > 13(864 n^2 + 72 n)$ and  $90n(12n-11)>0$, which is valid for $n\geq 1$.}%
\de{Die letzte Abschätzung ist dabei äquivalent zu $(228n-1)\cdot 54n > 13\cdot(864 n^2 + 72 n)$ bzw. zu $90n\cdot(12n-11)>0$, was für $n\geq 1$ erfüllt ist.}
\en{Thus we have proven:}%
\de{Somit ist Folgendes bewiesen:}
$$\frac{(6n)!}{(3n)!(n!)^3}<\frac{12^{3n}}{2(\pi n)^{3/2}}\cdot e^{-\frac{13}{54 n}}.$$
\en{With $A=13591409$ and $B=545140134$ we obtain:}%
\de{Mit $A=13591409$ und $B=545140134$ erhalten wir weiter:}
\begin{align*}
    |s_n| &= \frac{\left(6n\right)!}{\left(3n\right)!\left(n!\right)^3}\cdot\frac{A + B n}{640320^{3n}}
    <\frac{12^{3n}}{2(\pi n)^{3/2}}\cdot e^{-\frac{13}{54 n}}\cdot\frac{B n\cdot\left(1+\frac{A}{Bn}\right)}{640320^{3n}}
\end{align*}
\en{Here we use $1+x\leq\exp(x)$ and $\frac{A}{Bn}<\frac{1}{40n}<\frac{13}{54 n}$:}%
\de{Hier nutzen wir $1+x\leq\exp(x)$ und $\frac{A}{Bn}<\frac{1}{40n}<\frac{13}{54 n}$ und erhalten}
\begin{align*}
    |s_n| &<\frac{1}{\sqrt{n}\cdot 53360^{3n}}\cdot\frac{B}{2\pi^{3/2}}\cdot e^{-\frac{13}{54 n}}\cdot e^{\frac{A}{Bn}}<\frac{1}{\sqrt{n}\cdot 53360^{3n}}\cdot\frac{B}{2\pi^{3/2}}
\end{align*}
\en{Now we denote $C=\frac{12}{\sqrt{640320^3}}$. Since $\frac{1}{\pi} = C \cdot \sum_{n=0}^\infty s_n$ is an alternating series where $|s_n|$ decreases monotonously to zero, the error of $\frac{1}{\pi_N} = C \cdot \sum_{n=0}^{N-1} s_n$ is smaller than the next term's absolute value:}%
\de{Nun schreiben wir $C=\frac{12}{\sqrt{640320^3}}$. Weil $\frac{1}{\pi} = C \cdot \sum_{n=0}^\infty s_n$ eine alternierende Reihe ist, bei der $|s_n|$ streng monoton gegen Null fällt, ist der Fehler von $\frac{1}{\pi_N} = C \cdot \sum_{n=0}^{N-1} s_n$ kleiner als der nächste Summand:}
\begin{align*}
    \left|\frac{1}{\pi}-\frac{1}{\pi_N}\right|=C\cdot\left|\sum_{k=0}^\infty s_k-\sum_{k=0}^{N-1} s_k\right|=C\cdot\left|\sum_{k=N}^\infty s_k\right|< C\cdot\left| s_{N} \right|.
\end{align*}
\en{With $\frac{1}{\pi}-\frac{1}{\pi_N}=\frac{\pi_N-\pi}{\pi~\pi_N}$ and $|\pi_N|<3\ko1416$ and $\pi>3\ko1415$ this yields:}%
\de{Mit $\frac{1}{\pi}-\frac{1}{\pi_N}=\frac{\pi_N-\pi}{\pi~\pi_N}$ und $|\pi_N|<3\ko1416$ und $\pi>3\ko1415$ folgt:}\belowdisplayskip=-12pt
\begin{align*}
\left|\pi_N-\pi\right| &< \left|C~\pi~\pi_N~s_{N} \right| <\frac{C~\pi~\pi_N~B}{2\pi^{3/2}}\cdot\frac{53360^{-3N}}{\sqrt{N}}<11\ko315\cdot \frac{53360^{-3N}}{\sqrt{N}}.
\end{align*}
\end{proof}

\begin{bem}
\en{In a future paper, we will prove the sharper estimate}%
\de{In einem weiteren Aufsatz beweisen wir die schärfere Abschätzung}
$$\left|\pi_N-\pi\right|=53360^{-3N}\cdot\frac{A_0}{\sqrt{N}}\cdot\exp\left(-\frac{A_1}{N}-\frac{A_2}{N^2}+\frac{\delta_N}{N^3}\right)$$
\en{with the error term $0.00690<\delta_N<0.00843$ and the coefficients}%
\de{mit dem Fehlerterm $0.00690<\delta_N<0.00843$ und den Koeffizienten}
\begin{align*}
    A_0&=\frac{106720\cdot\sqrt{10005\pi}}{1672209},\\
    A_1&=\frac{1781843197433}{7456754505816},\\
    A_2&=\frac{1080096011925710088395}{3475199235000451148614116}.
\end{align*}
\end{bem}

\pagebreak

\begin{theo}\label{\en{EN}PiFormeln}
\en{The following ten formulae for calculating $\pi$ apply:}%
\de{Es gelten auch die folgenden zehn Formeln zur Berechnung von $\pi$:}
{\begin{align*}
\frac{\sqrt{15^3}}{3 \cdot\pi} &=  \sum_{n=0}^\infty\frac{\left(6n\right)!}{\left(3n\right)!\left(n!\right)^3}\cdot\frac{8 + 63\cdot n}{\left(-15^3\right)^n} & \text{\cite[\en{eq.}\de{Gl.} (1.4)]{\en{EN}chud1988}} && \left(\tau_{7}\right)\\
\frac{\sqrt{20^3}}{8 \cdot\pi} &=  \sum_{n=0}^\infty\frac{\left(6n\right)!}{\left(3n\right)!\left(n!\right)^3}\cdot\frac{3 + 28\cdot n}{\left(20^3\right)^n} & \text{\cite[\en{p.}\de{S.} 187]{\en{EN}Borwein:1987:agm}} && \left(\tau_{8}\right)\\
\frac{\sqrt{32^3}}{4 \cdot\pi} &=  \sum_{n=0}^\infty\frac{\left(6n\right)!}{\left(3n\right)!\left(n!\right)^3}\cdot\frac{15 + 154\cdot n}{\left(-32^3\right)^n} & \text{\cite[\en{eq.}\de{Gl.} (1.4)]{\en{EN}chud1988}} && \left(\tau_{11}\right)\\
\frac{\sqrt{2\cdot 30^3}}{72 \cdot\pi} &=  \sum_{n=0}^\infty\frac{\left(6n\right)!}{\left(3n\right)!\left(n!\right)^3}\cdot\frac{1 + 11\cdot n}{\left(2\cdot 30^3\right)^n} & \text{\cite[\en{eq.}\de{Gl.} (33)]{\en{EN}rama1914}} && \left(\tau_{12}\right)\\
\frac{\sqrt{2}\cdot\sqrt{66^3}}{48 \cdot\pi} &=  \sum_{n=0}^\infty\frac{\left(6n\right)!}{\left(3n\right)!\left(n!\right)^3}\cdot\frac{5 + 63\cdot n}{\left(66^3\right)^n} & \text{\cite[\en{p.}\de{S.} 187]{\en{EN}Borwein:1987:agm}} && \left(\tau_{16}\right)\\
\frac{\sqrt{96^3}}{12 \cdot\pi} &=  \sum_{n=0}^\infty\frac{\left(6n\right)!}{\left(3n\right)!\left(n!\right)^3}\cdot\frac{25 + 342\cdot n}{\left(-96^3\right)^n} & \text{\cite[\en{eq.}\de{Gl.} (1.4)]{\en{EN}chud1988}} && \left(\tau_{19}\right)\\
\frac{\sqrt{3\cdot160^3}}{36 \cdot \pi} &= \sum_{n=0}^\infty\frac{\left(6n\right)!}{\left(3n\right)!\left(n!\right)^3}\cdot\frac{31 + 506\cdot n}{\left(-3\cdot160^3\right)^n} & \text{\cite[\en{p.}\de{S.} 371]{\en{EN}Borwein1988}} && \left(\tau_{27}\right)\\
\frac{\sqrt{255^3}}{ 162 \cdot\pi} &= \sum_{n=0}^\infty\frac{\left(6n\right)!}{\left(3n\right)!\left(n!\right)^3}\cdot\frac{8 + 133\cdot n}{\left(255^3\right)^n} & \text{\cite[\en{eq.}\de{Gl.} (34)]{\en{EN}rama1914}} && \left(\tau_{28}\right)\\
\frac{\sqrt{960^3}}{36 \cdot \pi} &= \sum_{n=0}^\infty\frac{\left(6n\right)!}{\left(3n\right)!\left(n!\right)^3}\cdot\frac{263 + 5418\cdot n}{\left(-960^3\right)^n} & \text{\cite[\en{eq.}\de{Gl.} (1.4)]{\en{EN}chud1988}} && \left(\tau_{43}\right)\\
\frac{\sqrt{5280^3}}{12 \cdot \pi} &= \sum_{n=0}^\infty\frac{\left(6n\right)!}{\left(3n\right)!\left(n!\right)^3}\cdot\frac{10177 + 261702\cdot n}{\left(-5280^3\right)^n} & \text{\cite[\en{eq.}\de{Gl.} (1.4)]{\en{EN}chud1988}} && \left(\tau_{67}\right)
\end{align*}}
\en{Here we also listed the first known publication by Ramanujan (1914: \cite{\en{EN}rama1914}), the Borwein brothers (1987: \cite{\en{EN}Borwein:1987:agm}; 1988: \cite{\en{EN}Borwein1988}) or the Chudnovsky brothers (1988: \cite{\en{EN}chud1988}); and the value we used for $\tau$.}%
\de{Dabei ist jeweils noch die Erstveröffentlichung durch Ramanujan (1914: \cite{\en{EN}rama1914}), die Bor"-wein-Brüder (1987: \cite{\en{EN}Borwein:1987:agm}; 1988: \cite{\en{EN}Borwein1988}) oder die Chudnovsky-Brüder (1988: \cite{\en{EN}chud1988}) angegeben; und der Wert, der für $\tau$ eingesetzt wurde.}
\end{theo}
\begin{proof}
\en{Use the values from Tab.~\ref{\en{EN}tabJwerte} and from Tab.~\ref{\en{EN}tabel7} in the Main Theorem~\ref{\en{EN}hauptformel}, as in the proof of Thm.~\ref{\en{EN}theohud}.}%
\de{Setze (genau wie im Beweis von Thm.~\ref{\en{EN}theohud}) die Werte aus Tab.~\ref{\en{EN}tabJwerte} und aus Tab.~\ref{\en{EN}tabel7} ins Haupttheorem~\ref{\en{EN}hauptformel} ein.}
\end{proof}

\vfill\pagebreak
{\appendix\section{\en{On the Division Values of the \texorpdfstring{$\wp$}{\emph{p}}-Function}\de{Über die Teilungswerte der \texorpdfstring{$\wp$}{\emph{p}}-Funktion}}\label{\en{EN}kapdivi}
\renewcommand{\leftmark}{\en{On the Division Values of the $\wp$-Function}\de{Über die Teilungswerte der $\wp$-Funktion}}
\en{In this chapter we prove that $m\cdot\wp(u;L)$ is an algebraic integer of $\mathbb Z\mathopen{}\left[\frac 1 4 g_2(L); \frac 1 4 g_3(L)\right]\mathclose{}$ for all positive integers $m$ and for all $u\in\mathbb C-L$ with $m\cdot u\in L$.}%
\de{In diesem Kapitel beweisen wir, dass $m\cdot\wp(u;L)$ für alle natürlichen Zahlen $m\neq 0$ und für alle $u\in\mathbb C-L$ mit $m\cdot u\in L$ ganzalgebraisch in $\mathbb Z\mathopen{}\left[\frac 1 4 g_2(L); \frac 1 4 g_3(L)\right]\mathclose{}$ ist.}

\en{This appendix elaborates \cite[pp.~184-185]{\en{EN}Fricke2}.}%
\de{Dieser Anhang arbeitet \cite[pp.~184-185]{\en{EN}Fricke2} aus.}

\begin{defi}\label{\en{EN}defidivi}
\en{Given a complex number $m\neq 0$ and a lattice $L$.
Then $u$ is called a $m$-division-point and $\wp(u)$ is called a $m$-division value iff it holds:}%
\de{Gegeben sind eine komplexe Zahl $m\neq 0$ und ein Gitter $L$.
Dann nennen wir $u$ eine $m$-Teilungsstelle und $\wp(u)$ einen $m$-Teilungswert genau dann, wenn:}
$$m\cdot u\in L,\quad\text{ \en{but}\de{aber} }\quad u\notin L.$$
\en{We denote the following set $\DIV(m)$, which contains all $m$-division-points in the fundamental parallelogram $\mathcal P$ from Def.~\ref{\en{EN}fund}:}%
\de{Wir bezeichnen die folgende Menge mit $\DIV(m)$, sie enthält alle $m$-Teilungsstellen im Fundamentalparallelogramm $\mathcal P$ aus Def.~\ref{\en{EN}fund}:}
\begin{align*}
\DIV(m)&=\left\{~u\in\mathcal P~\left|~m\cdot u\in L,\text{ \en{but }\de{aber } }u\notin L\right.\right\}\\
\text{\en{where}\de{wobei}}\quad \mathcal P&=\left\{~a\omega_1+b\omega_2\in\mathbb C~\left|~0\leq a,b < 1\right.\right\}
\end{align*}
\en{Throughout Appendix~\ref{\en{EN}kapdivi}, $m$ will even be a positive integer.

\noindent Fig.~\ref{\en{EN}abbwege} on p.~\pageref{\en{EN}abbwege} shows $\overline{\mathcal P}$ and those points in $\overline{\mathcal P}$ which are equivalent to $u\in \DIV(2)$.}%
\de{Im gesamten Anhang~\ref{\en{EN}kapdivi} wird $m$ sogar positiv ganzzahlig sein.

\noindent Abb.~\ref{\en{EN}abbwege} auf S.~\pageref{\en{EN}abbwege} zeigt $\overline{\mathcal P}$ und die Stellen in $\overline{\mathcal P}$, die zu $u\in \DIV(2)$ äquivalent sind.}%
\end{defi}

\begin{lem}\label{\en{EN}lemmm}
\en{If $m\neq 0$ is integral, then the number of $m$-division-points in $\mathcal P$ is $m^2-1$.}%
\de{Wenn $m\neq 0$ ganzzahlig ist, dann gibt es $m^2-1$ $m$-Teilungs-Punkte in $\mathcal P$.}
\end{lem}
\begin{proof}
\en{$u\in \DIV(m)$ is equivalent to $m\cdot u = k\omega_1+l\omega_2$ with integral $k,l$ such that $u=\frac{k}{m}\cdot \omega_1+\frac{l}{m}\cdot\omega_2\in \mathcal P-\{0\}$.
So every pair $(k,l)\in\mathbb Z^2-\{(0,0)\}$ with $0\leq k,l < m$ yields one $u$. Since there are $m^2-1$ such pairs, the Lemma is proven.}%
\de{$u\in \DIV(m)$ ist äquivalent zu $m\cdot u = k\omega_1+l\omega_2$ mit ganzzahligen $k,l$ so dass $u=\frac{k}{m}\cdot \omega_1+\frac{l}{m}\cdot\omega_2\in \mathcal P-\{0\}$.
Also liefert jedes Paar $(k,l)\in\mathbb Z^2-\{(0,0)\}$ mit $0\leq k,l < m$ eine Stelle $u$. Weil es $m^2-1$ solche Paare gibt, ist das Lemma bewiesen.}
\end{proof}

\begin{lem}\label{\en{EN}lemdefifm}
\en{The following quotient of Weierstraß $\sigma$-functions (cf.~Def.~\ref{\en{EN}defisigma}) is an elliptic function with periods $\omega_{1;2}$:}%
\de{Der folgende Quotient Weierstraß'scher $\sigma$-Funktionen (Def.~\ref{\en{EN}defisigma}) ist eine elliptische Funktion mit Perioden $\omega_{1;2}$:}
$$F_m(z):=\frac{\sigma(m\cdot z)}{(\sigma(z))^{m\cdot m}}$$
\en{It has a pole of order $m^2-1$ at $z\in L$ (e.g.~at $z=0$), and $m^2-1$ zeros of order one at $z\in \DIV(m)$.
Modulo $L$, $F_m$ has no further poles or zeros.}%
\de{Sie hat bei $z\in L$ (z.B.~bei $z=0$) einen Pol der Ordnung $m^2-1$ und $m^2-1$ einfache Nullstellen bei $z\in \DIV(m)$.
Modulo $L$ hat $F_m$ keine weiteren Null- oder Polstellen.}
\end{lem}
\begin{proof}
\en{In Prop.~\ref{\en{EN}trafosigma} we proved}%
\de{Satz~\ref{\en{EN}trafosigma} besagt}
$$\sigma(z+\omega_k)=-\exp\lk\eta_k\cdot\left(z+\frac{\omega_k}{2}\right)\rk\cdot \sigma(z)$$
\en{From this we will prove by induction that}%
\de{Hieraus werden wir per vollständiger Induktion beweisen:}
\begin{align}
    \sigma(z+n\omega_k) &= (-1)^n \cdot \exp\lk n\cdot\eta_k\left(z+n\cdot\frac{\omega_k}{2}\right)\rk \cdot ~\sigma(z)\label{\en{EN}glgsigmatr}
\end{align}
\en{For $n=0$ we have $\sigma(z)=\sigma(z)$. Then we do the step $n-1\rightarrow n$, where we first use Prop.~\ref{\en{EN}trafosigma}:}%
\de{Zum Induktionsanfang $n=0$: Hier besagt Glg.~(\ref{\en{EN}glgsigmatr}) nur, dass $\sigma(z)=\sigma(z)$ ist.\\
Im Induktionsschritt $n-1\rightarrow n$ nutzen wir zunächst Satz~\ref{\en{EN}trafosigma}:}
\begin{align*}
    \sigma(z+n\omega_k) &= \sigma((z+(n-1)\omega_k)+~\omega_k)\\
    &= \underbrace{- \exp\lk \eta_k\left(z+(n-1)\omega_k+\frac{\omega_k}{2}\right)\rk}_{P} \cdot ~\sigma(z+(n-1)\omega_k)
\end{align*}
\en{Then we use the induction basis:}%
\de{Dann benutzen wir die Induktionsvoraussetzung:}
\begin{align*}
\sigma(z+(n-1)\omega_k) = \underbrace{(-1)^{n-1} \cdot \exp\lk (n-1)\cdot\eta_k\left(z+(n-1)\cdot\frac{\omega_k}{2}\right)\rk}_Q \cdot~\sigma(z)    
\end{align*}
\en{These factors can be combined to:}%
\de{Dann fassen wir diese Faktoren zusammen:}
\begin{align*}
  P\cdot Q &= (-1)^n \cdot \exp\lk \eta_k \cdot\left( z+(n-1)\omega_k+\frac{\omega_k}{2} + (n-1)\cdot\left(z+(n-1)\cdot\frac{\omega_k}{2}\right)   \right)\rk\\
  &= (-1)^n \cdot \exp\lk \eta_k \cdot\left(n\cdot z+\omega_k\cdot\left(n-1 + \frac 1 2 + \frac{(n-1)^2}{2}\right)   \right)\rk\\
  &= (-1)^n \cdot \exp\lk \eta_k \cdot\left(n\cdot z+\frac{\omega_k}{2}\cdot n^2   \right)\rk
\end{align*}
\en{This proves~(\ref{\en{EN}glgsigmatr}).
Now we prove the periodicity of $F_m$ using~(\ref{\en{EN}glgsigmatr}):}%
\de{Somit ist~(\ref{\en{EN}glgsigmatr}) bewiesen.
Dies benutzen wir nun, um die Perioden von $F_m$ zu beweisen:}
\begin{align*}
    F_m(z+\omega_k) &= \frac{\sigma(m\cdot z+m\omega_k)}{(\sigma(z+\omega_k))^{m\cdot m}}
    = \frac{(-1)^m \cdot \exp\lk m\cdot\eta_k\left(m\cdot z+m\cdot\frac{\omega_k}{2}\right)\rk \cdot ~\sigma(m\cdot z)}{\left(-\exp\lk\eta_k\cdot\left(z+\frac{\omega_k}{2}\right)\rk\cdot \sigma(z)\right)^{m\cdot m}}
\end{align*}
\en{But for all integral $m$ we have $(-1)^m=(-1)^{m\cdot m}$ and}%
\de{Aber für alle ganzzahligen $m$ gilt $(-1)^m=(-1)^{m\cdot m}$ und}
$$\exp\lk m\cdot\eta_k\left(m\cdot z+m\cdot\frac{\omega_k}{2}\right)\rk = \left(\exp\lk \eta_k\left(z+\frac{\omega_k}{2}\right)\rk\right)^{m\cdot m},$$
\en{thus we have proven $F_m(z+\omega_k)=F_m(z)$.
Since the $\sigma$-function has zeros of order one in all points of the lattice (cf.~Def.~\ref{\en{EN}defisigma}), the nominator is zero iff $mz\in L$ and the denominator is zero iff $z\in L$.
This proves that $F_m(z)$ has zeros of order one for all $z\in \DIV(m)$ and that $F_m$ has a pole of order $m^2-1$ at $z\in L$.}%
\de{somit haben wir $F_m(z+\omega_k)=F_m(z)$ bewiesen.
Da die $\sigma$-Funktion in allen Gitterpunkten einfache Nullstellen hat (siehe Def.~\ref{\en{EN}defisigma}), liefern genau diejenigen $z$ mit $mz\in L$ Nullstellen des Zählers und die mit $z\in L$ Nullstellen des Nenners.
Somit hat $F_m(z)$ einfache Nullstellen bei $z\in \DIV(m)$ und einen Pol der Ordnung $m^2-1$ bei $z\in L$.}
\end{proof}

\begin{lem}\label{\en{EN}lemnullst}
\en{For any positive integer $m$, the function}%
\de{Für alle natürlichen $m\neq 0$ hat die Funktion}
$$h_m(z) := m^2\cdot\prod_{u\in \DIV(m)}\left(\wp(z)-\wp(u)\right)$$
\en{has zeros of order two for all $z\in \DIV(m)$ and no further zeros modulo $L$.}%
\de{doppelte Nullstellen für alle $z\in \DIV(m)$ und keine weiteren Nullstellen modulo $L$.}%
\end{lem}
\begin{proof}
\en{If $z\in \DIV(m)$ and $z\notin \DIV(2)$, we deduce $2z\not\equiv 0$ and $z\not\equiv -z$.
Prop.~\ref{\en{EN}pgerade} tells $\wp(-z)=\wp(z)$, thus both the factor with $u\equiv z$ and the factor with $u\equiv -z$ vanish and we get a zero of order two for these $z$.}%
\de{Falls $z\in \DIV(m)$ und $z\notin \DIV(2)$ gilt, folgt $2z\not\equiv 0$ und $z\not\equiv -z$.
Dann besagt Satz~\ref{\en{EN}pgerade}, dass $\wp(-z)=\wp(z)$ gilt und folglich sowohl der Faktor mit $u\equiv z$ als auch der mit $u\equiv -z$ verschwindet, weshalb wir bei solchen $z$ eine \emph{doppelte} Nullstelle haben.}

\en{If there are $z\in \DIV(m)\cap \DIV(2)$, we combine these three factors like in Prop.~\ref{\en{EN}pstrichprod}:
$\left(\wp(z)-e_1\right) \cdot\left(\wp(z)-e_2\right) \cdot\left(\wp(z)-e_3\right)=\frac 1 4 \wp'(z)^2$.
From the zeros of $\wp'$ in Prop.~\ref{\en{EN}zerowp} we see that $h_m$ has zeros of order two for these $z$ as well.}%
\de{Falls noch $z\in \DIV(m)\cap \DIV(2)$ verbleiben, fassen wir diese drei Faktoren wie in Satz~\ref{\en{EN}pstrichprod} zusammen:
$\left(\wp(z)-e_1\right) \cdot\left(\wp(z)-e_2\right) \cdot\left(\wp(z)-e_3\right)=\frac 1 4 \wp'(z)^2$.
Aus den Nullstellen von $\wp'$ in Satz~\ref{\en{EN}zerowp} folgt, dass $h_m$ auch bei diesen $z$ doppelte Nullstellen hat.}

\en{Since there are $m^2-1$ factors in $h_m(z)$ (Lemma~\ref{\en{EN}lemmm}) and $\wp(z)$ has the order two (Def.~\ref{\en{EN}defiwp}), $h_m$ has the order $2\cdot(m^2-1)$.
Thus (by the third Liouville theorem, Prop.~\ref{\en{EN}liouville3}) $h_m$ has no further zeros modulo $L$.}%
\de{Da $m^2-1$ Faktoren in $h_m(z)$ sind (Lemma~\ref{\en{EN}lemmm}) und $\wp(z)$ die Ordnung zwei hat (Def.~\ref{\en{EN}defiwp}), hat $h_m$ die Ordnung $2\cdot(m^2-1)$.
Aus dem dritten Liouville'schen Satz~\ref{\en{EN}liouville3} folgt daher, dass $h_m$ modulo $L$ keine weiteren Nullstellen hat.}
\end{proof}

\begin{thm}\label{\en{EN}propfmprod}
\en{For the $F_m$ and $h_m$ from Lemma~\ref{\en{EN}lemdefifm} and \ref{\en{EN}lemnullst}, it holds $F_m^2(z)=h_m(z)$.}%
\de{Für die $F_m$ aus Lemma~\ref{\en{EN}lemdefifm} und die $h_m$ aus Lemma~\ref{\en{EN}lemnullst} gilt $F_m^2(z)=h_m(z)$.}
\end{thm}
\begin{proof}
\en{Both $F_m^2(z)$ and $h_m(z)$ are elliptic functions with no poles outside of $L$.
From $\sigma(z)\approx z$ (Def.~\ref{\en{EN}defisigma}) and from Lemma~\ref{\en{EN}lemdefifm} we see that the Laurent series of $F_m(z)^2$ starts with $m^2\cdot z^{-2(m^2-1)}$. The same is true for $h_m(z)$ (cf.~Lemma~\ref{\en{EN}lemmm} and $\wp(z)\approx 1/z^2$ from Prop.~\ref{\en{EN}laurentwp}), thus the quotient $q_m:=h_m/F_m^2$ takes the value $1$ at $z=0$, which proves that the quotient $q_m:=h_m/F_m^2$ has no pole at $z\in L$.}%
\de{Sowohl $F_m^2(z)$ als auch $h_m(z)$ sind elliptische Funktionen, die keine Pole außerhalb von $L$ haben.
Aus $\sigma(z)\approx z$ (Def.~\ref{\en{EN}defisigma}) und aus Lemma~\ref{\en{EN}lemdefifm} folgt, dass die Laurentreihe von $F_m(z)^2$ mit $m^2\cdot z^{-2(m^2-1)}$ beginnt. $h_m(z)$ beginnt ebenso (wegen $\wp(z)\approx 1/z^2$ aus Satz~\ref{\en{EN}laurentwp} und wegen Lemma~\ref{\en{EN}lemmm}), folglich hat der Quotient $q_m:=h_m/F_m^2$ bei $z=0$ den Wert $1$ (und keinen Pol bei $z\in L$).}

\en{The quotient $q_m$ could still have poles in the zeros of $F_m^2$, but since $F_m^2$ and $h_m$ have the same zeros (Lemma~\ref{\en{EN}lemdefifm} and~\ref{\en{EN}lemnullst}), $q_m$ is elliptic without poles. By the first Liouville theorem (Prop.~\ref{\en{EN}liouville1}), the quotient is constant.

We just proved $q_m(0)=1$, thus we obtain $q_m(z)=1$ and $F_m^2(z)=h_m(z)$.}%
\de{Der Quotient $q_m$ könnte noch Pole in den Nullstellen von $F_m^2$ haben, aber da $F_m^2$ und $h_m$ die gleichen Nullstellen haben (Lemma~\ref{\en{EN}lemdefifm} und~\ref{\en{EN}lemnullst}), ist $q_m$ eine elliptische Funktion ohne Pole. Aufgrund des ersten Liouville'schen Satzes~\ref{\en{EN}liouville1} ist der Quotient konstant.

Soeben haben wir $q_m(0)=1$ bewiesen, also gilt $q_m(z)=1$ und $F_m^2(z)=h_m(z)$.}
\end{proof}

\begin{bem}
\en{Prop.~\ref{\en{EN}propfmprod} gives $F_m^2(z)$ as a polynomial in $\wp(z)$. In the rest of this chapter, we will construct this polynomial recursively and use this recursion to prove that $m\cdot\wp(u)$ is an algebraic integer of~\,$\mathbb Z\mathopen{}\left[\frac 1 4 g_2(L);\frac 1 4 g_3(L)\right]\mathclose{}$.}%
\de{Satz~\ref{\en{EN}propfmprod} stellt $F_m^2(z)$ als Polynom in $\wp(z)$ dar. Im Rest dieses Kapitels werden wir dieses Polynom rekursiv konstruieren und diese Rekursion nutzen, um zu beweisen, dass $m\cdot\wp(u)$ für $u\in \DIV(m)$ ganzalgebraisch in $\mathbb Z\mathopen{}\left[\frac 1 4 g_2(L);\frac 1 4 g_3(L)\right]\mathclose{}$ ist.}
\end{bem}

\begin{lem}\label{\en{EN}lemwpsigma}
\en{For all $u,v\notin L$ it holds}%
\de{Für alle $u,v\notin L$ gilt}
$$\wp(v) - \wp (u) = \frac{\sigma(u+v)\sigma(u-v)}{\sigma^2(u)\sigma^2(v)}$$
\end{lem}
\begin{proof}
\en{We define the function}%
\de{Wir definieren die Funktion}
$$F(u,v):=\frac{\sigma(u+v)\sigma(u-v)}{\sigma^2(u)\sigma^2(v)}+\wp(u)-\wp(v)$$
\en{First we fix $v\notin L$ and analyze $F(u,v)$ as a function $g(u)$ which has no poles outside $L$. Around $u=0$, we use $\sigma(u)\approx u$ (cf.~Def.~\ref{\en{EN}defisigma}) and obtain:}%
\de{Zunächst fixieren wir $v\notin L$ und betrachten $F(u,v)$ als Funktion $g(u)$, welche keine Pole außerhalb $L$ hat. Um $u=0$ nutzen wir $\sigma(u)\approx u$ (vgl.~Def.~\ref{\en{EN}defisigma}) und erhalten:}
$$g(u) \approx\frac{\sigma(v)\sigma(-v)}{u^2\sigma^2(v)}+\frac{1}{u^2}-\wp(v)$$
\en{But since $\sigma(-v)=-\sigma(v)$ (cf.~Prop.~\ref{\en{EN}sigmaodd}), the two terms $\pm u^{-2}$ in the Laurent series of $g(u)$ cancel each other out.
And since $g(-u)=g(u)$, there can't be a pole of order one at $u=0$, which proves that $g(u)$ has no poles.

Prop.~\ref{\en{EN}trafosigma} shows that $g(u)$ is elliptic (details on this can be found in the calculation of eq.~(\ref{\en{EN}eqellip}) on p.~\pageref{\en{EN}eqellip}), thus $g(u)=F(u,v)$ is constant with respect to $u$.
In the same way we can prove that $F(u,v)$ is constant with respect to $v$, thus $F(u,v)$ is constant.
From Def.~\ref{\en{EN}defisigma} we get $\sigma(0)=0$ and thus $F(v,v)=0$, so $F(u,v)=0$ for all $u,v\notin L$.}%
\de{Satz~\ref{\en{EN}sigmaodd} besagt, dass $\sigma(-v)=-\sigma(v)$ gilt, also heben sich die beiden Terme $\pm u^{-2}$ in der Laurentreihe von $g(u)$ gegenseitig auf.
Weil weiter $g(-u)=g(u)$ gilt, kann es keinen Pol der Ordnung $1$ bei $u=0$ geben. Somit ist bewiesen, dass $g(u)$ keine Pole hat.

Aus Satz~\ref{\en{EN}trafosigma} folgt, dass $g(u)$ elliptisch ist (siehe hierzu auch die Rechnung zu Glg.~(\ref{\en{EN}eqellip}) auf S.~\pageref{\en{EN}eqellip}), also ist $g(u)=F(u,v)$ konstant bezüglich $u$.
Gleichermaßen können wir beweisen, dass $F(u,v)$ konstant bezüglich $v$ ist, also ist $F(u,v)$ konstant.
Aus Def.~\ref{\en{EN}defisigma} folgt $\sigma(0)=0$ und somit $F(v,v)=0$, also ist $F(u,v)=0$ für alle $u,v\notin L$.}
\end{proof}

\begin{lem}\label{\en{EN}lemwpnz}
\en{Using the $F_m$ from Lemma~\ref{\en{EN}lemdefifm}, it holds}%
\de{Für die $F_m$ aus Lemma~\ref{\en{EN}lemdefifm} gilt}
$$\wp(nz)= \wp(z) - \frac{F_{n-1}(z)\cdot F_{n+1}(z)}{F_n(z)^2}$$
\end{lem}
\begin{proof}
\en{Using Lemma~\ref{\en{EN}lemwpsigma} with $u=z$ and $v=nz$ yields:}%
\de{Lemma~\ref{\en{EN}lemwpsigma} liefert mit $u=z$ und $v=nz$:}
$$\wp(nz) - \wp (z) = \frac{\sigma((n+1)z)\sigma(-(n-1)z)}{\sigma^2(z)\sigma^2(nz)}$$
\en{Next we use $\sigma(z)=-\sigma(-z)$ (Prop.~\ref{\en{EN}sigmaodd}) and the definition of $F_k$ from Lemma~\ref{\en{EN}lemdefifm} which yields $\sigma(kz)=F_k(z)\cdot \sigma(z)^{k\cdot k}$:}
\de{Dann nutzen wir $\sigma(z)=-\sigma(-z)$ (Satz~\ref{\en{EN}sigmaodd}) und die Definition der $F_k$ aus Lemma~\ref{\en{EN}lemdefifm} in der Form $\sigma(kz)=F_k(z)\cdot \sigma(z)^{k\cdot k}$:}
\begin{align*}
    \wp(nz) - \wp (z) &= - \frac{F_{n+1}(z)\sigma(z)^{(n+1)\cdot(n+1)}\cdot F_{n-1}(z)\sigma(z)^{(n-1)\cdot(n-1)}}{\sigma^2(z)\cdot F_n(z)^2\sigma(z)^{2\cdot n\cdot n}}
\end{align*}
\en{Reducing this fraction by $\sigma(z)^{2n^2+2}$ proves the Lemma.}%
\de{Indem wir diesen Bruch mit $\sigma(z)^{2n^2+2}$ kürzen, ist das Lemma bewiesen.}
\end{proof}

\begin{lem}\label{\en{EN}lemwp2str}
\en{For all $z\in\mathbb C$ it holds}%
\de{Für alle $z\in\mathbb C$ gilt}
$\wp''(z) = 6\wp(z)^2-\frac 1 2 g_2$
\en{and}\de{und}
$$\wp(2z) = \frac 1 4 \cdot \left(\frac{\wp''(z)}{\wp'(z)}\right)^2-2\wp(z)$$
\end{lem}
\begin{proof}
\en{The algebraic differential equation of the $\wp$-function (Prop.~\ref{\en{EN}dglP}) yields:}%
\de{Die algebraische Differentialgleichung der $\wp$-Funktion aus Satz~\ref{\en{EN}dglP} lautet:}
\begin{align*}
    \wp'(z)^2 &= 4\wp(z)^3-g_2\wp(z)-g_3\qquad\left|\frac{d}{dz}\right.\\
    \Longrightarrow\quad 2\wp'(z)\wp''(z) &= 12 \wp(z)^2\wp'(z) - g_2\wp'(z)
\end{align*}
\en{Dividing by $2\cdot\wp'(z)$ yields $\wp''(z)$.}%
\de{Indem wir dies durch $2\cdot\wp'(z)$ dividieren, erhalten wir $\wp''(z)$.}

\en{Next we recall the Laurent series of $\wp(z)$ from Prop.~\ref{\en{EN}laurentwp}:}%
\de{Als nächstes rufen wir uns die Laurentreihe von $\wp(z)$ aus Satz~\ref{\en{EN}laurentwp} ins Gedächtnis:}
\begin{align*}
    \wp(z) &= z^{-2} + 3 G_4 z^2  + 5 G_6 z^4 + 7 G_8 z^6  + O(z^8)\\
    \wp'(z) &= -2z^{-3} + 6 G_4 z + 20 G_6 z^3 + 42 G_8 z^5  + O(z^7)\\
    \wp''(z) &= 6z^{-4} + 6 G_4   + 60 G_6 z^2 + 210 G_8 z^4  + O(z^6)
\end{align*}
\en{If we put this into $\wp''(z) = 6\wp(z)^2-\frac 1 2 g_2$ and compare the coefficients of $z^4$, we obtain:}%
\de{Wenn wir diese Reihen in $\wp''(z) = 6\wp(z)^2-\frac 1 2 g_2$ einsetzen und die Koeffizienten vor $z^4$ vergleichen, erhalten wir:}
$$210 G_8 = 6\cdot\left(2\cdot 7 G_8 + 9 G_4^2\right)\qquad\Longrightarrow\qquad 7G_8 = 3 G_4^2$$
\en{To prove the second equation, compare the Laurent series expansions of}%
\de{Um die zweite Gleichung zu beweisen, vergleichen wir die Laurentreihen von}
$\left(\frac{\wp''(z)}{2}\right)^2$ \en{and}\de{und} $\wp'(z)^2\cdot\left(\wp(2z)+2\wp(z)\right)$ \en{using}\de{unter Nutzung von} $7G_8=3 G_4^2$:
\begin{align*}
    \wp(z) &= z^{-2} + 3 G_4 z^2  + 5 G_6 z^4 + 3 G_4^2 z^6  + O(z^8)\\
    \wp'(z) &= -2z^{-3} + 6 G_4 z + 20 G_6 z^3 + 18 G_4^2 z^5  + O(z^7)\\
    \wp''(z) &= 6z^{-4} + 6 G_4   + 60 G_6 z^2 + 90 G_4^2 z^4  + O(z^6)\\
    \wp''(z)/2 &= 3z^{-4} + 3 G_4   + 30 G_6 z^2 + 45 G_4^2 z^4  + O(z^6)\\
    (\wp''(z)/2)^2 &= 9z^{-8} + 18 G_4 z^{-4} + 180 G_6 z^{-2} + 279 G_4^2 + O(z^2)
    \end{align*}\begin{align*}
    \wp'(z)^2 &= 4 z^{-6} - 24 G_4 z^{-2} - 80 G_6 - 36 G_4^2 z^2 + O(z^4)\\
    \wp(2z) &= 0\ko25 z^{-2} + 12 G_4 z^2  + 80 G_6 z^4 + 192 G_4^2 z^6  + O(z^8)\\
    \wp(2z)+2\wp(z) &= 2\ko25 z^{-2} + 18 G_4 z^2  + 90 G_6 z^4 + 198 G_4^2 z^6  + O(z^8)\\
    \wp'(z)^2\cdot\left(\wp(2z)+2\wp(z)\right) &= 9z^{-8} + 18 G_4 z^{-4} + 180 G_6 z^{-2} + 279 G_4^2 + O(z^2)
\end{align*}
\en{Since both $\left(\frac{\wp''(z)}{2}\right)^2$ and $\wp'(z)^2\cdot\left(\wp(2z)+2\wp(z)\right)$ are elliptic function without poles outside of $L$ (because the additional poles of $\wp(2z)$ are eliminated by the zeros of $\wp'(z)^2$ from Prop.~\ref{\en{EN}zerowp}) and their Laurent series are equal up to $O(z^2)$, the first Liouville theorem~\ref{\en{EN}liouville1} yields:}%
\de{Sowohl $\left(\frac{\wp''(z)}{2}\right)^2$ als auch $\wp'(z)^2\cdot\left(\wp(2z)+2\wp(z)\right)$ sind elliptische Funktionen, die keine Pole außerhalb von $L$ haben (weil die zusätzlichen Pole von $\wp(2z)$ durch die Nullstellen von $\wp'(z)^2$ aus Satz~\ref{\en{EN}zerowp} aufgehoben werden). Außerdem stimmen ihre Laurentreihen bis auf $O(z^2)$ überein, also liefert der erste Liouville'sche Satz~\ref{\en{EN}liouville1}:}
\begin{align*}
    \left(\frac{\wp''(z)}{2}\right)^2 - \wp'(z)^2\cdot\left(\wp(2z)+2\wp(z)\right) = 0
\end{align*}
\en{Solving this equation for $\wp(2z)$ produces the equation of Lemma~\ref{\en{EN}lemwp2str}.}%
\de{Wenn wir diese Gleichung nach $\wp(2z)$ auflösen, erhalten wir die zu beweisende Gleichung des Lemmas~\ref{\en{EN}lemwp2str}.}
\end{proof}

\begin{lem}\label{\en{EN}lemsigmasigma}
\en{For all $u, u_1, u_2, u_3 \in \mathbb C$ it holds}%
\de{Für alle $u, u_1, u_2, u_3 \in \mathbb C$ gilt}
\begin{align*}
  \sigma(u+u_1)\sigma(u-u_1)\sigma(u_2+u_3)\sigma(u_2-u_3)&\\
+~\sigma(u+u_2)\sigma(u-u_2)\sigma(u_3+u_1)\sigma(u_3-u_1)&\\
+~\sigma(u+u_3)\sigma(u-u_3)\sigma(u_1+u_2)\sigma(u_1-u_2)& = 0
\end{align*}
\en{where $\sigma$ denotes the Weierstraß $\sigma$-function.}%
\de{wobei $\sigma$ die Weierstraß'sche $\sigma$-Funktion bezeichnet.}
\end{lem}
\begin{proof}
\en{We distinguish two cases:}\de{Wir unterscheiden zwei Fälle:}
\en{In the \emph{first case}, at least two of the four complex numbers $u, u_1, u_2, u_3$ are equal modulo $L$. But then the equation of the Lemma is true, because one of the three terms is zero (since $\sigma(0)=0$) and the other two cancel each other out (since $\sigma(-z)=-\sigma(z)$).
For example $u_2=u_3$ yields:}%
\de{Im \emph{ersten Fall} sind mindestens zwei der vier komplexen Zahlen $u, u_1, u_2, u_3$ äquivalent modulo $L$. Dann liefert $\sigma(0)=0$, dass einer der drei Summanden Null ist, und die anderen beiden heben sich gegenseitig auf (weil $\sigma(-z)=-\sigma(z)$ ist).
Beispielsweise liefert $u_2=u_3$:}
\begin{align*}
  \sigma(u+u_1)\sigma(u-u_1)\sigma(u_2+u_2)\sigma(u_2-u_2)&\\
+~\sigma(u+u_2)\sigma(u-u_2)\sigma(u_2+u_1)\sigma(u_2-u_1)&\\
+~\sigma(u+u_2)\sigma(u-u_2)\sigma(u_1+u_2)\sigma(u_1-u_2)&\\
= 0 + \sigma(u+u_2)\sigma(u-u_2)\sigma(u_2+u_1)\sigma(u_2-u_1)&\\
    -~\sigma(u+u_2)\sigma(u-u_2)\sigma(u_2+u_1)\sigma(u_2-u_1)&=0
\end{align*}

\en{In the \emph{second case}, we have four complex numbers $u$, $u_1$, $u_2$ and $u_3$ which are pairwise distinct modulo $L$. Then we choose $u_4$ so that $u_4$ is not equivalent to any of the numbers $\{\pm u; \pm u_1; \pm u_2; \pm u_3; -u_4\}$ modulo $L$.}%
\de{Im \emph{zweiten Fall} haben wir vier komplexe Zahlen $u$, $u_1$, $u_2$ und $u_3$ welche paarweise inäquivalent modulo $L$ sind. Dann wählen wir $u_4$ so, dass $u_4$ inäquivalent zu allen Zahlen $\{\pm u; \pm u_1; \pm u_2; \pm u_3; -u_4\}$ modulo $L$ ist.}
\en{Next we define for $l\in\{1;2;3\}$ the functions}%
\de{Dann definieren für $l\in\{1;2;3\}$ die Funktionen}
$$f_l(z):=\frac{\sigma(z+u_l)\sigma(z-u_l)}{\sigma(z+u_4)\sigma(z-u_4)}$$
\en{From Prop.~\ref{\en{EN}trafosigma} we deduce that the $f_l(z)$ are elliptic functions:}%
\de{Aus Satz~\ref{\en{EN}trafosigma} folgt, dass die $f_l(z)$ elliptisch sind:}
\begin{align}
    f_l(z+\omega_k) &= \frac{\sigma(z+u_l+\omega_k)\cdot\sigma(z-u_l+\omega_k)}{\sigma(z+u_4+\omega_k)\cdot\sigma(z-u_4+\omega_k)}\nonumber\\
    &= \frac{\exp\lk\eta_k\left(z+u_l+\omega_k/2\right)\rk\cdot\exp\lk\eta_k\left(z-u_l+\omega_k/2\right)\rk}
    {\exp\lk\eta_k\left(z+u_4+\omega_k/2\right)\rk\cdot\exp\lk\eta_k\left(z-u_4+\omega_k/2\right)\rk}\cdot f_l(z) = f_l(z)\label{\en{EN}eqellip}
\end{align}
\en{From our choice of $u_4$, we see that it has two poles of order one at $\pm u_4$ (we have chosen $u_4$ such that the zeros of the denominator are inequivalent modulo $L$), thus it is an elliptic function of order two.}%
\de{Aus unserer Wahl von $u_4$ folgt, dass die $f_l$ zwei Pole erster Ordnung bei $\pm u_4$ haben (wir haben $u_4$ so gewählt, dass die Nullstellen des Nenners inäquivalent sind), also ist $f_l$ eine elliptische Funktion der Ordnung zwei.}
\en{Next we define the function}%
\de{Nun definieren wir die Funktion}
\begin{align*}
    f(z):= f_1(z)\cdot\sigma(u_2+u_3)\sigma(u_2-u_3)&\\
+f_2(z)\cdot\sigma(u_3+u_1)\sigma(u_3-u_1)&\\
+f_3(z)\cdot\sigma(u_1+u_2)\sigma(u_1-u_2)&
\end{align*}
\en{Since this is a linear combination of elliptic functions, it is also an elliptic function.}%
\de{Diese ist eine elliptische Funktion, weil sie eine Linearkombination elliptischer Funktionen ist.}
\en{But if we use $z\in\{u_1;u_2;u_3\}$, we obtain $f(z)=0$ (as in the first case). Since we know that these three numbers are pairwise inequivalent modulo $L$, the third Liouville theorem (Prop.~\ref{\en{EN}liouville3}) tells us that $f(z)$ must be a constant function (since a non-constant elliptic function of order two would have only two zeros modulo $L$), thus $f(z)$ is equal to zero.}%
\de{Aber für $z\in\{u_1;u_2;u_3\}$ erhalten wir $f(z)=0$ (wie im ersten Fall). Diese drei Zahlen sind aber paarweise inäquivalent modulo $L$. Aus dem dritten Liouville'schen Satz~\ref{\en{EN}liouville3} folgt nun, dass $f(z)$ konstant sein muss (weil eine nichtkonstante elliptische Funktion vom Grad zwei nur zwei Nullstellen modulo $L$ hätte) und somit  ist $f(z)=0$.}

\en{Finally we use $z=u$ with the originally given $u$ we get $f(u)=0$. But because $u_4$ is not equivalent to $\pm u$, we can multiply this with $\sigma(u+u_4)\sigma(u-u_4)$ and obtain the equation from the Lemma.}%
\de{Schließlich setzen wir $z=u$ mit dem ursprünglich gegebenen $u$ ein und erhalten $f(u)=0$.
Weiter gilt $u_4\not\equiv \pm u$ und wir dürfen $f(u)=0$ mit $\sigma(u+u_4)\sigma(u-u_4)$ multiplizieren. Somit ist das Lemma bewiesen.}
\end{proof}

\begin{lem}\label{\en{EN}lemfrekur}
\en{For the functions $F_m(z)$ from Lemma~\ref{\en{EN}lemdefifm}, it holds:}%
\de{Für die Funktionen $F_m(z)$ aus Lemma~\ref{\en{EN}lemdefifm} gilt:}
\begin{align*}
    F_{2n+1}(z) &=  F_{n+2}(z)\cdot F_n(z)^3 - F_{n-1}(z)\cdot F_{n+1}(z)^3 \\
    F_{2n}(z)\cdot F_2(z) &=  F_n(z)\cdot \left( F_{n+2}(z)\cdot F_{n-1}(z)^2 - F_{n-2}(z)\cdot F_{n+1}(z)^2\right)
\end{align*}
\end{lem}
\begin{proof}
\en{Using Lemma~\ref{\en{EN}lemsigmasigma} with $u=0$, $u_1=z$, $u_2=n\cdot z$ and $u_3=-(n+1)\cdot z$
yields:}%
\de{Wir setzen $u=0$, $u_1=z$, $u_2=n\cdot z$ und $u_3=-(n+1)\cdot z$ in Lemma~\ref{\en{EN}lemsigmasigma} ein:}
\begin{align*}
  \sigma(z)\sigma(-z)\sigma(-z)\sigma((2n+1)z)&\\
+~\sigma(nz)\sigma(-nz)\sigma(-nz)\sigma(-(n+2)z)&\\
+~\sigma(-(n+1)z)\sigma((n+1)z)\sigma((n+1)z)\sigma(-(n-1)z)& = 0
\end{align*}
\en{This can be simplified using $\sigma(-z)=-\sigma(z)$ from Prop.~\ref{\en{EN}sigmaodd}:}%
\de{Dann nutzen wir $\sigma(-z)=-\sigma(z)$ aus Satz~\ref{\en{EN}sigmaodd}:}
\begin{align*}
  \sigma(z)^3\sigma((2n+1)z)&=
  \sigma(nz)^3\sigma((n+2)z)
 -\sigma((n+1)z)^3\sigma((n-1)z)
\end{align*}
\en{Next we use the definition of the $F_k$ (cf.~Lemma~\ref{\en{EN}lemdefifm}) which yields $\sigma(kz)=F_k(z)\cdot \sigma(z)^{k\cdot k}$:}%
\de{Dann folgt aus der Definition der $F_k$ (in Lemma~\ref{\en{EN}lemdefifm}), dass $\sigma(kz)=F_k(z)\cdot \sigma(z)^{k\cdot k}$ gilt:}
\begin{align*}
  &~\sigma(z)^3 F_{2n+1}(z)\sigma(z)^{(2n+1)\cdot(2n+1)}\\
  =&~F_n(z)^3\sigma(z)^{3\cdot n\cdot n}\cdot F_{n+2}(z)\sigma(z)^{(n+2)\cdot(n+2)}\\
  &-F_{n+1}(z)^3\sigma(z)^{3\cdot (n+1)\cdot (n+1)}\cdot F_{n-1}(z)\sigma(z)^{(n-1)\cdot (n-1)}
\end{align*}
\en{Dividing by $\sigma(z)^{4n^2+4n+4}$ yields the first equation.}%
\de{Die erste Gleichung folgt hieraus durch Kürzen mit $\sigma(z)^{4n^2+4n+4}$.}

\en{To prove the second equation, we use Lemma~\ref{\en{EN}lemsigmasigma} with $u=\frac 1 2 z$, $u_1=\frac 3 2 z$, $u_2=\frac 1 2 \cdot(2n-1)\cdot z$ and $u_3=- \frac 1 2 \cdot(2n+1)\cdot z$:}%
\de{Um die zweite Gleichung zu beweisen, setzen wir $u=\frac 1 2 z$, $u_1=\frac 3 2 z$, $u_2=\frac 1 2 \cdot(2n-1)\cdot z$ und $u_3=- \frac 1 2 \cdot(2n+1)\cdot z$ in Lemma~\ref{\en{EN}lemsigmasigma} ein:}
\begin{align*}
  \sigma(2z)\sigma(-z)\sigma(-z)\sigma(2nz)&\\
+~\sigma(nz)\sigma(-(n-1)z)\sigma(-(n-1)z)\sigma(-(n+2)z)&\\
+~\sigma(-nz)\sigma((n+1)z)\sigma((n+1)z)\sigma(-(n-2)z)& = 0
\end{align*}
\en{Again, this can be simplified using $\sigma(-z)=-\sigma(z)$ to:}%
\de{Dann nutzen wir wieder $\sigma(-z)=-\sigma(z)$:}
\begin{align*}
 &~\sigma(2z)\sigma(z)^2\sigma(2nz)\\
=&~\sigma(nz)\sigma((n-1)z)^2\sigma((n+2)z)\\
 &-\sigma(nz)\sigma((n+1)z)^2\sigma((n-2)z)
\end{align*}
\en{As above, we use the definition of the $F_k$ which yields $\sigma(kz)=F_k(z)\cdot \sigma(z)^{k\cdot k}$:}%
\de{Dann folgt wie oben aus der Definition der $F_k$, dass $\sigma(kz)=F_k(z)\cdot \sigma(z)^{k\cdot k}$ gilt und somit}
\begin{align*}
 &~F_2(z)\sigma(z)^{2\cdot 2}\cdot \sigma(z)^2\cdot F_{2n}(z)\sigma(z)^{2n\cdot 2n}\\
=&~F_n(z)\sigma(z)^{n\cdot n}\cdot F_{n-1}(z)^2\sigma(z)^{2\cdot(n-1)\cdot(n-1)} \cdot F_{n+2}(z)\sigma(z)^{(n+2)\cdot(n+2)}\\
 &-F_n(z)\sigma(z)^{n\cdot n} \cdot F_{n+1}(z)^2\sigma(z)^{2\cdot(n+1)\cdot(n+1)} \cdot F_{n-2}(z)\sigma(z)^{(n-2)\cdot(n-2)}
\end{align*}
\en{Dividing by $\sigma(z)^{4n^2+6}$ yields the second equation.}%
\de{Die zweite Gleichung folgt hieraus durch Kürzen mit $\sigma(z)^{4n^2+6}$.}
\end{proof}

\pagebreak

\begin{defi}\label{\en{EN}defipm}
\en{For any positive integer $m$, we define the polynomial $P_m(x)$ as follows:}%
\de{Für alle natürlichen $m\neq 0$ definieren wir das Polynom $P_m(x)$ wie folgt:}
\begin{align*}
    P_1 &= 1;\quad P_2 = 1;\quad P_3 = 3x^4-6h_2x^2-12h_3x-h_2^2\\
    P_4 &= 2x^6-10h_2x^4-40h_3x^3-10h_2^2x^2-8h_2h_3x-16h_3^2+2h_2^3\\
   \intertext{\en{and for}\de{und für} $k\geq 1$:}
   P_{4k+1} &= 16(x^3-h_2x-h_3)^2\cdot P_{2k+2}\cdot P_{2k}^3 - P_{2k-1}\cdot P_{2k+1}^3\\
    P_{4k+2} &= P_{2k+1}\cdot \left( P_{2k+3}\cdot P_{2k}^2 - P_{2k-1}\cdot P_{2k+2}^2\right)\\
    P_{4k+3} &= P_{2k+3}\cdot P_{2k+1}^3 - 16(x^3-h_2x-h_3)^2\cdot P_{2k}\cdot P_{2k+2}^3\\
    P_{4k+4} &= P_{2k+2}\cdot \left( P_{2k+4}\cdot P_{2k+1}^2 - P_{2k}\cdot P_{2k+3}^2\right)
\end{align*}
\end{defi}

\begin{thm}\label{\en{EN}propfmpm}
\en{For any positive integer $m$, it holds}%
\de{Für alle natürlichen $m\neq 0$ gilt}
\begin{align*}
    F_m(z) &= \cas{-\wp'(z) \cdot P_m(\wp(z))}{P_m(\wp(z))}
\end{align*}
\en{where $\wp(z)$ denotes the Weierstraß $\wp$-function of a lattice $L$ with $h_2=\frac 1 4 g_2(L)$ and $h_3 = \frac 1 4 g_3(L)$.}%
\de{Hierbei bezeichne $\wp(z)$ die Weierstraß'sche $\wp$-Funktion eines Gitters $L$ mit $h_2=\frac 1 4 g_2(L)$ und $h_3 = \frac 1 4 g_3(L)$.}
\end{thm}
\begin{proof}
\en{Throughout the proof, we will use the abbreviations $x=\wp(z)$, $h_2 = \frac 1 4 g_2(L)$ and $h_3 = \frac 1 4 g_3(L)$.
First, we prove $m\leq 4$ as induction basis:}%
\de{Während des gesamten Beweises nutzen wir die Abkürzungen $x=\wp(z)$, $h_2 = \frac 1 4 g_2(L)$ und $h_3 = \frac 1 4 g_3(L)$.
Zunächst beweisen wir $m\leq 4$ als Induktionsanfang:}
\begin{enumerate}[leftmargin=*, itemsep=1ex]
    \item \en{From its definition in Lemma~\ref{\en{EN}lemdefifm}, we observe $F_1(z)=1$. Thus $F_1(z)=P_1(\wp(z))$.}%
    \de{In der Definition in Lemma~\ref{\en{EN}lemdefifm} lesen wir $F_1(z)=1$ ab. Also gilt $F_1(z)=P_1(\wp(z))$.}
    
    \item \en{In Lemma~\ref{\en{EN}lemdefifm} we proved that
    $F_2(z)=\frac{\sigma(2z)}{\sigma(z)^4}$ has a pole of order three at $z=0$ and three zeros at $z\in \DIV(2)$.
    Prop.~\ref{\en{EN}zerowp} shows that the same is true for $\wp'(z)$, thus $\frac{F_2(z)}{\wp'(z)}$ is constant by the first Liouville theorem (Prop.~\ref{\en{EN}liouville1}).
    Comparing the Laurent series at $z=0$ gives $F_2(z)=2z^{-3}+O(z^{-1})$ and $\wp'(z)=-2z^{-3} + O(z^{-1})$, thus $\frac{F_2(z)}{\wp'(z)}=-1$ and $F_2(z)=-\wp'(z)\cdot P_2(\wp(z))$.}%
    \de{In Lemma~\ref{\en{EN}lemdefifm} haben wir bewiesen, dass
    $F_2(z)=\frac{\sigma(2z)}{\sigma(z)^4}$ einen Pol dritter Ordnung bei $z=0$ and drei Nullstellen bei $z\in \DIV(2)$ hat.
    Satz~\ref{\en{EN}zerowp} besagt, dass dies auch auf $\wp'(z)$ zutrifft, also ist $\frac{F_2(z)}{\wp'(z)}$ aufgrund des ersten Liouville'schen Satzes~\ref{\en{EN}liouville1} konstant.
    Ein Vergleich der Laurentreihen um $z=0$ liefert $F_2(z)=2z^{-3}+O(z^{-1})$ und $\wp'(z)=-2z^{-3} + O(z^{-1})$, also $\frac{F_2(z)}{\wp'(z)}=-1$ und $F_2(z)=-\wp'(z)\cdot P_2(\wp(z))$.}
    
    \item \en{From Lemma~\ref{\en{EN}lemwpnz} with $n=2$ we obtain}%
    \de{Aus Lemma~\ref{\en{EN}lemwpnz} folgt mit $n=2$:}
    $\wp(2z)=\wp(z)-\frac{F_1(z)\cdot F_3(z)}{F_2(z)^2}$.
    \en{Here we use $F_1(z)=1$ and $F_2(z)=-\wp'(z)$ and deduce}%
    \de{Hier nutzen wir $F_1(z)=1$ und $F_2(z)=-\wp'(z)$ und erhalten:}
    $$F_3(z)=\wp'(z)^2\cdot\left(\wp(z)-\wp(2z)\right)$$
    \en{Here we use Lemma~\ref{\en{EN}lemwp2str} for $\wp(2z)$ and obtain}%
    \de{Hier setzen wir $\wp(2z)$ aus Lemma~\ref{\en{EN}lemwp2str} ein:}
    \begin{align*}
        F_3(z) &= \wp'(z)^2\cdot\left(\wp(z)-\frac 1 4 \left(\frac{\wp''(z)}{\wp'(z)}\right)^2+ 2 \wp(z)\right)
        = -\frac 1 4 \wp''(z)^2 + 3 \wp(z)\wp'(z)^2
    \end{align*}
    \en{Here we use the representation of $\wp''$ from Lemma~\ref{\en{EN}lemwp2str} and the algebraic differential equation from Prop.~\ref{\en{EN}dglP}:}%
    \de{Mit der Darstellung von $\wp''$ aus Lemma~\ref{\en{EN}lemwp2str} und der algebraischen Differentialgleichung der $\wp$-Funktion in Satz~\ref{\en{EN}dglP} folgt:}
    \begin{align*}
        F_3(z)
        &= -\frac 1 4 (6\wp(z)^2-2h_2)^2 + 3 \wp(z)\left(4\wp(z)^3-4h_2\wp(z)-4 h_3\right)
    \end{align*}
    \en{Using the abbreviation $x=\wp(z)$ yields}%
    \de{Mit der Abkürzung $x=\wp(z)$ führt das auf}
    \begin{align*}
        F_3(z)
        &= -\frac 1 4 (6x^2-2h_2)^2 + 3 x\left(4x^3-4h_2x-4 h_3\right)\\
        &= -9x^4+6h_2x^2-h_2^2+12x^4-12 h_2 x^2- 12 h_3 x = P_3(x)=P_3(\wp(z))
    \end{align*}
    
    \item \en{Again from the definition of the $F_m(z)$ in Lemma~\ref{\en{EN}lemdefifm} we obtain}%
    \de{Wir nutzen wieder die Definition der $F_m(z)$ aus Lemma~\ref{\en{EN}lemdefifm} und erhalten}
    $$F_4(z) = \frac{\sigma(4z)}{\sigma(z)^{16}}=\frac{\sigma(4z)}{\sigma(2z)^4}\cdot\left(\frac{\sigma(2z)}{\sigma(z)^4}\right)^4 = F_2(2z)\cdot F_2(z)^4$$
    \en{Using $F_2(z)=-\wp'(z)$ we deduce}%
    \de{Mit $F_2(z)=-\wp'(z)$ folgt}
    $$F_4(z) = -\wp'(2z)\cdot\wp'(z)^4$$
    \en{Next we write $\wp(2z)$ from Lemma~\ref{\en{EN}lemwp2str} only in terms of $x=\wp(z)$, $h_2 = \frac 1 4 g_2$ and $h_3 = \frac 1 4 g_3$:}%
    \de{Als nächstes drücken wir $\wp(2z)$ aus Lemma~\ref{\en{EN}lemwp2str} nur mit Hilfe von  $x=\wp(z)$, $h_2 = \frac 1 4 g_2$ und $h_3 = \frac 1 4 g_3$ aus:}
    \begin{align*}
    \wp(2z) &= \frac{1}{4}\cdot\frac{(6x^2-2h_2)^2}{4x^3-4h_2x-4h_3}-2x
    = \frac{9x^4-6h_2x^2+h_2^2}{4x^3-4h_2x-4h_3}-2x\\
    &= \frac{9x^4-6h_2x^2+h_2^2-2x(4x^3-4h_2x-4h_3)}{4x^3-4h_2x-4h_3}\\
    &= \frac{x^4+2h_2x^2+8h_3x+h_2^2}{4x^3-4h_2x-4h_3}=:f(x)
    \end{align*}
    \en{Deriving both sides with respect to $z$ yields $2\wp'(2z) = f'(x)\cdot\wp'(z)$, where}%
    \de{Beide Seiten nach $z$ ableiten liefert $2\wp'(2z) = f'(x)\cdot\wp'(z)$, wobei}
    \begin{align*}
        f'(x) &= \frac{\left(\begin{aligned}(4x^3+4h_2x+8h_3)\cdot(4x^3-4h_2x-4h_3)\\-(x^4+2h_2x^2+8h_3x+h_2^2)\cdot(12x^2-4h_2)\end{aligned}\right)}{(4x^3-4h_2x-4h_3)^2}\\
        &= \frac{4x^6-20h_2x^4-80h_3x^3-20h_2^2x^2-16h_2h_3x-32h_3^2+4h_2^3}{(4x^3-g_2x-g_3)^2}
        = \frac{2 P_4(x)}{\wp'(z)^4}
    \end{align*}
    \en{This yields}\de{Das liefert}
    $F_4(z) = -\wp'(2z)\cdot\wp'(z)^4 = - \frac 1 2 f'(x)\cdot\wp'(z)\cdot\wp'(z)^4 = -\wp'(z)\cdot P_4(x)$.
\end{enumerate}
\en{Now we have proven that the Proposition is true for $m\leq 4$.
To prove it for any $m\geq 5$, we assume that it is correct for all numbers less then $m$, and we distinguish four cases:}%
\de{Somit ist der Satz für $m\leq 4$ bewiesen.
Um ihn für $m\geq 5$ zu beweisen, nehmen wir die Korrektheit für alle Zahlen kleiner $m$ an und unterscheiden vier Fälle:}
\begin{enumerate}[leftmargin=*,itemsep=1ex]
    \item \en{If $m=4k+1$ with an integer $k\geq 1$, then Lemma~\ref{\en{EN}lemfrekur} tells (using $n=2k$):}%
    \de{Falls $m=4k+1$ ist mit natürlichem $k\geq 1$, liefert Lemma~\ref{\en{EN}lemfrekur} mit $n=2k$:}
            \begin{align*}
                F_{4k+1}(z) &= F_{2k+2}(z)\cdot F_{2k}(z)^3 - F_{2k-1}(z)\cdot F_{2k+1}(z)^3
            \end{align*}
            \en{Now, by induction, we get}\de{Aus der Induktionsvoraussetzung folgt:}
            \begin{align*}
                F_{4k+1}(z) &= -\wp'(z)P_{2k+2}(x)\cdot (-\wp'(z)P_{2k}(x))^3 - P_{2k-1}(x)\cdot P_{2k+1}(x)^3\\
                        &= \wp'(z)^4\cdot P_{2k+2}(x)\cdot P_{2k}(x)^3 - P_{2k-1}(x)\cdot P_{2k+1}(x)^3
            \end{align*}
            \en{Using}\de{Mit} $\wp'(z)^2=4\wp(z)^3-g_2\wp(z)-g_3 = 4(x^3-h_2x-h_3)$ \en{yields}\de{folgt}
            \begin{align*}
                F_{4k+1}(z) &= 16(x^3-h_2x-h_3)^2\cdot P_{2k+2}\cdot P_{2k}^3 - P_{2k-1}\cdot P_{2k+1}^3 = P_{4k+1}
            \end{align*}
            \en{which proves the Proposition in this case.}%
            \de{womit der Satz in diesem Fall bewiesen ist.}
    \item \en{If $m=4k+2$ with an integer $k\geq 1$, then Lemma~\ref{\en{EN}lemfrekur} tells (using $n=2k+1$):}%
    \de{Falls $m=4k+2$ mit natürlichem $k\geq 1$, liefert Lemma~\ref{\en{EN}lemfrekur} mit $n=2k+1$:}
            \begin{align*}
                F_{4k+2}\cdot F_2 &=  F_{2k+1}\cdot \left( F_{2k+3}\cdot F_{2k}^2 - F_{2k-1}\cdot F_{2k+2}^2\right)
            \end{align*}
            \en{Now, by induction, we get}\de{Aus der Induktionsvoraussetzung folgt:}
            \begin{align*}
                F_{4k+2}\cdot F_2 &=  P_{2k+1}\cdot \left( P_{2k+3}\cdot \wp'(z)^2 P_{2k}^2 - P_{2k-1}\cdot\wp'(z)^2 P_{2k+2}^2\right)
            \end{align*}
            \en{Dividing by $F_2=-\wp'(z)$ yields}\de{Hier dividieren wir durch $F_2=-\wp'(z)$ und erhalten:}
            \begin{align*}
                F_{4k+2} &=  -\wp'(z)\cdot P_{2k+1}\cdot \left( P_{2k+3}\cdot P_{2k}^2 - P_{2k-1}\cdot P_{2k+2}^2\right) = -\wp'(z)\cdot P_{4k+2}
            \end{align*}
            \en{Thus the Proposition is also true in this case.}\de{Somit ist der Satz auch in diesem Fall bewiesen.}
    \item \en{If $m=4k+3$ with an integer $k\geq 1$, then Lemma~\ref{\en{EN}lemfrekur} tells (using $n=2k+1$):}%
    \de{Falls $m=4k+3$ ist mit natürlichem $k\geq 1$, liefert Lemma~\ref{\en{EN}lemfrekur} mit $n=2k+1$:}
            \begin{align*}
                F_{4k+3} &= F_{2k+3}\cdot F_{2k+1}^3 - F_{2k}\cdot F_{2k+2}^3
            \end{align*}
           \en{Now, by induction, we get}\de{Aus der Induktionsvoraussetzung folgt:}
            \begin{align*}
                F_{4k+3} &= P_{2k+3}\cdot P_{2k+1}^3 - (-\wp'(z)P_{2k})\cdot(-\wp'(z) P_{2k+2})^3\\
                        &= P_{2k+3}\cdot P_{2k+1}^3 - \wp'(z)^4 P_{2k}\cdot P_{2k+2}^3
            \end{align*}
            \en{Again,}\de{Mit} $\wp'(z)^2=4\wp(z)^3-g_2\wp(z)-g_3 = 4(x^3-h_2x-h_3)$ \en{yields}\de{folgt wiederum}
             \begin{align*}
                F_{4k+3} = P_{2k+3}\cdot P_{2k+1}^3 - 16(x^3-h_2x-h_3)^2 P_{2k}\cdot P_{2k+2}^3 = P_{4k+3}
            \end{align*}
             \en{which proves the Proposition in this case.}%
            \de{womit der Satz auch in diesem Fall bewiesen ist.}
    \item \en{If $m=4k+4$ with an integer $k\geq 1$, then Lemma~\ref{\en{EN}lemfrekur} tells (using $n=2k+2$):}%
    \de{Falls $m=4k+4$ mit natürlichem $k\geq 1$, liefert Lemma~\ref{\en{EN}lemfrekur} mit $n=2k+2$:}
            \begin{align*}
                F_{4k+4}\cdot F_2 &=  F_{2k+2}\cdot \left( F_{2k+4}\cdot F_{2k+1}^2 - F_{2k}\cdot F_{2k+3}^2\right)
            \end{align*}
            \en{Now, by induction, we get}\de{Aus der Induktionsvoraussetzung folgt:}
            \begin{align*}
                F_{4k+4}\cdot F_2 &=  -\wp'(z) P_{2k+2}\cdot \left( -\wp'(z) P_{2k+4}\cdot P_{2k+1}^2 - (-\wp'(z)P_{2k})\cdot P_{2k+3}^2\right)
            \end{align*}
            \en{Dividing by $F_2=-\wp'(z)$ yields}\de{Hier dividieren wir durch $F_2=-\wp'(z)$ und erhalten:}
            \begin{align*}
                F_{4k+4} &=  -\wp'(z)\cdot P_{2k+2}\cdot \left( P_{2k+4}\cdot P_{2k+1}^2 - P_{2k}\cdot P_{2k+3}^2\right) = -\wp'(z)\cdot P_{4k+4}
            \end{align*}
            \en{Thus the Proposition is also true in this case.}\de{Somit ist der Satz auch in diesem Fall bewiesen.}
\end{enumerate}
\en{By induction, we have proven Prop.~\ref{\en{EN}propfmpm}.}%
\de{Satz~\ref{\en{EN}propfmpm} ist also für alle natürlichen $m\neq 0$ per vollständiger Induktion bewiesen.}
\end{proof}

\begin{theo}\label{\en{EN}thmbaker}
\en{For any positive integer $m$, it holds:}%
\de{Für alle natürlichen $m\neq 0$ gilt}
$$m^2\cdot\prod_{u\in \DIV(m)}\left(x-\wp(u)\right) = \cas{4\cdot (x^3-h_2x - h_3) \cdot P_m^2(x)}{P_m^2(x)}$$
\end{theo}
\begin{proof}
\en{We will prove that this is true for every $x\in \mathbb C$:}%
\de{Wir werden diese Aussage für alle $x\in \mathbb C$ beweisen:}
\en{Given $x\in\mathbb C$. Then chose $z\in\mathbb C$ such that $\wp(z)=x$. This is possible, since the $\wp$-function takes on every value (apply the third Liouville theorem~\ref{\en{EN}liouville3} to $f(u):=\wp(u)-x$ for the given $x$ and chose $z$ as one of the zeros of $f(u)$).}%
\de{Sei also ein beliebiges $x\in \mathbb C$ gegeben. Wählen dann ein $z\in\mathbb C$ mit $\wp(z)=x$. Das ist möglich, weil die $\wp$-Funktion jeden Wert annimmt (man wende den dritten Liouville'schen Satz~\ref{\en{EN}liouville3} auf $f(u):=\wp(u)-x$ an und wähle für $z$ eine der Nullstellen von $f(u)$).}
\en{Then Prop.~\ref{\en{EN}propfmprod} tells us that the left hand side is equal to $F_m(z)^2$.
Prop.~\ref{\en{EN}propfmpm} gives}%
\de{Aus Satz~\ref{\en{EN}propfmprod} folgt, dass die linke Seite gleich $F_m(z)^2$ ist.
Satz~\ref{\en{EN}propfmpm} liefert dann}
$$F_m(z)^2 = \cas{(-\wp'(z))^2 \cdot P_m^2(\wp(z))}{P_m^2(\wp(z))}$$
\en{Prop.~\ref{\en{EN}dglP} tells $(-\wp'(z))^2=4\wp(z)^3-g_2\wp(z)-g_3=4(x^3-h_2x-h_3)$, thus we have proven Thm.~\ref{\en{EN}thmbaker} for the given $x\in\mathbb C$.
This proves the Theorem, because $x$ was chosen arbitrarily.}%
\de{Satz~\ref{\en{EN}dglP} besagt $(-\wp'(z))^2=4\wp(z)^3-g_2\wp(z)-g_3=4(x^3-h_2x-h_3)$, also ist Thm.~\ref{\en{EN}thmbaker} für das gegebene $x\in\mathbb C$ und somit für alle $x$ bewiesen.}
\end{proof}

\begin{thm}\label{\en{EN}propindu}
\en{For any positive integer $m$, it holds:}%
\de{Für alle natürlichen $m\neq 0$ gilt}
\begin{enumerate}
    \item \en{$P_m$ is a polynomial in $x$, $h_2$ and $h_3$ with coefficients in $\mathbb Z$.}%
    \de{$P_m$ ist ein Polynom in $x$, $h_2$ und $h_3$ mit Koeffizienten aus $\mathbb Z$.}\\[2ex]
    \hspace*{-1.15cm}\en{Furthermore, when $P_m$ is regarded as a polynomial in $x$, it holds:}%
    \de{Weiter gilt, wenn man $P_m$ als Polynom in $x$ betrachtet:}
    \item \en{The degree of $P_m(x)$ is}\de{Der Grad von $P_m(x)$ ist} $d_m:=\cas{\frac{m^2-4}{2}}{\frac{m^2-1}{2}}$.
    \item \en{The leading coefficient of $P_m(x)$ is $\pm l_m$ with}\de{Der Leitkoeffizient von $P_m(x)$ ist $\pm l_m$ mit} $l_m:=\cas{\frac{m}{2}}{m}$.
    \item \en{The second highest coefficient of $P_m(x)$ is $0$.}%
    \de{Der zweithöchste Koeffizient von $P_m(x)$ ist $0$.}
\end{enumerate}
\end{thm}
\begin{proof}
\begin{enumerate}[leftmargin=*,itemsep=0.5ex]
    \item \en{Since the Definition~\ref{\en{EN}defipm} of $P_m$ only includes multiplication and summation, statement (1) is correct for all $m$ by induction over $m$.}%
        \de{Da die Definition~\ref{\en{EN}defipm} der $P_m$ nur Multiplikationen und Additionen enthält, folgt Aussage (1) durch vollständige Induktion über $m$.}
    \item \en{Lemma~\ref{\en{EN}lemmm} tells that there are $m^2-1$ values $u\in \DIV(m)$, thus the left-hand-side of Thm.~\ref{\en{EN}thmbaker} has the degree $m^2-1$. This proves statement (2).}%
        \de{Lemma~\ref{\en{EN}lemmm} besagt, dass es $m^2-1$ Werte $u\in \DIV(m)$ gibt, also hat die linke Seite in Thm.~\ref{\en{EN}thmbaker} den Grad $m^2-1$. Hieraus folgt Aussage (2).}
    \item \en{On the left-hand-side of Thm.~\ref{\en{EN}thmbaker}, we observe that the leading coefficient is $m^2$, thus the leading coefficient of $P_m(x)$ has to be either $l_m$ or $-l_m$.
    It could be proven by induction over $m$ that it is indeed $l_m$, but for our means it's enough to prove statement (3) up to a factor of $\pm 1$.}%
    \de{Auf der linken Seite von Thm.~\ref{\en{EN}thmbaker} erkennen wir den Leitkoeffizient $m^2$, also muss der Leitkoeffizient von $P_m(x)$ entweder $l_m$ oder $-l_m$ sein.
    Man könnte per vollständiger Induktion beweisen, dass es tatsächlich $l_m$ ist, aber für unsere Zwecke ist es ausreichend, Aussage (3) bis auf einen Faktor von $\pm 1$ zu beweisen.}
    \item \en{For proving (4) we use the fact that for any two polynomials $f(x)=\sum_{k=0}^{n} a_k x^k$ and $g(x)=\sum_{k=0}^{m} b_k x^k$ it holds:
The second highest coefficient of $f(x)\cdot g(x)$ is given by $a_n\cdot b_{m-1}+a_{n-1}\cdot b_m$.
This proves that if the second highest coefficients of the factors vanish, the second highest coefficient of the product also vanishes.
Further we observe that in the recursive definition of the $P_m$, we only add or substract polynomials of the same degree (cf.~(3)), which proves (4) by induction.}%
\de{Um (4) zu beweisen, nutzen wir die Tatsache, dass für zwei beliebige Polynome $f(x)=\sum_{k=0}^{n} a_k x^k$ und $g(x)=\sum_{k=0}^{m} b_k x^k$ gilt:
Der zweithöchste Koeffizient von $f(x)\cdot g(x)$ ist $a_n\cdot b_{m-1}+a_{n-1}\cdot b_m$.
Also: Wenn die zweithöchsten Koeffizienten der Faktoren Null sind, dann ist auch der zweithöchste Koeffizient des Produktes Null.
Des weiteren bemerken wir, dass in der rekursiven Definition der $P_m$ nur Polynome gleichen Grades addiert oder subtrahiert werden (vgl.~(3)), woraus (4) per Induktion folgt.}
\end{enumerate}\vspace*{-12pt}
\end{proof}

\begin{theo}\label{\en{EN}theowpsum}
\en{For any positive integer $m$, the following sum of $m$-division-values vanishes:}%
\de{Für alle natürlichen $m\neq 0$ verschwindet die folgende Summe von $m$-Teilungswerten:}
$$\sum_{u\in \DIV(m)}\wp(u)=0$$
\end{theo}
\begin{proof}
\en{Theorem~\ref{\en{EN}thmbaker} gives two equivalent representations of a polynomial in $x$.
We multiply out the left hand side of Thm.~\ref{\en{EN}thmbaker} and obtain}%
\de{Theorem~\ref{\en{EN}thmbaker} liefert zwei äquivalente Darstellungen eines Polynoms in $x$.
Wir multiplizieren die linke Seite des Thm.~\ref{\en{EN}thmbaker} aus und erhalten}
$$m^2\cdot x^{m^2-1} - \left(m^2\cdot\sum_{u\in \DIV(m)} \wp(u)\right)\cdot x^{m^2-2} + \sum_{k=0}^{m^2-3} a_k x^k$$
\en{Now Prop.~\ref{\en{EN}propindu} (4) tells us that the second highest coefficient of the right hand side of Thm.~\ref{\en{EN}thmbaker} is zero -- thus the same coefficient on the left hand side also vanishes. This proves Thm.~\ref{\en{EN}theowpsum}.}%
\de{Jetzt besagt Satz~\ref{\en{EN}propindu} (4), dass der zweithöchste Koeffizient auf der rechten Seite von Thm.~\ref{\en{EN}thmbaker} Null ist -- also muss dieser Koeffizient auch auf der linken Seite verschwinden. Dies beweist Thm.~\ref{\en{EN}theowpsum}.}
\end{proof}

\begin{theo}\label{\en{EN}theomwpu}
\en{Given a lattice $L$, let $h_2:=\frac{1}{4}g_2(L)$ and $h_3:=\frac{1}{4}g_3(L)$.
Then for any positive integer $m$ and for any $m$-division-point $u\in \DIV(m)$, the term}%
\de{Sei $L$ ein Gitter, sei $h_2:=\frac{1}{4}g_2(L)$ und $h_3:=\frac{1}{4}g_3(L)$.
Dann gilt für alle natürlichen $m\neq 0$ und für alle $m$-Teilungsstellen $u\in \DIV(m)$, dass}
$$m\cdot\wp(u)$$
\en{is an algebraic integer of~~$\mathbb Z[h_2;h_3]$.}%
\de{ganzalgebraisch in $\mathbb Z[h_2;h_3]$ ist.}\\
\en{Moreover, if $m$ is even, then $\frac{m}{2}\cdot\wp(u)$ is also an algebraic integer of~~$\mathbb Z[h_2;h_3]$.}%
\de{Und falls $m$ gerade ist, ist $\frac{m}{2}\cdot\wp(u)$ ebenfalls ganzalgebraisch in $\mathbb Z[h_2;h_3]$.}
\end{theo}
\begin{proof}
\en{Throughout this proof, we denote $\mathbb I := \mathbb Z[h_2;h_3]$.
From Thm.~\ref{\en{EN}thmbaker} we know that the $\wp(u)$ with $u\in \DIV(m)$ are either among the zeros of $x^3-h_2x-h_3$ or among the zeros of $P_m(x)$.

The zeros of $x^3-h_2x-h_3$ are algebraic integers of $\mathbb I$, so these $\wp(u)$ are algebraic integers of $\mathbb I$.
Since $P_1=P_2=1$, it remains to check the zeros of $P_m(x)$ for $m\geq 3$.}%
\de{Für diesen Beweis bezeichnen wir $\mathbb I := \mathbb Z[h_2;h_3]$.
Aus Thm.~\ref{\en{EN}thmbaker} wissen wir, dass die $\wp(u)$ mit $u\in \DIV(m)$ entweder Nullstellen von $x^3-h_2x-h_3$ oder von $P_m(x)$ sind.

Die Nullstellen von $x^3-h_2x-h_3$ sind ganzalgebraisch in $\mathbb I$, also sind diese $\wp(u)$ ganzalgebraisch in $\mathbb I$.
Wegen $P_1=P_2=1$ müssen wir nur noch die Nullstellen von $P_m(x)$ für $m\geq 3$ prüfen.}

\en{Again, we use the notation $d_m:=\cas{\frac{m^2-4}{2}}{\frac{m^2-1}{2}}$ for the degree of $P_m$ (as in Prop.~\ref{\en{EN}propindu}) and the notation $l_m:=\cas{\frac{m}{2}}{m}$.}%
\de{Hier nutzen wir wieder die Notation $d_m:=\cas{\frac{m^2-4}{2}}{\frac{m^2-1}{2}}$ für den Grad von $P_m$ (vgl.~Satz~\ref{\en{EN}propindu}) und die Notation $l_m:=\cas{\frac{m}{2}}{m}$.}

\en{Then we define the polynomial $\displaystyle h(x):=P_m\lk\frac{x}{l_m}\rk\cdot l_m^{d_m-1}$.
From Prop.~\ref{\en{EN}propindu} (3) we know that the leading coefficient of $P_m(x)$ is either $l_m$ or $-l_m$. This leads to}%
\de{Dann betrachten wir das Polynom $\displaystyle h(x):=P_m\lk\frac{x}{l_m}\rk\cdot l_m^{d_m-1}$.
Da der Leitkoeffizient von $P_m(x)$ nach Satz~\ref{\en{EN}propindu} (3) entweder $l_m$ oder $-l_m$ ist, erhalten wir}
\begin{align*}
h(x) &= \pm l_m \cdot \left(\frac{x}{l_m}\right)^{d_m}\cdot \left(l_m\right)^{d_m-1} + \sum_{k=0}^{d_m-2}b_k \left(\frac{x}{l_m}\right)^{k}\cdot \left(l_m\right)^{d_m-1}\\
&= \pm x^{d_m} + \sum_{k=0}^{d_m-2}b_k \cdot\left(l_m\right)^{d_m-1-k}\cdot x^k
\end{align*}
\en{where the $b_k$ are the coefficients of $P_m(x)$, for which Prop.~\ref{\en{EN}propindu}~(1) tells: $b_k\in\mathbb I$.}%
\de{wobei $b_k$ die Koeffizienten von $P_m(x)$ sind, für die nach Satz~\ref{\en{EN}propindu} (1) gilt: $b_k\in\mathbb I$.}

\en{This shows that $h$ is a monic polynomial in $x$ whose coefficients are in $\mathbb I$, so the zeros of $h(x)$ are algebraic integers of $\mathbb I$.
But it holds: $$h(x)=0\quad\Longleftrightarrow\quad P_m\lk\frac x {l_m}\rk=0$$
This proves that for all $\wp(u)$ which are zeros of $P_m(x)$, the term $l_m\cdot \wp(u)$ is an algebraic integer of $\mathbb I$.
Altogether, we have proved Thm.~\ref{\en{EN}theomwpu} for all $u\in \DIV(m)$.}%
\de{Dies zeigt, dass $h$ ein monisches Polynom in $x$ ist, dessen Koeffizienten in $\mathbb I$ sind -- also sind die Nullstellen von $h(x)$ ganzalgebraisch in $\mathbb I$.
Aber es gilt: $$h(x)=0\quad\Longleftrightarrow\quad P_m\lk\frac x {l_m}\rk=0$$
Hieraus folgt auch für alle $\wp(u)$, die Nullstellen von $P_m(x)$ sind, dass $l_m\cdot \wp(u)$ ganzalgebraisch in $\mathbb I$ ist.
Insgesamt haben wir Thm.~\ref{\en{EN}theomwpu} also für alle $u\in \DIV(m)$ bewiesen.}
\end{proof}

\vfill\pagebreak
\section{\en{Complex Multiplication}\de{Komplexe Multiplikation}}\label{\en{EN}kapmult}
\renewcommand{\leftmark}{\en{Complex Multiplication}\de{Komplexe Multiplikation}}
\en{We prove (using App.~\ref{\en{EN}kapdivi}) that $\sqrt{D}\cdot\frac{E_2^*(\tau)}{\eta^4(\tau)}\cdot(AC)^2$ is an algebraic integer if $\tau$ satisfies $C\tau^2+B\tau+A=0$ with discriminant $D$.}%
\de{Wir beweisen, dass $\sqrt{D}\cdot\frac{E_2^*(\tau)}{\eta^4(\tau)}\cdot(AC)^2$ ganzalgebraisch ist, falls $\tau$ eine Lösung von $C\tau^2+B\tau+A=0$ mit Diskriminante $D$ ist.}
\en{The proof elaborates \cite[Lemma A3]{\en{EN}Masser1975}.}%
\de{Der Beweis arbeitet \cite[Lem.~A3]{\en{EN}Masser1975} aus.}

\begin{defi}\label{\en{EN}bemkonv}
\en{Throughout this appendix, the following notations are used:}%
\de{In diesem ganzen Anhang gelten folgende Konventionen:}
\begin{itemize}
    \item \en{We fix a lattice $L=\mathbb Z \omega_1 + \mathbb Z \omega_2$ with complex multiplication as in Def.~\ref{\en{EN}defiCM},
    where $\tau$ denotes the period ratio $\tau:=\frac{\omega_2}{\omega_1}\in \mathbb H$ and \ldots}%
    \de{Wir betrachten ein festes Gitter $L=\mathbb Z \omega_1 + \mathbb Z \omega_2$ mit komplexer Multiplikation wie in Def.~\ref{\en{EN}defiCM},
    wobei wir das Periodenverhältnis mit $\tau:=\frac{\omega_2}{\omega_1}\in \mathbb H$ bezeichnen und \ldots}
    \begin{itemize}
    \item \en{$A$, $B$ and $C$ are integers with $\operatorname{gcd}(A;B;C)=1$ and $A+B\tau+C\tau^2=0$,}%
    \de{$A$, $B$ und $C$ seien ganze Zahlen mit $\operatorname{ggT}(A;B;C)=1$ und $A+B\tau+C\tau^2=0$,}
    \item \en{and $D$ denotes the discriminant of $\tau$, which is $D = B^2-4AC$.}%
    \de{und $D$ bezeichne die Diskriminante von $\tau$, also $D = B^2-4AC$.}
    \end{itemize}
    \item \en{As in Def.~\ref{\en{EN}defis2}, let $E_2^*(\tau):=E_2(\tau) - \frac{3}{\pi Im(\tau)}$ with $E_2$ from Thm.~\ref{\en{EN}fouriertheorem}.}%
    \de{Wie in Def.~\ref{\en{EN}defis2} sei $E_2^*(\tau):=E_2(\tau) - \frac{3}{\pi Im(\tau)}$ mit $E_2$ aus Thm.~\ref{\en{EN}fouriertheorem}.}
    \item \en{$\eta_1$ and $\eta_2$ denote the basic quasi periods of $L$ (cf.~Def.~\ref{\en{EN}defetak}).}%
    \de{$\eta_1$ und $\eta_2$ bezeichnen die Basis-Quasiperioden von $L$ (vgl.~Def.~\ref{\en{EN}defetak}).}
    \item \en{$\zeta(z)$ and $\wp(z)$ denote the Weierstraß functions from Def.~\ref{\en{EN}defizeta} and~\ref{\en{EN}defiwp}.}%
    \de{$\zeta(z)$ und $\wp(z)$ bezeichnen die Weierstraß'schen Funktionen aus Def.~\ref{\en{EN}defizeta} und~\ref{\en{EN}defiwp}.}
    \item \en{$\DIV(m)$ denotes the set of $m$-division points in the fundamental parallelogram $\mathcal P =\left\{~s\omega_1+t\omega_2~|~0\leq s,t < 1~\right\}$ as in Def.~\ref{\en{EN}defidivi}.}%
    \de{$\DIV(m)$ bezeichne die Menge der $m$-Teilungsstellen im Periodenparallelogramm $\mathcal P =\left\{~s\omega_1+t\omega_2~|~0\leq s,t < 1~\right\}$ wie in Def.~\ref{\en{EN}defidivi}.}
\end{itemize}
\end{defi}

\begin{defi}\label{\en{EN}defkappa}
\en{In the notations of Def.~\ref{\en{EN}bemkonv}, we define $\kappa$ by}%
\de{In der Notation von Def.~\ref{\en{EN}bemkonv} definieren wir $\kappa$ durch}
$$\kappa\omega_2 := A\eta_1 - C\tau\eta_2$$
\end{defi}

\begin{lem}\label{\en{EN}lemkapa1}
\en{In the notations of Def.~\ref{\en{EN}bemkonv}, it holds: }%
\de{In der Notation von Def.~\ref{\en{EN}bemkonv} gilt:}
$$\kappa = -\sqrt{D}\cdot\frac{\pi^2}{3\omega_1^2}\cdot E_2^*(\tau)$$
\end{lem}
\begin{proof}
\en{First we multiply Def.~\ref{\en{EN}defkappa} with $\omega_1$ and obtain:}%
\de{Zunächst multiplizieren wir Def.~\ref{\en{EN}defkappa} mit $\omega_1$ und erhalten:}
$$\kappa\omega_1\omega_2 = A\omega_1\eta_1 - C \tau \omega_1\eta_2$$
\en{Then we use Legendre's relation (Prop.~\ref{\en{EN}legendre}) and obtain:}%
\de{Dann nutzen wir die Legendre'sche Relation (Satz~\ref{\en{EN}legendre}):}%
\begin{align*}
    \kappa\omega_1\omega_2 = A\omega_1\eta_1 - C \tau (\eta_1\omega_2-2\pi i)
    &=\left(A-C\tau\frac{\omega_2}{\omega_1}\right)\cdot \omega_1\eta_1 + 2\pi i C \tau\\
    &=\left(A-C\tau^2\right)\cdot \omega_1\eta_1 + 2\pi i C \tau
\end{align*}
\en{Thm.~\ref{\en{EN}fouriertheorem} tells $\eta_1(L_\tau)=\frac{\pi^2}{3}\cdot E_2(\tau)$.
Now we have $L=\omega_1\cdot L_\tau$ and from Prop.~\ref{\en{EN}etatransf} we get $\eta_1(L)=\frac{1}{\omega_1}\cdot\eta_1(L_\tau)=\frac{1}{\omega_1}\cdot \frac{\pi^2}{3}\cdot E_2(\tau)$.
Thus $\omega_1\eta_1 = \frac{\pi^2}{3}\cdot E_2(\tau)$ and we obtain:}%
\de{In Thm.~\ref{\en{EN}fouriertheorem} haben wir $\eta_1(L_\tau)=\frac{\pi^2}{3}\cdot E_2(\tau)$ bewiesen.
Jetzt gilt $L=\omega_1\cdot L_\tau$ und Satz~\ref{\en{EN}etatransf} liefert $\eta_1(L)=\frac{1}{\omega_1}\cdot\eta_1(L_\tau)=\frac{1}{\omega_1}\cdot \frac{\pi^2}{3}\cdot E_2(\tau)$.
Somit folgt $\omega_1\eta_1 = \frac{\pi^2}{3}\cdot E_2(\tau)$ und}
\begin{align*}
    \kappa\omega_1\omega_2 &=\left(A-C\tau^2\right)\cdot \frac{\pi^2}{3}\cdot E_2(\tau) + 2\pi i C \tau
\end{align*}
\en{From}\de{Aus} $A+B\tau=-C\tau^2$ \en{we obtain}\de{folgt} $A-C\tau^2 = A+(A+B\tau)=2A+B\tau$ \en{and}\de{und}
\begin{align*}
   - \sqrt{D}\cdot \tau &= -\sqrt{D}\cdot\frac{-B+\sqrt{D}}{2C} = \frac{B\cdot\sqrt{D}-D}{2C} = \frac{B\cdot\sqrt{D}-B^2+4AC}{2C}\\
   &= 2A + \frac{B\cdot\sqrt{D}-B^2}{2C} = 2A +B\tau = A-C\tau^2
\end{align*}
\en{which yields}\de{Das führt zu}{\abovedisplayskip=-4pt
\begin{align*}
    \kappa\omega_1\omega_2 &=- \sqrt{D}\cdot \tau\cdot \frac{\pi^2}{3}\cdot E_2(\tau) + 2\pi i C \tau\\
    &= - \sqrt{D}\cdot \tau\cdot \frac{\pi^2}{3}\cdot \left(E_2(\tau) - \frac{2\pi i C \tau}{\sqrt{D}\cdot \tau\cdot \frac{\pi^2}{3}}\right)
\end{align*}}
\en{But from}\de{Aber aus} $\im(\tau)=\frac{\sqrt{-D}}{2C}= \frac{\sqrt{D}}{2C\cdot i}$ \en{we get}\de{folgt nun}
\begin{align*}
    \kappa\omega_1\omega_2 &= - \sqrt{D}\cdot \tau\cdot \frac{\pi^2}{3}\cdot \left(E_2(\tau) - \frac{3}{\pi\im(\tau)}\right)
\end{align*}
\en{Finally we divide by}\de{Schließlich dividieren wir durch} $\omega_1\cdot\omega_2$ \en{and obtain}\de{und erhalten} $\kappa = - \sqrt{D}\cdot \frac{\tau}{\omega_1\omega_2}\cdot \frac{\pi^2}{3}\cdot E_2^*(\tau)$.
\en{Here we simplify}\de{Hier vereinfachen wir noch} $\frac{\tau}{\omega_1\omega_2}=\frac{\omega_2/\omega_1}{\omega_1\omega_2}=\frac{1}{\omega_1^2}$ \en{which proves the Lemma.}\de{womit das Lemma bewiesen ist.}
\end{proof}

\pagebreak
\begin{lem}\label{\en{EN}lemfellip}
\en{In the notations of Def.~\ref{\en{EN}bemkonv}, the function}%
\de{In der Notation von Def.~\ref{\en{EN}bemkonv} ist die Funktion}
$$f(z):=-A\zeta(Cz)+C\tau\zeta(C\tau z) + C\tau\kappa z$$
\en{is elliptic with periods $\omega_1$ and $\omega_2$.}%
\de{elliptisch mit Perioden $\omega_1$ und $\omega_2$.}
\end{lem}
\begin{proof}
\en{To prove the first period we calculate}\de{Um die erste Periode zu beweisen, berechnen wir}
\begin{align*}
    f(z+\omega_1)-f(z) &= -A\cdot\underbrace{\left(\zeta(C(z+\omega_1))-\zeta(Cz)\right)}_{T_1}\\
    &~~~+ C \tau \cdot \underbrace{\left(\zeta(C\tau(z+\omega_1))-\zeta(C\tau z)\right)}_{T_2} +~ C\tau\kappa\omega_1
\end{align*}
\en{For calculating $T_1 = \zeta(Cz+C\omega_1)-\zeta(Cz)$ we apply $C$ times the Definition~\ref{\en{EN}defetak} of $\eta_1=\zeta(Cz+\omega_1)-\zeta(Cz)$ and get
$T_1=C\eta_1$.
Similarly, as $C\tau\omega_1=C\omega_2$, we get $T_2 = \zeta(C\tau z+C\tau\omega_1)-\zeta(C\tau z) = C\eta_2$. This yields}%
\de{Um den Wert von $T_1 = \zeta(Cz+C\omega_1)-\zeta(Cz)$ zu bestimmen, wenden wir $C$ mal die Def.~\ref{\en{EN}defetak} von $\eta_1=\zeta(Cz+\omega_1)-\zeta(Cz)$ an und erhalten
$T_1=C\eta_1$.
Aus $C\tau\omega_1=C\omega_2$ folgt genauso $T_2 = \zeta(C\tau z+C\tau\omega_1)-\zeta(C\tau z) = C\eta_2$. Dies führt zu}
\begin{align*}
    f(z+\omega_1)-f(z) &= - A\cdot C \eta_1 + C\tau \cdot C\eta_2 + C\kappa\omega_2 = C\cdot\left(\kappa\omega_2-A\eta_1+C\tau\eta_2\right)=0,
\end{align*}
\en{where we used Def.~\ref{\en{EN}defkappa} of $\kappa$ in the last step.}%
\de{wobei wir im letzten Schritt Def.~\ref{\en{EN}defkappa} von $\kappa$ benutzt haben.}
\en{Now the second period:}\de{Nun zur zweiten Periode:}
\begin{align*}
    f(z+\omega_2)-f(z) &= -A\cdot\underbrace{\left(\zeta(C(z+\omega_2))-\zeta(Cz)\right)}_{T_3}\\
    &~~~+ C \tau \cdot \underbrace{\left(\zeta(C\tau(z+\omega_2))-\zeta(C\tau z)\right)}_{T_4}
    +~C\tau\kappa\omega_2
\end{align*}
\en{As above we have $T_3=C\eta_2$. Then it holds}%
\de{Wie oben gilt $T_3=C\eta_2$. Außerdem ist}
$$C\tau\omega_2=C\tau^2\omega_1 = -(A+B\tau)\omega_1 = -A\omega_1-B\omega_2$$
\en{which yields $T_4 = -A\eta_1-B\eta_2$. This shows}%
\de{was $T_4=-A\eta_1-B\eta_2$ beweist und somit}
\begin{align*}
    f(z+\omega_2)-f(z) &= - A\cdot C \eta_2 + C\tau \cdot (-A\eta_1-B\eta_2) + C\tau\kappa\omega_2\\
    &= -AC\eta_2-AC\tau\eta_1-BC\tau\eta_2+C\tau(A\eta_1-C\tau\eta_2)\\
    &= -C\cdot\eta_2\cdot\left(A+B\tau+C\tau^2\right)=0
\end{align*}
\en{Thus the function $f(z)$ has the periods $\omega_1$ and $\omega_2$.}%
\de{Also hat die Funktion $f(z)$ die Perioden $\omega_1$ und $\omega_2$.}
\end{proof}

\begin{lem}\label{\en{EN}lemctau}
\en{In the notations of Def.~\ref{\en{EN}bemkonv}, it holds:}%
\de{In der Notation von Def.~\ref{\en{EN}bemkonv} gilt:}
\en{The number of $C\tau$-division points in $\mathcal P$ is $AC-1$. And every $C\tau$-division-point is also an $AC$-division-point. In brief:}%
\de{Die Anzahl der $C\tau$-Teilungsstellen in $\mathcal P$ ist $AC-1$. Und jede $C\tau$-Teilungsstelle ist auch eine $AC$-Teilungsstelle. Kurz:}
$$|\DIV(C\tau)|=AC-1\qquad\text{\en{and}\de{und}}\qquad \DIV(C\tau)\subset \DIV(AC)$$
\end{lem}
\begin{proof}\en{Let}\de{Sei} $L=\mathbb Z \omega_1 + \mathbb Z \omega_2$
\en{with}\de{mit} $\tau:=\frac{\omega_2}{\omega_1}\in CM$
\en{as in}\de{wie in} Def.~\ref{\en{EN}bemkonv}.\\
\en{Then it holds for all $u\in \DIV(C\tau)$:}\de{Dann gilt für alle $u\in \DIV(C\tau)$:}
\begin{align*}
    C\tau u \in L
    \quad\Longleftrightarrow\quad
     C\tau u &= -k\omega_1 + l \omega_2
    \quad\Longleftrightarrow\quad
    u = \frac{-k\omega_1 + l \omega_2}{C\tau} = \frac{-k\omega_1\bar\tau + l \omega_2\bar\tau}{C\tau\bar\tau}
\end{align*}
\en{Then we have $\tau\bar\tau = \frac{B^2-(B^2-4AC)}{4C^2} = \frac A C$ and $\tau+\bar\tau = -\frac B C$.\\
This leads to $\bar\tau = -\frac B C-\tau$ and $\omega_2\bar\tau=\omega_1\tau\bar\tau=\omega_1\cdot\frac{A}{C}$ and $\omega_1\tau=\omega_2$, thus:}%
\de{Weiter gilt $\tau\bar\tau = \frac{B^2-(B^2-4AC)}{4C^2} = \frac A C$ und $\tau+\bar\tau = -\frac B C$.\\
Hieraus folgt $\bar\tau = -\frac B C-\tau$ und $\omega_2\bar\tau=\omega_1\tau\bar\tau=\omega_1\cdot\frac{A}{C}$ sowie $\omega_1\tau=\omega_2$, also:}
\begin{align}
    u &= \frac{-k\cdot\omega_1\cdot\left(-\frac{B}{C}-\tau\right)+l\cdot\omega_1\cdot A/C}{C\cdot A/C} = \frac{k}{A}\cdot\omega_2+\frac{l\cdot A+k\cdot B}{AC}\cdot\omega_1\label{\en{EN}ctauu}
\end{align}
\en{Next, $u$ has to be in $\mathcal P$ which yields $0\leq \frac k A < 1$ and $0\leq \frac{l\cdot A+k\cdot B}{AC} <1$.
From the first we see that there are $A$ possible values of $k$. From the second we see that $0\leq l + k\cdot\frac{B}{A} < C$, thus we have $C$ possible values of $l$, no matter what value of $k$ we had chosen. This yields $AC$ values of $u$ in $\mathcal P$ such that $C\tau u\in L$. Since $u=0$ is not a $C\tau$-division-point, we have $AC-1$ of the $C\tau$-division-points in $\mathcal P$.}%
\de{Außerdem soll $u$ in $\mathcal P$ liegen, also muss $0\leq \frac k A < 1$ und $0\leq \frac{l\cdot A+k\cdot B}{AC} <1$ gelten.
Aus der ersten Ungleichung folgt, dass es $A$ mögliche Werte von $k$ gibt. Aus der zweiten erkennen wir $0\leq l + k\cdot\frac{B}{A} < C$, also gibt es (unabhängig vom Wert von $k$) genau $C$ mögliche Werte von $l$. Insgesamt haben wir also $AC$ Werte von $u$ in $\mathcal P$ mit $C\tau u\in L$.
Weil $u=0$ keine $C\tau$-Teilungsstelle ist, folgt $|\DIV(C\tau)|=AC-1$.}

\en{For all $u\in \DIV(C\tau)$, it holds eq.~(\ref{\en{EN}ctauu}), which implies $AC\cdot u \in L$ and thus $u\in \DIV(AC)$.
This proves that $\DIV(C\tau)\subset \DIV(AC)$.}%
\de{Für alle $u\in \DIV(C\tau)$ gilt Glg.~(\ref{\en{EN}ctauu}). Hieraus folgt $AC\cdot u \in L$ und somit $u\in \DIV(AC)$.
Somit haben wir auch $\DIV(C\tau)\subset \DIV(AC)$ bewiesen.}
\end{proof}

\begin{lem}\label{\en{EN}lemkonst}
\en{In the notations of Def.~\ref{\en{EN}bemkonv} and with the function $f(z)$ from Lemma~\ref{\en{EN}lemfellip} it holds:}%
\de{In der Notation von Def.~\ref{\en{EN}bemkonv} und mit der Funktion $f(z)$ aus Lemma~\ref{\en{EN}lemfellip} gilt:}
$$g(z):=f(z)-\left(1-\frac{A}{C}\right)\zeta(z)+ \frac{A}{C} \cdot\sum_{u\in \DIV(C)}\zeta(z-u) - \sum_{v\in \DIV(C\tau)}\zeta(z-v)$$
\en{is constant in the whole complex plane.}%
\de{ist konstant in der ganzen komplexen Ebene.}
\end{lem}
\begin{proof}
\en{First we observe that $f(z)$ has poles of order $1$ with residuum $-A/C$ in all $C$-division-points, and poles with residuum $\frac{C\tau}{C\tau}=1$ in all $C\tau$-division-points. In $z=0$, $f(z)$ has a pole with residuum $-\frac A C+1$.}%
\de{Zunächst bemerken wir, dass $f(z)$ Pole erster Ordnung mit Residuum $-A/C$ in allen $C$-Teilungsstellen hat, und Pole mit Residuum $\frac{C\tau}{C\tau}=1$ in allen $C\tau$-Teilungsstellen. In $z=0$ hat $f(z)$ einen Pol mit Residuum $-\frac A C+1$.}
\en{Thus, by definition of $g(z)$, $g(z)$ has no poles and is analytic in the whole complex plane.}%
\de{Aus der Definition von $g(z)$ folgt also, dass $g(z)$ keine Pole hat und in der ganzen komplexen Ebene analytisch ist.}
\en{We will now prove that $g$ is also elliptic. For this we observe that there are $C^2-1$ values in the $u$-summation (Lemma~\ref{\en{EN}lemmm}) and $AC-1$ values in the $v$-summation (Lemma~\ref{\en{EN}lemctau}).
This yields:}%
\de{Wir werden jetzt beweisen, dass $g$ auch elliptisch ist. Hierfür nutzen wir, dass es $C^2-1$ Werte in der $u$-Summation (Lemma~\ref{\en{EN}lemmm}) und $AC-1$ Werte in der $v$-Summation (Lemma~\ref{\en{EN}lemctau}) gibt.
Daraus folgt nämlich:}
\begin{align*}
    &~~~~~~~g(z+\omega_k)-g(z)\\
    &= f(z+\omega_k)-f(z)-\left(1-\frac{A}{C}\right)\cdot \eta_k + \frac A C \cdot (C^2-1)\cdot \eta_k - (AC-1)\cdot \eta_k\\
    &= f(z+\omega_k)-f(z)+\eta_k\cdot\left(-1+\frac{A}{C}+ AC-\frac A C - AC + 1\right) = f(z+\omega_k)-f(z)
\end{align*}
\en{Since $f(z+\omega_k)=f(z)$ (cf.~Lemma~\ref{\en{EN}lemfellip}) we get $g(z+\omega_k)-g(z)=0$. Since all elliptic functions without poles are constant (cf.~first Liouville theorem, Prop.~\ref{\en{EN}liouville1}), the Lemma is proven.}%
\de{Aus Lemma~\ref{\en{EN}lemfellip} folgt $f(z+\omega_k)=f(z)$ und somit $g(z+\omega_k)-g(z)=0$.
Da alle elliptischen Funktionen ohne Pole konstant sind (erster Liouville'scher Satz~\ref{\en{EN}liouville1}), ist das Lemma bewiesen.}
\end{proof}

\begin{lem}\label{\en{EN}lemkapa2}
\en{In the notations of Def.~\ref{\en{EN}bemkonv} and with $\kappa$ from Def.~\ref{\en{EN}defkappa} it holds:}%
\de{In der Notation von Def.~\ref{\en{EN}bemkonv} und mit $\kappa$ aus Def.~\ref{\en{EN}defkappa} gilt:}
$$C\tau\kappa = - \sum_{v\in \DIV(C\tau)}\wp(v)$$
\end{lem}
\begin{proof}
\en{From Def.~\ref{\en{EN}defiwp} we get $\zeta'(z+w)=-\wp(z+w)$ and thus around $z=0$ it holds:
$\zeta(z+w)=\zeta(w)-\wp(w)\cdot z + O(z^2)$ (if $w\notin L$).}%
\de{Aus Def.~\ref{\en{EN}defiwp} folgt $\zeta'(z+w)=-\wp(z+w)$, somit gilt um $z=0$ (falls $w\notin L$):
$\zeta(z+w)=\zeta(w)-\wp(w)\cdot z + O(z^2)$.}
\en{Def.~\ref{\en{EN}defizeta} along with Prop.~\ref{\en{EN}sigmaodd} proves that $\zeta(z)$ is odd:
$\zeta(-z)=\frac{\sigma'(-z)}{\sigma(-z)}=\frac{\sigma'(z)}{-\sigma(z)}=-\zeta(z)$.
In Prop.~\ref{\en{EN}laurentwp}, we proved that the Laurent series of $\wp(z)$ at $z=0$ has no constant term.
From $\wp(z)=-\zeta'(z)$ we deduce that around $z=0$ it holds $\zeta(z)=\frac 1 z + O(z^3)$.}%
\de{Aus Def.~\ref{\en{EN}defizeta} folgt mit Satz~\ref{\en{EN}sigmaodd}, dass $\zeta(z)$ ungerade ist:
$\zeta(-z)=\frac{\sigma'(-z)}{\sigma(-z)}=\frac{\sigma'(z)}{-\sigma(z)}=-\zeta(z)$.
In Satz~\ref{\en{EN}laurentwp} haben wir bewiesen, dass der konstante Term der Laurentreihe von $\wp(z)$ bei $z=0$ verschwindet.
Mit $\wp(z)=-\zeta'(z)$ folgt dann, dass um $z=0$ gilt: $\zeta(z)=\frac 1 z + O(z^3)$}

\en{Thus we obtain around $z=0$:}\de{Somit folgt um $z=0$:}
\begin{align*}
    g(z) =& -\frac{A}{Cz} + \frac{C\tau}{C\tau z} + C \tau \kappa z - \left(1-\frac A C\right)\cdot \frac 1 z
    + \frac A C \cdot\sum_{u\in \DIV(C)}\left(\zeta(-u)-\wp(-u)\cdot z\right)\\
    &- \sum_{v\in \DIV(C\tau)} \left( \zeta(-v) - \wp(-v)\cdot z \right) + O(z^2)
\end{align*}
\en{Since $g(z)$ is constant (cf.~Lemma~\ref{\en{EN}lemkonst}), the coefficient of $z$ in this Laurent series expansion is zero:}%
\de{Da $g(z)$ konstant ist (vgl.~Lemma~\ref{\en{EN}lemkonst}), verschwindet der Koeffizient vor $z$ in dieser Laurentreihe:}
\begin{align*}
    0 &= C \tau \kappa - \frac A C \cdot\sum_{u\in \DIV(C)}\wp(u) + \sum_{v\in \DIV(C\tau)} \wp(v)
\end{align*}
\en{Thm.~\ref{\en{EN}theowpsum} tells that $\sum_{u\in \DIV(C)}\wp(u)=0$.
It remains $0 = C\tau\kappa + \sum_{v\in \DIV(C\tau)}\wp(v)$ which proves the Lemma.}%
\de{Thm.~\ref{\en{EN}theowpsum} besagt, dass $\sum_{u\in \DIV(C)}\wp(u)=0$.
Es verbleibt $0 = C\tau\kappa + \sum_{v\in \DIV(C\tau)}\wp(v)$, womit das Lemma bewiesen ist.}
\end{proof}

\begin{thm}\label{\en{EN}lemkapppppa}
\en{In the notations of Def.~\ref{\en{EN}bemkonv}, it holds:}%
\de{In der Notation von Def.~\ref{\en{EN}bemkonv} gilt:}
$$\sqrt{D}\cdot E_2^*(\tau)\cdot\frac{\pi^2}{3\omega_1^2} = \frac{\sum_{v\in \DIV(C\tau)}\wp(v)}{C\tau}$$
\end{thm}
\begin{proof}
\en{We have proved two representations of $\kappa$ in Lemma~\ref{\en{EN}lemkapa1} and Lemma~\ref{\en{EN}lemkapa2}. Equating these yields Prop.~\ref{\en{EN}lemkapppppa}.}%
\de{Wir haben in Lemma~\ref{\en{EN}lemkapa1} und~\ref{\en{EN}lemkapa2} zwei Darstellungen von $\kappa$ bewiesen. Diese gleichzusetzen liefert Satz~\ref{\en{EN}lemkapppppa}.}
\end{proof}

\begin{theo}[\en{Remainder of Prop.}\de{Rest des Satzes}~\ref{\en{EN}satzezweistern}]\label{\en{EN}theoezweisternrest}
\en{Let $\eta(\tau)$ denote the Dedekind $\eta$-function with $1728\eta^{24}=E_4^3-E_6^2$.
Then, in the notations of Def.~\ref{\en{EN}bemkonv},}%
\de{$\eta(\tau)$ bezeichne die Dedekind'sche $\eta$-Funktion mit $1728\eta^{24}=E_4^3-E_6^2$.
Dann gilt mit den Notationen aus Def.~\ref{\en{EN}bemkonv}, dass }%
\begin{align*}
    \sqrt{D}\cdot\frac{E_2^*(\tau)}{\eta^4(\tau)}\cdot(AC)^2
\end{align*}
\en{is an algebraic integer if it holds $C\tau^2+B\tau+A=0$.}%
\de{ganzalgebraisch ist, falls $C\tau^2+B\tau+A=0$ gilt.}
\end{theo}

\begin{proof}\label{\en{EN}proofzweistern}
\en{For $\tau\in CM$ we denote the lattice}%
\de{Für $\tau\in CM$ bezeichnen wir das Gitter}
$\hat L:=\frac{\pi}{\sqrt{3}}\cdot\eta(\tau)^2\cdot L_\tau$.

\en{Then we have $\hat L=a\cdot L_\tau$ with $a=\frac{\pi}{\sqrt{3}}\cdot\eta(\tau)^2$ and thus Prop.~\ref{\en{EN}trafog23} and Thm.~\ref{\en{EN}fouriertheorem} yield:}%
\de{Dann gilt $\hat L=a\cdot L_\tau$ mit $a=\frac{\pi}{\sqrt{3}}\cdot\eta(\tau)^2$.
Somit folgt aus Satz~\ref{\en{EN}trafog23} und Thm.~\ref{\en{EN}fouriertheorem}:}
\begin{align*}
    h_2 := \frac 1 4 g_2(\hat L) &= \frac 1 4 \cdot a^{-4} \cdot g_2(L_\tau)=\frac 1 4 \cdot \left(\frac {\sqrt{3}}\pi\right)^4\cdot\frac{\frac 4 3\cdot\pi^4\cdot E_4(\tau)}{\eta(\tau)^8} = 3\cdot \frac{E_4(\tau)}{\eta(\tau)^8}\\
    h_3 := \frac 1 4 g_3(\hat L) &= \frac 1 4 \cdot a^{-6} \cdot g_3(L_\tau)=\frac 1 4\cdot \left(\frac {\sqrt{3}}\pi\right)^6\cdot\frac{\frac 8 {27}\cdot\pi^6\cdot E_6(\tau)}{\eta(\tau)^{12}} = 2\cdot \frac{E_6(\tau)}{\eta(\tau)^{12}}
\end{align*}
\en{Since we already proved in Prop.~\ref{\en{EN}satzezweistern} on p.~\pageref{\en{EN}satzezweistern} that $\frac{E_4(\tau)}{\eta(\tau)^8}$ and $\frac{E_6(\tau)}{\eta(\tau)^{12}}$ are algebraic integers, both $h_2 := \frac 1 4 g_2(\hat L)$ and $h_3 := \frac 1 4 g_3(\hat L)$ are algebraic integers.}%
\de{In Satz~\ref{\en{EN}satzezweistern} auf S.~\pageref{\en{EN}satzezweistern} haben wir bereits bewiesen, dass $\frac{E_4(\tau)}{\eta(\tau)^8}$ und $\frac{E_6(\tau)}{\eta(\tau)^{12}}$ ganzalgebraisch sind, also sind sowohl $h_2 := \frac 1 4 g_2(\hat L)$ als auch $h_3 := \frac 1 4 g_3(\hat L)$ ganzalgebraisch.}

\en{In Prop.~\ref{\en{EN}lemkapppppa} we proved for any lattice $L$ with complex multiplication:}%
\de{In Satz~\ref{\en{EN}lemkapppppa} haben wir für alle Gitter $L$ mit komplexer Multiplikation bewiesen:}
$$\sqrt{D}\cdot E_2^*(\tau)\cdot\frac{\pi^2}{3\omega_1^2} = \frac{\sum_{v\in \DIV(C\tau)}\wp(v;L)}{C\tau}.$$
\en{For the lattice $\hat L$ we have $\omega_1 = \frac{\pi}{\sqrt{3}}\cdot\eta(\tau)^2$ and thus $\frac{\pi^2}{3\omega_1^2}= \frac{1}{\eta^4(\tau)}$.
This proves that}%
\de{Für das Gitter $\hat L$ gilt $\omega_1 = \frac{\pi}{\sqrt{3}}\cdot\eta(\tau)^2$ und somit $\frac{\pi^2}{3\omega_1^2}= \frac{1}{\eta^4(\tau)}$.
Hieraus folgt}
\begin{align*}
\sqrt{D}\cdot \frac{E_2^*(\tau)}{\eta^4(\tau)}
= \frac{\sum_{v\in \DIV(C\tau)}\wp(v;\hat L)\cdot C\bar\tau}{C\tau\cdot C\bar\tau}
= \frac{\sum_{v\in \DIV(C\tau)}\wp(v;\hat L)\cdot C\bar\tau}{AC}
\end{align*}
\en{where we expanded the fraction with $C\bar\tau$ to obtain $C\tau\cdot C\bar\tau=AC$ in the denominator.}%
\de{wobei wir mit $C\bar\tau$ erweitert haben, um den Nenner von $C\tau\cdot C\bar\tau=AC$ zu erreichen.}\linebreak
\en{If we multiply this by $(AC)^2$, we obtain}%
\de{Durch Multiplikation mit $(AC)^2$ erhalten wir}
\begin{align*}
\sqrt{D}\cdot \frac{E_2^*(\tau)}{\eta^4(\tau)}\cdot (AC)^2 &= \left(\sum_{v\in \DIV(C\tau)}AC\cdot\wp(v;\hat L)\right)\cdot C\bar\tau
\end{align*}
\en{After Lemma~\ref{\en{EN}lemctau} we know that $v\in \DIV(C\tau)$ implies $v\in \DIV(AC)$.
And above we proved that $h_2$ and $h_3$ are algebraic integers, thus by Thm.~\ref{\en{EN}theomwpu} with $m=AC$, the summands $AC\cdot\wp(v;\hat L)$ are algebraic integers for all $v\in \DIV(AC)$ and for all $v\in \DIV(C\tau)$.
Thus $AC\cdot \sum_{v\in \DIV(C\tau)}\wp(v;\hat L)$ is an algebraic integer.\\
Since $C\bar\tau$ is a solution of $x^2+Bx+AC=0$, it is also an algebraic integer.\linebreak
As the product of algebraic integers is an algebraic integer itself, Thm.~\ref{\en{EN}theoezweisternrest} is proven.}%
\de{Wegen Lemma~\ref{\en{EN}lemctau} wissen wir, dass alle $v\in \DIV(C\tau)$ auch in $\DIV(AC)$ liegen.
Außerdem haben wir soeben bewiesen, dass $h_2$ und $h_3$ ganzalgebraisch sind.
Aus Thm.~\ref{\en{EN}theomwpu} folgt nun (mit $m=AC$), dass die Summanden $AC\cdot\wp(v;\hat L)$ für alle $v\in \DIV(AC)$ und somit für alle $v\in \DIV(C\tau)$ ganzalgebraisch sind.

Folglich ist $AC\cdot \sum_{v\in \DIV(C\tau)}\wp(v;\hat L)$ ganzalgebraisch.
Da $C\bar\tau$ eine Lösung von $x^2+Bx+AC=0$ ist, ist es ebenfalls ganzalgebraisch und Thm.~\ref{\en{EN}theoezweisternrest} ist bewiesen.}
\end{proof}

\begin{bem}
\en{The expression $X:=\sqrt{D}\cdot\frac{E_2^*(\tau)}{\eta^4(\tau)}$ is invariant under the transformations $\tau\rightarrow -1/\tau$ and $\tau\rightarrow \tau+N$ (for $N\in \mathbb Z$), but the value of $(AC)^2$ changes.
For example $\tau\rightarrow\tau'=\tau+1$ transforms $\tau^2-\tau+41=0$ into $(\tau'-1)^2-(\tau'-1)+41=0$ or $\tau'^2-3\tau'+43=0$.
Thus Thm.~\ref{\en{EN}theoezweisternrest} tells for this $\tau$, that $X\cdot 41^2$ as well as $X\cdot 43^2$ are algebraic integers.}%
\de{Der Ausdruck $X:=\sqrt{D}\cdot\frac{E_2^*(\tau)}{\eta^4(\tau)}$ ist invariant unter den Transformationen $\tau\rightarrow -1/\tau$ und $\tau\rightarrow \tau+N$ (für $N\in \mathbb Z$). Allerdings ändert sich dabei der Wert von $(AC)^2$.
Auf diese Weise kann man z.B.~mit $\tau\rightarrow\tau'=\tau+1$ aus der Gleichung $\tau^2-\tau+41=0$ die Gleichung $(\tau'-1)^2-(\tau'-1)+41=0$ bzw.~$\tau'^2-3\tau'+43=0$ erzeugen.
Thm.~\ref{\en{EN}theoezweisternrest} besagt also für dieses $\tau$, dass sowohl $X\cdot 41^2$ als auch $X\cdot 43^2$ ganzalgebraisch ist.}

\en{Euclid's algorithm yields $u,v\in\mathbb Z$ with $u\cdot 41^2+v\cdot43^2 = \operatorname{gcd}(41^2;43^2)=1$.
This proves that $X$ itself is an algebraic integer, because}%
\de{Der Euklidische Algorithmus liefert $u,v\in\mathbb Z$ mit $u\cdot 41^2+v\cdot43^2 = \operatorname{ggT}(41^2;43^2)=1$.
Hieraus folgt, dass auch $X$ selbst ganzalgebraisch ist, denn}
$$X=(u\cdot 41^2+v\cdot43^2)\cdot X = u\cdot X\cdot 41^2 + v\cdot X\cdot 43^2$$

\en{By applying this method repeatedly, one can prove (using the factor $\frac 1 2$ from Thm.~\ref{\en{EN}theomwpu}) that for all $\tau\in CM$, the value $2\cdot X$ is an algebraic integer, without the factor $(AC)^2$.

And if $AC$ is odd or if $B$ is even, one can prove with this method that $X$ itself is an algebraic integer (which is stated without proof in \cite[Prop.~5.10.6]{\en{EN}CohenStromberg}).}%
\de{Durch wiederholtes Anwenden dieser Methode kann man (mit dem Faktor $\frac 1 2$ aus Thm.~\ref{\en{EN}theomwpu}) für alle $\tau\in CM$ beweisen, dass $2\cdot X$ ganzalgebraisch ist, ohne den Faktor $(AC)^2$.

Und falls $AC$ ungerade ist oder $B$ gerade ist, kann man hiermit sogar beweisen, dass $X$ ganzalgebraisch ist (das wird in \cite[Prop.~5.10.6]{\en{EN}CohenStromberg} formuliert, dort aber ohne Beweis).}
\end{bem}

\vfill\pagebreak}
{\fancyhead[OL]{\emph{\en{References}\de{Literatur}}}
\section*{\en{Acknowledgements}\de{Danksagungen}}
\label{\en{EN}literat}
\en{For their invaluable help in discussions, emails and on {\small\href{https://mathoverflow.net}{\nolinkurl{mathoverflow.net}}} I am grateful to:}%
\de{Für ihre wertvolle Hilfe in Diskussionen, Emails und auf {\small\href{https://mathoverflow.net}{\nolinkurl{mathoverflow.net}}} danke ich:}

David \en{and}\de{und} Gregory Chudnovsky, 
Gregor Milla, 
Zavosh Amir Khosravi, 
Henri Cohen, 
Loïc Dreher, 
David Masser, 
Rolf Busam, 
François Brunault 
\en{and}\de{und}
Michael Griffin. 

\vfill
\bibliographystyle{plain}
{\raggedright

}

}
  \vfill\pagebreak
  
  \enfalse
  \invisiblepart{Deutsche Version}

\end{document}